%% file: operads_paper-arxiv.tex
\documentclass{article} 
\usepackage{latexsym}
\usepackage{amsfonts}
\usepackage[bbgreekl]{mathbbol}
\usepackage{amssymb}
\usepackage{amsmath}
\usepackage{amsthm}
\usepackage{thmtools}
\usepackage{amscd}
\usepackage{longtable}
\usepackage{array}
\usepackage{graphicx}
\usepackage{eucal}
\usepackage{pstricks}
\usepackage{pst-grad}
\usepackage{pst-arrow}
\usepackage{pstricks-add}
\usepackage{pst-node}
\usepackage{pst-plot}
\definecolor{antiquewhite}{rgb}{0.98, 0.92, 0.84}
\definecolor{myblue}{HTML}{15399A}
\definecolor{myred}{HTML}{B31333}
\definecolor{mygreen}{HTML}{118A4F}
\definecolor{myorange}{HTML}{DE6C06}
\definecolor{myyellow}{HTML}{E3AD30}
\definecolor{myblack}{HTML}{000000}
\definecolor{cblue}{HTML}{15399A}
\definecolor{corange}{HTML}{DE6C06}
\definecolor{myviolet}{HTML}{8601AF}
\usepackage[
style=numeric,
backend=biber,
citestyle=numeric,
doi=false,
isbn=false,
url=false,
maxcitenames=4, mincitenames=4,
maxbibnames=99, minbibnames=99
]{biblatex}
\usepackage{geometry}
\usepackage{multicol}
\usepackage{etoolbox}
\usepackage[all]{xy}
\SelectTips{cm}{11}
\UseTips
\usepackage{tikz}
\usetikzlibrary{decorations.markings,decorations.pathmorphing, arrows, calc}
\usepackage{tikz-cd}
\tikzset{mm/.style={execute at begin node=$\displaystyle, execute at end node=$}}
\tikzset{curve/.style={settings={#1},to path={(\tikztostart)
    .. controls ($(\tikztostart)!\pv{pos}!(\tikztotarget)!\pv{height}!270:(\tikztotarget)$)
    and ($(\tikztostart)!1-\pv{pos}!(\tikztotarget)!\pv{height}!270:(\tikztotarget)$)
    .. (\tikztotarget)\tikztonodes}},
    settings/.code={\tikzset{quiver/.cd,#1}
        \def\pv##1{\pgfkeysvalueof{/tikz/quiver/##1}}},
    quiver/.cd,pos/.initial=0.35,height/.initial=0}
\tikzset{tikzob/.style={commutative diagrams/every diagram, every cell}}
\tikzset{tikzar/.style={commutative diagrams/.cd, every arrow, every label, font={\small}}}
\tikzset{tikzsquiggle/.style={decorate, decoration={
    snake,
    segment length=8pt,
    amplitude=.9pt,post=lineto,
    post length=2pt}}}
\tikzset{cross line/.style={preaction={draw=white, -, line width=6pt}}}
\usepackage{pdfsync}
\usepackage{relsize}
\usepackage[unicode=true, pdfusetitle,
 bookmarks=true,bookmarksnumbered=false,
 breaklinks=false,
 backref=false,
 colorlinks=true,
 citecolor=black,
 urlcolor=blue!78!red,
 final
]{hyperref}
\usepackage[capitalise]{cleveref}

\newcommand{\newrefformat}[2]{}

\newcommand{\mb}{\mathbf}

\renewcommand{\hat}{\widehat}

\newcommand{\N}{\mathbb{N}}
\newcommand{\m}[1]{\mathcal{#1}}
\newcommand{\id}{\textrm{id}}
\newcommand{\un}{\underline}
\renewcommand{\SS}{\mathcal{S}}

\newcommand{\EL}{E\Lambda}
\newcommand{\ELn}{E\Lambda(\underline{n})}

\newcommand{\lmc}{\Lambda\mbox{-}\mb{MonCat}}
\newcommand{\wlmc}{\textrm{Wk}\mbox{-}\lmc}

\newcommand{\cat}{\ensuremath{\mb{Cat}}}

\newcolumntype{L}{>{$}l<{$}}
\newcolumntype{C}{>{$}c<{$}}
\newcolumntype{R}{>{$}r<{$}}

\newcommand{\cn}{\colon}

\newcommand{\trans}[2]{( #1 \, \, #2 )}
\newcommand{\coeq}[4]{#1(#4) \otimes_{#3(#4)} #2^{#4}}
\newcommand{\coequ}[4]{#1(#4) \otimes_{#3(#4)} #2_{#4}}
\newcommand{\coeqb}[3]{#1 \otimes_{#3} #2}
\newcommand{\coeqsig}[3]{#1(#3) \otimes_{\Sigma_#3} #2^#3}
\newcommand{\lcoll}{\Lambda\mbox{-}\mb{Coll}}

\crefname{lem}{Lemma}{Lemmas}
\crefname{thm}{Theorem}{Theorems}
\crefname{Defi}{Definition}{Definitions}
\crefname{nota}{Notation}{Notations}
\crefname{construction}{Construction}{Constructions}
\crefname{prop}{Proposition}{Propositions}
\crefname{rem}{Remark}{Remarks}
\crefname{remark}{Remark}{Remarks}
\crefname{cor}{Corollary}{Corollaries}
\crefname{scholium}{Scholium}{Scholia}
\crefname{figure}{Figure}{Figures}
\crefname{equation}{Equation}{Equations}
\crefname{eq}{Equation}{Equations}
\crefname{eqn}{Equation}{Equations}
\crefname{part}{Part}{Parts}
\crefname{example}{Example}{Examples}
\crefname{nonex}{Non-example}{Non-examples}
\crefname{term}{Terminology}{Terminologies}
\crefname{ax}{Axiom}{Axioms}

\newenvironment{eqn}{\begin{equation}}{\end{equation}}
\newenvironment{eq*}{\begin{equation*}}{\end{equation*}}
\newenvironment{eqn*}{\begin{equation*}}{\end{equation*}}
\tikzset{->-/.style={decoration={markings,mark=at position #1 with {\arrow{>}}},postaction={decorate}}} 

\pgfdeclarelayer{bg}
\pgfsetlayers{bg,main}

\newtheorem{thm}{Theorem}[section]
\newtheorem{prop}[thm]{Proposition}
\newtheorem{lem}[thm]{Lemma}
\newtheorem{cor}[thm]{Corollary}

\newtheoremstyle{example}{\topsep}{\topsep}%
     {}
     {}
     {\bfseries}
     {.}
     {2pt}
     {\thmname{#1}\thmnumber{ #2}\thmnote{ #3}}
  \theoremstyle{example}
  \newtheorem{conv}[thm]{Convention}
  \newtheorem{nota}[thm]{Notation}
  \newtheorem{example}[thm]{Example}
  \newtheorem{nonex}[thm]{Non-example}
  \newtheorem{Defi}[thm]{Definition}
  \newtheorem{rem}[thm]{Remark}
  \newtheorem{remark}[thm]{Remark}   
  \newtheorem{term}[thm]{Terminology}   
    
\usepackage[
color=orange!80, 
bordercolor=black,
textwidth=3cm,
textsize=small
,colorinlistoftodos]
{todonotes}

   \title{Operads and equivariance}

\author{
  Alexander S. Corner\\
  \texttt{alex.corner@shu.ac.uk}
  \and
  Nick Gurski\\
  \texttt{nick.gurski@case.edu}
}
\date{}

\newcounter{artpart}
\renewcommand{\theartpart}{\Roman{artpart}}
\newcommand{\artpart}[1]{%
  \refstepcounter{artpart}%
  \addtocontents{toc}{\smallskip}%
   \addtocontents{toc}{
    \protect\contentsline{section}{
       \protect\numberline{}{\large Part \theartpart: #1}
     }{}{}
  }
  \section*{Part \theartpart: #1}}

\begin{document}

\maketitle

\abstract{Operads were originally defined by May to have right actions of the symmetric groups, but later formulations have also used no groups actions at all or group actions by such families as the braid groups. We call such families  action operads, as they are the algebraic objects that encode parametrized group actions on operads. In Part I of this paper, we study the basic algebra of action operads $\Lambda$ and the $\Lambda$-operads they act upon. In Part II, we study $\Lambda$-operads in the 2-category of small categories.} 

\tableofcontents

\input{v2p1}
\input{v2p2}

\input{v2p3}




\printbibliography
\end{document}

%% file: v2p1.tex
\section{Introduction}

Operads are by now an established tool in fields ranging from algebra \cite{loday-vallette} and combinatorics \cite{mendez} to homotopy theory \cite{fie-br} and mathematical physics \cite{mss-op}.
First explicitly defined by May in \cite{maygeom}, following work of Whitehead \cite{whitehead}, Stasheff \cite{stasheff}, and Boardman-Vogt \cite{bv}, operads are one of many formalisms\footnote{Others include monads \cite{et-monads}, algebraic or Lawvere theories \cite{lawvere-thesis}, and PROPs \cite{mac_prop}.} used to give a presentation for different types of algebraic structures.
An operad $P$ consists of objects $P(n)$, indexed by natural numbers, together with composition operations
\[
\mu \colon P(n) \times P(k_1) \times \cdots P(k_n) \to P(k_1 + \cdots + k_n)
\]
and a unit element $\id \in P(1)$.
The prototypical operad is the endomorphism operad $\mathcal{E}_X$ on a set $X$ (see \cref{ex:endo}): the set $\mathcal{E}_X(n)$ consists of all functions $f \colon X^n \to X$, the composition operation is $n$-ary composition of functions
\begin{align*}
\mu(g; f_1, \ldots, f_n)\big(x_{1,1}, \ldots, x_{1, k_1}, x_{2,1}, \ldots, x_{n, k_n}\big) \qquad \\ \qquad = g \Big( f_1(x_{1,1}, \ldots, x_{1,k_1}), \ldots, f_n(x_{n,1}, \ldots, x_{n, k_n}) \Big),
\end{align*}
and the identity element is the identity function $1_X \colon X \to X$; the operad axioms were chosen to reflect the associativity and unit relations satisfied by the above structure.

May's definition also included symmetric group actions $P(n) \times \Sigma_n \to P(n)$, modeling the operation of permuting input variables; these operads are now often called \emph{symmetric operads} when the equivariance is being stressed.
Fiedorowicz later generalized these symmetric group actions to braid group actions in \cite{fie-br}, where he studied operations on double loop spaces. 
Wahl \cite{wahl-thesis} generalized these ideas further to ribbon braids and other families of groups in order to study spectra arising from stable mapping class groups.
Each of these authors leveraged different group actions on the operads involved in order to produce additional structure on the algebras over those operads.
This paper, an updated version of the preprints \cite{cg-preprint,g-borel}, seeks to give a treatment of the fundamental algebraic structure present in the families of groups $\Lambda = \{ \Lambda(n) \}_{n \in \mathbb{N}}$ --- such as the symmetric groups, braid groups, and ribbon braid groups --- arising in operad theory, here called \emph{action operads}.

\subsection*{Action operads}

Suppose that $\{ \Lambda(n) \}$ is a $\mathbb{N}$-indexed family of groups, acting on an operad $P$ via $P(n) \times \Lambda(n) \to P(n)$.
It is natural to impose some compatiblity between these group actions and operadic composition, and this compatibility requires that the groups $\Lambda(n)$ themselves form an operad.
We see from the example of the symmetric groups $\Sigma_n$ and May's symmetric operads that the operad consisting of the $\Lambda(n)$'s is not required to be an operad internal to the category of groups: operadic composition is not a group homomorphism, but rather satisfies a twisted version of the homomorphism axiom.
These considerations lead to the definition of an action operad in \cref{Defi:aop}, consisting of 
\begin{itemize}
\item an operad $\Lambda$ of sets, with each $\Lambda(n)$ equipped with a group structure; and
\item functions $\pi_n \colon \Lambda(n) \to \Sigma_n$ that are simultaneously group homomorphisms and assemble into a map of operads $\Lambda \to \Sigma$;
\end{itemize}
satisfying one additional axiom relating the operad and group structures.

As the work of May, Fiedorowicz, and Wahl shows, varying the group actions on an operad $P$ can lead to different structure on the algebras over $P$.
A given operad $P$ may have many different action operads $\Lambda$ acting upon it, with the simplest example being that any symmetric operad is also an operad in the non-symmetric sense, meaning that we can instead consider the group actions $P(n) \times T(n) \to P(n)$ where each $T(n)$ is the trivial group.
Many of our later results, particularly in \cref{sec:cart,sec:pscomm}, take into account both the operad $P$ and the action operad $\Lambda$ acting upon it.
Much of the operads literature has focused on the symmetric case, both because of its historical significance and its seeming universality.
This universality is best explained by the existence of a symmetrization functor (\cref{Defi:symmetrization}) that preserves algebra structures (\cref{cor:pi-star,cor:sym-preserve-alg}).
Unfortunately, symmetrization does not preserve all properties that one might desire on an operad, and \cref{rem:symm-and-contract} demonstrates this by examining how symmetrization interacts with contractability in the context of braided operads.
We take this example as evidence that one should be flexible with the equivariance conditions imposed on an operad, and choose the combination of an action operad $\Lambda$ acting on an operad $P$ that best balances whatever properties are desired.
To that end, we thoroughly investigate the basic algebra of action operads from both the elementary and categorical perspectives.

\subsection*{What has changed?}

As noted above, this paper synthesizes previous work of the authors \cite{cg-preprint,g-borel} that was, until now, available only in preprint form.
Reading those preprints in chronological order will differ significantly from reading the present paper: the earlier \cite{cg-preprint} contained most of the material from Part II of this paper, while the later \cite{g-borel} contained most of the material that now comprises Part I.
In each of these Parts, we have refined, expanded, reorganized, and sometimes removed material from those preprints, motivated by questions or comments from other researchers or, in one case, a later application of our work.
These references to our preprints was one of the motivating factors for finally publishing this material.
In particular, Yau \cite{yau_infinity_2021} has generalized much of our theory to an infinity-categorical context.
Our work on pseudo-commutative structures was also used in \cite{guillou_symmetric,guillou_multiplicative} in the construction of equivariant multiplicative $K$-theory machines.
Those authors identified a missing axiom as detailed in \cref{rem:updated}, the omission of which resulted in the error described in \cref{rem:EB-fail}; we have since removed the erroneous results.

\subsection*{Overview and main results}

The remainder of this preamble contains a short section of notation and conventions that we use throughout.

Part I studies action operads in generality. 
\begin{itemize}
\item \cref{sec:back-op} reviews the basic theory of operads, in their symmetric, plain (or non-symmetric), and braided variants.
\item \cref{sec:aop} provides the definition of an action operad, as well as some basic examples. Much of this section is devoted to the elementary algebra of action operads, and \cref{thm:charAOp} characterizes action operads in terms of block sum and duplication operations instead of operadic composition. 
\item \cref{sec:aop-examples} contains many of our key examples of action operads, together with a host of non-examples. 
\item The short \cref{sec:extension} situates action operads in relation to operads in the category $\mb{Grp}$ of groups; we construct kernels and images of maps of actions operads, and show in \cref{cor:extension} that every action operad is either an operad in $\mb{Grp}$ or an extension of the symmetric groups by an operad in $\mb{Grp}$. 
\item \cref{sec:pres-aop} is devoted to showing that the category of action operads is locally finitely presentable, and using that structure to establish a free action operad functor. We use the resulting adjunction between action operads and collections to define presentations for action operads, and we explicitly compute a presentation for the symmetric groups as an action operad. 
\item \cref{sec:forward-op} gives the definition of a $\Lambda$-operad for $\Lambda$ an action operad, as well as algebras over a $\Lambda$-operad. The main technical result of this section is a change-of-action-operad adjunction $f_{!} \dashv f^*$ associated to any map of action operads $f \colon \Lambda \to \Lambda'$, and we use this to define and then study the symmetrization functor. 
\item In \cref{sec:sub} we express $\Lambda$-operads as monoids under a substitution product (\cref{thm:operad=monoid}) on the category of $\Lambda$-collections. We do so in a coend-forward fashion rather than by elementary calculation, leveraging the internal hom-functor that is adjoint to the substitution product.
\end{itemize}

Part II specializes the theory to operads in the category or 2-category $\mb{Cat}$. 
\begin{itemize}
\item We begin with two background sections: \cref{sec:backgr-cat} covers the constructions we will need related to group actions and limits/2-limits, while 
\cref{sec:2monads} recalls the basic theory of 2-monads.
\item \cref{sec:op-to-2monad} interprets the algebras over a $\Lambda$-operad in $\mb{Cat}$ from the 2-monadic point of view. The most important construction in this section is that of the operad $\EL$ in \cref{cor:elambda_lambdaop}, and the subsequent definition of $\Lambda$-monoidal categories in \cref{Defi:lmc}. We finish this section by computing the free $\Lambda$-monoidal category on a category $X$ in \cref{prop:hom-set-lemma}.
\item \cref{sec:coherence} addresses two kinds of coherence theorems. The first is the abstract coherence theorem for the 2-monad induced by $\EL$, and is presented in \cref{cor:coherence-for-ulP}. The second is a standard strictification-type theorem, and it appears as \cref{thm:wlmc-to-lmc}.
\item \cref{sec:cart} studies the interaction between the group actions $P(n) \times \Lambda(n) \to P(n)$ and whether the 2-monad induced by $P$ is 2-cartesian. It is well-known in the categorical literature that every non-symmetric operad on $\mb{Sets}$ induces a cartesian monad, but that not every symmetric operad does. We characterize those symmetric operads in $\mb{Cat}$ that induce 2-cartesian 2-monads in \cref{cor:cart_cor} quite simply: they are the operads for which all the symmetric group actions are free. We extend this to other $\Lambda$ in \cref{thm:cart_thm}, and in particular show that the 2-monads induced by $\EL$ for any action operad $\Lambda$ are 2-cartesian in \cref{cor:EL-2cart}.
\item \cref{sec:club} begins by reviewing Kelly's notion of club, and then goes on to show that every action operad induces a club. In \cref{thm:club=operad} we characterize which clubs arise from action operads, and then compare Kelly's presentations of clubs with our presentations of action operads in \cref{thm:pres1}. This theorem shows how the form of the definition for a type of monoidal structure (eg, symmetric monoidal or braided monoidal) can be easily translated into a presentation for the action operad $\Lambda$ such that the type of monoidal categories in question are precisely the $\EL$-algebras.
\item \cref{sec:exex-cactus} gives an example of the theory in the previous section. We explain how the calculations with coboundary categories in \cite{hk-cobound} can be interpretted as computing a presentation for the action operad of cactus groups.
\item \cref{sec:pscomm} studies when a $\Lambda$-operad $P$ admits a pseudo-commutative structure. This additional structure equips the 2-category of algebras over $P$ with an additional tensor product and internal hom, making it closed monoidal in a 2-dimensional sense. We define a pseudo-commutative structure for a $\Lambda$-operad $P$ in \cref{Defi:ps-comm_operad}, and prove in \cref{thm:pscomm} that it induces a pseudo-commutative structure on the induced 2-monad. As a consequence, every contractible symmetric operad obtains a symmetric pseudo-commutative structure (\cref{cor:contractplussym-to-psc}).
\end{itemize}

\subsection*{Acknowledgements}
This research was supported by EPSRC Early Career Fellowship 134023 and LMS Research Reboot 42031.
The authors would also like to thank Nathaniel Arkor,  Daniel Graves, and Ang\'elica Osorno.

\section{Notation and Conventions}\label{sec:notation}

\begin{nota}[(Symmetric groups)]\label{nota:symm_sigma}
We denote the symmetric group on the symbols $1, 2, \ldots, n$ by $\Sigma_n$. Elements of a symmetric group are usually denoted by lowercase Greek letters or written in cycle notation.
\end{nota}

\begin{nota}[(Braid groups)]\label[nota]{nota:braid}
We denote the braid group on $n$ strands by $B_n$.
\end{nota}

\begin{nota}[(Identity elements)]\label[nota]{nota:e_identity}
The symbol $e$ will generically represent an identity element in a group. If we are considering a set of groups $\{ \Lambda(n) \}_{n \in \N}$ indexed by the natural numbers, then $e_{n}$ is the identity element in $\Lambda(n)$. We will often drop the subscripts and just write $e$ when the index $n$ in $\Lambda(n)$ is either clear from context or unimportant to the argument at hand.
\end{nota}

\begin{conv}[(Identity morphisms)]\label[conv]{conv:1_id}
We generically write an identity morphism $A \to A$ as either $1$ or $1_A$.
\end{conv}

\begin{nota}[(Group actions)]\label[nota]{nota:g-action}
For a group $G$, a right $G$-action on a set $X$ will be denoted $(x,g) \mapsto x \cdot g$ or $(x,g) \mapsto x g$. 
Similar notation will be used for left actions, and for multiplication in a group.
\end{nota}

\begin{conv}[(Indexed objects)]\label[conv]{conv:indexed}
We generically write $\{ \Lambda(n) \}_{n \in \N}$ for a $\N$-indexed family of objects $\Lambda(n)$. We will occasionally write $\Lambda_n$ in place of $\Lambda(n)$, especially in diagrams or when the objects $\Lambda_n$ have been independently defined, as in \cref{nota:symm_sigma,nota:braid}.
\end{conv}

\begin{conv}[(Products and quotients)]\label[conv]{conv:coeq}
We will often be interested in elements of a product of the form
$
A \times B(1) \times \cdots \times B(n) \times C
$ (or similar, for example without $A$ or $C$). We write elements of this set as $(a; b_{1}, \ldots, b_{n}; c)$, where $b_i \in B(i)$. 
In the case that we need to consider equivalence classes of such elements, these classes will be written as $[a; b_{1}, \ldots, b_{n}; c]$. 
The most common situation in which we consider such equivalence classes is that of a coequalizer of left and right group actions in the following sense. 
A coequalizer of maps
    \[
        \xy
            (0,0)*+{A \times G \times B}="00";
            (30,0)*+{A \times B}="10";
            (60,0)*+{\coeqb{A}{B}{G}}="20";
            {\ar@<1ex>^{1 \times \lambda} "00" ; "10"};
            {\ar@<-1ex>_{\rho \times 1} "00" ; "10"};
            {\ar^{\varepsilon} "10" ; "20"};
        \endxy
    \]
will be written as $\coeqb{A}{B}{G}$, where $\rho$ is a right action of $G$ on $A$ and $\lambda$ is a left action of $G$ on $B$. 
This notation is meant to emphasize the analogy with tensor products of $R$-modules, even when the monoidal structure involved is cartesian.
It also differentiates these coequalizers from pullbacks.
\end{conv}

\begin{conv}[(Tilde for maps respecting equivariance)]\label[conv]{conv:equiv-maps}
Suppose that $\coeqb{A}{B}{G}$ is a coequalizer as in \cref{conv:coeq}.
By definition, maps $f \colon \coeqb{A}{B}{G} \to X$ are in bijection with maps $A \times B \to X$ that coequalize $1 \times \lambda$ and $\rho \times 1$.
Given such a map $f$, we will always denote the corresponding map $A \times B \to X$ as $\tilde{f}$.
\end{conv}

\begin{conv}[(Pullbacks)]\label[conv]{conv:pb}
The pullback of the diagram
    \[
        \xy
            (10,0)*+{X}="00";
            (0,-10)*+{Y}="10";
            (10,-10)*+{A}="20";
            {\ar^{f} "00" ; "20"};
            {\ar_{g} "10" ; "20"};
        \endxy
    \]
will be written as $X \times_A Y$. 
\end{conv}

\begin{Defi}[(Underlying permutation)]\label[Defi]{Defi:underlying-perm}
Suppose that $f \colon G \to \Sigma_n$ is a given group homomorphism, and $x \in G$. The \emph{underlying permutation} of $x$ is the element $f(x) \in \Sigma_n$. If there is likely to be some confusion as to which homomorphism $f$ is being used, we will call $f(x)$ the \emph{underlying permutation with respect to $f$}.
\end{Defi}

\begin{nota}[(Applying underlying permutations)]\label[nota]{nota:perm_shorthand}
Throughout we will be using maps $\pi_n \colon O(n) \rightarrow \Sigma_n$, where $O(n)$ is the set of $n$-ary operations of an operad $O$ and $\Sigma_n$ is the symmetric group on $n$ elements.  
For any $\sigma \in O(n)$, we will write $\sigma(i)$ for $\pi_n(\sigma)(i)$, the image of $i$ with respect to the underlying permutation of $\sigma$; the notation $\sigma^{-1}(i)$ will be used for the inverse image of $i$ with respect to the underlying permutation of $\sigma$.
\end{nota}

\begin{rem}[(Left action of symmetric groups on tuples)]\label{rem:Sn-tuples}
The most common group action we will encounter is the left action of the symmetric group $\Sigma_n$ on a set of the form $X^n$.
We write this action as $\sigma \cdot (x_1, \ldots, x_n)$, and emphasize that it is given by the formula
\[
\sigma \cdot (x_1, \ldots, x_n) = (x_{\sigma^{-1}(1)}, \ldots, x_{\sigma^{-1}(n)}).
\]
\end{rem}

\begin{Defi}[(Block sum)]\label[Defi]{Defi:beta-s}
Let $k_1, \ldots, k_n$ be natural numbers and suppose that $\sigma_i \in \Sigma_{k_i}$ are permutations. The \emph{block sum} of $\sigma_1, \ldots, \sigma_n$, written 
\[
\beta( \sigma_1, \ldots, \sigma_n ),
\]
is the permutation in $\Sigma_K$, where $K = \sum_{i=1}^n k_i$, given as described below. 
For $1 \leq j \leq K$, define $c$ to be the unique integer such that
\[
k_1 + \cdots + k_c < j \leq k_1 + \cdots + k_c +k_{c+1}.
\]
Define
\[
\beta( \sigma_1, \ldots, \sigma_n )(j) = k_1 + \cdots + k_c + \sigma_{c+1}\big( j - (\sum_{i=1}^c k_i) \big).
\]
\end{Defi}

\begin{rem}\label{rem:beta-s}
The formula above expresses the idea that $\beta( \sigma_1, \ldots, \sigma_n )$ permutes the first $k_1$ elements using $\sigma_1$, the next $k_2$ elements using $\sigma_2$, and so on.
\end{rem}

\begin{Defi}[(Duplication)]\label[Defi]{Defi:delta-s}
Let $k_1, \ldots, k_n$ be natural numbers, and suppose that $\sigma \in \Sigma_n$ is a permutation. The \emph{duplication} of $\sigma$ with respect to $k_1, \ldots, k_n$, written
\[
\delta_{n; k_1, \ldots, k_n}(\sigma),
\]
is the permutation in $\Sigma_K$, where $K = \sum_{i=1}^n k_i$, given as described below. 
For $1 \leq j \leq K$, define $c$ to be the unique integer such that
\[
k_1 + \cdots + k_c < j \leq k_1 + \cdots + k_c +k_{c+1}.
\]
Define
\[
\delta_{n; k_1, \ldots, k_n}(\sigma)(j) = \big( \sum_{\sigma(k_i) < \sigma(k_{c+1})} k_i\big) + j - \big( \sum_{i=1}^c k_i \big).
\]
\end{Defi}

\begin{rem}\label{rem:delta-s}
The formula above for $\delta_{n; k_1, \ldots, k_n}(\sigma)$ is best explained by drawing the graph of $\sigma$ as follows. The function $\sigma$ can be represented by drawing two rows of $n$ dots each, and connecting dot $i$ in the top row to dot $\sigma(i)$ in the bottom row. Then $\delta_{n; k_1, \ldots, k_n}(\sigma)$ is obtained by 
\begin{itemize}
\item replacing dot $i$ in the top row with $k_i$ dots,
\item replacing dot $\sigma(i)$ in the bottom row with $k_i$ dots, and
\item connecting these two sets of $k_i$ dots in the unique way that preserves order.
\end{itemize}
Thus we see that the $i$th entry for $\sigma$ is duplicated $k_i$ times in $\delta_{n; k_1, \ldots, k_n}(\sigma)$.
\end{rem}

\begin{remark}\label{rem:perm_matrices}
Permutations, as elements of $\Sigma_n$, can be considered as permutation \emph{matrices} with exactly one $1$ in each row and column. E.g., the permutation $(1 \,\, 3 \,\, 2) \in \Sigma_3$ can be considered as a matrix which permutes three elements $\begin{bmatrix} a & b & c \end{bmatrix}$ upon pre-multiplication:
  \[
  \begin{bmatrix}
  0 & 1 & 0 \\
  0 & 0 & 1 \\
  1 & 0 & 0
  \end{bmatrix}
  \begin{bmatrix}
  a \\ b \\ c
  \end{bmatrix}
  =
  \begin{bmatrix}
  b \\ c \\ a
  \end{bmatrix}.
  \]
Then the block sum $\beta$ (\cref{Defi:beta-s}) corresponds to the process of taking the block diagonal matrix of the original permutation matrices. So given elements $\trans{1}{2} \in \Sigma_2$, $e_1 \in \Sigma_1$, and $(1 \,\, 2 \,\, 3) \in \Sigma_3$, then
  \[
  \beta(\trans{1}{2},e_1,(1 \,\, 2 \,\, 3)) =
  \begin{bmatrix}
  \begin{bmatrix}
  0 & 1 \\
  1 & 0
  \end{bmatrix} & 0 & 0 \\
  0 & \begin{bmatrix} 1 \end{bmatrix} & 0 \\
  0 & 0 &   \begin{bmatrix}
  0 & 0 & 1 \\
  1 & 0 & 0 \\
  0 & 1 & 0
  \end{bmatrix}
  \end{bmatrix}
  =
  \begin{bmatrix}
  0 & 1 & 0 & 0 & 0 & 0 \\
  1 & 0 & 0 & 0 & 0 & 0 \\
  0 & 0 & 1 & 0 & 0 & 0 \\
  0 & 0 & 0 & 0 & 0 & 1 \\
  0 & 0 & 0 & 1 & 0 & 0 \\
  0 & 0 & 0 & 0 & 1 & 0 \\
  \end{bmatrix}
  \]
and is the permutation $\trans{1}{2}(3)(4 \,\, 5 \,\, 6)$ in cycle notation.

Similarly, we can describe the duplication $\delta$ (\cref{Defi:delta-s}) as an operation on permutation matrices. For $\sigma \in \Sigma_n$, $\delta_{n;k_1,\ldots,k_n}(\sigma)$ takes a block diagonal of identity matrices $I_{k_1}$, $\ldots$, $I_{k_n}$ (which corresponds to $\beta(e_{k_1},\ldots,e_{k_n}) \in \Sigma_{k_1+\ldots+k_n}$), and permutes these according to the effect of the permutation $\sigma$. For example, given $\sigma = (1 \,\, 2 \,\, 3)$, then
  \[
    \delta_{3;2,1,3}(\sigma) =
    \sigma
    \ast
    \begin{bmatrix}
    I_2 & 0 & 0 \\
    0 & I_1 & 0 \\
    0 & 0 & I_3
    \end{bmatrix}
    =
      \begin{bmatrix}
      0 & 0 & 1 \\
      1 & 0 & 0 \\
      0 & 1 & 0
      \end{bmatrix}
    \ast
    \begin{bmatrix}
    I_2 & 0 & 0 \\
    0 & I_1 & 0 \\
    0 & 0 & I_3
    \end{bmatrix}
    =
    \begin{bmatrix}
    0 & 0 & I_3 \\
    I_2 & 0 & 0 \\
    0 & I_1 & 0
    \end{bmatrix}.
  \]
We make use of a similar interpretation of signed permutations and block diagonal matrices in a counterexample given in \cref{nonex:counterex2}.
\end{remark}

\begin{conv}[(Superscripts)]\label[conv]{conv:superscripts}
We generically use superscripts, when needed, to distinguish between operations of the same type associated to different structures. As an example, a monoid homomorphism $f \colon A \to B$ would have axioms written as
\begin{align*}
f(x \cdot^A y) & = f(x) \cdot^B f(y), \\
f(1^A) & = 1^B.
\end{align*}
\end{conv}

\artpart{Operads and Action Operads}

\section{Background: Operads}\label{sec:back-op}

This section will collect the basic background information on operads that we will later generalize in
\cref{sec:forward-op}. We begin with the most common type of operad, a symmetric operad, before defining two more types of operad: plain and braided. 

\begin{Defi}[(Symmetric operad)]\label[Defi]{Defi:sym-op}
A \textit{symmetric operad} $O$ (in the category of sets) consists of
\begin{itemize}
\item a set, $O(n)$, for each natural number $n$,
\item for each $n$, a right $\Sigma_{n}$-action on $O(n)$,
\item an element $\id \in O(1)$, and
\item functions
  \[
    \mu \colon  O(n) \times O(k_{1}) \times \cdots \times O(k_{n}) \rightarrow O(k_{1} + \cdots + k_{n}),
  \]
\end{itemize}
satisfying the following three axioms.
\begin{enumerate}
\item The element $\id \in O(1)$ is a two-sided unit for $\mu$, meaning that
  \begin{align*}
    \mu(\id;x) &= x,\\
    \mu(x;\id,\ldots,\id) &= x
  \end{align*}
for any $x \in O(n)$.
\item The functions $\mu$ are associative, meaning that the diagram below commutes.
 \[
    \xy
      (0,0)*+{\scriptstyle O(n) \times \left(\prod_{i=1}^n O(k_i)\right) \times \left(\prod_{i=1}^n\prod_{j=1}^{k_i} O(l_{i,j})\right)}="a";
      (65,0)*+{\scriptstyle O(n) \times \prod_{i=1}^n \left(O(k_i) \times \prod_{j=1}^{k_i} O(l_{i,j})\right)}="b";
      (65,-20)*+{\scriptstyle O(n) \times \prod_{i=1}^n O\left(\sum_{j=1}^{k_i} l_{i,j}\right)}="c";
      (65,-40)*+{\scriptstyle O\left(\sum_{i=1}^n \sum_{j=1}^{k_i} l_{i,j}\right)}="d";
      (0,-40)*+{\scriptstyle O\left(\sum_{i=1}^n k_i\right) \times \prod_{i=1}^n \prod_{j=1}^{k_i} O(l_{i,j})}="e";
      {\ar^{\cong} "a" ; "b"};
      {\ar^{1 \times \prod \mu} "b" ; "c"};
      {\ar^{\mu} "c" ; "d"};
      {\ar_{\mu \times 1} "a" ; "e"};
      {\ar_{\mu} "e" ; "d"};
    \endxy
  \]

\item The functions $\mu$ are equivariant with respect to the symmetric group actions, meaning that two equations hold.
\begin{itemize}
\item[3.1] Suppose that $x \in O(n)$, $y_i \in O(k_i)$ for $i = 1, \ldots, n$, and $\tau_i \in \Sigma_{k_i}$ for  $i = 1, \ldots, n$. Then the first equivariance axiom is the requirement that
\[
\mu(x;y_1 \cdot \tau_1,\ldots,y_n \cdot \tau_n) = \mu(x;y_1,\ldots,y_n)\cdot \beta(\tau_1,\ldots,\tau_n)
\]
holds, where $\beta$ is the function from \cref{Defi:beta-s}.
\end{itemize}
\begin{itemize}
\item[3.2] Suppose that  $x \in O(n)$, $y_i \in O(k_i)$ for $i = 1, \ldots, n$, and $\sigma \in \Sigma_{n}$. 
Then the second equivariance axiom is the requirement that
\[
 \mu(x \cdot \sigma; y_1, \ldots, y_n) = \mu\left(x;y_{\sigma^{-1}(1)},\ldots,y_{\sigma^{-1}(n)}\right)\cdot \delta_{n; k_1, \ldots, k_n}(\sigma)
\]
holds, where $\delta_{n; k_1, \ldots, k_n}$ is the function from \cref{Defi:delta-s}.
\end{itemize}
\end{enumerate}
\end{Defi}

\begin{term}[(Operadic multiplication, composition)]\label[term]{term:operadic-mult}
The functions $\mu$ in \cref{Defi:sym-op} are called \emph{operadic multiplication} or \emph{operadic composition} maps.
\end{term}

\begin{rem}\label{rem:op-def}
One is intended to think that $x \in O(n)$ is a function with $n$ inputs and a single output, as below.
  \begin{center}
  \begin{tikzpicture}
    \draw (0,0.9) -- (1,0.9);
    \draw (0,0.5) -- (1,0.5);
    \node [anchor=center] () at (0.5,0) {$\vdots$\strut};
    \draw (0,-0.5) -- (1,-0.5);
    \draw (0,-0.9) -- (1,-0.9);
    \draw (1,1) -- (1,-1) -- (3,0) -- cycle;
    \draw (3,0) -- (4,0);
    \node [anchor=center] () at (1.75,0) {$x$\strut};
  \end{tikzpicture}
  \end{center}
Operadic composition is then a generalization of function composition, with the pictorial representation below being $\mu(x; y_{1}, y_{2})$ for $\mu \colon O(2) \times O(2) \times O(3) \rightarrow O(5)$.
  \begin{center}
  \begin{tikzpicture}
    \draw (-2,0.8) -- (-1,0.8);
    \draw (-2,0.4) -- (-1,0.4);
    \draw (-1,1) -- (-1,0.2) -- (0,0.6) -- cycle;
    \node [anchor=center] () at (-0.7,0.6) {$y_1$\strut};    
    \draw (-2,-0.9) -- (-1,-0.9);
    \draw (-2,-0.3) -- (-1,-0.3);
    \draw (-1,-1) -- (-1,-0.2) -- (0,-0.6) -- cycle;
    \draw (-2,-0.6) -- (-1,-0.6);
    \node [anchor=center] () at (-0.7,-0.6) {$y_2$\strut};
    \draw (0,0.6) -- (1,0.6);
    \draw (0,-0.6) -- (1,-0.6);
    \draw (1,1) -- (1,-1) -- (3,0) -- cycle;
    \draw (3,0) -- (4,0);
    \node [anchor=center] () at (1.75,0) {$x$\strut};
  \end{tikzpicture}
  \end{center}
  \end{rem}

\begin{term}[($n$-ary operations)]\label[term]{term:nary-ops}
The set $O(n)$ in \cref{Defi:sym-op} is called the set of \emph{$n$-ary operations} of $O$.
\end{term}

\begin{rem}\label{rem:nary-ops-V}
One can change from operads in $\mb{Sets}$ to operads in another (symmetric) monoidal category $\mathcal{V}$ by requiring each $O(n)$ to be an object of $\mathcal{V}$ and replacing all instances of cartesian product with the appropriate tensor product in $\mathcal{V}$. One would also the element $\id \in O(1)$ with a map $I \rightarrow O(1)$ from the unit object of $\mathcal{V}$ to $O(1)$. In the case of symmetric operads, one would also express the right group actions as homomorphisms
\[
\Sigma_n^{op} \to \mathcal{V}\big( O(n), O(n) \big).
\]
If $O$ is an operad in a category other than $\mb{Sets}$, then we would call $O(n)$ the \emph{object} of $n$-ary operations.
\end{rem}

Here are two important examples of symmetric operads.

\begin{example}[(Symmetric operad of symmetric groups)]\label[example]{ex:Sigma}
The canonical example of \emph{a} symmetric operad is \emph{the} symmetric operad and we write this as $\Sigma$. The set $\Sigma(n)$ is the set of elements of the symmetric group $\Sigma_{n}$, and the group action is just multiplication on the right. The identity element $\id \in \Sigma(1)$ is just the identity permutation on a one-element set. Operadic composition in $\Sigma$ will then be given by a function
  \[
    \Sigma(n) \times \Sigma(k_{1}) \times \cdots \times \Sigma(k_{n}) \rightarrow \Sigma(k_{1} + \cdots + k_{n})
  \]
that takes permutations $\sigma \in \Sigma_{n}, \tau_{i} \in \Sigma_{k_{i}}$ and produces the following permutation in $\Sigma_{k_{1} + \cdots + k_{n}}$:
  \[
    \mu(\sigma; \tau_{1}, \ldots, \tau_{n}) = \delta_{n; k_1, \ldots, k_n}(\sigma) \cdot \beta(\tau_1,\ldots,\tau_n),
  \]
with $\beta$ and $\delta$ as in \cref{Defi:beta-s,Defi:delta-s}.

Below we have drawn the permutation for the composition
  \[
    \mu \colon \Sigma(3) \times \Sigma(2) \times \Sigma(4) \times \Sigma(3) \rightarrow \Sigma(9)
  \]
evaluated on the element $\left( (1 \,\, 2 \,\, 3); \trans{1}{2},\trans{1}{2}\trans{3}{4},\trans{1}{3} \right)$, in terms of $\beta$ and $\delta$. We expand on this in \cref{thm:charAOp}.
  \begin{center}
  \begin{tikzpicture}[scale=0.8]
  \draw (1,0) -- (2,-1) -- (6,-2.5);
  \draw (2,0) -- (1,-1) -- (5,-2.5);
  \draw (4,0) -- (5,-1) -- (8,-2.5);
  \draw (5,0) -- (4,-1) -- (7,-2.5);
  \draw (6,0) -- (7,-1) -- (10,-2.5);
  \draw (7,0) -- (6,-1) -- (9,-2.5);
  \draw (9,0) -- (11,-1) -- (4,-2.5);
  \draw (10,0) -- (10,-1) -- (3,-2.5);
  \draw (11,0) -- (9,-1) -- (2,-2.5);
  \node at (13.3,-0.5) {$\beta\big( \trans{1}{2},\trans{1}{2}\trans{3}{4}, \trans{1}{3} \big)$};
  \node at (13.3,-2) {$\delta_{3;2,4,3}\big( (1 \,\, 2 \,\, 3) \big)$};
  \end{tikzpicture}
  \end{center}
\end{example}

\begin{example}[(Endomorphism operad)]\label[example]{ex:endo}
Let $X$ be a set. The \emph{endomorphism operad} of $X$, denoted $\mathcal{E}_X$, consists of 
\begin{itemize}
\item the sets
\[
\mathcal{E}_X(n) = \mb{Sets}(X^n, X),
\]
\item the right group actions $\mathcal{E}_X(n) \times \Sigma_n \to \mathcal{E}_X(n)$ given by
\[
(f \cdot \sigma)(x_1, \ldots, x_n) = f( x_{\sigma^{-1}(1)}, \ldots, x_{\sigma^{-1}(n)}),
\]
\item the element $\id \in \mathcal{E}_X(1)$ being the identity function $1 \colon X \to X$, and
\item operadic multiplication given by
\[
\mu(g; f_1, \ldots, f_n) = g \circ (f_1 \times \cdots \times f_n).
\]
\end{itemize}
We leave verification of the axioms to the reader.
\end{example}

\begin{rem}[(Algebras and endomorphism operads)]
The intuition in \cref{rem:op-def} is connected with \cref{ex:endo} through the concept of an algebra, see \cref{sec:forward-op}.
\end{rem}

One can also drop the symmetric group actions entirely to obtain the notion of a non-symmetric or plain operad.

\begin{Defi}[(Non-symmetric operad)]\label[Defi]{Defi:non-sym-op}
A \emph{non-symmetric operad} $O$ consists of 
\begin{itemize}
\item a set, $O(n)$, for each natural number $n$,
\item an element $\id \in O(1)$, and
\item functions
  \[
    \mu \colon  O(n) \times O(k_{1}) \times \cdots \times O(k_{n}) \rightarrow O(k_{1} + \cdots + k_{n}),
  \]
\end{itemize}
satisfying axioms 1 and 2 from \cref{Defi:sym-op}.
\end{Defi}

\begin{rem}[(Underlying collections)]\label{rem:V-and-coll}
 Every symmetric operad has an underlying \textit{symmetric collection} that consists of the natural number-indexed set $\{ O(n) \}_{n \in \N}$ together with symmetric group actions, but without a chosen identity element or composition maps. The category of symmetric collections is a presheaf category, and we will equip it with a monoidal structure in which monoids are precisely operads in \cref{thm:operad=monoid}. A similar construction, but without reference to group actions, shows that every non-symmetric operad has an underlying (non-symmetric) collection which is now merely a $\N$-indexed collection of sets.
\end{rem}

\begin{example}[(Trimble's Operad $E$)]\label[example]{ex:non-sym}
In order to define weak $n$-categories through iterated enrichment, 
Trimble constructed an operad $E$ as follows:
    \begin{itemize}
        \item for $n \geq 0$, $E(n)$ is the space of continuous endpoint-preserving maps
            \[
                [0,1] \rightarrow [0,n],
            \]
        \item the identity element $1 \in E(1)$ is the identity map
            \[
                [0,1] \rightarrow [0,1],
            \]
        \item composition is described by substitution and reparameterisation.
    \end{itemize}
More details about Trimble's operad $E$ can be found in \cite[Definition $T_R$]{leinster-survey} and \cite[Example 1.7]{cg-comparison}, with variants considered in \cite{cg-cobordism}.
This is a non-symmetric operad in $\mb{Top}$ that, to the authors' knowledge, does not admit the structure of a symmetric operad.
\end{example}

In the original topological applications \cite{maygeom}, symmetric operads were the central figures. 
In \cite{fie-br} Fiedorowicz studied braided operads, in which the braid groups take the place of the symmetric groups. We sketch that definition below.

\begin{Defi}[(Braided operad, sketch)]\label[Defi]{Defi:broperad}
A \textit{braided operad} consists of
  \begin{itemize}
    \item a non-symmetric operad $O$ and
    \item for each $n$, a right action of the $n$th braid group $B_{n}$ on $O(n)$,
  \end{itemize}
such that the operadic multiplication functions $\mu$ are equivariant with respect to the braid group actions, meaning that two equations hold.
\begin{itemize}
\item[1] Suppose that $x \in O(n)$, $y_i \in O(k_i)$ for $i = 1, \ldots, n$, and $\tau_i \in B_{k_i}$ for  $i = 1, \ldots, n$. Then the first equivariance axiom is the requirement that
\[
\mu(x;y_1 \cdot \tau_1,\ldots,y_n \cdot \tau_n) = \mu(x;y_1,\ldots,y_n)\cdot \beta(\tau_1,\ldots,\tau_n)
\]
holds.
\item[2] Suppose that  $x \in O(n)$, $y_i \in O(k_i)$ for $i = 1, \ldots, n$, and $\sigma \in B_{n}$. 
Then the second equivariance axiom is the requirement that
\[
 \mu(x \cdot \sigma; y_1, \ldots, y_n) = \mu\left(x;y_{\sigma^{-1}(1)},\ldots,y_{\sigma^{-1}(n)}\right)\cdot \delta_{n; k_1, \ldots, k_n}(\sigma)
\]
holds.
\end{itemize}
\end{Defi}

\begin{rem}[(Block sum and duplication for braid groups)]\label{rem:br-op-needed}
The above sketch omits the definitions of $\beta, \delta$ for braids. 
Formulas for these can be found in \cite[Examples~5.1.11,~5.1.13]{yau_infinity_2021}, although the geometric interpretations are simple: $\beta$ takes the disjoint union of braids, and $\delta_{n; k_1, \ldots, k_n}(\tau)$ is obtained by replacing the $i$th strand of $\tau$ by $k_i$ parallel strands.
These operations are sometimes referred to as `cabling' operations for braids, as described in, for example, \cite{doucot_local_2025}.
\end{rem}

\begin{example}[(Braided operads)]\label[example]{ex:braided-op}
\begin{enumerate}
\item Let $C_2(n)$ be the $n$th space of the little 2-disks operad, and $\widetilde{C_2(n)}$ it universal cover. Then Example 3.1 of \cite{fie-br} shows that there is a braided operad structure on the spaces $\widetilde{C_2(n)}$.
\item The braid groups themselves define a braided operad, obtained by applying \cref{prop:gisgop}.
\end{enumerate}
\end{example}

We conclude this section by defining various categories of operads in $\mb{Sets}$, although the reader can generalize these to categories of operads in any symmetric monoidal category. We focus on the case of symmetric operads, and explain after how to modify the definitions for non-symmetric or braided operads.

\begin{Defi}[(Map of symmetric operads)]\label[Defi]{Defi:sym_op_map}
Let $O, O'$ be symmetric operads in $\mb{Sets}$. Then a \textit{map of symmetric operads} (or just \emph{operad map} for short, when it is clear that the intent is to respect the symmetric group actions) $f \colon O \rightarrow O'$ consists of functions $f_{n} \colon O(n) \rightarrow O'(n)$ for each natural number such that the following axioms hold for all $x \in O(n), y_i \in O(k_i), \sigma \in \Sigma_n$.

  \begin{align*}
    f\left(\id_O\right) &= \id_{O'}\\
    f\left(\mu^{O}(x;y_1,\ldots,y_n)\right) &= \mu^{O'}\left(f(x);f(y_1),\ldots,f(y_n)\right)\\
    f(x \cdot \sigma) & = f(x) \cdot \sigma
  \end{align*}
\end{Defi}

The next proposition states that symmetric operads and their maps form a category. We leave the proof to the reader.

\begin{prop}\label[prop]{prop:cat-of-sym-op}
There is a category with 
\begin{itemize}
\item objects the symmetric operads $O$ in $\mb{Sets}$, 
\item morphisms the maps of symmetric operads between them,
\item identities $1_O \colon O \to O$ given by
\[
(1_O)_n = 1_{O(n)} \colon O(n) \to O(n),
\]
and
\item composition given by
\[
(g \circ f)_n = g_n \circ f_n.
\]
\end{itemize}
\end{prop}

\begin{nota}[(The category of symmetric operads)]\label{nota:cat-of-sym-op}
The category in \cref{prop:cat-of-sym-op} is called the \emph{category of symmetric operads (in $\mb{Sets}$)}, and is denoted $\Sigma\mbox{-}\mb{Op}$.
\end{nota}

\begin{rem}[(The category of non-symmetric operads)]\label{rem:cat-of-nonsym-op}
Omitting symmetries entirely, we can also form the category of non-symmetric operads (in $\mb{Sets}$), denoted $\mb{Op}$. The objects are non-symmetric operads (\cref{Defi:non-sym-op}) and the morphisms have the same data as maps of symmetric operads (\cref{Defi:sym_op_map}) but only satisfy the first two axioms as there is no group action to preserve. Composition and identities are defined exactly as for symmetric operads.
\end{rem}

\begin{rem}[(The category of braided operads)]\label{rem:cat-of-br-op}
Replacing symmetries with braids, we can form the category of braided operads (in $\mb{Sets}$), denoted $B\mbox{-}\mb{Op}$. The objects are braided operads (\cref{Defi:broperad}). The morphisms have the same data as maps of symmetric operads (\cref{Defi:sym_op_map}) and satisfy identical looking axioms so long as the equivariance axiom is interpretted using braids rather than symmetries. Composition and identities are defined exactly as for symmetric operads.
\end{rem}

\section{Action Operads}\label{sec:aop}

The axioms for both symmetric and braided operads use the following features.
\begin{enumerate}
\item For each $n$, we have a group $\Lambda_n$ acting on the set $O(n)$ of $n$-ary operations of the operad. Each such group is equipped with a homomorphism $\pi_n \colon \Lambda_n \to \Sigma_n$, so that every element of $\Lambda_n$ has an underlying permutation.
\item The first equivariance axiom requires the additional data of a family of functions
\[
\beta \colon \Lambda_{k_1} \times \cdots \times \Lambda_{k_n} \to \Lambda_{k_1 + \cdots + k_n}.
\]
In order for this to be a well-defined function, the right group action axioms force these functions to be group homomorphisms.
\item The second equivariance axiom requires the additional data of a family of functions
\[
\delta_{n; k_1, \ldots, k_n} \colon \Lambda_{k_1} \times \cdots \times \Lambda_{k_n} \to \Lambda_{k_1 + \cdots + k_n}.
\]
These functions are not forced to be group homomorphisms, but do satisfy some additional axioms.
\end{enumerate}
In this section, we define \emph{action operads} in \cref{Defi:aop} in order to present a unified treatment of a family of groups satisfying the conditions above. 
In \cref{sec:forward-op}, we define for each action operad $\Lambda$ a notion of $\Lambda$-operad; symmetric operads will arise when $\Lambda = \Sigma$, non-symmetric operads will arise when $\Lambda$ is the action operad of trivial groups, and braided operads will arise when $\Lambda = B$.
Our definition of an action operad will not mention $\beta$ or $\delta$, but will instead use a single axiom relating the group structure, operadic multiplication, and underlying permutations. 
The main result of this section is \cref{thm:charAOp} in which we prove that action operads can be described entirely in terms of the functions $\beta, \delta$ as above.
We will give two examples of action operads (the symmetric groups and the trivial groups) in this section, and postpone the rest to \cref{sec:aop-examples}.

\begin{Defi}[(Action operad)]\label{Defi:aop}
An \textit{action operad} $(\Lambda, \pi)$ consists of
\begin{itemize}
\item an operad $\Lambda = \{ \Lambda(n) \}$ in the category of sets such that each $\Lambda(n)$ is equipped with the structure of a group and
\item a map $\pi \colon \Lambda \rightarrow \Sigma$ which is simultaneously a map of operads and a group homomorphism $\pi_{n} \colon \Lambda(n) \rightarrow \Sigma_{n}$ for each $n$
\end{itemize}
such that one additional axiom holds. Write
  \[
    \mu \colon  \Lambda(n) \times \Lambda(k_{1}) \times \cdots \times \Lambda(k_{n}) \rightarrow \Lambda(k_{1} + \cdots + k_{n})
  \]
for the multiplication in the operad $\Lambda$. Let 
\begin{align*}
(g; f_1, \ldots, f_n) & \in \Lambda(n) \times \Lambda(k_{1}) \times \cdots \times \Lambda(k_{n}), \\
(g'; f_1', \ldots f_n') & \in \Lambda(n) \times \Lambda(k_{g^{-1}(1)}) \times \cdots \times \Lambda(k_{g^{-1}(n)}).
\end{align*}
We require that
  \begin{eqn}\label{eqn:ao_axiom}
    \mu\left(g'; f_1', \ldots f_n'\right)  \mu\left(g; f_1, \ldots, f_n\right) = \mu\left(g'g; f_{g(1)}'f_{1}, \ldots, f_{g(n)}'f_{n}\right)
  \end{eqn}
in the group $\Lambda(k_{1} + \cdots + k_{n})$.
\end{Defi}

\begin{nota}\label{nota:suppress-pi}
We write an action $(\Lambda, \pi)$ as merely $\Lambda$.
The maps $\pi$ will be left implicit in the notation, as we will not have reason to study the case of a single operad $\Lambda$ equipped with two different action operad structures via $\pi$ and $\pi'$.
\end{nota}

\begin{rem}\label{rem:similar-defs}
Our definition of an action operad is the same as the \emph{operads from families of groups} appearing in Section 1.2 Wahl's thesis \cite{wahl-thesis}, but without the condition that each $\pi_{n}$ is surjective. It is also the same as the \emph{group operads} appearing in work of Zhang \cite{zhang-grp}, although we prove later (see \cref{lem:calclem}) that Zhang's condition of $e_{1} \in \Lambda(1)$ being the identity element follows from the rest of the axioms.
\end{rem}

We now give the two examples of action operads that have already appeared in this paper: the symmetric groups and the trivial groups.

\begin{example}[(Action operad of symmetric groups)]\label{example:aop-sym}
The symmetric operad $\Sigma$ has a canonical action operad structure. It is given by taking $\pi$ to be the identity map, and is the terminal object in the category of action operads (\cref{nota:cat_aop}).
\end{example}

\begin{example}[(Action operad of trivial groups)]\label{example:aop-triv}
The terminal operad $T$ in the category of sets has a unique action operad structure. Since $T(n)$ is a singleton for each $n$, the group structure is unique, as is the map $\pi$. The single action operad axiom is then automatic as both sides of \cref{eqn:ao_axiom} are the identity. This is the initial object in the category of action operads.
\end{example}

\begin{rem}
The final axiom is best explained using the operad $\Sigma$ of symmetric groups. Reading symmetric group elements as permutations from top to bottom, below is a pictorial representation of the final axiom for the map $\mu \colon \Sigma_{3} \times \Sigma_{2} \times \Sigma_{2} \times \Sigma_{2} \rightarrow \Sigma_{6}.$
  \begin{center}
  \begin{tikzpicture}[scale=0.5]
    \draw (1,0) -- (2,-1) -- (8,-2) -- (8,-3) -- (5,-4);
    \draw (2,0) -- (1,-1) -- (7,-2) -- (7,-3) -- (4,-4);
    \draw (4,0) -- (4,-1) -- (1,-2) -- (2,-3) -- (2,-4);
    \draw (5,0) -- (5,-1) -- (2,-2) -- (1,-3) -- (1,-4);
    \draw (7,0) -- (8,-1) -- (5,-2) -- (4,-3) -- (7,-4);
    \draw (8,0) -- (7,-1) -- (4,-2) -- (5,-3) -- (8,-4);
    \node at (9.5,-2) {$=$};
    \draw (11,0) -- (12,-1) -- (12,-2) -- (18,-3) -- (15,-4);
    \draw (12,0) -- (11,-1) -- (11,-2) -- (17,-3) -- (14,-4);
    \draw (14,0) -- (14,-1) -- (15,-2) -- (12,-3) -- (12,-4);
    \draw (15,0) -- (15,-1) -- (14,-2) -- (11,-3) -- (11,-4);
    \draw (17,0) -- (18,-1) -- (17,-2) -- (14,-3) -- (17,-4);
    \draw (18,0) -- (17,-1) -- (18,-2) -- (15,-3) -- (18,-4);
    \node at (4.5,-4.5) {$\scriptstyle \mu\big((23);(12),(12), e_2\big) \ \cdot \ \mu\big((132); (12), e_2, (12)\big)$};
    \node at (14.5,-4.5) {$\scriptstyle \mu\big((23)\cdot (132); e_2 \cdot (12), (12) \cdot e_2, (12) \cdot (12)\big)$};
  \end{tikzpicture}
  \end{center}
\end{rem}

Action operads are themselves the objects of a category, $\mb{AOp}$. The morphisms of this category are defined below.
\begin{Defi}[(Map of action operads)]\label{Defi:mapaop}
A \textit{map of action operads} $f \colon  \Lambda \rightarrow \Lambda'$ consists of a map $f \colon \Lambda \rightarrow \Lambda'$ of the underlying operads such that
  \begin{enumerate}
    \item $\pi^{\Lambda'} \circ f = \pi^{\Lambda}$ (i.e., $f$ is a map of operads over $\Sigma$) and
    \item each $f_{n} \colon \Lambda(n) \rightarrow \Lambda'(n)$ is a group homomorphism.
  \end{enumerate}
\end{Defi}

\begin{prop}\label{prop:cat-of-aop}
There is a category with 
\begin{itemize}
\item objects the action operads $O$ in $\mb{Sets}$, 
\item morphisms as defined in \cref{Defi:mapaop},
\item identities $1_\Lambda \colon \Lambda \to \Lambda$ given by the identity morphism of $\Lambda$ as an operad,
and
\item composition given by composition of maps of operads.
\end{itemize}
\end{prop}

\begin{nota}[(The category of action operads)]\label{nota:cat_aop}
The category in \cref{prop:cat-of-aop} is called the \emph{category of action operads (in $\mb{Sets}$)}, and is denoted $\mb{AOp}$.
\end{nota}

\begin{prop}\label{prop:pi-in-aop}
Let $(\Lambda, \pi)$ be an action operad. The map $\pi \colon \Lambda \rightarrow \Sigma$ is a map of action operads.
\end{prop}

We now study some of the structure on the groups $\Lambda(n)$ for small values of $n$. Recall from \cref{nota:e_identity} that we write $e_{n}$ for the identity element in the group $\Lambda(n)$.
Many of our proofs rely on the following version of the Eckmann-Hilton argument \cite{eh}.

\begin{prop}[Eckmann-Hilton argument]\label{prop:EH}
Let $G$ be a group with identity element $e$, and suppose $\varphi \colon G \times G \to G$ is a function. If $\varphi$ is a homomorphism, meaning that
\[
\varphi(g',h') \cdot \varphi(g,h) = \varphi(g' \cdot g, h' \cdot h),
\]
and $\varphi(g,e) = g = \varphi(e,g)$ for all elements $g \in G$, then 
\[
\varphi(g, h) = g \cdot h
\]
and $G$ is abelian.
\end{prop}

\begin{lem}\label{lem:calclem}
Let $\Lambda$ be an action operad.
\begin{enumerate}
\item In $\Lambda(1)$, the identity element for the group structure, $e_1$, is equal to the identity element for the operad structure, $\id$.
\item The equation
  \[
    \mu(e_{n}; e_{i_{1}}, \ldots, e_{i_{n}}) = e_{I}
  \]
holds for any natural numbers $n, i_{j}, I = \sum_{j=1}^n i_{j}$.
\item The group $\Lambda(1)$ is abelian.
\end{enumerate}
\end{lem}
\begin{proof}
For the first claim, we will prove that $\id \cdot e_1 = \id \cdot \id$, so $e_1 = \id$ by cancellation.
Note that since the only element of $\Sigma_1$ is the identity permutation, the action operad axiom \cref{eqn:ao_axiom} is
\[
\mu(g';f')\cdot \mu(g;f) = \mu(g'g; f'f)
\]
when $g, g' \in \Lambda(1)$.
Thus we obtain 
\begin{align*}
\id \cdot e_1 & = \mu(\id; \id) \cdot \mu(\id; e_{1}) \\
& = \mu(\id \cdot \id; \id \cdot e_{1}) \\
    &= \mu(\id \cdot \id; \id) \\
    &= \id \cdot \id
\end{align*}
using that $\id$ is the identity element for operadic multiplication, the $n=1$ action operad axiom explained above, that $e_1$ is the identity for group multiplication, and that $\id$ is the identity for operadic multiplication again.
Therefore $\id \cdot e_1 = \id \cdot \id$ as desired, and $e_1 = \id$.

For the second claim, we write $\mu(e_{n}; e_{i_{1}}, \ldots, e_{i_{n}})$ as $\mu(e; \underline{e})$, and consider the square of this element. We find that
  \begin{align*}
    \mu(e; \underline{e}) \cdot \mu(e; \underline{e}) & = \mu(e \cdot e; \underline{e} \cdot \underline{e}) \\
    &= \mu(e; \underline{e}),
  \end{align*}
where the first equality follows from the last action operad axiom together with the fact that $e$ gets mapped to the identity permutation; here $\underline{e} \cdot \underline{e}$ is the sequence $e_{i_{1}} \cdot e_{i_{1}}, \ldots, e_{i_{n}} \cdot e_{i_{n}}$. Thus $\mu(e; \underline{e})$ is an idempotent element of the group $\Lambda(I)$, so must be the identity element $e_{I}$.

For the final claim, note that the specific operadic multiplication map $\mu \colon \Lambda(1) \times \Lambda(1) \rightarrow \Lambda(1)$ is a group homomorphism following from the action operad axioms, and $\id = e_{1}$ is a two-sided unit, so \cref{prop:EH} shows that $\mu$ is actually group multiplication and that $\Lambda(1)$ is abelian.
\end{proof}

\begin{lem}\label{lem:e0-unit}
Let $\Lambda$ be an action operad, and $g_i \in \Lambda(k_i)$ for $i=2, \ldots, n$. Then
\[
\mu(e_n; e_0, g_2, \ldots, g_n) = \mu(e_{n-1}; g_2, \ldots, g_n).
\]
Similarly, $\mu(e_n; h_1, \ldots, h_{n-1}, e_0) = \mu(e_{n-1}; h_1, \ldots, h_{n-1})$ for any $h_i \in \Lambda(k_i)$ for $i=1, \ldots, n-1$.
\end{lem}
\begin{proof}
We will only check the first claim, as the second follows by analogous calculations.
The equalities
\begin{align*}
\mu(e_n; e_0, g_2, \ldots, g_n) & = \mu \big( \mu(e_2; e_1, e_{n-1}); e_0, g_2, \ldots, g_n \big) \\
& = \mu\big( e_2; \mu(e_1; e_0), \mu(e_{n-1}; g_2, \ldots, g_n) \big) \\
& = \mu\big( e_2; e_0, \mu(e_{n-1}; g_2, \ldots, g_n) \big)
\end{align*}
follow from the second part of \cref{lem:calclem}, operadic associativity, and the first part of \cref{lem:calclem}, respectively.
Therefore the first equality in the lemma follows from the special case when $n=2$ and the equality
\begin{equation}\label{eq:n=2-e0}
\mu(e_2; e_0, g) = g,
\end{equation}
by subsituting $g = \mu(e_{n-1}; g_2, \ldots, g_n)$.
In order to prove \cref{eq:n=2-e0}, we use the same methods as above to obtain
  \begin{align*}
   g & = \mu(e_1; g) \\
    &= \mu(\mu(e_2; e_0, e_1); g) \\
    & = \mu(\mu(e_2; e_0, e_1); \mu(e_1; g)) \\
    & = \mu(e_2; e_0, g).
  \end{align*}
This calculation verifies \cref{eq:n=2-e0}, and so completes the proof of the first equality in the statement of the lemma.
\end{proof}

\begin{cor}\label{cor:G0abel}
Let $\Lambda$ be an action operad. For any $g, h \in \Lambda(0)$, the equation
\[
g \cdot h = \mu(e_2; g,h)
\]
holds. As a consequence, $\Lambda(0)$ is abelian.
\end{cor}
\begin{proof}
The function $\Lambda(0) \times \Lambda(0) \to \Lambda(0)$ given by
\[
g, h \mapsto \mu(e_2; g, h)
\]
is a group homomorphism by the action operad axiom \cref{eqn:ao_axiom} as we verify below.
\[
\mu(e_2; g', h') \cdot \mu(e_2; g, h) = \mu(e_2 \cdot e_2; g' \cdot g, h' \cdot h) = \mu(e_2; g' \cdot g, h' \cdot h)
\]
In order to apply \cref{prop:EH} and conclude that $g \cdot h = \mu(e_2; g, h)$, we must verify that 
\[
\mu(e_2; e_0, g) = g = \mu(e_2; g, e_0)
\]
for all $g \in \Lambda(0)$, but this follows immediately from \cref{lem:e0-unit}.
Thus the function $\mu(e_2; -,-)$ satisfies the hypotheses in \cref{prop:EH}.
  Therefore $g \cdot h = \mu(e_2; g,h)$ and $\Lambda(0)$ is abelian.
\end{proof}

The symmetric operad structure on the symmetric groups in \cref{ex:Sigma} was constructed using the functions $\beta, \delta$ from \cref{Defi:beta-s} and \cref{Defi:delta-s}, respectively.
We are now ready to show that any action operad can be described in this way, as promised in the introductory remarks to this section.

\begin{thm}\label{thm:charAOp}
An action operad $\Lambda$ determines, and is uniquely determined by, the following: 
\begin{itemize}
\item groups $\Lambda(n)$ together with group homomorphisms $\pi_{n} \colon \Lambda(n) \rightarrow \Sigma_{n}$,
\item a group homomorphism
  \[
    \Lambda(k_{1}) \times \cdots \times \Lambda(k_{n}) \stackrel{\beta}{\longrightarrow} \Lambda(k_{1} + \cdots + k_{n}),
  \]
for each $n > 0$ and $k_{1}, \ldots, k_{n}$, and
\item a function of sets
  \[
    \Lambda(n) \stackrel{\delta_{n; k_{1}, \ldots, k_{n}}}{\longrightarrow} \Lambda(k_{1} + \cdots + k_{n})
  \]
for each $n, k_{1}, \ldots, k_{n}$,
\end{itemize}
subject to the axioms below. In what we write below, we use the following notational conventions.
\begin{itemize}
\item The symbols $f,g,h$, with or without subscripts, always refer to an element of some group $\Lambda(n)$.
\item The symbols $j,k,m,n,p$ are all natural numbers, and $i$ is a natural number between 1 and $n$.
\end{itemize}
Axioms:
\begin{enumerate}
\item\label{eq1} The homomorphisms $\beta$ are natural with respect to the maps $\pi_{n}$, where $K = k_{1} + \cdots + k_{n}$.
  \[
    \xy
      (0,0)*+{\Lambda(k_{1}) \times \cdots \times \Lambda(k_{n}) } ="00";
      (0,-15)*+{\Sigma_{k_{1}} \times \cdots \times \Sigma_{k_{n}}  } ="01";
      (40,0)*+{\Lambda(K) } ="20";
      (40,-15)*+{\Sigma_{K} } ="21";
      {\ar^{\beta} "00" ; "20"};
      {\ar^{\pi} "20" ; "21"};
      {\ar_{\pi_1 \times \cdots \times \pi_n} "00" ; "01"};
      {\ar_{\beta} "01" ; "21"};
    \endxy
  \]

\item\label{eq2} The homomorphism $\beta \colon \Lambda(k) \rightarrow \Lambda(k)$ is the identity.
\item\label{eq3} The homomorphisms $\beta$ are associative in the sense that the equation
\[
  \beta(\underline{h_1},\ldots,\underline{h_n}) = \beta(\beta(\underline{h_1}),\ldots,\beta(\underline{h_n}))
\]
holds, where $\underline{h_i} = h_{i1},\ldots,h_{ij_i}$.
\item\label{eq4} The functions $\delta_{n; k_{1}, \ldots, k_{n}}$ are natural with respect to the maps $\pi_{n}$, where $K = k_1 + \cdots + k_n$.
  \[
    \xy
      (0,0)*+{\Lambda(n)} ="00";
      (40,0)*+{\Lambda(k_{1} + \cdots + k_{n}) } ="20";
      (0,-15)*+{\Sigma_{n}  } ="01";
      (40,-15)*+{\Sigma_{k_{1} + \cdots + k_{n}} } ="21";
      {\ar^{\delta} "00" ; "20"};
      {\ar^{\pi} "20" ; "21"};
      {\ar_{\pi} "00" ; "01"};
      {\ar_{\delta} "01" ; "21"};
    \endxy
  \]

\item\label{eq5} The function $\delta_{n; 1, \ldots, 1} \colon \Lambda(n) \rightarrow \Lambda(n)$ is the identity. The function $\delta_{n;k_1,\ldots,k_n} \colon \Lambda(n) \to \Lambda(k_1 + \cdots + k_n)$ maps $e_n$ to $e_{k_1 + \cdots + k_n}$
\item\label{eq6} The equation $\delta_{n; k_1, \ldots, k_n}(g) \delta_{n; j_1, \ldots, j_n}(h) = \delta_{n; j_1,\ldots,j_n}(gh)$ holds when

  \[
    k_{i} = j_{h^{-1}(i)}.
  \]
\item\label{eq7} The functions $\delta$ are associative in the sense that the equation
  \[
    \delta_{m_1 + \cdots + m_n; \underline{p_1},\ldots,\underline{p_n}}\left( \delta_{n; m_{1}, \ldots, m_{n}}(g) \right) = \delta_{n; P_{1}, \ldots, P_{n}}(g)
  \]
holds, where $P_{i} = p_{i1} + \cdots + p_{im_{i}}$ and $\underline{p_i} = p_{i1}, \ldots, p_{im_i}$. We note that when $m_i = 0$, then the list $\underline{p_i}$ is empty and $P_i$ is defined to be $0$.
\item\label{eq8} The equation
  \[
    \delta_{n;k_1,\ldots,k_n}(g) \beta(h_{1}, \ldots, h_{n}) = \beta(h_{g^{-1}(1)}, \ldots,  h_{g^{-1}(n)}) \delta_{n;k_{g^{-1}(1)},\ldots,k_{g^{-1}(n)}}(g)
  \]
holds, where $h_{i} \in \Lambda(k_{i})$.
\item\label{eq9} The equation
  \[
    \beta(\delta_{1}(g_{1}), \ldots, \delta_{n}(g_{n})) = \delta_{c}(\beta(g_{1}, \ldots, g_{n}))
  \]
holds, where $\delta_{i}(g_{i})$ is shorthand for $\delta_{k_{i}; m_{i1}, \ldots, m_{ik_{i}}}(g_{i})$ and $\delta_{c}$ is shorthand for
  \[
    \delta_{k_{1}+\cdots + k_{n}; m_{11}, m_{12}, \ldots, m_{1k_{1}}, m_{21}, \ldots, m_{nk_{n}}}.
  \]
\end{enumerate}
\end{thm}

\begin{proof}
Let $\Lambda$ be an action operad, and define 
\begin{align*}\label{eqn:bd-from-aop}
\beta(g_1, \ldots, g_n) & = \mu(e_n; g_1, \ldots, g_n),\\
\delta_{n; k_1, \ldots, k_n}(g) & = \mu(g; e_{k_1}, \ldots, e_{k_n}).
\end{align*}
Since $\pi \colon \Lambda \rightarrow \Sigma$ is an operad map, Axioms \eqref{eq1} and \eqref{eq4} hold by the definition of the operad structure on $\Sigma$ in \cref{ex:Sigma}. 
Since $\Lambda$ is an operad of sets, Axioms \eqref{eq2} and \eqref{eq5} follow from the operad unit axioms and the first part of \cref{lem:calclem}, and Axioms \eqref{eq3}, \eqref{eq7}, and \eqref{eq9} follow from the operad associativity axiom and the second part of \cref{lem:calclem}. Axioms \eqref{eq6} and \eqref{eq8} are special cases of the additional action operad axiom, as is the fact that $\beta$ is a group homomorphism.

Conversely, given the data above, we need only define the operad multiplication, verify the operad unit and multiplication axioms, and finally check the action operad axiom. Multiplication is given by
  \begin{equation}\label{eqn:mu-from-betadelta}
    \mu(g; h_{1}, \ldots, h_{n}) = \delta_{n; k_{1}, \ldots, k_{n}}(g) \beta(h_{1}, \ldots, h_{n})
  \end{equation}
where $h_{i} \in \Lambda(k_{i})$. The identity element $\id$ for the operad structure is $e_1 \in \Lambda(1)$.

We now verify the operad unit axioms. Let $g, h \in \Lambda(n)$. Then
  \begin{align*}
    \mu(e_1; g) &= \delta(e_1)\beta(g) \\
    &= e_1 \cdot g \\
    &= g, \\
    \mu(h; e_1, \ldots, e_1) &= \delta_{n; 1, \ldots, 1}(h)\beta(e_1, \ldots, e_1) \\
    &= h \cdot e_n \\
    &= h
  \end{align*}
by Axioms \eqref{eq2} and \eqref{eq5}, together with the fact that $\beta$ is a group homomorphism.
Thus $e_1$ satisfies the identity axioms for operadic multiplication.

For the operad associativity axiom, let
\begin{itemize}
\item $f \in \Lambda(m),$
\item $g_{i} \in \Lambda(n_{i})$ for $i=1, \ldots, m$, and
\item $h_{ij} \in \Lambda(p_{i,j})$ for $i=1, \ldots, m$ and $j=1, \ldots, n_{i}$.
\end{itemize}
Further, let $P_{i} = p_{i1} + \cdots + p_{in_{i}}$ and $\underline{h_i}$ denote the list $h_{i1}, h_{i2}, \ldots, h_{in_{i}}$. We must then show that
  \[
    \mu\big( f; \mu(g_{1}; \underline{h_1}), \ldots, \mu(g_{m}; \underline{h_m}) \big) = \mu\big( \mu(f; g_{1}, \ldots, g_{m}); \underline{h_1}, \ldots, \underline{h_m} \big).
  \]
By definition, the left side of this equation is
  \[
    \delta_{m; P_{1}, \ldots, P_{m}}(f) \beta\big( \mu(g_{1}; \underline{h_1}), \ldots, \mu(g_{m}; \underline{h_m}) \big),
  \]
and
  \[
    \mu\left(g_{i}; \underline{h_i}\right) = \delta_{n_{i}; p_{i1}, \ldots, p_{in_{i}}}(g_{i})\beta\left(h_{i1}, \ldots, h_{in_{i}}\right).
  \]
From this point, we suppress subscripts on the $\delta$'s.
Since $\beta$ is a group homomorphism, we can then rewrite the left side as
  \[
    \delta(f)\beta\big(\delta(g_{1}), \ldots, \delta(g_{m})\big)\beta\big(\beta(\underline{h_1}), \ldots, \beta(\underline{h_m})\big)
  \]
where we have suppressed the subscripts on the $\delta$'s. By Axiom \eqref{eq3},
  \[
    \beta\big(\beta(\underline{h_1}), \ldots, \beta(\underline{h_m})\big) = \beta\left(\underline{h_1},\ldots,\underline{h_m}\right).
  \]
Furthermore, Axiom \eqref{eq9} above shows that
  \[
    \beta\big(\delta(g_{1}), \ldots, \delta(g_{m})\big) = \delta\big(\beta(g_{1}, \ldots, g_{m})\big).
  \]
Thus we have shown that the left side of the operad associativity axiom is equal to
  \[
    \delta(f)\delta\big(\beta(g_{1}, \ldots, g_{m})\big)\beta\left(\underline{h_1},\ldots,\underline{h_m}\right).
  \]
Now the right side is
  \[
    \mu\big( \mu (f; g_{1}, \ldots, g_{m}); \underline{h_1}, \ldots, \underline{h_m} \big),
  \]
which is by definition
  \[
    \delta\big(\mu (f; g_{1}, \ldots, g_{m})\big)\beta\left(\underline{h_1}, \ldots, \underline{h_m}\right).
  \]
Cancelling the $\beta\left(\underline{h_1}, \ldots, \underline{h_m}\right)$ terms, verifying the operad associativity axiom reduces to showing
\begin{eqn}\label{eqn:opass}
\delta(f)\delta\big(\beta(g_{1}, \ldots, g_{m})\big) = \delta\big(\mu (f; g_{1}, \ldots, g_{m})\big).
\end{eqn}
By the definition of $\mu$,
  \[
    \delta\big(\mu (f; g_{1}, \ldots, g_{m})\big) = \delta\big(\delta(f)\beta(g_{1}, \ldots, g_{m}) \big)
  \]
which is itself equal to
\begin{eqn}\label{eqn:opass2}
\delta\big(\delta(f)\big) \delta\big(\beta(g_{1}, \ldots, g_{m})\big)
\end{eqn}\noindent
by Axiom \eqref{eq6} above.

Now the $\delta(f)$ on the left side of \cref{eqn:opass} uses $\delta_{n; P_{1}, \ldots, P_{n}}$, while the $\delta(\delta(f))$ in \cref{eqn:opass2} is actually
  \[
    \delta_{m_1 + \cdots + m_{n}; q_{ij}}(\delta_{n; m_{1}, \ldots, m_{n}} (f))
  \]
where the $q_{ij}$ are defined, by Axiom \eqref{eq6}, to be given by
  \[
    q_{ij} = p_{i,g_{i}^{-1}(j)}
  \]
using the compatibility of $\beta$ and $\pi$ in Axiom \eqref{eq1}. By Axiom \eqref{eq7}, this composite of $\delta$'s  is then $\delta_{n; Q_{1}, \ldots, Q_{n}}$ where $Q_{i} = q_{i1} + \cdots + q_{im_{i}}$. But by the definition of the $q_{ij}$, we immediately see that $Q_{i} = P_{i}$, so the $\delta(f)$ in \cref{eqn:opass} is equal to the $\delta(\delta(f))$ appearing in \cref{eqn:opass2}, concluding the proof of the operad associativity axiom.

Writing $\mu(g;\underline{h}) = \mu\left(g; h_{1}, \ldots, h_{n}\right)= $ and $\mu(g';\underline{h'}) = \mu\left(g'; h_{1}', \ldots, h_{n}'\right)$, the action operad axiom is now the calculation below, and uses Axioms \eqref{eq4} and \eqref{eq8}.
\begin{small}
  \begin{align*}
    \mu(g;\underline{h})\mu(g';\underline{h'}) &= \delta\left(g\right) \beta\left(h_{1}, \ldots, h_{n}\right) \delta\left(g'\right) \beta\left(h_{1}', \ldots, h_{n}'\right) \\
    &= \delta\left(g\right) \delta\left(g'\right) \beta\left(h_{g'(1)}, \ldots, h_{g'(n)}\right)  \beta\left(h_{1}', \ldots, h_{n}'\right) \\
    &= \delta\left(gg'\right) \beta\left(h_{g'(1)}h_{1}', \ldots, h_{g'(n)}h_{n}'\right) \\
    &= \mu\left(gg'; h_{g'(1)}h_{1}', \ldots, h_{g'(n)}h_{n}'\right)
  \end{align*}
\end{small}
\end{proof}

\begin{prop}[Corollary 2.17,  \cite{zhang-grp}]\label{prop:surjortriv} Let $\Lambda$ be an action operad. Then the homomorphisms $\pi_n \colon \Lambda(n) \rightarrow \Sigma_n$ are either all surjective or all the zero map.
\end{prop}
\begin{proof}
We will prove each case separately. The two cases coincide for $n = 0, 1$ as both $\Sigma_0, \Sigma_1$ are the trivial group and therefore any homomorphism with one of them as its codomain is both surjective and the zero map. Since $\Sigma_2$ only has one non-identity element, any homomorphism $G \rightarrow \Sigma_2$ must necessarily be surjective or the zero map.

Suppose that $\pi_2 \colon \Lambda(2) \rightarrow \Sigma_2$ is surjective, so there exists $g \in \Lambda(2)$ such that $\pi_2(g) = \trans{1}{2}$. Let $n>2$. Since $\Sigma_n$ is generated by the adjacent transpositions $\trans{a}{a+1}$, we will show that each such element is in the image of $\pi_n$. Write $\underline{x}^i$ for the $i$-tuple $x, x, \ldots, x$. Then $\trans{a}{a+1} = \beta(\underline{e_1}^{a-1} ,\trans{1}{2},\underline{e_1}^{n-a-1})$ in $\Sigma$, so
  \begin{align*}
      \trans{a}{a+1} &= \beta(\underline{e_1}^{a-1},\trans{1}{2},\underline{e_1}^{n-a-1}) \\ &= \beta\big(\underline{\pi_1(e_1)}^{a-1},\pi_2(g),\underline{\pi_1(e_1)}^{n-a-1}\big) \\
      &= \pi_n\big(\beta(\underline{e_1}^{a-1},g,\underline{e_1}^{n-a-1})\big)
  \end{align*}
by Axiom \eqref{eq1} of \ref{thm:charAOp}.
Thus $\pi_n$ is surjective for all $n>2$ if $\pi_2$ is surjective.

Now we will consider the case where $\pi_2$ is the zero map.
Suppose that there exists $g \in \Lambda(n)$ such that $\pi_n(g) = \sigma \neq e_n$ in $\Sigma_n$.
Then we can find $1 \leq i < j \leq n$ such that $\sigma(j) < \sigma(i)$.
Consider the element 
\[
h = \delta_{n; \underline{0}^{i-1}, 1, \underline{0}^{j-i-1}, 1, \underline{0}^{n-j}}(g) \in \Lambda(2).
\]
By the assumption that $\pi_2$ is the zero map, we must have that $\pi_2(h) = e_2$, but by Axiom \eqref{eq4} of \ref{thm:charAOp} we also compute
\[
\pi_2(h) = \delta_{n; \underline{0}^{i-1}, 1, \underline{0}^{j-i-1}, 1, \underline{0}^{n-j}}\big( \pi_n(g) \big) =  \delta_{n; \underline{0}^{i-1}, 1, \underline{0}^{j-i-1}, 1, \underline{0}^{n-j}}\big( \sigma \big).
\]
The element $\delta_{n; \underline{0}^{i-1}, 1, \underline{0}^{j-i-1}, 1, \underline{0}^{n-j}}\big( \sigma \big)$ is equal to $\trans{1}{2}$ by the choice of $i, j$ and \cref{Defi:delta-s}.
These two computations of $\pi_2(h)$ are in contradiction, so there must be no such $g \in \Lambda(n)$.
Thus if $\pi_2$ is the zero map, so is $\pi_n$ for all $n>2$.
\end{proof}

\section{Examples}\label{sec:aop-examples}

In this section, we expand our collection of examples  and non-examples of action operads. 
In all but one case, \cref{ex:abgp-aop}, the examples we provide have appeared elsewhere.
The non-examples we provide were largely sourced from questions received after preliminary talks on this research by the authors. 

\begin{example}[(Action operad of braid groups)]\label{example:aop-braid}
One can form an operad $B$ where $B(n)$ is the underlying set of the $n$th braid group, $B_{n}$.
We define the operad structure using the functions $\beta, \delta$ from \cref{rem:br-op-needed}.
Yau checks that these groups and functions satisfy the axioms of an action operad in \cite[Prop 5.2.5]{yau_infinity_2021}, but we note that each of the nine axioms in \cref{thm:charAOp} follows immediately by using the geometric definitions of $\beta, \delta$.
\end{example}

\begin{example}[(Action operad of ribbon braid groups)]\label{example:aop-rbraid}
For each $n \in \mathbb{N}$, the \emph{ribbon braid group} $RB_{n}$ is defined to be the semidirect product $\mathbb{Z}^n \rtimes B_n$, where the action of $B_n$ on $\mathbb{Z}^n$ is given, using the underlying permutation of a braid $\gamma$, by the formula
\[
\gamma \cdot (a_1, \ldots, a_n) = (a_{\gamma^{-1}(1)}, \ldots, a_{\gamma^{-1}(n)}).
\]
Alternatively, $RB_n$ can be described as the group of isotopy classes of braids equipped with a framing \cite{KS-framed,KS-framed2,MT-framed}. In \cite{KS-framed,KS-framed2} the action above is described without inverses, but we choose to present it this way for consistency with \Cref{Defi:broperad}. A purely algebraic presentation of $RB_n$ is given by generators
    \[
        \sigma_1, \ldots, \sigma_{n-1}, t_1, \ldots, t_n
    \]
where 
\begin{itemize}
\item the $\sigma_i$ are the usual braid generators and satisfy the relations of the braid group, and 
\item the $t_i$ are the \emph{full twists} and satisfy the additional equations
\begin{align*}
t_i t_j & = t_j t_i, \\
\sigma_i t_j & = t_{\sigma_i^{-1}(j)} \sigma_i
\end{align*}
for all $i, j$.
\end{itemize}
In particular, for any given braid $\sigma \in B_n$, the relations imply that $\sigma t_i = t_{\sigma^{-1}(i)} \sigma$.

The action operad structure is much like that for the braid groups in \cref{rem:br-op-needed} and its properties are investigated in \cite{wahl-thesis}. Each ribbon braid has an underlying permutation given by its underlying braid and the operation $\beta$ is the disjoint union of ribbon braids in the same way as for braids. In particular we can describe the twist of each ribbon as $t_i = \beta(e_{i-1},t,e_{n-i})$ where $t \in RB_1$ is the twist of the single ribbon. The `cabling' operation $\delta$ for ribbon braids is slightly more intricate, however, than that for braids. For a ribbon braid without any twists this is simple enough to describe and is the same as the standard cabling operation: the function $\delta_{n;k_1,\ldots,k_n}$ replaces the first ribbon with $k_1$ ribbons, and so on, before these are braided in blocks according to the underlying braid. If twists are involved, then we must take care to describe this carefully. First we describe the `Garside half-twist' \cite[Section 2]{garside}:
    \[
        \gamma_n = (\sigma_1 \sigma_2 \cdots \sigma_n)(\sigma_1 \sigma_2 \cdots \sigma_{n-1})\cdots(\sigma_1 \sigma_2)(\sigma_1).
    \]
This is the unique positive (meaning it can be written as a product of only $\sigma_i$) braid where any two strands of the braid cross exactly once. Its underlying permutation is the order-reversing permutation, meaning that the underlying permutation of its square (the full Garside twist) is the identity. Similarly the underlying permutation of each ribbon twist $t_i \in RB_n$ is simply the identity $e_n$. An important fact about the full Garside twist $\gamma_n^2$ is that it commutes with any other braid on $n$ strings, i.e., $\gamma_n^2 \sigma = \sigma \gamma_n^2$ for any $\sigma \in B_n$.

Now an arbitrary element in $RB_n$ can be described using the description of the elements as framed braids \cite{KS-framed}. Each element can be uniquely expressed in the form
    \[
        t_1^{m_1}\cdots t_n^{m_n}\sigma
    \]
where $m_i \in \mathbb{Z}$ and $\sigma \in B_n$; we interpret this as $t_i^{m_i}$ describing the number of full twists in ribbon $i$, after braiding the ribbons according to $\sigma$. For the full ribbon cabling operation $\delta$, we describe a general form:
    \[
        \delta^{RB}_{n;k_1\cdots k_n}(t_1^{m_1} \cdots t_n^{m_n}\sigma) = \beta(\underline{t^{m_1}},\ldots,\underline{t^{m_n}})\beta(\gamma_{k_1}^{2m_1},\ldots,\gamma_{k_n}^{2m_n})\delta^B_{n;k_1,\ldots,k_n}(\sigma).
    \]
In this expression, $\underline{t^{m_i}}$ represents $t^{m_i}$ repeated $k_i$ times.

Abstractly, when a twist is `duplicated' each new ribbon gains a full twist and then all of the ribbons in that block are braided via a full Garside twist. Physically, this corresponds to putting a full twist in a large wide ribbon before cutting it into $k$ new ribbons whilst preserving the endpoints, before shuffling the twists on each new ribbon to the top of the ribbon braid. A similar interpretation with many examples is described in \cite[Chapter 6]{yau_infinity_2021}. We do not check the axioms in detail here as they are largely similar to the braid groups or use the commutativity relations given above for the $t_i$'s and the Garside twists.
\end{example}

\begin{example}[(Action operad of cactus groups)]\label{ex:cactus-aop}
The operad of $n$-fruit cactus groups defined by Henriques and Kamnitzer in \cite{hk-cobound} has an action operad structure that we will discuss in \cref{sec:exex-cactus}.
\end{example}

\begin{example}[(Action operad from an abelian group)]\label{ex:abgp-aop}
Every abelian group $A$ gives rise to an action operad $A^{\bullet}$ as follows. The group $A^{\bullet}(n)$ is the direct sum of $n$ copies of $A$, $A^{n}$. The identity element is required to be $e \in A^{1}$, and the multiplication is defined by
  \[
    \mu((a_{1}, \ldots, a_{n}); \underline{b_1}, \ldots, \underline{b_n}) = (a_{1}+\underline{b_{1}}, a_{2} + \underline{b_{2}}, \ldots, a_{n} + \underline{b_{n}})
  \]
where $\underline{b_{i}}$ is the string $b_{i1}, \ldots, b_{ik_{i}}$, and $a_{i} + \underline{b_{i}}$ is
  \[
    a_{i} + b_{i1}, a_{i} + b_{i2}, \ldots, a_{i} + b_{ik_{i}}.
  \]
The map $\pi_n \colon A^\bullet(n) \rightarrow \Sigma_n$ is the zero map. 
\end{example}

The characterisation of action operads in terms of maps $\pi$, $\beta$, and $\delta$ as in the above \cref{thm:charAOp} allows us to more easily check for counterexamples, as we show below. Some of these, such as the cyclic groups, reflexive groups, and hyperoctahedral groups, do however form crossed simplicial groups which are discussed in relation to action operads in \cref{rem:crossed}.

\begin{nonex}[(Subgroups of symmetric groups)]\label{nonex:counterex1}
By \cref{prop:surjortriv}, the only action operad $\pi \colon \Lambda \to \Sigma$ for which the homomorphisms $\pi_n$ are injective but not surjective is the action operad of trivial groups. 
Thus there is no family of proper, nontrivial subgroups of the symmetric groups that admits an action operad structure.
In particular, the families of cyclic groups $\{ C_n \}$, reflexive groups $\Lambda(n) = C_2$ of \cite{Kra87}, and alternating groups $\{ A_n \}$ do not admit action operad structures.
\end{nonex}

\begin{nonex}[(Hyperoctahedral groups)]\label{nonex:counterex2}
In Example 2.28 of \cite{zhang-grp}, Zhang describes one way in which the sequence of hyperoctahedral groups $H_n = C_2 \wr \Sigma_n$ do not form an action operad. We clarify that counterexample here, while also describing another. The group $H_n$ can be described in many ways, but we will use the description of them as signed permutation matrices, i.e., invertible $n \times n$-matrices whose entries consist of $-1$, $0$, or $1$ and in which each row and column has exactly one non-zero entry, similar to the permutation matrices in \Cref{rem:perm_matrices}.

In order to describe the hyperoctahedral groups as an action operad, we could use \cref{thm:charAOp} and define maps $\pi$, $\beta$, and $\delta$. The obvious map $\pi_n \colon H_n \rightarrow \Sigma_n$ takes the `absolute value' of a signed permutation matrix, giving back the underlying permutation. We then define $\beta$ to be the block sum of signed permutation matrices in much the same way as for the symmetric groups.

For the maps $\delta$ there seem to be two sensible options to try. The first captures Zhang's counterexample by first taking $r_n$ to be the order-reversing signed permutation matrix where all entries are $-1$, i.e., the $n \times n$ matrix with $-1$ in every entry of the anti-diagonal. Then we define $\delta_{n;k_1,\ldots,k_n}(\sigma)$ to be the block sum $\beta(r_{k_1},\ldots,r_{k_n})$ acted on by the product of $r_n$ and $\sigma$. For example,
  \begin{align*}
    \delta_{3;2,1,3}\left(
    \begin{bmatrix}
      -1 & 0 & 0 \\
      0 & 0 & -1 \\
      0 & 1 & 0
      \end{bmatrix}
      \right)
      &=
      \left(
      r_3
      \begin{bmatrix}
      -1 & 0 & 0 \\
      0 & 0 & -1 \\
      0 & 1 & 0
      \end{bmatrix}
      \right)
      \ast
      \begin{bmatrix}
      r_2 & 0 & 0 \\
      0 & r_1 & 0 \\
      0 & 0 & r_3
      \end{bmatrix}\\
&=
    \begin{bmatrix}
      0 & -1 & 0 \\
      0 & 0 & 1 \\
      1 & 0 & 0
      \end{bmatrix}
      \ast
      \begin{bmatrix}
      r_2 & 0 & 0 \\
      0 & r_1 & 0 \\
      0 & 0 & r_3
      \end{bmatrix}\\
      &= \begin{bmatrix}
      0 & -r_1 & 0\\
      0 & 0 & r_3 \\
      r_2 & 0 & 0
      \end{bmatrix}.
    \end{align*}
This choice gives $\delta_{1;n}([-1]) = r_n$ as in \cite{zhang-grp}. Taking $\sigma = (\trans{2}{3};-1,1,-1)$, as above, we can show that Axiom \eqref{eq8} fails in \cref{thm:charAOp}.
The left-hand side of the axiom would be $\delta_{1;3}([-1])\beta(\sigma) = r_n \cdot \sigma = ((1 \,\, 3 \,\, 2); 1, -1, 1)$.
However, the right-hand side of the axiom would be $\beta(\sigma)\delta_{1;3}([-1]) = \sigma \cdot r_n = ((1 \,\, 2 \,\, 3);1,-1,1)$.
Clearly defining $\delta_{n;k_1,\ldots,k_n}(\sigma) = (\sigma \cdot r_n) \ast \beta(r_{k_1},\ldots,r_{k_n})$ would run into the same problem.

An alternative way of defining $\delta$ is to take $\delta_{n;k_1,\ldots,k_n}(\sigma) = \sigma \ast \beta(e_{k_1},\ldots,e_{k_n})$, without involving the order-reversing permutation $r_n$, having the effect of making $\delta_{1;n}([-1]) = -I_n$. This then does satisfy Axiom \eqref{eq8}, but fails Axiom \eqref{eq6} instead; working through the following counterexample shows this to be the case:
  \begin{align*}
    \delta_{3;2,1,3}\left(
    \begin{bmatrix}
    -1 & 0 & 0 \\
    0 & 0 & 1 \\
    0 & -1 & 0
    \end{bmatrix}
    \right)
    &\delta_{3;3,1,2}\left(
    \begin{bmatrix}
    0 & 0 & 1 \\
    0 & -1 & 0 \\
    -1 & 0 & 0
    \end{bmatrix}
    \right)\\
    &\neq
    \delta_{3;3,1,2}\left(
    \begin{bmatrix}
    -1 & 0 & 0 \\
    0 & 0 & 1 \\
    0 & -1 & 0
    \end{bmatrix}
    \begin{bmatrix}
    0 & 0 & 1 \\
    0 & -1 & 0 \\
    -1 & 0 & 0
    \end{bmatrix}
    \right).
  \end{align*}

While hyperoctahedral groups do not form an action operad in the sense of \cref{Defi:aop}, this family of groups does satisfy the conditions of what is in \cite[Section 3.2]{accg-actionclones} called a \emph{general} action operad. 
\end{nonex}

\begin{rem}[(Crossed Simplicial Groups)]\label{rem:crossed}
The crossed simplicial groups of both Krasauskas \cite{Kra87} and Fiedoriwicz and Loday \cite{FL91} are related to action operads in the following way. 
On objects, we define a functor $C \colon \mathbf{AOp} \rightarrow \mathbf{CSGrp}$ from the category of action operads to the category of crossed simplicial groups by defining $C(\Lambda)(n) = \Lambda_{n+1}$. We must check that the family of groups $\{\Lambda_n\}_{n \in \mathbb{N}}$ forms a simplcial set. The face and degeneracy maps are defined using the operadic composition of $\Lambda$ as in \cite[Construction~1.1]{Kra96} and \cite[Section~2]{ber-simplicial}. We can interpret these maps in terms of the map $\delta$ of \Cref{thm:charAOp}, where we consider $g \in \Lambda_n$ to act via $\pi_n(g) \in \Sigma_n$ on indices $\{0,\ldots,n-1\}$, as for the standard indexing for simplicial sets, instead of $\{1,\ldots,n\}$ :
    \begin{itemize}
        \item The face maps $d_{n;i} \colon \Lambda_{n+1} \rightarrow \Lambda_n$, $i \in \{0,\ldots,n\}$, are defined as
            \[
                d_{n;i}(g) = \delta^i_{n+1;\underline{1},0,\underline{1}}(g)
            \]
        where $0$ is in position $g^{-1}(i)$.
        \item The degeneracy map $s_{n;i} \colon \Lambda_{n+1} \rightarrow \Lambda_{n+2}$, $i \in \{0,\ldots,n\}$, are defined as
            \[
                s_{n;i}(g) = \delta^i_{n+1;\underline{1},2,\underline{1}(g)}
            \]
        where $2$ is in position $g^{-1}(i)$.
    \end{itemize}
Checking the simplicial identities amounts to repeated use of Axiom \eqref{eq7}. For example, taking $i \leq j$, we require that $d_jd_i = d_id_{j+1}$. In terms of our maps this is the requirement that
    \[
        \delta^j_{n;\underline{1},0,\underline{1}}(\delta^i_{n;\underline{1},0,\underline{1}}(g)) = \delta^i_{n;\underline{1},0,\underline{1}}(\delta^{j+1}_{n;\underline{1},0,\underline{1}}(g)).
    \]
In terms of Axiom \eqref{eq7} this is found to be the case by taking $m_{g^{-1}(i)}=0$, $p_{g^{-1}(i)} = \emptyset$, and $P_{g^{-1}(i)}=0$ on the left hand side and $m_{{g^{-1}(j+1)}} = 0$, $p_{g^{-1}(j+1)} = \emptyset$, and $P_{g^{-1}(j+1)}=0$ on the right hand side. For example, the left-most function is being interpreted here as $\delta^j_{n;\underline{1},\emptyset,\underline{1},0,\underline{1}}$. Since $i \leq j$, both of these are equal to the same element
    \[
        \delta_{n+1;\underline{1},0,\underline{1},0,\underline{1}}(g),
    \]
where each $0$ is in position $g^{-1}(i) = (\delta_{n;\underline{1},0,\underline{1}}^j(g))^{-1}(i)$ and $g^{-1}(j+1) = (\delta_{n;\underline{1},0,\underline{1}}^{i})(g)^{-1}(j)$, from which the identity follows. Similar arguments gives the remaining simplicial identities, again using Axiom \eqref{eq7} and Axiom \eqref{eq5}.

Following \cite[Definition~1.3]{Kra87} or \cite[Proposition~1.7]{FL91}, it remains to check the following axioms:
    \begin{itemize}
        \item the maps $\pi^{C(\Lambda)}_n = \pi^\Lambda_{n+1} \colon \Lambda_{n+1} \rightarrow \Sigma_{n+1}$ constitute a simplicial map,
        \item each $\pi^{C(\Lambda)}_n$ is a homomorphism,
        \item $d_{n;i}(gh) = d_{n;i}(g)d_{n;g^{-1}(i)}(h)$, 
        \item $s_{n;i}(gh) = s_{n;i}(g)d_{n;g^{-1}(i)}(h)$.
    \end{itemize}
The latter two axioms follow immediately from Axiom \eqref{eq6}, the second is immediate, while the first requires the commutativity of the following diagrams, where in the right-hand diagram $\delta_i$ and $\delta_{g^{-1}(i)}$ are the face maps in the simplicial category $\Delta$.
        \[
            \xy
                (0,0)*+{[n+1]}="00";
                (20,0)*+{[n]}="10";
                (0,-15)*+{[n+1]}="01";
                (20,-15)*+{[n]}="11";
                {\ar^{\sigma_{g^{-1}(i)}} "00" ; "10"};
                {\ar^{g} "10" ; "11"};
                {\ar_{s_{i}(g)} "00" ;  "01"};
                {\ar_{\sigma_{i}} "01" ; "11"};
                (50,0)*+{[n-1]}="00";
                (70,0)*+{[n]}="10";
                (50,-15)*+{[n-1]}="01";
                (70,-15)*+{[n]}="11";
                {\ar^{\delta_{g^{-1}(i)}} "00" ; "10"};
                {\ar^{g} "10" ; "11"};
                {\ar_{d_{i}(g)} "00" ;  "01"};
                {\ar_{\delta_{i}} "01" ; "11"};
            \endxy
        \]
Both of these axioms follow automatically from Axiom \eqref{eq4} and the fact that $\Sigma$ is a crossed simplicial group and hence already satisfies these axioms. The case for the degeneracy maps is shown below, with the argument for the face maps being almost identical.
    \begin{align*}
        \sigma_i(\pi_{n+2}(s_i^\Lambda(g))) &= \sigma_i(\pi_{n+2}(\delta^{i,\Lambda}_{n+1;\underline{1},2,\underline{1}}(g))) \\
        &= \sigma_i(\delta^{i,\Sigma}_{n+1;\underline{1},2,\underline{1}}(\pi_{n+1}(g))) \\
        &= \sigma_i(s_i^{\Sigma}(\pi_{n+1}(g))) \\
        &= \pi_{n+1}(g)(\sigma_{\pi_{n+1}(g)^{-1}(i)})
    \end{align*}

On morphisms, for a map of action operads $f \colon \Lambda \rightarrow \Lambda'$, we define $C(f)_n = f_{n+1}$. From the axioms in \Cref{Defi:mapaop} we automatically have that $\pi^{\Lambda'}_n \circ f_n = \pi^{\Lambda}_n$ for each $n$, $C(f)$ is a simplicial map since $f$ is a map of operads, and each $f_n$ is a homomorphism. Functoriality follows directly.

This functor is neither faithful nor conservative due to the removal of the group $\Lambda_0$. 
Zhang \cite[Section 2.4]{zhang-grp} also describes this connection with action operads in some detail, while Yoshida \cite{yoshida2018group} connects action operads with the related notion of crossed interval groups \cite{bm-crossed}. Semidetnov's \emph{operadic} crossed simplicial groups likely correspond to action operads where $\Lambda_0$ is the trivial group \cite[Definition 2.4]{semidetnov}.
\end{rem}

\begin{rem}[(Cloning Systems)]\label{rem:cloning}
Another algebraic structure that has a close relationship with action operads is that of a cloning system \cite{zaremsky2017user}. A cloning system consists of an $\mathbb{N}$-indexed family of groups $(G_n)_{n \in \mathbb{N}}$ with various structure maps, including directed system morphisms $G_m \rightarrow G_n$ for each $m \leq n$, representation maps $G_n \rightarrow \Sigma_n$ for each $n$, and cloning maps $G_n \rightarrow G_{n+1}$ for each $n$ which closely mirror the duplication maps $\delta$ in \cref{Defi:aop}.

Witzel and Zerensky \cite{wz-cloning} define these in the context of studying generalized Thompson groups for families of groups, while Thumann \cite{thumann2017operad} also considers Thompson-like groups arising from the fundamental groups of categories associated to braided operads. These two approaches to studying Thompson groups are refined and related in \cite{accg-actionclones}, wherein it is shown that action operads directly correspond to restricted operadic cloning systems and restricted cloning PROs.
\end{rem}

\section{Action Operads as Extensions}\label{sec:extension}

In this short section, we situate action operads between operads in the category of groups and symmetric operads.
We prove two main results.
In \cref{prop:Z}, we prove that operads in the category of groups are precisely the same as action operads for which the homomorphisms $\pi_n$ are all the zero map; these are called ``non-crossed group operads'' in \cite{zhang-grp}.
We then turn to studying kernels, images, and short exact sequences of action operads.
We finally prove, in \cref{cor:extension}, that every action operad with surjective $\pi_n$'s can be expressed as an extension of the action operad $\Sigma$ by an operad in the category of groups.

\begin{rem}[(Operads in the category of groups)]\label{rem:op-in-grp}
The category $\mb{Grp}$ of groups and group homomorphisms is symmetric monoidal using the cartesian product of groups.
Thus we can form the category of operads in the category of groups, denoted $\mb{Op}(\mb{Grp})$, as in \cref{rem:V-and-coll}.
The objects of this category are operads $P$ in $\mb{Sets}$ with the additional data of a group structure on each $P(n)$ such that operadic multiplication is a group homomorphism and $\id = e_1$ in $P(1)$; morphisms $f \colon P \to Q$ are those maps of operads such that $f_n \colon P(n) \to Q(n)$ is a group homomorphism for each $n$.
\end{rem}

\begin{prop}\label[prop]{prop:Z-obj}
Let $P$ be an operad in $\mb{Grp}$. Then there is an action operad, denoted $Z(P)$, with
\begin{itemize}
\item $Z(P)(n) = P(n)$,
\item the same operadic multiplication as $P$, and
\item each $\pi_n \colon P(n) \to \Sigma_n$ the zero map.
\end{itemize}
Furthermore, if $\Lambda$ is an action operad for which each $\pi_n$ is the zero map, then the groups $\Lambda(n)$ define an operad in $\mb{Grp}$ using the operadic multiplication of $\Lambda$.
\end{prop}
\begin{proof}
It is easy to verify, using \cref{eqn:ao_axiom} of \cref{Defi:aop}, that the operadic multiplication $\mu$ of an action operad is a group homomorphism if and only if $\pi_{n}$ is zero for all $n$.
\end{proof}

\begin{prop}\label{prop:Z}
The assignment on objects $P \mapsto Z(P)$ extends to a functor 
\[
Z \colon  \mb{Op}(\mb{Grp}) \rightarrow \mb{AOp}.
\]
This functor is full, faithful, and its image at the level of objects is precisely the collection of action operads $\Lambda$ for which each $\pi_n$ is the zero map.
\end{prop}
\begin{proof}
Let $f \colon P \to Q$ be a morphism in $\mb{Op}(\mb{Grp})$, meaning that $f$ consists of a family of group homomorphisms $f_n \colon P(n) \to Q(n)$ that define a map of operads.
Define $Z(f)_n = f_n$. We must check that these functions define a map of action operads; functoriality will follow immediately, as composition and identities in both $\mb{Op}(\mb{Grp})$ and $\mb{AOp}$ are given levelwise.
Since each action operad $Z(P)$ has $\pi_n$ the zero map for all $n$, the first numbered axiom in \cref{Defi:mapaop} is satisfied trivially.
The second numbered axiom follows from the definition of a morphism in $\mb{Op}(\mb{Grp})$.
This completes the proof that $Z(f)$ is a map of action operads, and the same reasoning shows that every map of action operads $g \colon Z(P) \to Z(Q)$ is $Z(g')$ for a unique $g' \colon P \to Q$ in $\mb{Op}(\mb{Grp})$, thus $Z$ is full and faithful.
\end{proof}

\begin{prop}\label[prop]{prop:im-and-ker-aop}
Let $f \colon \Lambda \to \Lambda'$ be a map of action operads.
\begin{enumerate}
\item The groups
  \[
    \mathrm{Ker}\,f_n = \{x \in \Lambda(n)~\colon~f(x) = e_{n} \}
  \]
form an action operad for which the inclusion $\mathrm{Ker}\,f \hookrightarrow \Lambda$ is a map of action operads.
\item The groups
  \[
    \mathrm{Im}\,f_n = \{f(x)~\colon~x \in \Lambda(n)\}
  \]
form an action operad for which the inclusion $\mathrm{Im}\,\pi \hookrightarrow \Lambda'$ is a map of action operads.
\end{enumerate}
\end{prop}
\begin{proof}
For the first part, we start by defining $\pi_n^{\mathrm{Ker}\,f} \colon \mathrm{Ker}\,f_n \to \Sigma_n$ as the composite group homomorphism
\[
\mathrm{Ker}\,f_n \hookrightarrow \Lambda(n) \stackrel{\pi_n^{\Lambda}}{\to} \Sigma_n.
\]
Since $\pi^{\Lambda} = \pi^{\Lambda'} \circ f$, the composites $\pi_n^{\mathrm{Ker}\,f}$ are all the zero map.
Next we verify that the subgroups $\mathrm{Ker}\,f_n$ are closed under operadic multiplication. 
Let $y \in \mathrm{Ker}\,f_n$ and $x_i \in \mathrm{Ker}\,f_{k_i}$ for $i=1, \ldots, n$.
Then
\begin{align*}
f \big( \mu(y; x_1, \ldots, x_n) \big) & = \mu\big( f(y); f(x_1), \ldots, f(x_n) \big) \\
& = \mu(e_n; e_{k_1}, \ldots, e_{k_n} ) \\ 
& = e_{k_1 + \cdots + k_n}
\end{align*}
by the assumption that $f$ is a map of operads, that $y$ and each $x_i$ is in the kernel, and \cref{lem:calclem}, showing that the kernel subgroups are closed under operadic multiplication.
The operadic identity $\id \in \Lambda(1)$ is an element of $\mathrm{Ker}\,f_1$ because it is equal to $e_1$ by \cref{lem:calclem}.
Thus the groups $\mathrm{Ker}\,f_n$ form a sub-operad of $\Lambda$, and the action operad axiom \cref{eqn:ao_axiom} of \cref{Defi:aop} for $\mathrm{Ker}\,f$ follows immediately from the same axiom for $\Lambda$.
This completes the proof of the first claim, and in fact shows, via \cref{prop:Z-obj}, that these groups constitute an operad in $\mb{Grp}$.

For the second part, we start by defining $\pi_n^{\mathrm{Im}\,f} \colon \mathrm{Im}\,f_n \to \Sigma_n$ as the composite group homomorphism
\[
\mathrm{Im}\,f_n \hookrightarrow \Lambda'(n) \stackrel{\pi_n^{\Lambda'}}{\to} \Sigma_n.
\]
These subgroups are closed under operadic multiplication in $\Lambda'$ using that $f$ is a map of action operads.
The operadic identity $\id \in \Lambda'(1)$ is an element of $\mathrm{Im}\,f_1$ because it is equal to $e_1$ by \cref{lem:calclem}.
This completes the proof that the groups $\mathrm{Im}\,f_n$ form a sub-operad of $\Lambda'$, and the action operad axiom \cref{eqn:ao_axiom} of \cref{Defi:aop} for $\mathrm{Im}\,f$ follows immediately from the same axiom for $\Lambda'$, finishing the proof of the second claim.
\end{proof}

\begin{example}[(Action operads of pure braids, pure ribbon braids)]\label[example]{ex:pure-braids}
The $n$th pure braid group, $PB_n$, is defined as the kernel of the homomorphism $\pi_n \colon B_n \to \Sigma_n$, or equivalently as the subgroup of the $n$th braid group consisting of those braids with underlying permutation the identity.
\cref{prop:im-and-ker-aop} gives a simple proof that the pure braid groups form an operad in the category of groups.
Similarly, the pure ribbon braid group, $PRB_n$, is defined as the kernel of $\pi_n \colon RB_n \to \Sigma_n$, and these groups also constitute an operad in the category of groups.

\end{example}

\begin{rem}[(Kernels and images of $\pi$)]\label{rem:im-and-ker-Sigma}
We note that if $(\Lambda, \pi)$ is an action operad, then we can apply the results of \cref{prop:im-and-ker-aop} to $\pi$ by \cref{prop:pi-in-aop}.
The action operad $\mathrm{Ker}\,\pi_n$ will then be an operad in groups, and the action operad $\mathrm{Im}\,\pi_n$ will be a sub-action operad of $\Sigma$.
By \cref{prop:surjortriv}, this means that the action operad $\mathrm{Im}\,\pi_n$ is either $\Sigma$ or the trivial action operad $T$ (\cref{example:aop-triv}).
\end{rem}

\begin{Defi}\label[Defi]{Defi:ses-aop}
A \emph{short exact sequence of action operads} consists of action operads $\Lambda_1, \Lambda_2, \Lambda_3$ and maps of action operads $f \colon \Lambda_1 \to \Lambda_2$, $g \colon \Lambda_2 \to \Lambda_3$ such that
\begin{itemize}
\item the action operad $\mathrm{Ker}\,f$ is the trivial action operad $T$,
\item the action operad $\mathrm{Im}\,f$ is the action operad $\mathrm{Ker}\,g$, and
\item the action operad $\mathrm{Im}\,g$ is the action operad $\Lambda_3$.
\end{itemize}
We denote such a short exact sequence as
\[
T \to \Lambda_1 \stackrel{f}{\to} \Lambda_2 \stackrel{g}{\to} \Lambda_3 \to T,
\]
and we say that a short exact sequence exhibits \emph{$\Lambda_2$ as an extension of $\Lambda_3$ by $\Lambda_1$}.
\end{Defi}

The following corollary puts \cref{rem:im-and-ker-Sigma} into the language of short exact sequences and extensions.

\begin{cor}\label[cor]{cor:extension}
Let $(\Lambda, \pi)$ be an action operad and assume that every homomorphism $\pi_n \colon \Lambda(n) \to \Sigma_n$ is surjective.
Then there is a short exact sequence of action operads
  \[
    T \rightarrow \mathrm{Ker}\,\pi \hookrightarrow \Lambda \stackrel{\pi}{\longrightarrow} \Sigma \rightarrow T.
  \]
  In particular, every action operad $\Lambda$ is either an operad in groups or an extension of $\Sigma$ by an operad in groups.
\end{cor}

\section{Presentations for Action Operads}\label{sec:pres-aop}

This section details how to provide presentations for action operads using the theory of locally finitely presentable (lfp) categories.
We refer the reader to \cite{ar} for a full treatment of lfp categories.
Our treatment here diverges slightly from how one might give presentations for symmetric operads because it is necessary to build in the underlying permutation map $\pi \colon \Lambda \to \Sigma$ from the beginning.
In the theory of plain operads, the starting point is a collection: sets $\{ P(n) \}$ indexed by the natural numbers.
In the theory of symmetric operads, these are enhanced to symmetric collections: sets $\{ P(n) \}$ indexed by the natural numbers, together with a right action $P(n) \times \Sigma_n \to P(n)$.
Our analogue of collections (see \cref{Defi:coll-over-SS}) are now sets $P(n)$, indexed by natural numbers, equipped with functions $\pi_n \colon P(n) \to \Sigma_n$.
Thus the natural notion of the arity of an element in an action operad is not a natural number $n$, but rather a pair $(n, \sigma)$ where $\sigma \in \Sigma_n$.

\begin{Defi}[(Collections over $\SS$)]\label[Defi]{Defi:coll-over-SS}
  Let $\SS$ be the set that is the disjoint union of the underlying sets of all the symmetric groups. Then $\mb{Sets}/\SS$ is the slice category over $\SS$ with objects $(X,f)$ where $X$ is a set and $f \colon X \rightarrow \SS$ and morphisms $(X_{1}, f_{1}) \rightarrow (X_{2}, f_{2})$ are those functions $g \colon X_{1} \rightarrow X_{2}$ such that $f_{1} = f_{2}g$. We call an object $(X,f)$ a \textit{collection over $\SS$}, and say that an element $x \in X$ has \emph{underlying permutation $\sigma$} if $f(x) = \sigma$.
\end{Defi}

\begin{nota}\label[nota]{nota:Xofsigma}
If $(X, f)$ is a collection over $\SS$, we write $X(\sigma)$ for $f^{-1}(\sigma)$. In other words, $X(\sigma)$ is the set of elements of $X$ with underlying permutation $\sigma \in \Sigma_n \subseteq \SS$.
\end{nota}

\begin{thm}\label[thm]{thm:aop-lfp}
The category $\mb{AOp}$ of action operads is a variety of $\SS$-sorted finitary algebras, and therefore is a finitary monadic category over $\mb{Sets}/\SS$. In particular, $\mb{AOp}$ is locally finitely presentable.
\end{thm}
\begin{proof}
In order to prove that $\mb{AOp}$ is a variety of $\SS$-sorted finitary algebras (henceforth shortened to \emph{$\mb{AOp}$ is a variety}), we must define a set $\mathcal{O}$ of operation symbols and a set of equations $E$ such that action operads are the 
\begin{itemize}
\item collections $(X, f)$ over $\SS$, 
\item equipped with functions
\[
X(\theta) \colon X(\sigma_1) \times \cdots \times X(\sigma_n) \to X(\sigma)
\]
for each operation symbol $\theta \in \mathcal{O}$ of type $\theta \colon \sigma_1, \ldots, \sigma_n \to \sigma$,
\item satisfying the equations in $E$.
\end{itemize}
The set $\mathcal{O}$ of operation symbols is defined to have the elements given below.
\begin{enumerate}
\item For each $\sigma, \tau \in \Sigma_n$, we define an operation symbol
\[
\star[\sigma, \tau] \colon \sigma, \tau \to \sigma \tau,
\]
where the target $\sigma \tau$ is the product of these permutations in $\Sigma_n$.
\item For each natural number $n$, we define an operation symbol
\[
U_n \colon \to e_n,
\]
where the source is the empty list of permutations and the target is the identity element $e_n \in \Sigma_n$.
\item For each $\sigma \in \Sigma_n$, we define an operation symbol
\[
i[\sigma] \colon \sigma \to \sigma^{-1}.
\]
\item Let $\mu$ denote the operadic multiplication in the operad of symmetric groups, $\Sigma$, from \cref{ex:Sigma}. For each $\sigma \in \Sigma_n$ and $\tau_i \in \Sigma_{k_i}$ for $i=1, \ldots, n$, we define an operation symbol $\theta[\sigma; \tau_i]$ of type
\[
\theta[\sigma; \tau_i] \colon \sigma, \tau_1, \ldots, \tau_n \to \mu(\sigma; \tau_1, \ldots, \tau_n).
\]
\end{enumerate}
The set $E$ of equations is defined to have the elements below.
\begin{enumerate}
\item We write $x \star y$ for $\star[\sigma, \tau](x,y)$, where $x$ is a variable of type $\sigma$ and $y$ is a variable of type $\tau$. For each triple $\rho, \sigma, \tau \in \Sigma_n$, there is an equation
\[
(x \star y) \star z = x \star (y \star z).
\]
\item An $\mathcal{O}$-algebra $X$ has, for each $n$, an element $u_n \in X(e_n)$ given by $U_n$. For each $\sigma \in \Sigma_n$, there are equations
\begin{align*}
u_n \star x & = x,\\
x \star u_n & = x.
\end{align*}
\item We write $x^{-1}$ for $i[\sigma](x)$, where $x$ is a variable of type $\sigma$. For each $\sigma \in \Sigma_n$, there are equations
\begin{align*}
x^{-1} \star x & = u_n,\\
x \star x^{-1} & = u_n.
\end{align*}
\item We write $\theta(x; y_1, \ldots, y_n)$ for $\theta[\sigma; \tau_i](x, y_1, \ldots, y_n)$, where $x$ is a variable of type $\sigma \in \Sigma_n$ and for each $i= 1, \ldots, n$ $y_i$ is a variable of type $\tau_i \in \Sigma_{k_i}$. Then for each 
\begin{itemize}
\item $\rho \in \Sigma_n$;
\item $\sigma_i \in \Sigma_{k_i}$, for $i= 1, \ldots, n$; and
\item $\tau_{i,j} \in \Sigma_{h_{i,j}}$, for $i=1, \ldots, n$ and $j=1, \ldots, k_i$;
\end{itemize}
there is an equation
\begin{multline*}
\theta\big( \theta(x; y_1, \ldots, y_n); z_{1,1}, \ldots, z_{1, k_1}, \ldots, z_{n, k_n} \big) = \\ \theta\big( x; \theta(y_1; z_{1,1}, \ldots, z_{1, k_1}), \ldots, \theta(y_n; z_{n, 1}, \ldots, z_{n, k_n}) \big).
\end{multline*}
\item For each $\sigma \in \Sigma_n$, there are equations
\begin{align*}
\theta(u_1; x)& = x, \\
\theta(x; u_1, \ldots, u_1) & = x.
\end{align*}
\item Let $x$ be a variable of type $\sigma \in \Sigma_i$, $y_i$ be a variable of type $\tau_i \in \Sigma_{k_i}$ for $i=1, \ldots, n$, $x'$ be a variable of type $\sigma' \in \Sigma_n$, and $y_i'$ be a variable of type $\tau_i' \in \Sigma_{\sigma^{-1}(i)}$. Then for each such choice of permutations, there is an equation
\[
\theta(x'; y_1', \ldots, y_n') \star \theta(x; y_1, \ldots, y_n) = \theta(x' \star x; y'_{\sigma(1)} \star y_1, \ldots, y'_{\sigma(n)} \star y_n).
\]
\end{enumerate}

The category of $\mathcal{O}$-algebras satisfying the equations in $E$ is isomorphic to $\mb{AOp}$ as follows. Given such an algebra $(X, f)$, define an action operad $(\Lambda^X, \pi)$ by defining 
\[
\Lambda^X(n) = \coprod_{\sigma \in \Sigma_n} X(\sigma)
\]
and defining $\pi_n \colon X(n) \to \Sigma_n$ to be $f$ restricted to $\Lambda^X(n) \subseteq X$. Each $X(n)$ is a group using $\star$ as its multiplication and $u_n$ as its identity element, and $\pi_n$ is a homomorphism since $\pi_n(x \star y) = \sigma \tau$ when $x \in X(\sigma), y \in X(\tau)$ by the definition of the source and target of $\star$. The operadic multiplication is given by the operation symbols $\theta[\sigma; \tau_i]$, and equations 4 and 5 in $E$ are the operadic associativity and unit axioms, using the first part of \cref{lem:calclem}. The additional action operad axiom is equation 6 in $E$. A morphism of $\mathcal{O}$-algebras is easily seen to define a map of action operads, and these assignments are an isomorphism between the category $\mb{AOp}$ and the category of $\mathcal{O}$-algebras satisfying the equations in $E$.
This completes the proof that $\mb{AOp}$ is the variety defined by $\mathcal{O}$ and $E$. It is therefore a finitary monadic category over $\mb{Sets}/\SS$ by \cite[Thm~3.18]{ar} and locally finitely presentable by \cite[Cor~3.7]{ar}.
\end{proof}

For our purposes, the most important consequence of \cref{thm:aop-lfp} is that we can freely generate an action operad from a collection over $\SS$, as stated below.

\begin{cor}\label[cor]{cor:free-aop-fun}
The underlying collection functor $U \colon \mb{AOp} \to  \mb{Sets}/\SS$ has a left adjoint $F \colon \mb{Sets}/\SS \rightarrow \mb{AOp}$, the free action operad functor.
\end{cor}

\begin{Defi}[(Presentation for action operads)]\label[Defi]{Defi:pres-aop}
  A \textit{presentation} for an action operad $\Lambda$ consists of
  \begin{itemize}
    \item a pair of collections over $\SS$ denoted $\mathbf{g}, \mathbf{r}$,
    \item a pair of maps $s_{1}, s_{2} \colon F\mathbf{r} \rightarrow F\mathbf{g}$ between the associated free action operads, and
    \item a map $p \colon F\mathbf{g} \rightarrow \Lambda$ of action operads exhibiting $\Lambda$ as the coequalizer of $s_{1},s_{2}$.
  \end{itemize}
\end{Defi}

\begin{example}[(The presentation for $\Sigma$)]\label[example]{ex:sigma-pres}
Here we explicitly give a presentation for the action operad of symmetric groups.
Recall that the symmetric group $\Sigma_n$ has a presentation, as a \emph{group}, with
\begin{itemize}
\item generators $\sigma_{1;n}, \ldots, \sigma_{n-1;n}$ and
\item relations
\begin{enumerate}
\item $\sigma_{i;n}^2 = e_n$ for all $i$,
\item $\sigma_{i;n} \sigma_{j;n} = \sigma_{j;n} \sigma_{i;n}$ for $i, j$ satisfying $|i-j| \geq 2$, and
\item $\sigma_{i;n} \sigma_{i+1;n} \sigma_{i;n} = \sigma_{i+1;n} \sigma_{i;n} \sigma_{i+1;n}$ for $1 \leq i < n-1$.
\end{enumerate}
\end{itemize}
Note the nonstandard naming of the generators as $\sigma_{i;n}$ instead of merely $\sigma_i$. 
We have included this additional information in our generators as it is necessary to distinguish between the generator of $\Sigma_2$ traditionally denoted $\sigma_1$ and the generator of $\Sigma_3$ traditionally denoted with the same notation, for example. 

Utilizing the operad structure, we notice that
\[
\sigma_{i;n} = \beta(e_{i-1}, \sigma_{1;2}, e_{n-i-1}).
\]
Furthermore, the second relation above is a consequence of this expression, the fact that $\beta$ is a group homomorphism, and the second part of \cref{lem:calclem}.
Thus as an action operad, $\Sigma$ is generated by the single element $\sigma_{1;2} \in \Sigma_2$.

We now define the relations for this presentation of $\Sigma$ as an action operad.
The first relation is 
\begin{equation}\label{eq:rel1}
\sigma^2 = e_2.
\end{equation}
The second relation is
\begin{equation}\label{eq:rel2}
\mu(\sigma; e_1, e_2) = \mu(e_2; e_1, \sigma) \cdot \mu(e_2; \sigma, e_1).
\end{equation}
We therefore claim that $\Sigma$ has a presentation given by
\begin{itemize}
\item $\mathbf{g} = \{ \sigma \}$, defined as a collection over $\SS$ by the function sending $\sigma$ to $(1 \ 2) \in \Sigma_2$;
\item $\mathbf{r} = \{ \rho_1, \rho_2 \}$, defined as a collection over $\SS$ by the function sending $\rho_1$ to $e_2 \in \Sigma_2$ and $\rho_2$ to $(1 \ 3 \ 2) \in \Sigma_3$;
\item $s_1 \colon F\mathbf{r} \to F\mathbf{g}$ defined uniquely by requiring
\begin{align*}
s_1(\rho_1) & = \sigma^2,\\
s_1(\rho_2) & = \mu(\sigma; e_1, e_2);
\end{align*}
\item and $s_2 \colon F\mathbf{r} \to F\mathbf{g}$ defined uniquely by requiring
\begin{align*}
s_2(\rho_1) & = e_2,\\
s_2(\rho_2) & = \mu(e_2; e_1, \sigma) \cdot \mu(e_2; \sigma, e_1).
\end{align*}
\end{itemize}
We note that both $s_1, s_2$ map $\rho_1$ to $(1 \ 2)(1 \ 2) = e_2 \in \Sigma_2$, and both $s_1, s_2$ map $\rho_2$ to $(1 \ 3 \ 2) = (2 \ 3)(1 \ 2)$, thus defining maps of collections over $\SS$.

In order to prove that the above is a presentation for $\Sigma$, we must define a map of action operads $t \colon F\mathbf{g} \to \Sigma$ that exhibits $\Sigma$ as the coequalizer of $s_1, s_2$.
Define $t$ by requiring $t(\sigma) = (1 \ 2)$.
The calculations above prove that $t \circ s_1 = t \circ s_2$, and now we must prove that $t$ is the universal map of action operads coequalizing $s_1, s_2$.
Let $\Lambda$ be an action operad, and $f \colon F\mathbf{g} \to \Lambda$ a map of action operads such that $f \circ s_1 = f \circ s_2$.
We construct a unique map of action operads $\tilde{f} \colon \Sigma \to \Lambda$ such that $f = \tilde{f} \circ t$.
If such an $\tilde{f}$ exists, it must map the transposition $\sigma_{1;2} = (1 \ 2) \in \Sigma_2$ to $f(\sigma)$.
Since each other generator (of $\Sigma_n$ as a group) $\sigma_{i;n}$ is the image
\begin{equation}\label{eq:sig-in}
\sigma_{i;n} = \beta(e_{i-1}, \sigma_{1;2}, e_{n-i-1}) = \mu(e_3; e_{i-1}, \sigma_{1;2}, e_{n-i-1})
\end{equation}
of $\sigma_{1;2}$ under an operadic multiplication, any map of action operads $\tilde{f}$ satisfying $f = \tilde{f} \circ t$ is unique if it exists, and $\tilde{f}_n$ must be 
defined on the generators of $\Sigma_n$ using \cref{eq:sig-in} by
\begin{align*}
\tilde{f}_n(\sigma_{i;n}) & = \tilde{f} \big( \mu(e_3; e_{i-1}, \sigma_{1;2}, e_{n-i-1}) \big) \\
& = \mu\big( \tilde{f}_n(e_3); \tilde{f}_n(e_{i-1}), \tilde{f}_n(\sigma_{1;2}), \tilde{f}_n(e_{n-i-1}) \big) \\
& =  \mu(e_3; e_{i-1}, f(\sigma), e_{n-i-1}).
\end{align*}
In order to show that the formula 
\[
\tilde{f}_n(\sigma_{i;n}) = \mu(e_3; e_{i-1}, f(\sigma), e_{n-i-1})
\]
defines a unique homomorphism $\tilde{f}_n \colon \Sigma_n \to \Lambda(n)$, we must check that it respects the relations in the presentation of $\Sigma_n$ given above.
We only check the third axiom, and only in the case $i=1, n=3$; the rest we leave as a simple exercise for the reader.
In order to verify that $\tilde{f}_3$ respects this relation, we must show that
\begin{equation}\label{eq:3term}
\tilde{f}_3( \sigma_{1;3}) \ \tilde{f}_3(\sigma_{2;3} ) \ \tilde{f}_3(\sigma_{1;3} ) = \tilde{f}_3(\sigma_{2;3} ) \ \tilde{f}_3(\sigma_{1;3} ) \ \tilde{f}_3( \sigma_{2;3}).
\end{equation}
By \cref{lem:e0-unit}, the left side of the above is 
\begin{align*}
\mu\big( e_2; f(\sigma), e_1 \big) \mu\big( e_2; e_1, f(\sigma) \big) \mu\big( e_2; f(\sigma), e_1 \big) & = f \big( \mu(e_2; \sigma, e_1)\mu(e_2;e_1,  \sigma)\mu(e_2; \sigma, e_1) \big).
\end{align*}
Since $f$ is a map of action operads and coequalizes $s_1, s_2$, we obtain
\begin{align*}
f \big( \mu(e_2; \sigma, e_1)\mu(e_2;e_1,  \sigma)\mu(e_2; \sigma, e_1) \big) & = f\big( \mu(e_2; \sigma, e_1) \mu(\sigma; e_1, e_2) \big)
\end{align*}
by the equality $fs_1(\rho_2) = fs_2(\rho_2)$.
Finally, the action operad axiom shows that
\[
f\big( \mu(e_2; \sigma, e_1) \mu(\sigma; e_1, e_2) \big) = f\big( \mu(\sigma; e_1, \sigma) \big).
\]
A similar argument shows that the right side of \cref{eq:3term} is equal to 
\[
f \big(\mu(e_2;e_1, \sigma) \mu(e_2; \sigma, e_1)\mu(e_2;e_1,  \sigma) \big),
\]
and once by coequalizing $s_1, s_2$ is therefore $f\big( \mu(\sigma; e_1, \sigma) \big)$.
We have now verified that $\tilde{f}_n$ respects the relations for the presentation of $\Sigma_n$, and therefore defines a unique group homomorphism $\tilde{f}_n \colon \Sigma_n \to \Lambda(n)$.

By \cref{thm:charAOp}, to show that the homomorphisms $\tilde{f}_n$ defined by
\[
\tilde{f}_n(\sigma_{i;n}) = \mu(e_3; e_{i-1}, f(\sigma), e_{n-i-1})
\]
give a map of action operads, it suffices to check that they commute with the operations $\delta, \beta$ and preserve underlying permutations.
We sketch this proof below, and leave the routine details to the reader.
\begin{itemize}
\item First, check that the equality
\[
\tilde{f}_N\big( \beta(\tau_1, \ldots, \tau_k) \big) = \beta\big( \tilde{f}_{n_1}(\tau_1), \ldots, \tilde{f}_{n_k}(\tau_k) \big)
\]
follows from the special case when all the $\tau_i$ are identity elements except one, and that $\tau_i$ is $\sigma$; this reduction uses that $\beta$ and the $\tilde{f}_{n_k}$'s are homomorphisms. Check that special case using Axiom \eqref{eq3}.
\item Second, check that the equality
\[
\tilde{f}_N\big( \delta_{n; k_1, \ldots, k_n}(\tau) \big) = \delta_{n; k_1, \ldots, k_n}\big( \tilde{f}_n(\tau) \big) 
\]
follows once it is verified in the special cases that $\tau = \sigma_{i;n}$ for some $i$; this reduction uses that $\tilde{f}_n$ is a homomorphism and Axiom \eqref{eq6}.
Then show, using Axiom \eqref{eq9}, that it suffices to check the case of $\sigma_{1;2}$ only.
\item Third, we check
\[
\tilde{f}_{i+j}\big( \delta_{2; i,j}(\sigma_{1;2}) \big) = \delta_{2; i,j}\big( \tilde{f}_2(\sigma_{1;2}) \big) 
\]
by induction. Fixing $i$ and inducting on $j$, we start on the right and compute
\begin{align*}
\delta_{2; i, j+1} \tilde{f}_2(\sigma_{1;2}) & = \mu\big( \tilde{f}_2(\sigma_{1;2}); e_i, e_{j+1} \big)  \\
& = \mu \big( f(\sigma); e_i, e_{j+1} \big) \\
& = \mu \Big( \mu\big( f(\sigma); e_1, e_2 \big); e_i, e_1, e_j \Big) \\
& = \mu \Big( f\big(\mu(\sigma; e_1, e_2) \big); e_i, e_1, e_j \Big) \\
& = \mu \Big( f\big( \mu(e_2; e_1, \sigma) \mu(e_2; \sigma, e_1) \big); e_i, e_1, e_j \Big) \\
& = \mu \Big( \mu\big(e_2; e_1, f(\sigma)\big) \mu\big(e_2; f(\sigma), e_1 \big); e_i, e_1, e_j \Big) \\
& = \mu \Big( \mu\big(e_2; e_1, f(\sigma)\big); e_1, e_i, e_j\Big)
\mu\Big( \mu\big(e_2; f(\sigma), e_1 \big); e_i, e_1, e_j \Big) \\
& = \mu \Big(  e_2; e_1, \mu \big( f(\sigma); e_i, e_j \big) \Big) \mu \Big(  e_2; \mu \big( f(\sigma); e_i, e_1 \big), e_j \Big) \\
& = \mu \Big(  e_2; e_1, \tilde{f}_{i+j}\big(\mu ( \sigma_{1;2}; e_i, e_j) \big) \Big) \mu \Big(  e_2; \tilde{f}_{i+1} \big(\mu (\sigma_{1;2}; e_i, e_1) \big), e_j \Big) \\
& = \tilde{f}_{i+j+1}\Big( \mu \big(  e_2; e_1, \mu ( \sigma_{1;2}; e_i, e_j) \big) \Big) \tilde{f}_{i+j+1}\Big(\mu \big(  e_2;  \mu (\sigma_{1;2}; e_i, e_1), e_j\big) \Big) \\
& = \tilde{f}_{i+j+1}\Big( \mu \big(  e_2; e_1, \mu ( \sigma_{1;2}; e_i, e_j) \big) \mu \big(  e_2;  \mu (\sigma_{1;2}; e_i, e_1), e_j\big) \Big)  \\
& = \tilde{f}_{i+j+1}\Big( \mu \big(  e_2; e_1, \mu ( t(\sigma); e_i, e_j) \big) \mu \big(  e_2;  \mu (t(\sigma); e_i, e_1), e_j\big) \Big)  \\
& = \tilde{f}_{i+j+1}\Big( \mu\big( \mu(e_2; e_1, t(\sigma))\mu(e_2; t(\sigma), e_1); e_i, e_1, e_j\big)\Big)  \\
& = \tilde{f}_{i+j+1}\Big( \mu\big( t(\mu(\sigma;e_1, e_2)); e_i, e_1, e_j\big)\Big)  \\
& = \tilde{f}_{i+j+1}\Big( \mu\big( \mu(\sigma_{1;2};e_1, e_2); e_i, e_1, e_j\big)\Big)  \\
& = \tilde{f}_{i+j+1}( \sigma_{1;2}; e_i, e_{j+1})  \\
& = \tilde{f}_{i+j+1}\delta_{2; i, j+1}(\sigma_{1;2}).  \\
\end{align*}
The equalities above are derived, in order, from the following:
\begin{enumerate}
\item the definition of $\delta$, 
\item the definition of $\tilde{f}_2$, 
\item operad associativity,
\item that $f$ is a map of action operads,
\item that $f$ coequalizes $s_1$ and $s_2$,
\item that $f$ is a map of action operads,
\item the action operad axiom \cref{eqn:ao_axiom},
\item operad associativity, 
\item induction on $j$,
\item that $\tilde{f}$ commutes with $\beta$,
\item that $\tilde{f}$ is a group homomorphism,
\item the definition $\sigma_{1;2} = t(\sigma)$,
\item the action operad axiom \cref{eqn:ao_axiom},
\item that $t$ is an action operad map coequalizing $s_1$ and $s_2$,
\item that $t$ is an action operad map and $\sigma_{1;2} = t(\sigma)$,
\item operad associativity, and
\item the definition of $\delta$.
\end{enumerate}
The argument for fixing $j$ and inducting on $i$ is similar.
\item Fourth, we check the base case for induction. When $i=j=0$, we note that $\delta_{2;0,0}(\sigma_{1;2}) = e_0$. 
Since $\tilde{f}_0$ is a group homomorphism, we therefore must check that the element
$\delta_{2;0,0}\big( \tilde{f}_2(\sigma_{1;2}) \big) = \delta_{2;0,0}\big( f(\sigma) \big)$
equals $e_0$. We do this by showing that $ \delta_{2;0,0}\big( f(\sigma) \big) = \mu\big(f(\sigma); e_0, e_0 \big)$ squares to itself as follows, using similar methods as above:
\begin{align*}
\mu\big(f(\sigma); e_0, e_0 \big) & = \mu\Big(\mu \big( f(\sigma);e_1, e_2\big); e_0, e_0, e_0 \Big) \\
& = \mu\Big(\mu \big(e_2;e_1, f(\sigma)\big)\mu \big(e_2; f(\sigma), e_1 \big); e_0, e_0, e_0 \Big) \\
& = \mu\Big(\mu \big(e_2;e_1, f(\sigma)\big); e_0, e_0, e_0 \Big)\mu\Big(\mu \big(e_2; f(\sigma), e_1 \big); e_0, e_0, e_0 \Big) \\
& = \mu\Big(e_2; e_0, \mu\big(f(\sigma); e_0,e_0\big) \Big)\mu\Big(e_2; \mu\big(f(\sigma); e_0,e_0\big),e_0 \Big) \\
& = \mu\big(f(\sigma); e_0,e_0\big)\mu\big(f(\sigma); e_0,e_0\big).
\end{align*}
This concludes the induction argument, and shows that the $\tilde{f}$ maps commute with the $\delta$'s. Therefore the $\tilde{f}$'s assemble to define a map of operads.
\item Finally, we must check that $\tilde{f}$ is a map of operads over $\Sigma$, meaning that $\pi \circ \tilde{f} = 1_{\Sigma}$.
Since every generator $\sigma_{i;n}$ is an operadic composition of $\sigma_{1;2}$ and identity elements, and $\tilde{f}$ preserves operadic composition and identities, it suffices to check the equality
\[
\pi \circ \tilde{f}(\sigma_{1;2}) = \sigma_{1;2}.
\]
By definition, $\tilde{f}(\sigma_{1;2}) = f(\sigma)$, and $f \colon F\mathbf{g} \to \Lambda$ is a map of action operads, so $\pi\big( f(\sigma) \big) = \sigma_{1;2}$ by the definition of $\mathbf{g}$ as a  collection over $\SS$. This step completes the proof that the homomorphisms $\tilde{f}_n$ define a map of action operads $\tilde{f} \colon \Sigma \to \Lambda$.
\end{itemize}
\end{example}

%% file: v2p2.tex
\section{\texorpdfstring{$\Lambda$}{L}-Operads and their Algebras}\label{sec:forward-op}

This section presents the definition of a $\Lambda$-operad (\cref{Defi:lamop}), where $\Lambda$ is an action operad.
This definition unifies the various types (non-symmetric, symmetric, and braided) of operads discussed in \cref{sec:back-op} under one umbrella term.
The different group actions arise from different choices of $\Lambda$.
We also define algebras over a $\Lambda$-operad in \cref{Defi:aop-alg}, and prove a change-of-action operad result in \cref{thm:pbaop}.

\begin{Defi}[($\Lambda$-operads)]\label{Defi:lamop}
  Let $\Lambda$ be an action operad. A \textit{$\Lambda$-operad} $P$ (in $\mb{Sets}$) consists of
    \begin{itemize}
      \item a non-symmetric operad $P$ in $\mb{Sets}$ and
      \item for each $n$, an action $P(n) \times \Lambda(n) \rightarrow P(n)$ of $\Lambda(n)$ on $P(n)$
    \end{itemize}
such that the following two equivariance axioms hold.
  \begin{itemize}
    \item For each $p \in P(n)$, $q_i \in P(k_i)$, and $g_i \in \Lambda(k_i)$ for $i=1, \ldots, n$:
  \[
    \mu^{P}(p; q_{1} \cdot g_{1}, \ldots, q_{n} \cdot g_{n}) = \mu^{P}(p; q_{1}, \ldots, y_{n}) \cdot \beta^{\Lambda}(g_{1}, \ldots, g_{n}).
  \]
    \item For each $p \in P(n)$, $g \in \Lambda(n)$, and $q_i \in \Lambda(k_i)$ for $i=1, \ldots, n$:
  \[
    \mu^{P}(p \cdot g; q_{1}, \ldots, q_{n}) = \mu^{P}\left(p; q_{g^{-1}(1)}, \ldots, q_{g^{-1}(n)}\right) \cdot \delta_{n;k_1,\ldots,k_n}^{\Lambda}(g).
  \]
  \end{itemize}
\end{Defi}

\begin{rem}\label{rem:single-axiom}
Using \cref{eqn:mu-from-betadelta} from the proof of \cref{thm:charAOp}, the two equivariance axioms in \cref{Defi:lamop} can be combined into the single equality
\[
    \mu\left(p; p_1, \ldots p_n\right)  \mu\left(\tau;\tau_1, \ldots, \tau_n\right) = \mu\left(p \cdot \tau; p_{\tau(1)}\cdot \tau_{1}, \ldots,  p_{\tau(n)}\cdot \tau_{n}\right).
\]
\end{rem}

\begin{rem}\label{oscg}
The operadic crossed simplicial groups of Semidetnov appear to correspond to action operads with trivial $\Lambda_0$, although they are defined using $\circ_i$ operations instead of an operadic composition $\mu$. Continuing the comparison, his $G_\ast$-shifted operads \cite[Definition 3.5]{semidetnov} satisfy two axioms that are the same as those for $\Lambda$-operads in \Cref{Defi:lamop}. 
\end{rem}

\begin{Defi}[(Map of $\Lambda$-operads)]\label{Defi:aop_map}
Let $P$ and $Q$ be $\Lambda$-operads. A \emph{map $f \colon P \rightarrow Q$ of $\Lambda$-operads} consists of an operad map (\cref{rem:cat-of-nonsym-op}) that is levelwise equivariant with respect to the $\Lambda(n)$-actions, i.e., for each $n \in \mathbb{N}$ the following diagram commutes.
  \[
    \xy
      (0,0)*+{P(n) \times \Lambda(n)}="00";
      (30,0)*+{Q(n) \times \Lambda(n)}="10";
      (30,-15)*+{Q(n)}="11";
      (0,-15)*+{P(n)}="01";
      {\ar^{f_n \times 1} "00";"10"};
      {\ar "10";"11"};
      {\ar "00";"01"};
      {\ar_{f_n} "01";"11"};
    \endxy
  \]
\end{Defi}

\begin{prop}\label{prop:cat-of-L-op}
There is a category with 
\begin{itemize}
\item objects the $\Lambda$-operads $P$ in $\mb{Sets}$, 
\item morphisms the maps of $\Lambda$-operads between them,
\item identities $1_P \colon P \to P$ given by
\[
(1_P)_n = 1_{P(n)} \colon P(n) \to P(n),
\]
and
\item composition given by
\[
(g \circ f)_n = g_n \circ f_n.
\]
\end{itemize}
\end{prop}

\begin{nota}[(The category of $\Lambda$-operads)]\label{nota:cat-of-L-op}
The category in \cref{prop:cat-of-sym-op} is called the \emph{category of $\Lambda$-operads (in $\mb{Sets}$)}, and is denoted $\Lambda\mbox{-}\mb{Op}$.
\end{nota}

\begin{example}[(Non-symmetric, symmetric, and braided operads expressed as $\Lambda$-operads)]\label{ex:lop-exs}
We can express the non-symmetric, symmetric, and braided operads of \cref{sec:back-op} as the $\Lambda$-operads for the appropriate choice of $\Lambda$.
  \begin{enumerate}
    \item Let $T$ denote the terminal operad in $\mb{Sets}$ equipped with its unique action operad structure. Then a $T$-operad is just a non-symmetric operad in $\mb{Sets}$.
    \item Let $\Sigma$ denote the operad of symmetric groups with $\pi \colon \Sigma \rightarrow \Sigma$ the identity map. Then a $\Sigma$-operad is a symmetric operad in the category of sets.
    \item Let $B$ denote the operad of braid groups with $\pi_{n} \colon B_{n} \rightarrow \Sigma_{n}$ the canonical projection of a braid onto its underlying permutation. Then a $B$-operad is a braided operad in the sense of Fiedorowicz \cite{fie-br}.
  \end{enumerate}
\end{example}

A further example of a $\Lambda$-operad is given by the underlying operad, $\Lambda$, of $\Lambda$ itself.
\begin{prop}\label{prop:gisgop}
Let $\Lambda$ be an action operad. Then the operad $\Lambda$ is itself a $\Lambda$-operad.
\end{prop}
\begin{proof}
The underlying operad $\Lambda$ is of course an operad in $\mb{Sets}$. The right group action $\Lambda(n) \times \Lambda(n) \rightarrow \Lambda(n)$ is given simply by the group multiplication in $\Lambda(n)$. The two equivariance axioms are then both instances of the action operad axiom of $\Lambda$.
\end{proof}

An operad is intended to be an abstract description of a certain type of algebraic structure, and the particular instances of that structure are the algebras for that operad. 
We give the general definition first in \cref{Defi:aop-alg}, and then recover algebras over non-symmetric, symmetric, and braided operads in \cref{ex:lop-alg-exs}.

\begin{rem}\label{rem:aop-alg-pre}
In preparation for the definition of an algebra over a $\Lambda$-operad, we make the following two remarks.
\begin{enumerate}
\item If $\Lambda$ is an action operad and $X$ is a set, then $\Lambda(n)$ acts on $X^n$ by 
\[
\Lambda(n) \times X^n \stackrel{\pi_n \times 1}{\to} \Sigma_n \times X^n \stackrel{\kappa_n}{\to} X^n,
\]
where $\kappa_n$ is defined by the formula
\[
\kappa_n(\sigma; x_1, \ldots, x_n) = (x_{\sigma^{-1}(1)}, \ldots, x_{\sigma^{-1}(n)}).
\]
Thus we would write
\[
g \cdot (x_1, \ldots, x_n) = (x_{g^{-1}(1)}, \ldots, x_{g^{-1}(n)})
\]
using \cref{nota:perm_shorthand}.
\item Following the previous item, we define $\coeq{P}{X}{\Lambda}{n}$ as in \cref{conv:coeq}. An algebra over $P$ will involve maps with source object $\coeq{P}{X}{\Lambda}{n}$, so we remind the reader of the tilde notation for maps respecting coequalizers, \cref{conv:equiv-maps}.
\end{enumerate}
\end{rem}

\begin{Defi}[($P$-algebras)]\label{Defi:aop-alg}
Let $\Lambda$ be an action operad, and $P$ be a $\Lambda$-operad. An \textit{algebra} for $P$, or \emph{$P$-algebra}, consists of a set $X$ together with maps 
\[
\alpha_{n} \colon \coeq{P}{X}{\Lambda}{n} \rightarrow X
\]
such that the maps $\tilde{\alpha}_{n}$ satisfy the following axioms.
\begin{enumerate}
\item The element $\id \in P(1)$ is a unit in the sense that
  \[
    \tilde{\alpha}_{1}(\id; x) = x
  \]
for all $x \in X$.
\item The maps $\tilde{\alpha}_{n}$ are associative in the sense that the following diagram commutes.
  \[
    \xy
      (0,0)*+{P(n) \times \prod_{i=1}^n \left(P(k_i) \times X^{k_i}\right)}="ul";
      (60,0)*+{P(n) \times X^{n}}="ur";
      (0,-12)*+{P(n) \times \left(\prod_{i=1}^n P(k_i)\right) \times \prod_{i=1}^n \left(X^{k_i}\right) }="ml";
      (0,-24)*+{P(\sum k_{i}) \times X^{\sum k_{i}}}="bl";
      (60,-24)*+{X}="br";
      {\ar^>>>>>>>>>>>>>>{1 \times \tilde{\alpha}_{k_{1}} \times \cdots \times \tilde{\alpha}_{k_{n}}} "ul"; "ur"};
      {\ar^{\tilde{\alpha}_{n}} "ur"; "br"};
      {\ar_{\cong} "ul"; "ml"};
      {\ar_{\mu \times 1} "ml"; "bl"};
      {\ar_{\tilde{\alpha}_{\sum k_{i}}} "bl"; "br"};
    \endxy
  \]
\end{enumerate}
\end{Defi}

\begin{rem}
It is worth reiterating that the equivariance required for a $P$-algebra is built into the definition above by requiring that the maps $\alpha_{n}$ be defined on coequalizers, even though the algebra axioms then only use the maps $\tilde{\alpha}_{n}$. Since every $\Lambda$-operad has an underlying non-symmetric operad (see \cref{thm:pbaop}, applied to the unique map $T \rightarrow \Lambda$), this reflects the fact that the algebras for the $\Lambda$-equivariant version are always algebras for the plain version, but not conversely.
\end{rem}

\begin{example}[(Algebras over non-symmetric, symmetric, and braided operads as $\Lambda$-operads)]\label{ex:lop-alg-exs}
We can recover standard notions of algebras over non-symmetric, symmetric, and braided operads as algebras over a $\Lambda$-operad.
  \begin{enumerate}
    \item For the action operad $T$ of trivial groups, a $T$-operad is a non-symmetric operad. The coequalizer $\coeq{P}{X}{T}{n}$ is isomorphic to $P(n) \times X^n$, so without loss of generality we can assume that $\tilde{\alpha}_n=\alpha_n$. This recovers the usual notion of an algebra over a non-symmetric operad, see \cite[Definition 1.20]{mss-op}.
    \item For the action operad $\Sigma$ of symmetric groups, a $\Sigma$-operad is a symmetric operad. \cref{Defi:aop-alg} is equivalent to May's original definition \cite[Definition 1.1]{maygeom} by \cref{rem:aop-alg-pre}.
    \item For the action operad $B$ of braid groups, a $B$-operad is a braided operad in the sense of Fiedorowicz \cite{fie-br}. Once again, \cref{Defi:aop-alg} is equivalent to Fiedorowicz's definition \cite[Definition 3.2]{fie-br} by \cref{rem:aop-alg-pre}.
  \end{enumerate}
\end{example}

\begin{Defi}[(Map of $P$-algebras)]\label{Defi:map-palg}
Let $P$ be a $\Lambda$-operad, and let $(X, \alpha)$ and $(Y, \beta)$ be $P$-algebras. Then a \emph{map of $P$-algebras} $f \colon  (X, \alpha) \rightarrow (Y, \beta)$ is a function $f \colon X \rightarrow Y$ such that the following diagram commutes for every $n$.
  \[
    \xy
      {\ar^{1 \times f^{n}} (0,0)*+{P(n) \times X^{n}}; (50,0)*+{P(n) \times Y^{n}} };
      {\ar^{\tilde{\beta}_{n}} (50,0)*+{P(n) \times Y^{n}}; (50,-15)*+{Y} };
      {\ar_{\tilde{\alpha}_{n}} (0,0)*+{P(n) \times X^{n}}; (0,-15)*+{X} };
      {\ar_{f} (0,-15)*+{X}; (50,-15)*+{Y} };
    \endxy
  \]
\end{Defi}

\begin{prop}\label{prop:cat-of-Palg}
Let $\Lambda$ be an action operad and $P$ be a $\Lambda$-operad.
There is a category with 
\begin{itemize}
\item objects the $P$-algebras $(X, \alpha)$, 
\item morphisms the maps of $P$-algebras between them,
\item identities $1_{(X, \alpha)} \colon (X, \alpha) \to (X, \alpha)$ given by the identities $1_X$, and
\item composition given by composition of the underlying functions.
\end{itemize}
\end{prop}

\begin{nota}[(The category of $P$-algebras)]\label{nota:cat-of-palg}
The category in \cref{prop:cat-of-Palg} is called the \emph{category of $P$-algebras (in $\mb{Sets}$)}, and is denoted $P\mbox{-}\mb{Alg}$.
\end{nota}

The final goal of this section is to explore the adjunctions induced by a map of action operads. 
Using these, we will recast the category of algebras over a $\Lambda$-operad $P$ using the endomorphism operad of \cref{ex:endo}.
We begin by proving a change-of-action operad result.

\begin{thm}\label{thm:pbaop}
Let $f \colon \Lambda \rightarrow \Lambda'$ be a map of action operads. 
\begin{enumerate}
\item The map $f$ induces a functor $f^{*} \colon \Lambda'\mbox{-}\mb{Op} \rightarrow \Lambda\mbox{-}\mb{Op}$ with the property that $(f^{*}Q)(n) = Q(n)$ for every $\Lambda'$-operad $Q$.
\item The map $f$ induces a functor $f_{!} \colon \Lambda\mbox{-}\mb{Op} \rightarrow \Lambda'\mbox{-}\mb{Op}$, where $(f_{!}P)(n)$ is defined by the coequalizer below,
    \[
        \xy
            (-10,0)*+{P(n) \times \Lambda(n) \times \Lambda'(n)}="00";
            (30,0)*+{P(n) \times \Lambda'(n)}="10";
            (60,0)*+{\coeqb{P(n)}{\Lambda'(n)}{\Lambda(n)}}="20";
            {\ar@<1ex>^(.52){1 \times (\star \circ f_n \times 1)} "00" ; "10"};
            {\ar@<-1ex>_(.52){\rho \times 1} "00" ; "10"};
            {\ar^(.45){\varepsilon} "10" ; "20"};
        \endxy
    \]
where $\star \colon \Lambda'(n) \times \Lambda'(n) \to \Lambda'(n)$ is group multiplication and $\rho \colon P(n) \times \Lambda(n) \to P(n)$ is the right action given by the $\Lambda$-operad structure. 
\item The functor $f_{!}$ is left adjoint to $f^*$.
\end{enumerate}
\end{thm}
\begin{proof}
The right action of $\Lambda(n)$ on $(f^{*}Q)(n) = Q(n)$ is given as the composite
\[
Q(n) \times \Lambda(n) \stackrel{1 \times f_n}{\to} Q(n) \times \Lambda'(n) \to Q(n),
\]
where the second map is the action given by the $\Lambda'$-operad structure on $Q$.
This group action, together with the operadic multiplication maps for $Q$ as a $\Lambda'$-operad, give $f^*Q$ a $\Lambda$-operad structure.
Given a map $h \colon P \to Q$ of $\Lambda'$-operads, we must check that the maps $h_n \colon P(n) \to Q(n)$ also constitute a map of $\Lambda$-operads $f^*P \to f^*Q$.
The functions $h_n$ give a map of underlying operads by definition, so we need only verify the equivariance with respect to the $\Lambda(n)$-actions.
This equivariance diagram commutes by the functoriality of products and the definition of $h$ as a map of $\Lambda'$-operads.
It is then straightforward to check the functoriality of these assignments, finishing the proof of the first claim.

For the second claim, we first observe that elements of this coequalizer are equivalence classes $[p,\tau]$ where $p \in P(n)$, $\tau \in \Lambda'(n)$, and the relation is given by $[p\cdot \sigma, \tau] = [p, f(\sigma)\tau]$ for $\sigma \in \Lambda(n)$. Then $(f_{!}P)(n)$ inherits a right $\Lambda'(n)$-action by multiplication in the second coordinate.
The coequalizer displayed in the proposition statement is easily seen to be reflexive, with a common section given by mapping a pair $(p, \tau) \in P(n) \times \Lambda'(n)$ to $(p, e_n, \tau) \in P(n) \times \Lambda(n) \times \Lambda'(n)$. Since the product of reflexive coequalizers is again a coequalizer, we define operadic multiplication
\[
m \colon f_{!}P(n) \times f_{!}P(k_1) \times \cdots \times f_{!}P(k_n) \to f_{!}P\left( \sum k_i\right)
\]
to be uniquely determined by the universal property of its source as the coequalizer of a pair of maps 
\begin{equation}\label{eqn:coeq-m}
\big( P(n) \Lambda(n) \Lambda'(n) \big) \times \prod_{i=1}^n \big( P(k_i) \Lambda(n) \Lambda'(k_i) \big) \rightrightarrows \big( P(n) \times \Lambda'(n) \big) \times \prod_{i=1}^n \big( P(k_i) \times \Lambda'(k_i) \big).
\end{equation}
Define a function
\[
\tilde{m} \colon \big( P(n) \times \Lambda'(n) \big) \times \prod_{i=1}^n \big( P(k_i) \times \Lambda'(k_i) \big) \to f_{!}P\left( \sum k_i\right)
\]
by 
\[
\tilde{m}\big( (p,\tau); (p_1, \tau_1), \ldots, (p_n, \tau_n) \big) = \left[ \mu^P(p; p_{\tau^{-1}(1)}, \ldots, p_{\tau^{-1}(n)}), \mu^{\Lambda'}(\tau; \tau_1, \ldots, \tau_n) \right].
\]
The function $\tilde{m}$ will induce the operadic multiplication on $f_{!}P$ once we verify that it coequalizes the two maps in \cref{eqn:coeq-m}.
In order to do so, we must check that
\[
\tilde{m}\big( (p\cdot \sigma,\tau); (p_1\cdot\sigma_1, \tau_1), \ldots, (p_n\cdot\sigma_n, \tau_n) \big) = \tilde{m}\big( (p,f(\sigma)\tau); (p_1, f(\sigma_1)\tau_1), \ldots, (p_n, f(\sigma_n)\tau_n) \big).
\]
By definition the left side is
\begin{equation}\label{eqn:lhs-fshriek}
\left[ \mu^{P}(p\cdot\sigma; p_{\tau^{-1}(1)}\cdot\sigma_{\tau^{-1}(1)}, \ldots, p_{\tau^{-1}(n)}\cdot\sigma_{\tau^{-1}(n)}),\  \mu^{\Lambda'}(\tau; \tau_1, \ldots, \tau_n)            \right],
\end{equation}
while the right side is
\begin{equation}\label{eqn:rhs-fshriek}
\left[\mu^P(p; p_{(f(\sigma)\tau)^{-1}(1)}, \ldots, p_{(f(\sigma)\tau)^{-1}(n)}), \ \mu^{\Lambda'}\big( f(\sigma)\tau; f(\sigma_1)\tau_1, \ldots, f(\sigma_n)\tau_n\big)          \right].
\end{equation}
By the action operad axioms and the fact that $f$ is a map of action operads, we have the equalities
\begin{align*}
\mu\big( f(\sigma)\tau; f(\sigma_1)\tau_1, \ldots, f(\sigma_n)\tau_n\big) &= \mu\big( f(\sigma); f(\sigma_{\tau^{-1}(1)}), \ldots, f(\sigma_{\tau^{-1}(n)})\big)\mu\big( \tau; \tau_1, \ldots, \tau_n\big) \\
& = f\left(\mu\big( \sigma; \sigma_{\tau^{-1}(1)}, \ldots, \sigma_{\tau^{-1}(n)}\big) \right)\mu\big( \tau; \tau_1, \ldots, \tau_n\big).
\end{align*}
By the equality $[p\cdot \sigma, \tau] = [p, f(\sigma)\tau]$, we therefore conclude that \cref{eqn:rhs-fshriek} equals
\[
\left[\mu^P(p; p_{(f(\sigma)\tau)^{-1}(1)}, \ldots, p_{(f(\sigma)\tau)^{-1}(n)}) \cdot \mu\big( \sigma; \sigma_{\tau^{-1}(1)}, \ldots, \sigma_{\tau^{-1}(n)}\big), \mu\big( \tau; \tau_1, \ldots, \tau_n\big) \right].
\]
By the single-equality version of the $\Lambda$-operad equivariance axioms from \cref{rem:single-axiom}, we obtain that the first coordinate can be rewritten as
\begin{align*}
\mu^P(p; p_{(f(\sigma)\tau)^{-1}(1)}, \ldots, p_{(f(\sigma)\tau)^{-1}(n)}) \cdot \mu\big( \sigma; \sigma_{\tau^{-1}(1)}, \ldots, \sigma_{\tau^{-1}(n)}\big) = \\
\mu^P(p\cdot \sigma; p_{(f(\sigma)\tau)^{-1}\big( \sigma(1) \big)}\cdot \sigma_{\tau^{-1}(1)}, \ldots, p_{(f(\sigma)\tau)^{-1}\big( \sigma(n) \big)}\cdot \sigma_{\tau^{-1}(n)}). 
\end{align*}
The indices on the terms $p_{(f(\sigma)\tau)^{-1}\big( \sigma(1) \big)}$ can be simplified using that $f$ is a map of action operads, and hence preserves underlying permutations. Thus 
\begin{align*}
(f(\sigma)\tau)^{-1}\big( \sigma(i) \big) & = \tau^{-1}\left( f(\sigma)^{-1}\left( \sigma(i) \right) \right) \\
& = \tau^{-1}(i)
\end{align*}
because the underlying permutation of $\sigma$ is equal to that of $f(\sigma)$. After substituting these simplifications into the above, we obtain \cref{eqn:lhs-fshriek}, completing the proof that $\tilde{m}$ coequalizes the two maps in \cref{eqn:coeq-m} and therefore induces a unique operadic multiplication on $f_{!}P$. By definition, the operadic multiplication map $m \colon f_{!}P(n) \times f_{!}P(k_1) \times \cdots \times f_{!}P(k_n) \to f_{!}P\left( \sum k_i\right)$ is therefore defined to be
\begin{equation}\label{eqn:defn-m}
m\left( [p,\tau]; [p_1, \tau_1], \ldots, [p_n, \tau_n]\right) = \left[ \mu^P(p; p_{\tau^{-1}(1)}, \ldots, p_{\tau^{-1}(n)}), \mu^{\Lambda'}(\tau; \tau_1, \ldots, \tau_n) \right].
\end{equation}

The unit for $f_{!}P$ is $[\id, e_1]$ where $\id \in P(1)$ is the  unit for $P$ and $e_1 \in \Lambda'(1)$ is the unit for $\Lambda'$. It is straightforward to show that $[\id, e_1]$ acts as a unit for $f_{!}P$ using the unit axioms for $P$ and $\Lambda'$, and we leave these calculations to the reader.

Next we check the associativity axiom for $f_{!}P$.
We do so by checking that the function
\[
\hat{m} \colon \big( P(n) \times \Lambda'(n) \big) \times \prod_{i=1}^n \big( P(k_i) \times \Lambda'(k_i) \big) \to P\left( \sum k_i\right) \times \Lambda'\left(\sum k_i\right)
\]
given by the formula
\[
\hat{m}\left(  (p,\tau); (p_1, \tau_1), \ldots, (p_n, \tau_n)             \right)  = \left( \mu^P(p; p_{\tau^{-1}(1)}, \ldots, p_{\tau^{-1}(n)}), \mu^{\Lambda'}(\tau; \tau_1, \ldots, \tau_n) \right)
\]
satisfies associativity with respect to operadic composition.
The actual operadic composition is then obtained from $\hat{m}$ by passing to equivalence classes via the coequalizer in \cref{eqn:coeq-m}, so associativity of $\hat{m}$ will ensure the associativity of $m$.
First note that the second coordinate of the operadic composite $\hat{m}\left(  (p,\tau); (p_1, \tau_1), \ldots, (p_n, \tau_n)             \right)$ is $\mu^{\Lambda'}(\tau; \tau_1, \ldots, \tau_n)$, so associativity in the second coordinate follows immediately from associativity of $\Lambda'$. 

In order to check associativity in the first coordinate, we introduce the following notation.
We write a list $a_1, \ldots, a_n$ as $\underline{a_i}$. If $\sigma \in \Sigma_n$, we write $\sigma \bullet_i \underline{a_i}$ for the list $a_{\sigma^{-1}(1)}, \ldots, a_{\sigma^{-1}(n)}$. Then the first coordinate of
\[
\hat{m}\left( (p,\tau); \underline{ \hat{m}\left( (p_i, \tau_i); \underline{ (p_{i,j}, \tau_{i,j})} \right) } \right)
\]
is $\mu\left( p; \tau \bullet_i \underline{ \mu(p_i; \tau_i \bullet_j \underline{p_{i,j}}) } \right)$. 
By the associativity of $P$, the first coordinate is therefore equal to
\begin{equation}\label{eqn:firstcoord}
\mu\left( \mu(p; \tau \bullet_i \underline{p_i}); \tau \bullet_i \underline{\tau_i \bullet_j \underline{p_{i,j}}} \right).
\end{equation}
On the other hand, the first coordinate of
\[
\hat{m}\left( \hat{m}\left( (p,\tau); \underline{ (p_i, \tau_i) }\right); \underline{ (p_{i,j}, \tau_{i,j}) } \right)
\]
is $\mu\left( \mu(p; \tau \bullet_i \underline{p_i}); \mu(\tau; \underline{\tau_i}) \bullet_{i,j} \underline{p_{i,j}} \right)$ where $\underline{p_{i,j}}$ is ordered lexicographically and $\bullet_{i,j}$ means that $\mu(\tau; \underline{\tau_i})$ acts upon this list. By \cref{eqn:mu-from-betadelta} from the proof of \cref{thm:charAOp}, we have
\begin{align*}
\mu(\tau; \underline{\tau_i}) \bullet_{i,j} \underline{p_{i,j}} & = \delta(\tau)\beta(\underline{\tau_i}) \bullet_{i,j} \underline{p_{i,j}} \\
& = \delta(\tau) \bullet_{i,j} \underline{ \tau_i \bullet_j \underline{p_{i,j}}} \\
& = \tau \bullet_i \underline{\tau_i \bullet_j \underline{p_{i,j}}}.
\end{align*}
Therefore the first coordinate of $\hat{m}\left( \hat{m}\left( (p,\tau); \underline{ (p_i, \tau_i) }\right); \underline{ (p_{i,j}, \tau_{i,j}) } \right)$ is
\begin{align*}
\mu\left( \mu(p; \tau \bullet_i \underline{p_i}); \mu(\tau; \underline{\tau_i}) \bullet_{i,j} \underline{p_{i,j}} \right) = \mu\left( \mu(p; \tau \bullet_i \underline{p_i}); \tau \bullet_i \underline{\tau_i \bullet_j \underline{p_{i,j}}} \right),
\end{align*}
matching \cref{eqn:firstcoord} and verifying operadic associativity.

Finally we check the equivariance axiom in the form of \cref{rem:single-axiom}. This axiom is the equality
\[
m\left( [p,\tau]; \underline{[p_i, \tau_i]}\right)\cdot \mu(\omega; \underline{\omega_i}) = 
m\left( [p,\tau]\cdot \omega; [p_{\omega(1)}, \tau_{\omega(1)}]\cdot \omega_1, \ldots, [p_{\omega(n)}, \tau_{\omega(n)}]\cdot \omega_n \right).
\]
Using the formula for $m$, it is straightforward to verify that both sides are equal to
\[
\left[ \mu(p; \tau \bullet_i \underline{p_i}), \mu(\tau \omega; \tau_{\omega(1)}\omega_1, \ldots, \tau_{\omega(n)}\omega_n) \right].
\]
Thus the equivariance axiom for a $\Lambda'$-operad holds for $f_{!}P$, and completes the proof that $f_{!}P$ is a $\Lambda'$-operad.

Next we turn to the functoriality of $f_{!}$. Let $g \colon P \to Q$ be a map of $\Lambda$-operads. Define $f_{!}g \colon f_{!}P \to f_{!}Q$ by the formula
\begin{equation}\label{eqn:fshriek-g}
(f_{!}g)_n\left( [p, \tau] \right) = [g(p), \tau],
\end{equation}
for $[p, \tau] \in f_{!}P(n)$.
Since $g$ is a map of $\Lambda$-operads, it preserves the identity elements, so
\[
(f_{!}g)\left( [\id, e_1] \right) = [g(\id), e_1] = [\id, e_1],
\]
showing that $f_{!}g$ also preserves identity elements. It is clear that $f_{!}g$ is levelwise equivariant with respect to the $\Lambda'(n)$-actions, so we need only check that it preserves operadic multiplication. We leave this calculation to the reader, as it is a straightforward application of the fact that $g$ is a map of $\Lambda$-operads so preserves operadic multiplication. This completes the definition of $f_{!}$ on morphisms. It is immediate that $f_{!}$ preserves composition and identity maps, and is thus a functor 
\[ f_{!} \colon \Lambda\mbox{-}\mb{Op} \rightarrow \Lambda'\mbox{-}\mb{Op}\]
as desired, completing the proof of the second claim in the statement.

We now prove the third claim, that $f_{!}$ is left adjoint to $f^*$.
The unit of this adjuntion has a component at the $\Lambda$-operad $P$ given by a map $\eta^P \colon P \to f^*f_{!}P$, and is defined by the formula
\[
\eta^{P}_n(p) = [p, e_n]
\]
for all $p \in P(n)$. The map $\eta^P$ preserves the operadic identity by definition, and it preserves operadic multiplication by a simple application of \cref{eqn:defn-m} defining operadic multiplication in $f_{!}P$. Furthermore, $\eta^P$ preserves the right $\Lambda$-actions because
\begin{align*}
\eta^{P}(p\cdot \tau) & = [p \cdot \tau, e_n] \\
& = [p, f(\tau)e_n] \\
& = [p, e_n] \cdot \tau \\
& = \eta^P(p) \cdot \tau.
\end{align*}
Thus $\eta^P$ is a map of $\Lambda$-operads. It is then simple to check that these components define a natural transformation $1 \Rightarrow f^* f_{!}$.

The counit of this adjunction has a component at the $\Lambda'$-operad $Q$ given by a map $\varepsilon^Q \colon f_{!}f^*Q \to Q$, and is defined by the formula
\[
\varepsilon^Q_n\left( [q,\tau] \right) = q \cdot \tau.
\]
Since $f_{!}f^*Q(n)$ is defined to be $f^*Q(n) \otimes_{\Lambda(n)} \Lambda'(n)$ and the action of $\Lambda(n)$ on  $f^*Q(n)$ is given by $q \cdot_{f^*Q} \sigma = q \cdot_{Q} f(\sigma)$, the equivalence relation defining $f_!f^*Q(n)$ is generated by the equalities
\[
[q \cdot_{Q} f(\sigma), \tau] = [q, f(\sigma)\tau]
\]
for all $\sigma \in \Lambda(n)$. In particular, the formula for $\varepsilon^Q_n$ is well-defined on equivalence classes. It is obvious that $\varepsilon^Q$ preserves the operadic identity and the right $\Lambda'(n)$-actions, so we only need to show that it preserves operadic multiplication. We therefore compute that
\begin{align*}
\varepsilon^Q \left( m\left( [q, \tau]; \underline{ [q_i, \tau_i] } \right) \right) & = \varepsilon^Q\left( [\mu(q; \tau \bullet_i \underline{q_i}), \mu(\tau; \underline{\tau_i}) ] \right) \\
& = \mu(q; \tau \bullet_i \underline{q_i}) \cdot \mu(\tau; \underline{\tau_i}) \\
& = \mu(q \cdot \tau; \underline{q_i \cdot \tau_i}) \\
& = m\left( \varepsilon^Q[q,\tau]; \underline{\varepsilon^Q[q_i, \tau_i]}\right)
\end{align*}
by the definition of $m$, the definition of $\varepsilon^Q$, the equation in \cref{rem:single-axiom}, and the definition of $\varepsilon^Q$ again.
Thus $\varepsilon^Q$ preserves operadic multiplication, and is a map of $\Lambda'$-operads. As with $\eta$, naturality in $Q$ is simple to check.

For a $\Lambda$-operad $P$, the composite $\varepsilon^{f_{!}P} \circ f_{!}\eta^{P}$ is given by
\begin{align*}
\varepsilon^{f_{!}P} \circ f_{!}\eta^{P}[p, \tau] & = \varepsilon^{f_{!}P}\left[ [p,e_n], \tau\right] \\
& = [p, e_n \tau]\\
& = [p, \tau],
\end{align*}
so is the identity. Likewise, for a $\Lambda'$-operad $Q$ the composite $f^*\varepsilon^Q \circ \eta^{f^*Q}$ is given by
\begin{align*}
f^*\varepsilon^Q \circ \eta^{f^*Q}(q) & = f^*\varepsilon^Q[q,e_n] \\
& = q \cdot e_n \\
& = q,
\end{align*}
so is also the identity. These two calculations verify the triangle identities, and therefore prove that $f_{!} \dashv f^*$.
\end{proof}

We will examine the adjunction from \cref{thm:pbaop} further in the case that the map $f$ is the underlying permutation map $\pi \colon \Lambda \to \Sigma$ for an action operad $(\Lambda, \pi)$. 

\begin{Defi}\label{Defi:symmetrization}
Let $(\Lambda, \pi)$ be an action operad. The left adjoint $\pi_{!}$ from \cref{thm:pbaop} is called the \emph{symmetrization functor}.
\end{Defi}

\begin{prop}\label{prop:unit-pi-star}
Let $(\Lambda, \pi)$ be an action operad. The unit of the adjunction $\pi_{!} \dashv \pi^*$ from \cref{thm:pbaop} is an isomorphism if and only if $\Lambda \cong \Sigma$ via $\pi$.
\end{prop}
\begin{proof}
The component $\eta^{P} \colon P \to \pi^* \pi_{!} P$ of the unit at a $\Lambda$-operad $P$ is defined by the formula
\[
\eta^{P}_n(p) = [p, e_n].
\]
Taking $\Lambda$ as a $\Lambda$-operad, we see that $\pi^*\pi_{!}\Lambda$ can be chosen to be $\Sigma_n$, in which case the map $\eta^{\Lambda}_n$ is just $\pi_n$.
The unit $\eta$ can only be an isomorphism if each $\eta^P$ is, so it is a necessary condition that $\Lambda \cong \Sigma$ via $\pi$.
Sufficiency of this condition is obvious by the formulas defining $f^*, f_{!}$ applied to an isomorphism $f$, completing the proof.
\end{proof}

\begin{prop}\label{prop:counit-pi-star}
Let $(\Lambda, \pi)$ be an action operad. The counit of the adjunction $\pi_{!} \dashv \pi^*$ from \cref{thm:pbaop} is an isomorphism if and only if $\pi \colon \Lambda \to \Sigma$ is surjective.
\end{prop}
\begin{proof}
For a symmetric operad $Q$, the counit $\varepsilon^Q \colon \pi_{!}\pi^*Q \to Q$ is defined on the $n$-ary operations by
\[
\varepsilon^Q_n\left( [q,\tau] \right) = q \cdot \tau,
\]
where $q \in Q(n)$ and $\tau \in \Sigma_n$. This function is surjective for any action operad $\Lambda$ and any $\Lambda$-operad $Q$, so we only check injectivity. 

When $\pi \colon \Lambda \to \Sigma$ is not surjective, it is the zero map by \cref{prop:surjortriv}.
The equivalence relation defining $\pi_{!}\pi^*Q(n)$ is $[q \cdot_{\pi^*Q} \lambda, \tau] = [q, \pi(\lambda)\tau]$ for $q \in Q(n)$, $\lambda \in \Lambda(n)$, and $\tau \in \Sigma_n$, but $\pi(\lambda) = e_n \in \Sigma_n$. Furthermore, the action $q \cdot_{\pi^*Q} \lambda$ is computed in $\pi^*Q(n)$, so is actually 
\[
q \cdot_{\pi^*Q} \lambda = q \cdot_{Q} \pi(\lambda) = q.
\]
These calculations show that the equivalence relation degenerates to $[q,\tau] = [q, \tau]$, so $\pi_{!}\pi^*Q(n) \cong Q(n) \times \Sigma_n$, and the counit map is just given by the group action of $\Sigma_n$ on $Q(n)$, so in particular is not injective.

Now assume that $\pi \colon \Lambda \to \Sigma$ is surjective, and assume that $\varepsilon^Q_n\left( [q_1,\tau_1] \right) = \varepsilon^Q_n\left( [q_2,\tau_2] \right)$, or equivalently that
\[
q_1 \cdot \tau_1 = q_2 \cdot \tau_2.
\]
We must check that $[q_1, \tau_1] = [q_2, \tau_2]$. Then for any $\lambda \in \Lambda(n)$ such that $\pi(\lambda) = \tau_1 \tau_2^{-1}$, we have
\begin{align*}
[q_1, \tau_1] & = [q_1, \pi(\lambda)\pi(\lambda^{-1})\tau_1] \\
& = [q_1 \pi(\lambda), \pi(\lambda^{-1})\tau_1] \\
& = [q_1\tau_1\tau_2^{-1}, \tau_2\tau_1^{-1}\tau_1] \\
& = [q_2, \tau_2],
\end{align*}
proving that $\varepsilon^Q_n$ is injective and completing the proof that if $\pi$ is surjective then $\varepsilon$ is an isomorphism.
\end{proof}

The following statement appears in \cite{maygeom} as Lemma 1.4.
\begin{lem}\label{lem:alg=map}
Let $P$ be a symmetric operad and $X$ be a set. $P$-algebra structures on $X$, given by $\{ \alpha_n \colon  \coeq{P}{X}{\Sigma}{n} \rightarrow X\}$ as in \cref{Defi:aop-alg}, are in bijection with maps of symmetric operads $\alpha \cn P \to \mathcal{E}_X$.
\end{lem}
\begin{proof}
A function $\tilde{\alpha}_n \colon P(n) \times X^n \to X$ corresponds, under the hom-tensor adjunction $- \times X^n \dashv [X^n, -]$, to a function $\hat{\alpha}_n \colon P(n) \to \mathcal{E}_X(n)$. 
The unit axiom in \cref{Defi:sym_op_map} for the function $\hat{\alpha}_1$ corresponds to the first axiom in \cref{Defi:aop-alg} for $P$ as a nonsymmetric operad.
The associativity axiom in \cref{Defi:sym_op_map} for the functions $\hat{\alpha}_n$ corresponds to the second axiom in \cref{Defi:aop-alg} for $P$ as a nonsymmetric operad.
The function $\hat{\alpha}_n$ commutes with the right $\Sigma_n$-actions if and only if the function $\tilde{\alpha}_n$ is obtained as in \cref{conv:equiv-maps} from a function $\alpha_n \colon \coeq{P}{X}{\Sigma}{n} \rightarrow X$, completing the proof of the bijection.
\end{proof}

\begin{cor}\label{cor:pi-star}
Let $\Lambda$ be an action operad with underlying permutation map $\pi \colon \Lambda \to \Sigma$. 
 For any $\Lambda$-operad $P$ and any set $X$, $P$-algebra structures on $X$ are in bijection with 
\begin{itemize}
\item maps of $\Lambda$-operads $\alpha \colon P \to \pi^* \mathcal{E}_X$ or
\item maps of symmetric operads $\alpha' \colon \pi_{!}P \to  \mathcal{E}_X$.
\end{itemize}
\end{cor}
\begin{proof}
By the definition of the $\Lambda(n)$-actions on $P(n)$ and $X^n$, the maps \[\coeq{P}{X}{\Lambda}{n} \rightarrow X\] defining a $P$-algebra structure in \cref{Defi:aop-alg} are in bijection with 
$\Lambda(n)$-equivariant maps $P(n) \to \pi^*\mathcal{E}_X(n)$. The correspondence between the axioms in \cref{Defi:aop-alg} and the map of $\Lambda$-operad axioms (\cref{Defi:aop_map}) follows just as in the proof of \cref{lem:alg=map}. The equivalence between the $\Lambda$- and symmetric operads versions follows immediately from adjointness, \cref{thm:pbaop}.
\end{proof}

\begin{Defi}[(Endofunctor induced by a $\Lambda$-operad)]\label{Defi:und-P}
Let $P$ be a $\Lambda$-operad. Then $P$ induces an endofunctor of $\mb{Sets}$, denoted $\underline{P}$, by the following formula.
  \[
	 \underline{P}(X) = \coprod_n \coeq{P}{X}{\Lambda}{n}
  \]
\end{Defi}

We now have the following proposition; its proof is standard (see the discussion in Construction 2.4 of \cite{maygeom}), and we leave it to the reader.

\begin{prop}\label{prop:op=monad1}  Let $P$ be a $\Lambda$-operad.
  \begin{enumerate}
    \item The $\Lambda$-operad structure on $P$ induces a monad structure on $\underline{P}$ via the operadic multiplication and operadic identities for $P$. We denote this monad $(\underline{P}, \mu, \id)$, or just $\underline{P}$ when $\mu, \id$ are understood.
    \item The category of algebras for the $\Lambda$-operad $P$ is isomorphic to the category of algebras for the monad $(\underline{P}, \mu, \id)$.
  \end{enumerate}
\end{prop}

In the case that we take the operad $P$ to also be $\Lambda$, we do not get algebras more interesting than monoids.
\begin{prop}\label{prop:Lalg=monoid}
Let $\Lambda$ be an action operad. The category of algebras for $\Lambda$ taken as a $\Lambda$-operad, $\Lambda\mbox{-}\mb{Alg}$, is isomorphic to the category of monoids.
\end{prop}
\begin{proof}
The category of monoids is $\underline{T}\mbox{-}\mb{Alg}$ where $T$ is the terminal action operad, so we will produce an isomorphism of monads $R \colon \underline{T} \cong \underline{\Lambda}$.
For a set $X$, 
\[
\underline{T}(X) = \coprod_n \coeq{T}{X}{T}{n} \cong \coprod_n X^n, 
\]
while 
\[
\underline{\Lambda}(X) = \coprod_n \coeq{\Lambda}{X}{\Lambda}{n}.
\]
The elements of the coequalizer $\coeq{\Lambda}{X}{\Lambda}{n}$ are equivalence classes $[g; x_1, \ldots, x_n]$ under the equivalence relation
	\[
		(gh; x_1, \ldots, x_n) \sim \left(g; x_{h^{-1}(1)}, \ldots, x_{h^{-1}(n)}\right).
	\]
The functions $R_{X;n} \colon \coeq{\Lambda}{X}{\Lambda}{n} \to X^n$ defined by
\[
R_{X;n}\big( \ [g; x_1, \ldots, x_n] \ \big) = \big( x_{g^{-1}(1)}, \ldots, x_{g^{-1}(n)}\big)
\]
are bijections, and are easily seen to be natural in $X$. 
Define $R_X = \coprod_n R_{X;n}$.
We leave it to the reader that these components also commute with the multiplication and unit of the monads $\underline{T}, \underline{\Lambda}$, so produce the desired isomorphism of monads.
The isomorphism of monads $R$ then induces an isomorphism between categories of algebras, proving the desired claim.
\end{proof}

The monad-theoretic incarnation of \cref{cor:pi-star} is then the following.

\begin{cor}\label{cor:sym-preserve-alg}
Let $(\Lambda, \pi)$ be an action operad.
For any $\Lambda$-operad P, there exists a natural isomorphism of monads between $\underline{P}$ and $\underline{\pi_{!}P}$. In particular, these monads (and hence operads) have isomorphic categories of algebras.
\end{cor}
\begin{proof}
The isomorphism is induced by the universal property of the coequalizer by noting that $\coeq{P}{X}{\Lambda}{n}$, the $n$th summand in $\underline{P}(X)$, and $\big(\coequ{P}{\Sigma}{\Lambda}{n}\big)\otimes_{\Sigma_n} X^n$, the $n$th summand in $\underline{\pi_{!}P}(X)$, both satisfy the same universal property.
\end{proof}

\section{The Substitution Product}\label{sec:sub}

In this section, we will show that $\Lambda$-operads are the monoids in the category of $\Lambda$-collections equipped with an appropriate substitution product. 
Such a result is fairly standard (see \cite[Section 1.8]{mss-op}), and in both the symmetric and non-symmetric cases can easily be proven directly. 
Since we work with an arbitrary action operad, however, it will be more economical to take the abstract approach using coends and Day convolution.

\begin{Defi}[($\Lambda$-collections)]\label{Defi:BLambda}
Let $\Lambda$ be an action operad.
\begin{enumerate}
\item The category $B\Lambda$ has 
\begin{itemize}
\item objects natural numbers $n \in \N$, and
\item morphism sets $B\Lambda(m,n)$ empty when $m\neq n$
\[
B\Lambda(n,n) = \Lambda(n),
\]
with composition given by group multiplication and identities given by the elements $e_n$.
\end{itemize}
\item The category $\Lambda\mb{\mbox{-}Coll}$ of \emph{$\Lambda$-collections} is the presheaf category
  \[
[B\Lambda^{\textrm{op}}, \mb{Sets}].
  \]
\end{enumerate}
\end{Defi}

\begin{rem}
The definition of $\Lambda\mb{\mbox{-}Coll}$ does not require that $\Lambda$ be an action operad, only that one has a natural number-indexed set of groups. 
\end{rem}

\begin{Defi}[(The substitution product $\circ$)]\label{Defi:sub-prod}
Let $\Lambda$ be an action operad, and let $X, Y$ be $\Lambda$-collections. We define the $\Lambda$-collection $X \circ Y$ by
  \[
    X \circ Y (k) = \left(\left( \coprod_{k_{1} + \cdots + k_{r} = k} X(r) \times Y(k_{1}) \times \cdots \times Y(k_{r}) \right) \times \Lambda(k)\right) / \sim
  \]
where the equivalence relation is generated by the following.
\begin{enumerate}
\item For $x \in X(r)$, $h \in \Lambda(r)$, $y_i \in Y(k_i)$ for $i = 1, \ldots, r$, and $g \in \Lambda(k)$, we have
\[
 \left(xh; y_{1}, \ldots, y_{r}; g\right) \sim \left(x; y_{h^{-1}(1)}, \ldots, y_{h^{-1}(r)}; \delta_{r; k_1, \ldots, k_r}(h)g\right).
\]
\item For $x \in X(r)$, $y_{i} \in Y(k_{i})$ for $i = 1, \ldots, r$, $g_{i} \in \Lambda(k_{i})$  for $i = 1, \ldots, r$, and $g \in \Lambda(k)$, we have
\[
 \left(x; y_{1}g_{1}, \ldots, y_{r}g_{r}; g\right) \sim \left(x; y_{1}, \ldots, y_{r}; \beta(g_{1}, \ldots, g_{r})g\right).
\]
\end{enumerate}
The group $\Lambda(k)$ acts on the right by multiplication,
\[
[x; y_1, \ldots, y_r; g] \cdot h = [x; y_1, \ldots, y_r; gh], 
\]
and this action is compatible with the equivalence relation above, so defines $X \circ Y$ as a presheaf on $B \Lambda$.
\end{Defi}

We will now develop the tools to prove that the category $\Lambda\mb{\mbox{-}Coll}$ has a monoidal structure given by $\circ$, and that operads are the monoids with respect to this monoidal structure.
We provide the statement here.

\begin{thm}\label{thm:operad=monoid}
Let $\Lambda$ be an action operad.
  \begin{enumerate}
    \item The category $\Lambda\mb{\mbox{-}Coll}$ has a monoidal structure with tensor product given by $\circ$ and unit given by the collection $I$ with $I(n) = \emptyset$ when $n \neq 1$ and $I(1) = \Lambda(1)$.
    \item The category $\mb{Mon}(\Lambda\mb{\mbox{-}Coll})$ of monoids in $\Lambda\mb{\mbox{-}Coll}$ is equivalent to the category of $\Lambda$-operads.
  \end{enumerate}
\end{thm}

While this theorem can be proven by direct calculation using the equivalence relation given above, such a proof is unenlightening. 
Furthermore, we want to consider $\Lambda$-operads in categories other than sets, so an element-wise proof might not apply. 
Instead we will develop general machinery that will apply to $\Lambda$-operads in any cocomplete symmetric monoidal category, by which we mean a category that is cocomplete, equipped with a symmetric monoidal structure, and the functors $X \otimes -, - \otimes X$ preserve colimits for every object $X$ (as is the case if the monoidal structure is closed). 
Our construction of the monoidal structure on the category of $\Lambda$-collections will require the Day convolution product \cite{day-thesis}, and we begin by proving that $B\Lambda$ has a monoidal structure.

\begin{prop}\label{prop:Gmonoidal}
The action operad structure of $\Lambda$ gives $B\Lambda$ a strict monoidal structure.
\end{prop}
\begin{proof}
The tensor product on $B\Lambda$ is given by addition on objects, with unit object 0; we denote tensor product by $+$. 
On morphisms, $+$ must be given by a group homomorphism
  \[
    + \colon \Lambda(n) \times \Lambda(m) \rightarrow \Lambda(n+m),
  \]
 and is defined by the formula
  \[
    +(g,h) = \beta(g,h).
  \]
By \cref{thm:charAOp}, $\beta$ is a homomorphism as desired, and we now write $+(g,h)$ as $g+h$.

Addition of objects is strictly associative and unital.
Strict associativity at the level of morphisms follows from Axiom \eqref{eq3}, and strict unitality at the level of morphisms follows from Axiom \eqref{eq3} and \cref{lem:e0-unit}.
Thus $B\Lambda$ is a strict monoidal category as desired, completing the proof.
\end{proof}

Now that $B\Lambda$ has a monoidal structure, there is also a monoidal structure on the category of $B\Lambda$-collections using Day convolution, denoted $\star$. 

\begin{Defi}[(Day convolution, \cite{day-thesis})]\label{Defi:day-conv}
Given collections $X, Y$, their \emph{convolution product} $X \star Y$ is given by the coend formula
  \[
    X \star Y (k) = \int^{m,n \in B\Lambda} X(m) \times Y(n) \times B\Lambda(k, m+n).
  \]
  \end{Defi}
  
\begin{rem}\label{rem:DC-simple}
Given that $  B\Lambda(k, m+n)$ is empty unless $k=m+n$, the coend in \cref{Defi:day-conv} can be rewritten as
  \[
    X \star Y (k) = \int^{m+n=k} X(m) \times Y(n) \times \Lambda(k).
  \]
In this formulation, $\Lambda(m) \times \Lambda(n)$ acts on $X(m) \times Y(n)$ by the product of their separate actions, and acts on $\Lambda(k)$ by $(g, h) \cdot t = \beta(g,h) t$.
\end{rem}
  
\begin{rem}[($n$-fold Day convolution)]\label{rem:nfold-DC}
The $n$-fold Day convolution product of a $\Lambda$-collection $Y$ with itself is given by the following coend formula.
  \[
    Y^{\star n}(k) = \int^{k_{1}+ \cdots + k_{n}=k } Y(k_{1}) \times \cdots \times Y(k_{n}) \times \Lambda(k)
  \]
  \end{rem}
  
Computations with Day convolution will necessarily involve heavy use of the calculus of coends, and we refer the unfamiliar reader to \cite{maclane-catwork} or \cite{loregian}. 
Our goal is to express the substitution tensor product as a coend just as in \cite{kelly-op}, and to do that we need one final result about the Day convolution product.

\begin{lem}\label{lem:calclem2}
Let $\Lambda$ be an action operad, $Y$ be a $\Lambda$-collection, and $k$ be a fixed natural number. Then the assignment
  \[
    n \mapsto Y^{\star n}(k)
  \]
can be given the structure of a functor $B\Lambda \rightarrow \mb{Sets}$.
\end{lem}
\begin{proof}
Since the convolution product is given by a coend, it is the universal object with maps
  \[
    \theta_{k_1, \ldots, k_n; k} \colon Y(k_{1}) \times \cdots \times Y(k_{n}) \times \Lambda(k) \rightarrow Y^{\star n}(k),
  \]
for $k= k_1 + \cdots + k_n$, such that the following diagram commutes for every $g_{1} \in \Lambda(k_{1}), \ldots, g_{n} \in \Lambda(k_{n})$.
  \[
    \xy
      (0,0)*+{Y(k_1) \times \cdots \times Y(k_n)  \times \Lambda(k)}="a";
      (70,0)*+{Y(k_1) \times \cdots \times Y(k_n) \times \Lambda(k)}="b";
      (70,-20)*+{Y^{\star n}(k)}="c";
      (0,-20)*+{Y(k_1) \times \cdots \times Y(k_n)  \times \Lambda(k)}="d";
      {\ar^{(-\cdot g_1, \ldots, -\cdot g_n) \times 1} "a" ; "b"};
      {\ar_{1 \times \big( (g_1 + \cdots + g_n)\cdot -\big)} "a" ; "d"};
      {\ar^{\theta_{k_1, \ldots, k_n; k}} "b" ; "c"};
      {\ar_{\theta_{k_1, \ldots, k_n; k}} "d" ; "c"};
    \endxy
  \]

Let $f \in \Lambda(n)$, considered as a morphism $n \rightarrow n$ in $B\Lambda$. We induce a map 
\[
f \bullet - \colon Y^{\star n}(k) \rightarrow Y^{\star n}(k)
\]
using the universal property of the coend. For each $k$ and $k_1, \ldots, k_n$ such that $k = k_1 + \cdots + k_n$, define
\[
f[k_1, \ldots, k_n] \colon Y(k_1) \times \cdots \times Y(k_n) \times \Lambda(k) \to  Y(k_{f^{-1}(1)}) \times \cdots \times Y(k_{f^{-1}(n)}) \times \Lambda(k)
\]
by
\[
f[k_1, \ldots, k_n](y_1, \ldots, y_n; g) = \big( y_{f^{-1}(1)}, \ldots, y_{f^{-1}(n)}; \delta_{n; k_1, \ldots, k_n}(f) g\big).
\]
We now check that the following diagram commutes, where underlined elements represent lists of the indicated elements indexed from $1$ to $n$, e.g.,
  \[
      \beta(\underline{g_i}) = \beta(g_1,\cdots,g_n) = g_1 + \cdots + g_n.
  \]
  \[
    \xy
      (0,0)*+{\prod_{i=1}^n Y(k_i)  \times \Lambda(k)}="a";
      (50,0)*+{\prod_{i=1}^n Y(k_i) \times \Lambda(k)}="b";
      (80,-20)*+{\prod_{i=1}^n Y(k_{f^{-1}(i)}) \times \Lambda(k)}="c";
      (80,-50)*+{Y^{\star n}(k)}="d";
      (0,-30)*+{\prod_{i=1}^n Y(k_i)  \times \Lambda(k)}="e";
      (30,-50)*+{\prod_{i=1}^n Y(k_{f^{-1}(i)}) \times \Lambda(k)}="f";
      (30,-20)*+{\prod_{i=1}^n Y(k_{f^{-1}(i)}) \times \Lambda(k)}="g";
      (40,-10)*+{\raisebox{.5pt}{\textcircled{\raisebox{-.9pt} {1}}}}="1";
      (15,-25)*+{\raisebox{.5pt}{\textcircled{\raisebox{-.9pt} {2}}}}="2";
      (60,-35)*+{\raisebox{.5pt}{\textcircled{\raisebox{-.9pt} {3}}}}="3";
      %
      {\ar^{\left(\underline{-\cdot g_i}\right) \times 1} "a" ; "b"};
      {\ar^{f[k_1, \ldots, k_n]} "b" ; "c"};
      {\ar^{\theta_{\underline{k_{f^{-1}(i)}}; k}} "c" ; "d"};
      {\ar_{1 \times \left( \beta\left(\underline{g_i}\right)\cdot -\right)} "a" ; "e"};
      {\ar_{f[k_1, \ldots, k_n]} "e" ; "f"};
      {\ar_{\theta_{\underline{k_{f^{-1}(i)}}; k}} "f" ; "d"};
      {\ar^{f[k_1, \ldots, k_n]} "a" ; "g"};
      {\ar^{1 \times \left(\beta\left(\underline{g_{f^{-1}(i)}}\right) \cdot - \right)} "g" ; "f"};
      {\ar_{\left(\underline{- \cdot g_{f^{-1}(i)}} \right) \times 1} "g" ; "c"};
    \endxy
  \]
Square $\raisebox{.5pt}{\textcircled{\raisebox{-.9pt} {1}}}$ commutes by naturality of symmetries, square $\raisebox{.5pt}{\textcircled{\raisebox{-.9pt} {2}}}$ commutes by Axiom \eqref{eq8} from \cref{thm:charAOp}, and square $\raisebox{.5pt}{\textcircled{\raisebox{-.9pt} {3}}}$ commutes by the definition of the coend. 
Therefore by the universality property, there is a unique map $f \bullet - \colon Y^{\star n}(k) \rightarrow Y^{\star n}(k)$ such that
\begin{equation}\label{eq:fbullet}
\theta_{k_{f^{-1}(1)}, \ldots, k_{f^{-1}(n)}; k} \circ f[k_1, \ldots, k_n] = (f \bullet -) \circ \theta_{k_1, \ldots, k_n; k}
\end{equation}
for all $k$ and $k_1, \ldots, k_n$ such that $k = k_1 + \cdots + k_n$.
Given $f_1, f_2 \in \Lambda(n)$, we have
\begin{align*}
(f_2 \bullet -) \circ (f_1 \bullet -) \circ \theta_{\underline{k_i};k} &= (f_2 \bullet -) \circ \theta_{\underline{k_{f_1^{-1}(i)}};k} \circ f_1\left[\underline{k_i}\right] \\
&= \theta_{\underline{k_{f_1^{-1}(f_2^{-1}(i))}};k} \circ f_2\left[k_{f_1^{-1}(1)}, \ldots, k_{f_1^{-1}(n)}\right] \circ f_1\left[\underline{k_i}\right] \\
&= \theta_{\underline{k_{(f_2f_1)^{-1}(i)}};k} \circ (f_2f_1)\left[\underline{k_i}\right] \\
&= ((f_2f_1) \bullet -) \circ \theta_{\underline{k_i};k}
\end{align*}
by \cref{eq:fbullet} twice, the left action of $\Sigma_n$ on $n$-tuples as in \cref{rem:Sn-tuples}, and Axiom \eqref{eq6} from \cref{thm:charAOp}. By the universal property of the coend, we conclude that $(f_2 \bullet -) \circ (f_1 \bullet -) = \big( (f_2f_1) \bullet -\big)$, verifying functoriality and completing the proof.
\end{proof}

\begin{rem}[(Yoneda via coends)]\label{rem:yoneda-coend}
We make heavy use of the following consequence of the Yoneda lemma: given any functor $F \colon B\Lambda \rightarrow \mb{Sets}$ and a fixed object $a \in B\Lambda$, there is a natural isomorphism
  \[
    \int^{n \in B\Lambda} B\Lambda(n,a) \times F(n) \cong F(a)
  \]
given by sending the pair $(g, x)$, for $g \in B\Lambda(n,a)$ and $x \in F(n)$, to $F(g)(x)$.
There is a corresponding result for $F \colon B\Lambda^{\textrm{op}} \rightarrow \mb{Sets}$, using representables of the form $B\Lambda(a,n)$ instead.
\end{rem}

We are now ready for the abstract description of the substitution tensor product.

\begin{lem}\label{lem:star-of-circle}
  Let $X, Y$ be $\Lambda$-collections. Then there is a natural isomorphism
  \[
  X \circ Y \cong  \int^{n} X(n) \times Y^{\star n},
  \]
  induced by the colimit structures.
\end{lem}
\begin{proof}
The coend $\int^{n} X(n) \times Y^{\star n}(k)$ can be expanded as follows, using \cref{rem:nfold-DC}, the fact that $A \times -$ preserves colimits for any $A$, and the Fubini theorem for coends \cite[Theorem 1.3.1]{loregian}.
\begin{align*}
\int^{n} X(n) \times Y^{\star n}(k) & \cong \int^{n} X(n) \times \left( \int^{k_{1}+ \cdots + k_{n}=k } Y(k_{1}) \times \cdots \times Y(k_{n}) \times \Lambda(k) \right) \\
& \cong \int^{n, k_{1}+ \cdots + k_{n}=k} X(n) \times Y(k_{1}) \times \cdots \times Y(k_{n}) \times \Lambda(k).
\end{align*}
This final coend, when written out as a coequalizer, gives the formula in \cref{Defi:sub-prod}.
The two isomorphisms above are natural in both variables by the universal property of the colimits involved.
\end{proof}

\begin{cor}\label{cor:circ-adj}
Let $Y$ be a $\Lambda$-collection.
\begin{enumerate}
\item The functor $- \circ Y \colon \Lambda\mb{\mbox{-}Coll} \to \Lambda\mb{\mbox{-}Coll}$ has a right adjoint $[Y, -]$.
\item For any other $\Lambda$-collection $X$, there is a natural isomorphism
\[
X^{\star n} \circ Y \cong (X \circ Y)^{\star n},
\]
induced by the colimit structures.
\end{enumerate}
\end{cor}
\begin{proof}
We define the $\Lambda$-collection $[Y, Z]$ by
\[
[Y,Z](k) = \Lambda\mb{\mbox{-}Coll}(Y^{\star k}, Z)
\]
on objects and \cref{lem:calclem2} on morphisms via precomposition. 
Then
\begin{align*}
\Lambda\mb{\mbox{-}Coll}\left( X \circ Y, Z \right) & \cong \Lambda\mb{\mbox{-}Coll}\left( \int^{n} X(n) \times Y^{\star n}, Z \right) \\
& \cong \int_{n} \Lambda\mb{\mbox{-}Coll}\big(  X(n) \times Y^{\star n}, Z \big) \\
& \cong \int_{n} \mb{Sets}\big(  X(n), \Lambda\mb{\mbox{-}Coll}(Y^{\star n}, Z) \big) \\
& \cong \Lambda\mb{\mbox{-}Coll}\big( X, [Y,Z] \big)
\end{align*}
by \cref{lem:star-of-circle}, the representable functor $\Lambda\mb{\mbox{-}Coll}\left( -, Z \right)$ mapping coends to ends, the copowering of collections over sets, and the identification of the set of natural transformations as an end. 
Each of these isomorphisms is visibly natural in all three variables, so  $[Y, -]$ is right adjoint to $- \circ Y$, completing the proof of the first claim.

The second claim follows immediately from the first, as $X \mapsto X^{\star n}$ is a colimit, hence preserved by $- \circ Y$. 
\end{proof}

\begin{lem}\label{lem:Istarn(k)}
Let $I$ be the $\Lambda$-collection defined by
\[
I(k) = \left\{ \begin{array}{lr}
\emptyset & k \neq 1, \\
\Lambda(1) & k=1.
\end{array} \right.
\]
Then $I^{\star n} (k)$ is empty unless $k=n$, and then is isomorphic to $\Lambda(n)$.
\end{lem}
\begin{proof}
By definition, we have
 \[
    I^{\star n}(k) = \int^{k_{1}+ \cdots + k_{n}=k } I(k_{1}) \times \cdots \times I(k_{n}) \times \Lambda(k).
  \]
The only non-empty terms appear when $k_1 = k_2 = \cdots = k_n = 1$, from which we derive $k=n$. The coend is therefore the coequalizer $\coeqb{\Lambda(1)^n}{\Lambda(n)}{\Lambda(1)^n}$, where $\Lambda(1)^n$ acts on itself by right multiplication and on $\Lambda(n)$ by
\[
(g_1, \ldots, g_n) \cdot h = \beta(g_1, \ldots, g_n)h.
\]
This coequalizer is again $\Lambda(n)$, via the map above, completing the proof.
\end{proof}

Finally we are in a position to prove \cref{thm:operad=monoid}. 

\begin{proof}[Proof of \cref{thm:operad=monoid}]
First we must show that $\Lambda\mbox{-}\mb{Coll}$ has a monoidal structure using $\circ$. To prove this, we must give the unit and associativity isomorphisms and then check the monoidal category axioms. Define the unit object to be $I = B\Lambda(-,1)$. Then for the left unit isomorphism, we find that
  \begin{align*}
    I \circ Y (k) &= \int^{n} B\Lambda(n,1) \times Y^{\star n}(k) \\
    &\cong Y^{\star 1}(k) \\
    &\cong Y(k),
  \end{align*}
where both isomorphisms are induced by the universal property of the coend. For the right unit isomorphism, we have that
  \begin{align*}
    X \circ I (k) &= \int^{n} X(n) \times I^{\star n}(k) \\
    &\cong X(k)
  \end{align*}
by \cref{lem:Istarn(k)}.

Next we turn to constructing the associativity isomorphisms.
We first compute that
\begin{align*}
\big[ Y, [Z,W] \big](k) & = \lcoll \big( Y^{\star k}, [Z, W] \big) \\
& \cong \lcoll ( Y^{\star k} \circ Z, W) \\
& \cong \lcoll \big( (Y \circ Z)^{\star k}, W \big) \\
& = \lcoll \big[ Y \circ Z, W \big] (k)
\end{align*}
by the definition of the internal hom from the first part of \cref{cor:circ-adj} and the preservation of colimits from the second part of \cref{cor:circ-adj}.
These isomorphisms are compatible with the right $\Lambda(k)$-actions, so constitute an isomorphism that we denote
\[
\overline{a} \colon \big[ Y, [Z,W] \big] \cong \big[ Y \circ Z, W \big].
\]
The associativity isomorphism is defined to be the one induced, by Yoneda, from the composite below, in which each unmarked isomorphism is obtained from an adjunction of the form $- \circ A \dashv [A, -]$.
\begin{align*}
\lcoll \big( (X \circ Y) \circ Z, W \big) & \cong \lcoll \big( X \circ Y, [Z, W] \big) \\
& \cong \lcoll \big( X, \big[ Y, [Z,W] \big] \big) \\
& \stackrel{\overline{a}}{\cong} \lcoll \big( X, \big[ Y \circ Z, W \big] \big) \\
& \cong \lcoll \big(X \circ (Y \circ Z), W \big)
\end{align*}

In order to finish the proof that $(\lcoll, \circ, I)$, with the unit and associativity isomorphisms above, is a monoidal category, we must check two axioms.
These axioms follow immediately from the fact that the unit and associativity isomorphisms were all induced by the universal property of the colimit constructing their domains.

Now we must show that monoids in $(\lcoll, \circ, I)$ are operads. By the Yoneda lemma, a map of $\Lambda$-collections $\eta \colon I \rightarrow X$ corresponds to an element $\id \in X(1)$ since $I = B\Lambda(-,1)$. A map $\mu \colon X \circ X \rightarrow X$ is given by, for each $k$, a $\Lambda(k)$-equivariant map $(X \circ X) (k) \rightarrow X(k)$. By the universal property of the coend, this is equivalent to giving maps
  \[
  \mu_{n; k_1, \ldots, k_n;k} \colon X(n) \times X(k_{1}) \times \cdots \times X(k_{n}) \times \Lambda(k) \to X(k)
  \]
that are compatible with the following group actions as specified.
\begin{itemize}
\item $\Lambda(n)$ acts on $X(n)$ on the right by the given action, and on $X(k_1) \times \cdots \times X(k_n) \times \Lambda(k)$ on the left by permutations and $\delta$. The map $\mu$ must coequalize these.
\item The group $\Lambda(k_i)$ acts on the factor $X(k_i)$ on the right by the given action, and on the left of $\Lambda(k)$ by group multiplication and $\beta$. The map $\mu$ must coequalize these.
\item $\Lambda(k)$ acts on the right of $\Lambda(k)$ by group multiplication, and on $X(k)$ on the right by the given action. The map $\mu$ must preserve this action.
\end{itemize}
Given such a monoid structure, we define the operadic multiplication on the $\Lambda$-collection $X$ by
\[
    \mu (x; y_{1}, \ldots, y_{n}) = \mu_{n; k_1, \ldots, k_n; k}(x; y_{1}, \ldots, y_{n}; e_{k}).
  \]
Conversely, given an operad $P$, we make the underlying $\Lambda$-collection into a monoid under $\circ$ by defining
\[
 \mu_{n; k_1, \ldots, k_n; k}(x; y_{1}, \ldots, y_{n}; g) =  \mu (x; y_{1}, \ldots, y_{n}) \cdot g.
\]
We leave checking the remaining details to the reader.
\end{proof}

%% file: v2p3.tex
\artpart{Operads in Categories}\label{part:op-in-cat}

\section{Background: Group Actions and 2-limits}\label{sec:backgr-cat}

We assume familiarity with basic $2$-category theory \cite{jy-2dim,KS}, but recall some 1- and 2-dimensional aspects of $\mb{Cat}$ itself here.

\begin{conv}[(Sets and discrete categories)]\label{conv:set-disc}
By abuse of notation, any set $S$ will be identified with the discrete, small category $dS$ with object set $S$.
In this way, we also view any action operad $\Lambda$ as an operad in $\mb{Cat}$, and we view any group $G$ as a discrete, strict monoidal category.
\end{conv}

\begin{conv}[(Group actions on categories)]\label{conv:group-act-cat}
A group action on a category is meant here in the strict sense, not in the up-to-isomorphism sense.
Thus if $G$ acts on $C$, the equations
\begin{align*}
g \cdot (h \cdot x) & = (gh) \cdot x,\\
e \cdot x & = x
\end{align*}
hold for all $x$, where $x$ is allowed to be either an object or morphism of $C$.
Furthermore, a group action on a category is functorial, so 
\begin{align*}
g \cdot \id_{c} & = \id_{g \cdot c},\\
(g \cdot p) \circ (g \cdot q) & = g \cdot (p \circ q)
\end{align*}
hold for all objects $c$ and composable pairs of morphisms $p, q$.
\end{conv}

\begin{Defi}[(Free actions)]\label{Defi:free-action}
Suppose that a group $G$ acts on a category $C$, and let $x$ denote either an object or a morphism of $C$.
We say that the action is \emph{free} if, for any $x$, $g \cdot x = x$ implies that $g$ is the identity element $e \in G$.
This is equivalent to the condition that, for all $x, y$ there exists at most one $g \in G$ such that $g \cdot x = y$.
\end{Defi}

\begin{lem}\label{lem:free-on-obj}
Let $G$ be a group, $C$ be a category, and suppose that $G$ acts on $C$. Then the action of $G$ on $C$ is free if and only if the action of $G$ on the set of objects of $C$ is free.
\end{lem}
\begin{proof}
If the action of $G$ on $C$ is not free, then there is element $g \in G$ and either an object $c$ or a morphism $f$ such that $g \cdot c = c$ or $g \cdot f = f$, respectively. If there is such an object $c$, then the action of $G$ on the objects of $C$ is not free; if there is such an $f \colon c \to d$, then $g \cdot f = f$ implies that $g \cdot c = c$ and once again the action of $G$ on the objects of $C$ is not free. Finally, if the action of $G$ on $C$ is free, it is immediate that the action of $G$ on the objects of $C$ is free, completing the proof.
\end{proof}

\begin{lem}\label{lem:coeq-lem}
Let $G$ be a group, $C$ be a category, and $\mu \colon G \times C \to C$ be an action of $G$ on $C$. Suppose that the action of $G$ on $C$ is free.
\begin{enumerate}
\item Then there is a category $C/G$ with
\begin{itemize}
\item objects $[c]$, where $c$ is an object of $C$ and $[c]$ denotes its orbit under the $G$-action; and
\item morphisms $[p] \colon [c] \to [d]$, where $p \colon c_1 \to d_1$, $c_1 \in [c]$, $d_1 \in [d]$, and $[p]$ denotes the orbit of $p$ under the $G$-action.
\end{itemize}
\item The category $C/G$ is the coequalizer
    \[
        \xy
            (0,0)*+{G \times C}="00";
            (30,0)*+{C}="10";
            (60,0)*+{C/G}="20";
            {\ar@<1ex>^{\mu} "00" ; "10"};
            {\ar@<-1ex>_{\pi_2} "00" ; "10"};
            {\ar^{\varepsilon} "10" ; "20"};
        \endxy
    \]
where the top map is the action of $G$ on $C$, the bottom map is the projection onto $C$, and the coequalizing functor $\varepsilon$ is defined by sending an object or morphism to its orbit in $C/G$.
\end{enumerate}

\end{lem}
\begin{proof}
In order to prove part one of the lemma, we must define identities and composition, and check the axioms for a category. Define the identity morphism $\id_{[c]} \colon [c] \to [c]$ to be $[\id_{c}]$. For morphisms $[p] \colon [a] \to [b]$ and $[q] \colon [b] \to [c]$ represented by $p \colon a_1 \to b_1$ and $q \colon b_2 \to c_2$, let $g \in G$ be the unique element, since the action is free, such that
\[
g \cdot b_2 = b_1.
\]
Define $[q] \circ [p] = [g\cdot q \ \circ \ p]$. Now we check the axioms.
\begin{itemize}
\item For $[p] \colon [a] \to [b]$ represented by $p \colon a_1 \to b_1$, the composite $[p] \circ \id_{[a]}$ is $[g \cdot p \ \circ \ \id_{a}]$ where $g$ is the unique element such that $g \cdot a_1 = a$. Since $g \cdot p \ \circ \ \id_{a} = g \cdot p$, the composite  $[p] \circ \id_{[a]}$ equals $[g \cdot p] = [p]$ as desired.
\item For $[p] \colon [a] \to [b]$ represented by $p \colon a_1 \to b_1$, the composite $\id_{[b]} \circ [p]$ is $[g \cdot \id_{b} \ \circ \ p]$ where $g$ is the unique element such that $g \cdot b_1 = b$. Since the action of $G$ on $C$ is functorial, $g \cdot \id_{b} = \id_{g \cdot b} = \id_{b_1}$, so $[g \cdot \id_{b} \ \circ \ p] = [\id_{b_1} \circ p] = [p]$ as desired.
\item For $[p] \colon [a] \to [b]$ represented by $p \colon a_1 \to b_1$, $[q] \colon [b] \to [c]$ represented by $q \colon b_2 \to c_2$, and $[r] \colon [c] \to [d]$ represented by $r \colon c_3 \to d_3$, we compute
\[
[r] \circ \big([q] \circ [p]\big) = [h \cdot r \ \circ \ g \cdot q \ \circ \ p],
\]
where $g \cdot b_2 = b_1$ and $h \cdot c_3 = g \cdot c_2$, and
\[
\big([r] \circ [q]\big) \circ [p] = [g \cdot \big( j \cdot r \ \circ \ q \big) \ \circ p],
\]
where $g \cdot b_2 = b_1$ and $j \cdot c_3 = c_2$. By functoriality of the $G$-action, $g \cdot \big( j \cdot r \ \circ \ q \big) = gj \cdot r \ \circ g \cdot q$, so to prove associativity we need only check that $gj = h$. This follows from the assumption that the action is free together with the equations $h \cdot c_3 = g \cdot c_2$ and $j \cdot c_3 = c_2$.
\end{itemize}
Thus $C/G$ is a category.

For the second claim in the lemma, first note that $\varepsilon$, defined by $\varepsilon(x) = [x]$ for $x$ an object or morphism of $C$, is a functor. Furthermore, we see that 
\[
\varepsilon \pi_2(g, x) = \varepsilon(x) = [x] = [g \cdot x] = \varepsilon \mu(g,x),
\]
so $\varepsilon$ does coequalize $\mu$ and $\pi_2$. In order to check universality, let $F \colon C \to D$ be any other functor that coequalizes. We must check that there is a unique functor $\overline{F} \colon C/G \to D$ such that $\overline{F} \circ \varepsilon = F$. Any such $\overline{F}$ must be defined by $\overline{F}\big( [c] \big) = F(c)$ on objects, and since $F$ coequalizes $\mu, \pi_2$ this function is well-defined. The same argument applies to morphisms, so $\overline{F}\big( [p] \big) = F(p)$. As for objects, this function is well-defined, and also forces the uniqueness of $\overline{F}$. We need only check functoriality to finish the proof. By construction $\overline{F}$ preserves identity morphisms. For composition, we have
\begin{align*}
\overline{F} \left( [q] \circ [p] \right) & = \overline{F}\left( [g\cdot q \ \circ \ p] \right) \\
& = F(g\cdot q \ \circ \ p) \\
& = F(g \cdot q) \circ F(p) \\ 
& \stackrel{c}{=} F(q) \circ F(p) \\
& \overline{F}\left( [q] \right) \circ \overline{F}\left( [p] \right),
\end{align*}
where the equality labeled $c$ is a consequence of $F$ coequalizing $\mu, \pi_2$. This calculation shows that $\overline{F}$ is a functor, so $C/G$ is the coequalizer of $\mu, \pi_2$.
\end{proof}

\begin{rem}[(Free versus non-free actions)]\label{rem:obj-coeq}
We note that if the action of $G$ on $C$ is not free, then the coequalizer
    \[
        \xy
            (0,0)*+{G \times C}="00";
            (30,0)*+{C}="10";
            (60,0)*+{\textrm{coeq}(\mu, \pi_2)}="20";
            {\ar@<1ex>^{\mu} "00" ; "10"};
            {\ar@<-1ex>_{\pi_2} "00" ; "10"};
            {\ar^{\varepsilon} "10" ; "20"};
        \endxy
    \]
does not admit a simple description in general, although the set of objects of $\textrm{coeq}(\mu, \pi_2)$ is still given by the set of orbits of the action of $G$ on the objects of $C$ because $\textrm{ob} \colon \mb{Cat} \to \mb{Set}$ is a left adjoint. In particular, if $P$ is a $\Lambda$-operad in $\mb{Cat}$ and the action of $\Lambda(n)$ on $P(n)$ is not free, then the set of  objects of $\coeq{P}{A}{\Lambda}{n}$ is given by quotienting that set $\textrm{ob}P(n) \times \textrm{ob}A^n$ by the action of $\Lambda(n)$.
\end{rem}

\begin{Defi}[(2-limits)]\label{Defi:2limit}
Let $F \colon \mathbb{D} \rightarrow \mathcal{K}$ be a 2-functor. The (strict) 2-limit of $F$ consists of:
    \begin{itemize}
        \item an object $\lim F$ in $\mathcal{K}$,
        \item for each object $d \in \mathbb{D}$, a 1-cell $\pi_d \colon \lim F \rightarrow Fd$,
    \end{itemize}
such that:
    \begin{enumerate}
        \item For any 1-cell $f \colon d \rightarrow d'$ in $\mathbb{D}$, the following diagram commutes.
          \[
            \xy
              (0,0)*+{\lim F}="00";
              (20,10)*+{Fd}="10";
              (20,-10)*+{Fd'}="11";
              %
              {\ar^{\pi_d} "00" ; "10"};
              {\ar_{\pi_{d'}} "00" ; "11"};
              {\ar^{Ff} "10" ; "11"};
            \endxy
          \]
        \item For any 1-cells $f, g \colon d \rightarrow d'$ and any 2-cell $\alpha \colon f \Rightarrow g$ in $\mathbb{D}$,
            \[
                F\alpha \ast \id_{\pi_d} = \id_{\pi_{d'}}.
            \]
    \end{enumerate}
These data then satisfy the following universal properties:
    \begin{enumerate}
        \item For any object $X$ and 1-cells $\chi_d \colon X \rightarrow Fd$ satisfying
          \[
            \xy
              (0,0)*+{X}="00";
              (20,10)*+{Fd}="10";
              (20,-10)*+{Fd'}="11";
              %
              {\ar^{\chi_d} "00" ; "10"};
              {\ar_{\chi_{d'}} "00" ; "11"};
              {\ar^{Ff} "10" ; "11"};
            \endxy
          \]
          and the equations
           \[
                F\alpha \ast \id_{\chi_d} = \id_{\chi_{d'}}
            \]
            for all $\alpha \colon f \Rightarrow g$,   
            there exists a unique 1-cell $h \colon X \rightarrow \lim F$ in $\mathcal{K}$ such that $\pi_d \circ h = \chi_d$.
        \item For any 1-cells $h, k \colon X \rightarrow \lim F$ and 2-cells
          \[
            \xy
              (0,0)*+{X}="00";
              (20,0)*+{\lim F}="10";
              (20,-15)*+{Fd}="11";
              (0,-15)*+{\lim F}="01";
              {\ar^{h} "00" ; "10"};
              {\ar_{k} "00" ; "01"};
              {\ar^{\pi_d} "10" ; "11"};
              {\ar_{\pi_d} "01" ; "11"};
              {\ar@{=>}^{\varphi_c} (12,-5.5) ; (8,-9.5)};
            \endxy
          \]
        there exists a unique 2-cell $\gamma \colon h \rightarrow k$ such that
            \[
                \id_{\pi_d} \ast \gamma = \varphi_d.
            \]
    \end{enumerate}
\end{Defi}

\begin{rem}[(Limits versus 2-limits)]\label{rem:lim-v-2lim}
Let $C$ be a small category, and $F \colon C \to \mb{Cat}$ be a functor.
We can treat $C$ as a locally discrete 2-category, and then $F$ becomes a 2-functor.
The limit of $F$, as a functor, is then also the 2-limit of $F$, as a 2-functor by standard methods in enriched category theory \cite[Section 3.8]{Kelly2005}. In particular, every limit (or colimit) of such a diagram automatically inherits a 2-dimensional aspect to its universal property.
\end{rem}

\begin{conv}[(Naming 2-limits)]\label{conv:name-2lim}
For familiar limits such as products, terminal objects, pullbacks, and equalizers, we prepend 2- and write 2-products, 2-terminal objects, 2-pullbacks, and 2-equalizers instead. 
\end{conv}

\begin{Defi}[(Preservation of 2-limits)]\label{Defi:preserve-2lim}
Let $\mathcal{K}, \mathcal{L}$ be a 2-categories with all 2-limits of shape $\mathbb{D}$, and $F \colon \mathcal{K} \to \mathcal{L}$ a 2-functor between them. Then $F$ \emph{preserves 2-limits of shape $\mathbb{D}$} if, for every $P \colon \mathbb{D} \to \mathcal{K}$, the morphism
\[
F(\lim P) \to \lim FP
\]
induced by the universal property is an isomorphism.
\end{Defi}

\section{Background: $2$-monads and their Algebras}\label{sec:2monads}

To investigate operads in $\cat$ we will make use of $2$-monads and their algebras, specifically the notion of a pseudoalgebra for a $2$-monad. We recall the required definitions and theory related to $2$-monads here. For further reference, we refer the reader to \cite{BKP} and \cite{power-gen}.

\begin{Defi}[($2$-monad)]\label{Defi:2monad}
Let $\m{K}$ be a 2-category. A \emph{$2$-monad} on $\m{K}$ consists of
    \begin{itemize}
        \item a strict $2$-functor $T \colon \m{K} \rightarrow \m{K}$,
        \item a $2$-natural transformation $\mu \colon T^2 \Rightarrow T$,
        \item a $2$-natural transformation $\eta \colon \id_{\m{K}} \Rightarrow T$,
    \end{itemize}
satisfying the following axioms.
    \begin{itemize}
        \item The following diagram commutes.
        \[
            \xy
                (0,0)*+{T^3X}="00";
                (20,0)*+{T^2X}="10";
                (0,-15)*+{T^2X}="01";
                (20,-15)*+{TX}="11";
                {\ar^{T\mu_X} "00" ; "10"};
                {\ar^{\mu_x} "10" ; "11"};
                {\ar_{\mu_{TX}} "00" ;  "01"};
                {\ar_{\mu_X} "01" ; "11"};
            \endxy
        \]
        \item The following diagram commutes.
        \[
            \xy
                (0,0)*+{TX}="00";
                (20,0)*+{T^2X}="10";
                (40,0)*+{TX}="20";
                (20,-15)*+{TX}="11";
                {\ar^{\eta_{TX}} "00" ; "10"};
                {\ar_{T\eta_X} "20" ; "10"};
                {\ar|{\mu_X} "10" ; "11"};
                {\ar_{\id_{TX}} "00" ; "11"};
                {\ar^{\id{TX}} "20" ; "11"};
            \endxy
        \]
\end{itemize}
\end{Defi}

\begin{Defi}[(Pseudoalgebra, $2$-monad version)]\label{Defi:pseudoalgebra}
Let $T \colon \m{K} \rightarrow \m{K}$ be a $2$-monad. A $T$-\textit{pseudoalgebra} consists of an object $X$, a $1$-cell $\alpha \colon TX \rightarrow X$ in $\m{K}$, and invertible $2$-cells of $\m{K}$
    \[
        \xy
            (0,0)*+{T^2X}="00";
            (20,0)*+{TX}="10";
            (0,-15)*+{TX}="01";
            (20,-15)*+{X}="11";
            {\ar^{T\alpha} "00" ; "10"};
            {\ar^{\alpha} "10" ; "11"};
            {\ar_{\mu_X} "00" ;  "01"};
            {\ar_{\alpha} "01" ; "11"};
            {\ar@{=>}^{\Phi} (10,-5.5) ; (10,-9.5)};
            (40,0)*+{X}="20";
            (52.5,-15)*+{TX}="31";
            (72.5,-15)*+{X}="41";
            {\ar_{\eta_X} "20" ; "31"};
            {\ar_{\alpha} "31" ; "41"};
            {\ar@/^1.5pc/^{1_X} "20" ; "41"};
            {\ar@{=>}^{\Phi_{\eta}} (54.5,-5.5) ; (54.5,-9.5)};
        \endxy
    \]

satisfying the following axioms.
    \begin{itemize}
        \item The following equality of pasting diagrams holds.
    \[
        \xy
            (5,0)*+{T^3X}="t3xl";
            (29,0)*+{T^2X}="t2xl1";
            (5,-17.5)*+{T^2X}="t2xl2";
            (24,-35)*+{TX}="txl1";
            (48,-17.5)*+{TX}="txl2";
            (48,-35)*+{X}="xl";
            (24,-17.5)*+{T^2X}="t2xl3";
            {\ar^{T^2\alpha} "t3xl" ; "t2xl1"};
            {\ar^{T\alpha} "t2xl1" ; "txl2"};
            {\ar^{\alpha} "txl2" ; "xl"};
            {\ar_{\mu_{TX}} "t3xl" ; "t2xl2"};
            {\ar_{\mu_X} "t2xl2" ; "txl1"};
            {\ar_{\alpha} "txl1" ; "xl"};
            {\ar_{T\mu_X} "t3xl" ; "t2xl3"};
            {\ar^{T\alpha} "t2xl3" ; "txl2"};
            {\ar_{\mu_X} "t2xl3" ; "txl1"};
            {\ar@{=>}_{T\Phi} (26,-6) ; (26,-10)};
            {\ar@{=>}^{\Phi} (36,-24) ; (36,-28)};
            (64,0)*+{T^3X}="t3xr";
            (88,0)*+{T^2X}="t2xr1";
            (64,-17.5)*+{T^2X}="t2xr2";
            (83,-35)*+{TX}="txr1";
            (107,-17.5)*+{TX}="txr2";
            (107,-35)*+{X}="xr";
            (88,-17.5)*+{TX}="txr3";
            {\ar^{T^2\alpha} "t3xr" ; "t2xr1"};
            {\ar^{T\alpha} "t2xr1" ; "txr2"};
            {\ar^{\alpha} "txr2" ; "xr"};
            {\ar_{\mu_{TX}} "t3xr" ; "t2xr2"};
            {\ar_{\mu_X} "t2xr2" ; "txr1"};
            {\ar_{\alpha} "txr1" ; "xr"};
            {\ar_{T\alpha} "t2xr2" ; "txr3"};
            {\ar_{\alpha} "txr3" ; "xr"};
            {\ar_{\mu_X} "t2xr1" ; "txr3"};
            {\ar@{=>}_{\Phi} (98,-15) ; (98,-19)};
            {\ar@{=>}^{\Phi} (85,-24) ; (85,-28)};
            {\ar@{=} (54,-20) ; (56,-20)};
        \endxy
    \]

    \item The following pasting diagram is an identity.
    \[
        \xy
            (0,0)*+{TX}="txl1";
            (15,-15)*+{T^2X}="t2x";
            (15,-30)*+{TX}="txl2";
            (35,-15)*+{TX}="txl3";
            (35,-30)*+{X}="xl";
            {\ar@/^1.7pc/^{1_{TX}} "txl1" ; "txl3"};
            {\ar@/_1.7pc/_{1_{TX}} "txl1" ; "txl2"};
            {\ar_{T\eta_X} "txl1" ; "t2x"};
            {\ar^{T\alpha} "t2x" ; "txl3"};
            {\ar_{\mu_X} "t2x" ; "txl2"};
            {\ar_{\alpha} "txl2" ; "xl"};
            {\ar^{\alpha} "txl3" ; "xl"};
            {\ar@{=>}^{T\Phi_\eta} (17,-5.5) ; (17,-9.5)};
            {\ar@{=>}^{\Phi} (25,-20.5) ; (25,-24.5)};
        \endxy
    \]

    \end{itemize}
\end{Defi}

\begin{rem}[(Omitted third axiom)]\label{rem:third-axiom}
Power's definition of a pseudoalgebra includes a third axiom relating to the unit of the $2$-monad \cite[Definition 2.4, Axiom 2.1]{power-gen}. However, following an argument of Marmolejo \cite[Lemma 9.1]{marm-doct} this extra axiom is redundant and is omitted here.
\end{rem}

\begin{Defi}[(Strict algebra, $2$-monad version)]\label{Defi:strictalgebra}
Let $T \colon \m{K} \rightarrow \m{K}$ be a $2$-monad. A \textit{strict $T$-algebra} is a pseudoalgebra in which all of the isomorphisms $\Phi$ are identities.
\end{Defi}

\begin{Defi}[(Pseudomorphism, $2$-monad version)]\label{Defi:pseudomorphism}
Let $T$ be a $2$-monad and let $(X,\alpha,\Phi,\Phi_\eta)$, $(Y,\beta,\Psi,\Psi_\eta)$ be $T$-pseudoalgebras. A \textit{pseudomorphism} $(f, \bar{f})$ between these pseudoalgebras consists of a $1$-cell $f \colon X \rightarrow Y$ along with an invertible $2$-cell
    \[
        \xy
            (0,0)*+{TX}="00";
            (20,0)*+{TY}="10";
            (0,-15)*+{X}="01";
            (20,-15)*+{Y}="11";
            {\ar^{Tf} "00" ; "10"};
            {\ar^{\beta} "10" ; "11"};
            {\ar_{\alpha} "00" ; "01"};
            {\ar_{f} "01" ; "11"};
            {\ar@{=>}^{\bar{f}} (10,-5.5) ; (10,-9.5)};
        \endxy
    \]

satisfying the following axioms.
    \begin{itemize}
        \item The following equality of pasting diagrams holds.
                \[
        \xy
            (5,0)*+{T^2X}="t3xl";
            (29,0)*+{T^2Y}="t2xl1";
            (5,-17.5)*+{TX}="t2xl2";
            (24,-35)*+{TX}="txl1";
            (48,-17.5)*+{TY}="txl2";
            (48,-35)*+{Y}="xl";
            (24,-17.5)*+{TX}="t2xl3";
            {\ar^{T^2f} "t3xl" ; "t2xl1"};
            {\ar^{T\beta} "t2xl1" ; "txl2"};
            {\ar^{\beta} "txl2" ; "xl"};
            {\ar_{\mu_X} "t3xl" ; "t2xl2"};
            {\ar_{\alpha} "t2xl2" ; "txl1"};
            {\ar_{f} "txl1" ; "xl"};
            {\ar^{T\alpha} "t3xl" ; "t2xl3"};
            {\ar^{Tf} "t2xl3" ; "txl2"};
            {\ar_{\alpha} "t2xl3" ; "txl1"};
            {\ar@{=>}^{T\bar{f}} (24,-6) ; (24,-10)};
            {\ar@{=>}^{\bar{f}} (36,-24) ; (36,-28)};
            {\ar@{=>}^{\Phi} (13.5,-15.5) ; (13.5,-19.5)};
            (64,0)*+{T^2X}="t3xr";
            (88,0)*+{T^2Y}="t2xr1";
            (64,-17.5)*+{TX}="t2xr2";
            (83,-35)*+{TX}="txr1";
            (107,-17.5)*+{TY}="txr2";
            (107,-35)*+{Y}="xr";
            (88,-17.5)*+{TX}="txr3";
            {\ar^{T^2f} "t3xr" ; "t2xr1"};
            {\ar^{T\beta} "t2xr1" ; "txr2"};
            {\ar^{\beta} "txr2" ; "xr"};
            {\ar_{\mu_{X}} "t3xr" ; "t2xr2"};
            {\ar_{\alpha} "t2xr2" ; "txr1"};
            {\ar_{f} "txr1" ; "xr"};
            {\ar_{Tf} "t2xr2" ; "txr3"};
            {\ar_{\beta} "txr3" ; "xr"};
            {\ar_{\mu_Y} "t2xr1" ; "txr3"};
            {\ar@{=>}_{\Psi} (98,-15) ; (98,-19)};
            {\ar@{=>}^{\bar{f}} (85,-24) ; (85,-28)};
            {\ar@{=} (54,-20) ; (56,-20)};
        \endxy
    \]
    \item The following equality of pasting diagrams holds.
            \[
                        \xy
            (0,0)*+{X}="00";
            (20,0)*+{Y}="10";
            (0,-20)*+{TX}="01";
            (20,-20)*+{TY}="11";
            (10,-35)*+{X}="02";
            (30,-35)*+{Y}="12";
            {\ar^{f} "00" ; "10"};
            {\ar@/^1.5pc/^{1_Y} "10" ; "12"};
            {\ar_{\eta_X} "00" ; "01"};
            {\ar_{\eta_Y} "10" ; "11"};
            {\ar_{Tf} "01" ; "11"};
            {\ar_{\alpha} "01" ; "02"};
            {\ar_{f} "02" ; "12"};
            {\ar^{\beta} "11" ; "12"};
            {\ar@{=>}^{\bar{f}} (15,-25.5) ; (15,-29.5)};
            {\ar@{=>}^{\Psi_{\eta}} (25,-17) ; (25,-21)};
            (50,0)*+{X}="30";
            (70,0)*+{Y}="40";
            (50,-20)*+{TX}="31";
            (60,-35)*+{X}="32";
            (80,-35)*+{Y}="42";
            {\ar^{f} "30" ; "40"};
            {\ar_{\eta_X} "30" ; "31"};
            {\ar_{\alpha} "31" ; "32"};
            {\ar_{f} "32" ; "42"};
            {\ar@/^1.5pc/^{1_X} "30" ; "32"};
            {\ar@/^1.5pc/^{1_Y} "40" ; "42"};
            {\ar@{=>}^{\Phi_{\eta}} (55,-17) ; (55,-21)};
        \endxy
        \]
\end{itemize}
\end{Defi}

\begin{Defi}[(Strict morphism, $2$-monad verison)]\label{Defi:strictmorphism}
Let $T$ be a $2$-monad and let $(X,\alpha,\Phi,\Phi_\eta)$ and $(Y,\beta,\Psi,\Psi_\eta)$ be $T$-pseudoalgebras. A \textit{strict morphism} $(f, \bar{f})$ consists of a pseudomorphism in which $\bar{f}$ is an identity.
\end{Defi}

\begin{rem}
The strict algebras and strict morphisms are exactly the same as algebras and morphisms for the underlying monad on the underlying category of $\m{K}$.
\end{rem}

\begin{Defi}[($T$-transformation, $2$-monad version)]\label{Defi:Ttrans}
Let $(f, \overline{f}), (g, \overline{g}) \colon X \rightarrow Y$ be pseudomorphisms of $T$-algebras. A \textit{$T$-transformation} consists of a $2$-cell $\gamma \colon f \Rightarrow g$ such that the following equality of pasting diagrams holds.
    \[
        \xy
            (0,0)*+{TX}="00";
            (30,0)*+{TY}="10";
            (0,-20)*+{X}="01";
            (30,-20)*+{Y}="11";
            {\ar@/^1.5pc/^{Tf} "00" ; "10"};
            {\ar_{Tg} "00" ; "10"};
            {\ar^{\beta} "10" ; "11"};
            {\ar_{\alpha} "00" ; "01"};
            {\ar_{g} "01" ; "11"};
            {\ar@{=>}^{T \gamma} (13.5,5.5) ; (13.5,1.5)};
            {\ar@{=>}^{\overline{g}} (13.5,-8) ; (13.5,-12)};
            (60,0)*+{TX}="x00";
            (90,0)*+{TY}="x10";
            (60,-20)*+{X}="x01";
            (90,-20)*+{Y}="x11";
            {\ar^{Tf} "x00" ; "x10"};
            {\ar^{\beta} "x10" ; "x11"};
            {\ar_{\alpha} "x00" ; "x01"};
            {\ar^{f} "x01" ; "x11"};
            {\ar@/_1.5pc/_{g} "x01" ; "x11"};
            {\ar@{=>}^{\gamma} (75,-21.5) ; (75,-25.5)};
            {\ar@{=>}^{\overline{f}} (75,-8) ; (75,-12)};
            {\ar@{=} (42.75,-10) ; (46.75,-10)};
        \endxy
    \]
\end{Defi}

There are many different possible choices of 2-categories in which the objects are some kind of algebra over a $2$-monad $T$.
Here are the two that will be the most important for us.

\begin{Defi}[(2-categories of algebras, $2$-monad version)]\label{Defi:2cat-of-algs}
Let $T$ be a $2$-monad.
\begin{itemize}
\item The $2$-category $T\mbox{-}\mb{Alg}_{s}$ consists of strict $T$-algebras, strict morphisms, and $T$-transformations.
\item The $2$-category $\mb{Ps}\mbox{-}T\mbox{-}\mb{Alg}$ consists of $T$-pseudoalgebras, pseudomorphisms, and $T$-transformations.
\end{itemize}
\end{Defi}

\section{\texorpdfstring{$\Lambda$}{L}-Operads in Cat as $2$-monads}\label{sec:op-to-2monad}

This section begins our study of algebras over a $\Lambda$-operad $P$ in $\mb{Cat}$.
This theory blends together standard results in both $2$-monad theory \cite{lack-cod,BKP} and operad theory \cite{maygeom}.

\begin{rem}[($\Lambda$-operads in $\mb{Cat}$)]\label{rem:lop-in-cat}
Here we explicitly describe the structure of a $\Lambda$-operad $P$ in $\mb{Cat}$, following \cref{rem:nary-ops-V}.
A $\Lambda$-operad $P$ in $\mb{Cat}$ consists of
\begin{itemize}
\item a category, $P(n)$, for each natural number $n$,
\item for each $n$, a right $\Lambda(n)$-action on $P(n)$ as per \cref{conv:group-act-cat},
\item an object $\id \in P(1)$, and
\item functors
  \[
    \mu \colon  P(n) \times P(k_{1}) \times \cdots \times P(k_{n}) \rightarrow P(k_{1} + \cdots + k_{n}),
  \]
\end{itemize}
satisfying the first two axioms from \cref{Defi:sym-op} and the two equivariance axioms from \cref{Defi:lamop}.
\end{rem}

\begin{Defi}[(Pseudoalgebra, $\Lambda$-operad version)]\label{def:ps-alg-lop}
Let $P$ be a $\Lambda$-operad in $\mb{Cat}$. A \textit{pseudoalgebra} for $P$ consists of: 
    \begin{itemize}
        \item a category $X$,
        \item a family of functors
            \[
                \left(\alpha_n \colon \coeq{P}{X}{\Lambda}{n} \rightarrow X \right)_{n \in \mathbb{N}},
            \]
        \item for each $n, k_1, \ldots, k_n \in \mathbb{N}$, a natural isomorphism $\phi_{k_1, \ldots, k_n}$ (corresponding, via \cref{conv:equiv-maps,rem:lim-v-2lim}) to a natural isomorphism
            \[
                \xy
                    (0,0)*+{\scriptstyle P_n \times \prod_{i=1}^n \left(P_{k_i} \times X^{k_i}\right)}="00";
                    (0,-10)*+{\scriptstyle P_n \times \prod_{i=1}^n P_{k_i} \times X^{\Sigma k_i}}="01";
                    (0,-20)*+{\scriptstyle P_{\Sigma k_i} \times X^{\Sigma k_i}}="02";
                    (60,-20)*+{\scriptstyle X}="12";
                    (60,0)*+{\scriptstyle P_n \times X^n}="11";
                    {\ar_{} "00" ; "01"};
                    {\ar^{1 \times \prod \tilde{\alpha}_{k_i}} "00" ; "11"};
                    {\ar^{\tilde{\alpha}_n} "11" ; "12"};
                    {\ar_{\mu^P \times 1} "01" ; "02"};
                    {\ar_>>>>>>>>>>>>>>>>>>>{\tilde{\alpha}_{\Sigma k_i}} "02" ; "12"};
                    {\ar@{=>}^{\tilde{\phi}_{k_1, \ldots, k_n}} (30,-8) ; (30,-12)};
                \endxy
            \]

               \item and a natural isomorphism $\phi_{\eta}$ corresponding to a natural isomorphism
            \[
                \xy
                    (0,0)*+{X}="00";
                    (0,-15)*+{1 \times X}="x10";
                    (0,-30)*+{P(1) \times X}="10";
                    (30,-30)*+{X}="11";
                    {\ar_{\eta^P \times 1} "x10" ; "10"};
                    {\ar_{\tilde{\alpha}_1} "10" ; "11"};
                    {\ar^{1} "00" ; "11"};
                    {\ar_{\cong} "00" ; "x10"};
                    {\ar@{=>}^{\tilde{\phi}_\eta} (10,-18) ; (10,-22)};
                \endxy
            \]

    \end{itemize}
satisfying the following axioms.
    \begin{itemize}
        \item For all $n, k_i, m_{ij} \in \mathbb{N}$, the following equality of pasting diagrams holds.
            \[
                \xy
                    (0,0)*+{\scriptstyle P_n \times \prod_i\left(P_{k_i} \times \prod_j \left(P_{m_{ij}} \times X^{m_{ij}}\right)\right)}="00";
                    (60,0)*+{\scriptstyle P_n \times \prod_i \left(P_{k_i} \times X^{k_i}\right)}="10";
                    (0,-30)*+{\scriptstyle P_{\Sigma k_i} \times \prod_i\prod_j\left(P_{m_{ij}} \times X^{m_{ij}}\right)}="02";
                    (30,-50)*+{\scriptstyle P_{\Sigma\Sigma m_{ij}} \times X^{\Sigma \Sigma m_{ij}}}="04";
                    (80,-20)*+{\scriptstyle P_n \times X^n}="12";
                    (80,-50)*+{\scriptstyle X}="14";
                    {\ar^>>>>>>>>>>>>>>{1 \times \prod\left(1 \times \prod \tilde{\alpha}_{m_ij}\right)} "00" ; "10"};
                    {\ar^{1 \times \prod \tilde{\alpha}_{k_i}} "10" ; "12"};
                    {\ar^{\tilde{\alpha}_n} "12" ; "14"};
                    {\ar_{\mu^P \times 1} "00" ; "02"};
                    {\ar_{\mu^P \times 1} "02" ; "04"};
                    {\ar_{\tilde{\alpha}_{\Sigma\Sigma m_{ij}}} "04" ; "14"};
                    (30,-20)*+{\scriptstyle P_n \times \prod_i\left(P_{\Sigma m_{ij}} \times X^{\Sigma m_{ij}}\right)}="22";
                    {\ar^{\mu^P \times 1} "00" ; "22"};
                    {\ar^{1 \times \prod \tilde{\alpha}_{\Sigma m_{ij}}} "22" ; "12"};
                    {\ar^{\mu^P \times 1} "22" ; "04"};
                    (0,-70)*+{\scriptstyle P_n \times \prod_i\left(P_{k_i} \times \prod_j \left(P_{m_{ij}} \times X^{m_{ij}}\right)\right)}="b00";
                    (50,-70)*+{\scriptstyle P_n \times \prod_i \left(P_{k_i} \times X^{k_i}\right)}="b10";
                    (0,-100)*+{\scriptstyle P_{\Sigma k_i} \times \prod_i\prod_j\left(P_{m_{ij}} \times X^{m_{ij}}\right)}="b02";
                    (20,-120)*+{\scriptstyle P_{\Sigma\Sigma m_{ij}} \times X^{\Sigma \Sigma m_{ij}}}="b04";
                    (80,-90)*+{\scriptstyle P_n \times X^n}="b12";
                    (80,-120)*+{\scriptstyle X}="b14";
                    {\ar^>>>>>>>>>{1 \times \prod\left(1 \times \prod \tilde{\alpha}_{m_ij}\right)} "b00" ; "b10"};
                    {\ar^{1 \times \prod \tilde{\alpha}_{k_i}} "b10" ; "b12"};
                    {\ar^{\tilde{\alpha}_n} "b12" ; "b14"};
                    {\ar_{\mu^P \times 1} "b00" ; "b02"};
                    {\ar_{\mu^P \times 1} "b02" ; "b04"};
                    {\ar_{\tilde{\alpha}_{\Sigma\Sigma m_{ij}}} "b04" ; "b14"};
                    (50,-100)*+{\scriptstyle P_{\Sigma k_i} \times X^{\Sigma k_i}}="b22";
                    {\ar_{\mu^P \times 1} "b10" ; "b22"};
                    {\ar^>>>>>>>>>>>>>>>>{1 \times \prod\prod \tilde{\alpha}_{m_{ij}}} "b02" ; "b22"};
                    {\ar^{\tilde{\alpha}_{\Sigma k_i}} "b22" ; "b14"};
                    {\ar@{=>}^{1 \times \prod_i \tilde{\phi}_{m_{i1}, \ldots, m_{ik_{i}}}} (35,-8) ; (35,-12)};
                    {\ar@{=>}^{\tilde{\phi}_{\Sigma m_{1j}, \ldots, \Sigma m_{nj}}} (50,-33) ; (50,-37)};
                    {\ar@{=>}^{\tilde{\phi}_{k_1,\ldots,k_n}} (60,-92) ; (60,-96)};
                    {\ar@{=>}^{\tilde{\phi}_{m_{11}, \ldots, m_{nk_n}}} (30,-108) ; (30,-112)};
                    {\ar@{=} (45,-58) ; (45,-62)};
                \endxy
            \]
        \item Each pasting diagram of the following form is an identity.
            \[
                \xy
                    (0,0)*+{P_n \times X^n}="00";
                    (12,-12)*+{P_n \times (1 \times X)^n}="11";
                    (24,-24)*+{P_n \times (P_1 \times X)^n}="22";
                    (60,-24)*+{P_n \times X^n}="32";
                    (60,-48)*+{X}="34";
                    (24,-36)*+{P_n \times P_1^n \times X^n}="23";
                    (24,-48)*+{P_n \times X^n}="24";
                    {\ar@/^2.5pc/^{1} "00" ; "32"};
                    {\ar^{\tilde{\alpha}_n} "32" ; "34"};
                    {\ar^{\cong} "00" ; "11"};
                    {\ar^>>>{1 \times \left(\eta^P \times 1\right)^n} "11" ; "22"};
                    {\ar^>>>>>>{1 \times \tilde{\alpha}_1^n} "22" ; "32"};
                    {\ar@/_3pc/_{1} "00" ; "24"};
                    {\ar_{\cong} "22" ; "23"};
                    {\ar_{\mu^P \times 1} "23" ; "24"};
                    {\ar_{\tilde{\alpha}_n} "24" ; "34"};
                    {\ar@{=>}^{1 \times \tilde{\phi}_\eta^n} (32,-8) ; (32,-12)};
                    {\ar@{=>}^{\tilde{\phi}_{1,\ldots,1}} (40,-34) ; (40,-38)};
                \endxy
            \]
    \end{itemize}

\end{Defi}

\begin{Defi}[(Strict algebra, $\Lambda$-operad version)]\label{Defi:strictalgebra-lop}
Let $P$ be a $\Lambda$-operad. A \textit{strict algebra} for $P$ consists of a pseudoalgebra in which all of the isomorphisms $\phi$ are identities.
\end{Defi}

\begin{Defi}[(Pseudomorphism, $\Lambda$-operad version)]\label{Defi:pseudomorphism-lop}
Let $(X, \alpha_n,\phi,\phi_\eta)$ and $(Y, \beta_n,\psi,\psi_{\eta})$ be pseudoalgebras for a $\Lambda$-operad $P$. A $P$-\textit{pseudomorphism} consists of
    \begin{itemize}
        \item a functor $f \colon X \rightarrow Y$
        \item for each $n \in \mathbb{N}$, a natural isomorphism $f_n$ (corresponding, via \cref{conv:equiv-maps,rem:lim-v-2lim}) to a natural isomorphism
            \[
                \xy
                    (0,0)*+{P_n \times X^n}="00";
                    (20,0)*+{X}="10";
                    (0,-15)*+{P_n \times Y^n}="01";
                    (20,-15)*+{Y}="11";
                    {\ar^>>>>>{\tilde{\alpha}_n} "00" ; "10"};
                    {\ar^{f} "10" ; "11"};
                    {\ar_{1 \times f^n} "00" ; "01"};
                    {\ar_>>>>>{\tilde{\beta}_n} "01" ; "11"};
                    {\ar@{=>}^{\overline{f}_n} (10,-5.5) ; (10,-9.5)};
                \endxy
            \]

        \end{itemize}
satisfying the following axioms.
    \begin{itemize}
        \item The following equality of pasting diagrams holds.
            \[
                \xy
                    (0,0)*+{\scriptstyle P_n \times \prod_i (P_{k_i} \times X^{k_i})}="00";
                    (50,0)*+{\scriptstyle P_n \times \prod_i (P_{k_i} \times Y^{k_i})}="10";
                    (0,-25)*+{\scriptstyle P_{\Sigma k_i} \times X^{\Sigma k_i}}="01";
                    (50,-25)*+{\scriptstyle P_{\Sigma k_i} \times Y^{\Sigma k_i}}="11";
                    (75,-15)*{\scriptstyle P_n \times Y^n}="21";
                    (75,-40)*+{\scriptstyle Y}="22";
                    (25,-40)*+{\scriptstyle X}="02";
                    {\ar^{1 \times \prod(1 \times f^{k_i})} "00" ; "10"};
                    {\ar^{1 \times \prod \tilde{\beta}_{k_i}} "10" ; "21"};
                    {\ar_{\mu^P \times 1} "00" ; "01"};
                    {\ar_{\tilde{\alpha}_{\Sigma k_i}} "01" ; "02"};
                    {\ar_{f} "02" ; "22"};
                    {\ar^{1 \times f^{\Sigma k_i}} "01" ; "11"};
                    {\ar_{\tilde{\beta}_{\Sigma k_i}} "11" ; "22"};
                    {\ar_{\mu^P \times 1} "10" ; "11"};
                    {\ar^{\tilde{\beta}_n} "21" ; "22"};
                    {\ar@{=>}^{\overline{f}_n} (37.5,-30.5) ; (37.5,-34.5)};
                    {\ar@{=>}^{\tilde{\psi}_{k_1,\ldots,k_n}} (57.5,-16.5) ; (57.5,-20.5)};
                    (0,-55)*+{\scriptstyle P_n \times \prod_i (P_{k_i} \times X^{k_i})}="b00";
                    (50,-55)*+{\scriptstyle P_n \times \prod_i (P_{k_i} \times Y^{k_i})}="b10";
                    (0,-80)*+{\scriptstyle P_{\Sigma k_i} \times X^{\Sigma k_i}}="b01";
                    (25,-70)*+{\scriptstyle P_n \times X^n}="b11";
                    (75,-70)*{\scriptstyle P_n \times Y^n}="b21";
                    (75,-95)*+{\scriptstyle Y}="b22";
                    (25,-95)*+{\scriptstyle X}="b02";
                    {\ar^{1 \times \prod(1 \times f^{k_i})} "b00" ; "b10"};
                    {\ar^{1 \times \prod \tilde{\beta}_{k_i}} "b10" ; "b21"};
                    {\ar_{\mu^P \times 1} "b00" ; "b01"};
                    {\ar_{\tilde{\alpha}_{\Sigma k_i}} "b01" ; "b02"};
                    {\ar_{f} "b02" ; "b22"};
                    {\ar^{\tilde{\beta}_n} "b21" ; "b22"};
                    {\ar^{1 \times \prod \tilde{\alpha}_{k_i}} "b00" ; "b11"};
                    {\ar^{1 \times f^n} "b11" ; "b21"};
                    {\ar_{\tilde{\alpha}_n} "b11" ; "b02"};
                    {\ar@{=>}^{\overline{f}_n} (50,-80.5) ; (50,-84.5)};
                    {\ar@{=>}^{1 \times \prod\overline{f}_{k_i}} (37.5,-60.5) ; (37.5,-64.5)};
                    {\ar@{=>}^{\tilde{\phi}_{k_1,\ldots,k_n}} (9,-72) ; (9,-76)};
                    {\ar@{=} (37.5,-45.5) ; (37.5,-49.5)};
                \endxy
            \]
            \item The following equality of pasting diagrams holds.
                \[
                    \xy
                        (0,0)*+{X}="00";
                        (20,0)*+{Y}="10";
                        (0,-15)*+{1 \times X}="01";
                        (20,-15)*+{1 \times Y}="11";
                        (0,-30)*+{P_1 \times X}="02";
                        (20,-30)*+{P_1 \times Y}="12";
                        (20,-45)*+{X}="r02";
                        (40,-45)*+{Y}="r12";
                        {\ar^{f} "00" ; "10"};
                        {\ar@/^2pc/^{1} "10" ; "r12"};
                        {\ar_{\cong} "00" ; "01"};
                        {\ar_{\eta^P \times 1} "01" ; "02"};
                        {\ar_{\tilde{\alpha}_1} "02" ; "r02"};
                        {\ar^{1 \times f} "01" ; "11"};
                        {\ar^{1 \times f} "02" ; "12"};
                        {\ar^{\tilde{\beta}_1} "12" ; "r12"};
                        {\ar_{\cong} "10" ; "11"};
                        {\ar_{\eta^P \times 1} "11" ; "12"};
                        {\ar_{f} "r02" ; "r12"};
                        {\ar@{=>}^{\overline{f}_1} (20,-35.5) ; (20,-39.5)};
                        {\ar@{=>}^{\tilde{\psi}_{\eta}} (30,-20) ; (30,-24)};
                        (60,0)*+{X}="x00";
                        (80,0)*+{Y}="x10";
                        (60,-15)*+{1 \times X}="x01";
                        (60,-30)*+{P_1 \times X}="x02";
                        (80,-45)*+{X}="xr02";
                        (100,-45)*+{Y}="xr12";
                        {\ar^{f} "x00" ; "x10"};
                        {\ar@/^2pc/^{1} "x10" ; "xr12"};
                        {\ar_{\cong} "x00" ; "x01"};
                        {\ar_{\eta^P \times 1} "x01" ; "x02"};
                        {\ar_{\tilde{\alpha}_1} "x02" ; "xr02"};
                        {\ar_{f} "xr02" ; "xr12"};
                        {\ar@/^2pc/^{1} "x00" ; "xr02"};
                        {\ar@{=>}^{\tilde{\phi}_\eta} (70,-20) ; (70,-24)};
                        {\ar@{=} (45,-22.5) ; (49,-22.5)};
                    \endxy
                \]
    \end{itemize}
\end{Defi}

\begin{Defi}[(Strict morphism, $\Lambda$-operad version)]\label{Defi:strictmorphism-lop}
Let $(X, \alpha_n,\phi,\phi_\eta)$ and $(Y, \beta_n,\psi,\psi_{\eta})$ be pseudoalgebras for a $\Lambda$-operad $P$. A \textit{strict morphism} is a pseudomorphism in which all of the isomorphisms $\overline{f}_{n}$ are identities.
\end{Defi}

\begin{rem}
A strict algebra for a $\Lambda$-operad $P$ in $\mb{Cat}$ is precisely the same thing as an algebra for $P$ considered as an operad in the \textit{category} of small categories and functors. A strict morphism between strict algebras is then just a map of $P$-algebras in the standard sense. We could also consider the notion of a lax algebra for an operad, or a lax morphism of algebras, simply by considering natural transformations in place of isomorphisms in the definitions.
\end{rem}

\begin{rem}[(Equivariance axioms, or lack thereof)]\label{rem:eq-lack}
In the version of \cref{Defi:pseudomorphism-lop} that appeared in the original preprint \cite[Definition 2.4]{cg-preprint}, we did not state clearly that the isomorphisms $\overline{f}_n$ should satisfy an equivariance condition. This was highlighted in Remark 2.22 of Rubin's thesis \cite{rubin-thesis}. Similarly, this omission is also explicity stated as Definition 2.23 of \cite{guillou_symmetric}, as mentioned in \cite{guillou_multiplicative}. These equivariance axioms are a consequence of \cref{conv:equiv-maps,rem:lim-v-2lim}. In \cref{Defi:pseudomorphism-lop} we require the existence of natural isomorphisms $f_n$ in order to induce corresponding natural isomorphisms $\overline{f}_n$. That the $\overline{f}_n$ are induced by the $f_n$ corresponds to the fact that the $\overline{f}_n$ satisfy an equivariance condition, namely that for $(\sigma, g, x_1, \ldots, x_n) \in P(n) \times \Lambda(n) \times X^n$, we have
  \[
    \left(\overline{f}_n\right)_{\left(\sigma \cdot g, x_1, \ldots, x_n\right)} = \left(\overline{f}_n\right)_{\left(\sigma,x_{g^{-1}(1)},\ldots,x_{g^{-1}(n)}\right)}.
  \]
\end{rem}

\begin{Defi}[($P$-transformation, $\Lambda$-operad version)]\label{Defi:Ptrans}
Let $P$ be a $\Lambda$-operad and let $f, g \colon (X, \alpha, \phi, \phi_\eta) \rightarrow (Y, \beta, \psi, \psi_\eta)$ be pseudomorphisms of $P$-pseudoalgebras. A \textit{$P$-transformation} is then a natural transformation $\gamma \colon f \Rightarrow g$ such that the following equality of pasting diagrams holds, for all $n$.
    \[
        \xy
            (0,0)*+{P_n \times X^n}="00";
            (30,0)*+{P_n \times Y^n}="10";
            (0,-20)*+{X}="01";
            (30,-20)*+{Y}="11";
            {\ar@/^1.5pc/^{1 \times f^n} "00" ; "10"};
            {\ar_{1 \times g^n} "00" ; "10"};
            {\ar^{\tilde{\beta}_n} "10" ; "11"};
            {\ar_{\tilde{\alpha}_n} "00" ; "01"};
            {\ar_{g} "01" ; "11"};
            {\ar@{=>}^{1 \times \gamma^n} (13.5,5.5) ; (13.5,1.5)};
            {\ar@{=>}^{\overline{g}_n} (13.5,-8) ; (13.5,-12)};
            (60,0)*+{P_n \times X^n}="x00";
            (90,0)*+{P_n \times Y^n}="x10";
            (60,-20)*+{X}="x01";
            (90,-20)*+{Y}="x11";
            {\ar^{1 \times f^n} "x00" ; "x10"};
            {\ar^{\tilde{\beta}_n} "x10" ; "x11"};
            {\ar_{\tilde{\alpha}_n} "x00" ; "x01"};
            {\ar^{f} "x01" ; "x11"};
            {\ar@/_1.5pc/_{g} "x01" ; "x11"};
            {\ar@{=>}^{\gamma} (75,-21.5) ; (75,-25.5)};
            {\ar@{=>}^{\overline{f}_n} (75,-8) ; (75,-12)};
            {\ar@{=} (42.75,-10) ; (46.75,-10)};
        \endxy
    \]
\end{Defi}

We can form various $2$-categories using these cells.

\begin{Defi}[(2-categories of algebras, $\Lambda$-operad version)]\label{Defi:2cat-of-algs-lop}
Let $P$ be a $\Lambda$-operad.
\begin{itemize}
\item The $2$-category $P\mbox{-}\mb{Alg}_{s}$ consists of strict $P$-algebras, strict morphisms, and $P$-transformations.
\item The $2$-category $\mb{Ps}\mbox{-}P\mbox{-}\mb{Alg}$ consists of $P$-pseudoalgebras, pseudomorphisms, and $P$-transformations.
\end{itemize}
\end{Defi}

Our first main result in this section is the following, showing that one can consider algebras and higher cells, in either strict or pseudo strength, using either the operadic or $2$-monadic incarnation of a $\Lambda$-operad $P$. This theorem extends \cref{prop:op=monad1} to the 2-dimensional context.

\begin{thm}\label{thm:2monad=op}
Let $P$ be a $\Lambda$-operad in $\mb{Cat}$, and let $\underline{P}$ denote the monad on the category of small categories from \cref{Defi:und-P}.
\begin{itemize}
\item The monad $\underline{P}$ is the underlying monad of a 2-monad on the 2-category $\mb{Cat}$ that we also denote $\underline{P}$.
\item There is an isomorphism of $2$-categories
    \[
        P\mbox{-}\mb{Alg}_{s} \cong \underline{P}\mbox{-}\mb{Alg}_{s}.
    \]
\item There is an isomorphism of $2$-categories
    \[
        \mb{Ps}\mbox{-}P\mbox{-}\mb{Alg} \cong \mb{Ps}\mbox{-}\underline{P}\mbox{-}\mb{Alg}
    \]
    extending the one above.
\end{itemize}
\end{thm}
\begin{proof}
We begin by noting that we will suppress the difference between $2$-cells $\Gamma$ and the corresponding 2-cells $\tilde{\Gamma}$ by applying the 2-dimensional part of \cref{rem:lim-v-2lim} to \cref{conv:equiv-maps}. 
A proof of the first statement follows from our proof of the second by inserting identities where appropriate. Thus we begin by constructing a $2$-functor $R \colon \mb{Ps}\mbox{-}\underline{P}\mbox{-}\mb{Alg} \rightarrow \mb{Ps}\mbox{-}P\mbox{-}\mb{Alg}$. We map a $\underline{P}$-pseudoalgebra $(X,\alpha,\Phi,\Phi_\eta)$ to the following $P$-pseudoalgebra structure on the same category $X$. First we define the functor $\alpha_n$ to be the composite
    \[
        \xy
            (0,0)*+{\coeq{P}{X}{\Lambda}{n}}="00";
            (35,0)*+{\underline{P}(X)}="10";
            (55,0)*+{X.}="20";
            {\ar@{^{(}->} "00" ; "10"};
            {\ar^{\alpha} "10" ; "20"};
        \endxy
    \]
The isomorphisms $\phi_{k_1,\ldots,k_n}$ are defined using $\Phi$ as in the following diagram

    \[
        \xy
            (-10,0)*+{\scriptstyle P_n \times \prod_{i=1}^n\left(P_{k_i} \times X^{k_i}\right)}="00";
            (30,0)*+{\scriptstyle P_n \times \prod_i \left( P_{k_i} \times_{\Lambda_{k_i}} X^{k_i} \right)}="10";
            (60,0)*+{\scriptstyle P_n \times \underline{P}(X)^n}="20";
            (90,0)*+{\scriptstyle P_n \times X^n}="30";
            (-10,-20)*+{\scriptstyle P_n \times \prod_{i} P_{k_i} \times X^{\Sigma k_I}}="01";
            (-10,-40)*+{\scriptstyle P_{\Sigma k_i} \times X^{\Sigma k_{i}}}="02";
            (60,-10)*+{\scriptstyle P_n \times_{\Lambda_n} \underline{P}(X)^n}="21";
            (60,-20)*+{\scriptstyle \underline{P}^2(X)}="22";
            (90,-10)*+{\scriptstyle P_n \times_{\Lambda_n} X^n}="31";
            (90,-20)*+{\scriptstyle \underline{P}(X)}="32";
            (30,-40)*+{\scriptstyle P_{\Sigma k_i} \times_{\Lambda_{\Sigma k_i}} X^{\Sigma k_i}}="12";
            (60,-40)*+{\scriptstyle \underline{P}(X)}="23";
            (90,-40)*+{\scriptstyle X}="33";
            {\ar "00" ; "10"};
            {\ar "00" ; "01"};
            {\ar_{\mu^P \times 1} "01" ; "02"};
            {\ar@{^{(}->} "10" ; "20"};
            {\ar "20" ; "21"};
            {\ar^{1 \times \alpha^n} "20" ; "30"};
            {\ar "30" ; "31"};
            {\ar@{^{(}->} "21" ; "22"};
            {\ar^{\underline{P}\alpha} "22" ; "32"};
            {\ar@{^{(}->} "31" ; "32"};
            {\ar_{\mu_X} "22" ; "23"};
            {\ar_{\alpha} "23" ; "33"};
            {\ar^{\alpha} "32" ; "33"};
            {\ar "02" ; "12"};
            {\ar@{^{(}->} "12" ; "23"};
            {\ar@{=>}^{\Phi} (75,-28) ; (75,-32)};
        \endxy
    \]
\noindent whilst $\phi_\eta$ is defined to be $\Phi_{\eta}$, since the composition of $\alpha$ with the composite of the coequalizer and inclusion map from $P(1) \times X$ into $\underline{P}(X)$ is just $\tilde{\alpha_1}$. 
It is straightforward to verify the $P$-pseudoalgebra axioms from the $\underline{P}$-pseudoalgebra on components, and we leave that to the reader.

For a $1$-cell $(f,\overline{f}) \colon (X, \alpha) \rightarrow (Y, \beta)$, we send $f$ to itself whilst sending $\overline{f}$ to the obvious family of isomorphisms, as follows.
    \[
        \xy
            (-30,0)*+{P(n) \times X^n}="-10";
            (-30,-15)*+{P(n) \times Y^n}="-11";
            (0,0)*+{\coeq{P}{X}{\Lambda}{n}}="00";
            (30,0)*+{\underline{P}(X)}="10";
            (60,0)*+{X}="20";
            (0,-15)*+{\coeq{P}{Y}{\Lambda}{n}}="01";
            (30,-15)*+{\underline{P}(Y)}="11";
            (60,-15)*+{Y}="21";
            {\ar@{^{(}->} "00" ; "10"};
            {\ar^{\alpha} "10" ; "20"};
            {\ar_{1 \times f^n} "00" ; "01"};
            {\ar_{\underline{P}f} "10" ; "11"};
            {\ar^{f} "20" ; "21"};
            {\ar@{^{(}->} "01" ; "11"};
            {\ar_{\beta} "11" ; "21"};
            {\ar "-10" ; "00"};
            {\ar "-11" ; "01"};
            {\ar_{1 \times f^n} "-10" ; "-11"};
            {\ar@{=>}^{\overline{f}} (45,-5.5) ; (45,-9.5)};
        \endxy
    \]
\noindent It is easy to check that the above data satisfy the axioms for being a pseudomorphism of $P$-pseudoalgebras, following from the axioms for $(f,\overline{f})$ being a pseudomorphism of $\underline{P}$-pseudoalgebras. A $\underline{P}$-transformation $\gamma \colon (f, \bar{f}) \Rightarrow (g, \bar{g})$ immediately gives a $P$-transformation $\bar{\gamma}$ between the families of isomorphisms we previously defined, with the components of $\bar{\gamma}$ being precisely those of $\gamma$. It is then easily shown that $R$ is a $2$-functor.

For there to be an isomorphism of $2$-categories, we require an inverse to $R$, namely a $2$-functor $S \colon \mb{Ps}\mbox{-}P\mbox{-}\mb{Alg} \rightarrow \mb{Ps}\mbox{-}\underline{P}\mbox{-}\mb{Alg}$. Now assume that $(X, \alpha_n, \phi_{\underline{k}_i}, \phi_\eta)$ is a $P$-pseudoalgebra. We will give the same object $X$ a $\underline{P}$-pseudoalgebra structure. We can induce a functor $\alpha \colon \underline{P}(X) \rightarrow X$ by using the universal property of the coproduct.
    \[
        \xy
            (-30,0)*+{P(n) \times X^n}="-10";
            (0,0)*+{\coeq{P}{X}{\Lambda}{n}}="00";
            (30,0)*+{\underline{P}(X)}="10";
            (30,-15)*+{X}="11";
            {\ar "-10" ; "00"};
            {\ar^{\alpha_n} "00" ; "11"};
            {\ar@{^{(}->} "00" ; "10"};
            {\ar^{\exists ! \alpha} "10" ; "11"};
            {\ar_{\tilde{\alpha}_n} "-10" ; "11"};
        \endxy
    \]
\noindent This can be induced using either $\alpha_n$ or $\tilde{\alpha}_n$, each giving the same functor $\alpha$ by uniqueness. The components of the isomorphism $\Phi \colon \alpha \circ \underline{P}(\alpha) \Rightarrow \alpha \circ \mu_X$ can be given as follows. Let $\left|\underline{x}_i\right|$ denote the number of objects in the list $\underline{x}_i$. Then define the component of $\Phi$ at the object
    \[
        \big[p;\left[q_1;\underline{x}_1\right],\ldots,\left[q_n;\underline{x}_n\right]\big]
    \]
to be the component of $\phi_{\left|\underline{x}_1\right|, \ldots, |\underline{x}_n|}$ at the same object. 
Define the isomorphism $\Phi_{\eta}$ to be $\phi_\eta$.

Now given a $1$-cell $f$ with structure $2$-cells $\overline{f}_n$ we define a $1$-cell $(F,\overline{F})$ with underlying $1$-cell $f$ and structure $2$-cell $\overline{F}$ with components
    \[
        \overline{F}_{[p;x_1, \ldots, x_n]} := \left(\overline{f}_{n}\right)_{(p;x_1,\ldots,x_n)}.
    \]
The mapping for $2$-cells sends $\gamma$ to $\gamma$ as before. 
It is now easy to verify that $S$ is an inverse for $R$, completing the proof of the isomorphism.
\end{proof}

\begin{rem}\label{rem:lop-alg-E}
Every category $C$ determines an endomorphism operad $\mathcal{E}_C$ in $\mb{Cat}$ by defining
\[
\mathcal{E}_C(n) = [C^n, C],
\]
where the square brackets indicate the functor category.
While $\mathcal{E}_C$ is naturally a symmetric operad, it can be given the structure of a $\Lambda$-operad for any action operad $(\Lambda, \pi)$ using $\pi^*$ from \cref{thm:pbaop}.
The reader can verify that strict $P$-algebra structures are in bijection with strict maps of $\Lambda$-operads $P \to \pi^*\mathcal{E}_C$, and pseudo-$P$-algebra structures are in bijection with pseudomorphisms of $\Lambda$-operads $P \to \pi^*\mathcal{E}_C$.
It is possible to develop analogues of \cref{lem:alg=map} and \cref{cor:pi-star}, but we do not pursue this line of research here.
\end{rem}

We finish this section by studying a special case of algebras over a $\Lambda$-operad in $\mb{Cat}$ that we call $\Lambda$-monoidal categories.
These generalize the various kinds of monoidal categories (plain, symmetric, and braided) to any action operad $\Lambda$.
In order to define $\Lambda$-monoidal categories, we must first construct the operads for which they will be algebras.

\begin{Defi}[($E$ and $B$)]\label{Defi:e_b}
We define the constructions $E$ and $B$ as follows.
  \begin{enumerate}
    \item Let $X$ be a set. We define the \textit{translation category} $EX$ to have objects the elements of $X$ and morphisms consisting of a unique isomorphism between any two objects.
    \item Let $G$ be a group. The category $BG$ has a single object $*$, and hom-set $BG(*,*) = G$ with composition and identity given by multiplication and the unit element in the group, respectively.
  \end{enumerate}
\end{Defi}

The following lemma is straightforward to verify.

\begin{lem}\label{lem:symmoncor}
The functor $E \colon \mb{Sets} \rightarrow \mb{Cat}$ is right adjoint to the set of objects functor. Therefore $E$ preserves all limits, and in particular is a symmetric monoidal functor when both categories are equipped with their cartesian monoidal structures.
\end{lem}

\begin{cor}\label{cor:elambda_lambdaop}
Let $\Lambda$ be an action operad. Then $\EL = \{ E\Lambda(n) \}_{n \in \mathbb{N}}$ is a $\Lambda$-operad in $\mb{Cat}$.
\end{cor}
\begin{proof}
We have already defined the categories $E\Lambda(n)$, and the right $\Lambda(n)$-action on $E\Lambda(n)$ is given by multiplication in the group $\Lambda(n)$ on objects and then uniquely determined on morphisms. The object $\id \in E\Lambda(1)$ is $e_1 \in \Lambda(1)$. The operadic multiplication
\[
\mu \colon E\Lambda(n) \times E\Lambda(k_{1}) \times \cdots \times E\Lambda(k_{n}) \rightarrow E\Lambda(k_{1} + \cdots + k_{n})
\]
corresponds by adjointness to a function
\[
\mu' \colon \textrm{ob}\big( E\Lambda(n) \times E\Lambda(k_{1}) \times \cdots \times E\Lambda(k_{n}) \big) \to \Lambda(k_{1} + \cdots + k_{n}).
\]
Since the set of objects functor itself preserves products, and we have an equality $\textrm{ob}ES = S$, we define $\mu'$ to be the operadic multiplication for $\Lambda$. The axioms then all follow from the \cref{prop:gisgop}.
\end{proof}

\begin{Defi}[($\Lambda$-monoidal categories, functors, and transformations)]\label{Defi:lmc}
Let $\Lambda$ be an action operad.
\begin{itemize}
\item A \emph{$\Lambda$-monoidal category} is a strict algebra for the $\Lambda$-operad $\EL$. 
\item A \emph{$\Lambda$-monoidal functor} is a strict morphism for the $\Lambda$-operad $\EL$. 
\item A \emph{$\Lambda$-transformation} is an $\EL$-transformation.
\end{itemize}
\end{Defi}

\begin{rem}[($\EL$-algebras are $\underline{\EL}$-algebras)]\label{rem:elambda=el}
In each of the items above, we could have expressed the same concept using the $2$-monad $\underline{\EL}$ instead of the $\Lambda$-operad $\EL$ by \cref{thm:2monad=op}.
The same substitution can be made throughout without changing any of the results.
We have just chosen to state definitions and results in terms of operads rather than $2$-monads.
\end{rem}

\begin{rem}[(Strictness of $\Lambda$-monoidal categories)]\label{rem:strictness-lmc}
Note that our definition of a $\Lambda$-monoidal category involves a strict underlying monoidal structure.
We will briefly explore a version suitable for general monoidal categories in \cref{sec:coherence}, and prove a strictification result in \cref{thm:wlmc-to-lmc}.
\end{rem}

\begin{Defi}[(The 2-category of $\Lambda$-monoidal categories)]
The $2$-category $\lmc$ is the $2$-category $\EL\mbox{-}\mb{Alg}_{s}$ of strict algebras, strict morphisms, and algebra $2$-cells for $\EL$.
\end{Defi}

We end this section with a computation of the free $\Lambda$-monoidal category generated by a category $X$, the free algebra $\EL(X)$.
We will eventually show in \cref{thm:pres1,ex:S-moncats} that $\Lambda$-monoidal categories can be given in more familiar terms, as in Chapter 19 of \cite{yau_infinity_2021}.

\begin{rem}\label{rem:operadcoeq-as-quotient}
Recall that any right action $\mu \colon C \times G \to C$ can be viewed as a left action $\mu' \colon G \times C \to C$ via
\[
\mu'(g,c) = \mu(c, g^{-1}).
\]
Suppose that $P$ is a $\Lambda$-operad in $\mb{Cat}$ such that the action of $\Lambda(n)$ on $P(n) \times X^n$, given by
\[
\lambda \cdot (p, \underline{x_i}) = (p \cdot \lambda^{-1}, \underline{x_{\lambda^{-1}(i)}}),
\]
 is free for every category $X$; this hypothesis is easily verified in the case that the action of $\Lambda(n)$ on $P(n)$ is itself free, such as when $P = \EL$.
 Then the coequalizer $\coeqb{P(n)}{X^n}{\Lambda(n)}$ coincides with the one in the second part of \cref{lem:coeq-lem}, and can therefore be computed as $\left( P(n) \times X^n \right)/\Lambda(n)$.
 \end{rem}

\begin{prop}\label{prop:hom-set-lemma}
Let $\Lambda$ be an action operad and $X$ be a category. The free $\Lambda$-monoidal category generated by $X$, $\EL(X)$, is isomorphic to a category with
\begin{itemize}
\item object set $\coprod_{n \in \mathbb{N}} (\textrm{ob} X)^n$ and
\item morphism sets
\[
\EL(X)\left( (x_1, \ldots, x_m), (y_1, \ldots, y_n) \right) = \left\{ \begin{array}{cr} \emptyset, & m \neq n \\
\coprod_{g \in \Lambda(n)} \prod_{i=1}^{n} X\left(x_i, y_{g(i)}\right), & m=n. \end{array} \right.
\]
\end{itemize}
\end{prop}
\begin{proof}
The $2$-monad $\underline{\EL}$ has underlying 2-functor given by
  \[
    X \mapsto \EL(X) = \coprod_{n \geq 0} \coeq{\EL}{X}{\Lambda}{n}.
  \]
The coequalizer $\coeq{\EL}{X}{\Lambda}{n}$ can be computed as the quotient $\left( \EL(n) \times X^n \right)/\Lambda(n)$ from \cref{lem:coeq-lem} using the method from \cref{rem:operadcoeq-as-quotient}.
Therefore the set of objects of $\coeq{\EL}{X}{\Lambda}{n}$ is in bijection with the set of orbits of the $\Lambda(n)$-action on $\EL(n) \times X^n$. We have the equality of orbits
\[
[g; x_{1}, \ldots, x_n] = [e_n; x_{g^{-1}(1)}, \ldots, x_{g^{-1}(n)}]
\]
for any $g \in \Lambda(n)$. Moreover, since the action is free, there is an equality of orbits
\[
[e_n; x_1, \ldots, x_n] = [e_n; y_1, \ldots, y_n]
\]
if and only if $x_i = y_i$ for all $i$. Thus the function assigning to each orbit the unique representative with group element the identity $e_n$ is an isomorphism from the set of objects of $\coeq{\EL}{X}{\Lambda}{n}$ to the set of objects of $X^n$. The formula for the morphisms given in \cref{lem:coeq-lem} then reduces to the one above.
\end{proof}

\begin{cor}\label{cor:el1=bl}
Let $\Lambda$ be an action operad. The free $\Lambda$-monoidal category on one object, $\EL(1)$, has its underlying strict monoidal category given by $B\Lambda$ with the monoidal structure from \cref{prop:Gmonoidal}.
\end{cor}

\section{Coherence}\label{sec:coherence}

This section addresses questions of coherence for $2$-monads induced by $\Lambda$-operads in $\mb{Cat}$. Coherence theorems take various forms, and we will primarily be concerned with strictification-style coherence theorems. The prototypical example here is the coherence theorem for monoidal categories. In a monoidal category we require associator isomorphisms
    \[
        \left( A \otimes B \right) \otimes C \cong A \otimes \left( B \otimes C \right)
    \]
for all objects in the category. The coherence theorem tells us that, for any monoidal category $M$, there exists a strict monoidal category that is equivalent to $M$. In other words, we can treat the associators in $M$ as identities, and similarly for the unit isomorphisms.

By \cref{thm:2monad=op}, we can study the algebras for a $\Lambda$-operad $P$ directly, or do so by studying the algebras for the corresponding $2$-monad $\underline{P}$. We first note that the $2$-monads induced by $\Lambda$-operads are finitary, using standard arguments. Second, we show that the Lack's generalised version \cite[Theorem 4.10]{lack-cod} of Power's coherence theorem \cite[Theorem 3.4]{power-gen} applies to all such $2$-monads and allows us to show that each pseudo-$\underline{P}$-algebra is equivalent to a strict $\underline{P}$-algebra.

\begin{prop}
Let $P$ be a $\Lambda$-operad. Then $\underline{P}$ is finitary.
\end{prop}
\begin{proof}
The argument is identical to that for braided operads in Section 4.1 of \cite{lack-cod}.
\end{proof}

We now give an abstract coherence theorem for algebras over a $\Lambda$-operad $P$ in $\mb{Cat}$ following the method of John Power \cite{power-gen} and generalized by Lack \cite{lack-cod}. In order to do so, we recall the notion of an enhanced factorization system and Power's coherence result.

\begin{Defi}[(Enhanced factorization system)]\label{Defi:efs}
Let $K$ be a 2-category. An \emph{enhanced factorization system} on $K$ consists of two classes of 1-cells $\mathcal{L},\mathcal{R}$ satisfying the following properties.
\begin{enumerate}
\item Given a commutative square of 1-cells
     \[
        \xy
            (0,0)*+{A}="00";
            (15,0)*+{C}="10";
            (0,-15)*+{B}="01";
            (15,-15)*+{D}="11";
            {\ar^{f} "00" ; "10"};
            {\ar^{r} "10" ; "11"};
            {\ar_{l} "00" ; "01"};
            {\ar_{g} "01" ; "11"};
        \endxy
     \]
where $l \in \m{L}$ and $r \in {R}$, there exists a unique 1-cell $m \colon B \rightarrow C$ such that $rm = g$ and $ml = f$.
\item Given two commuting squares  of 1-cells as above, $rf_1 = g_1l$ and $rf_2 = g_2l$ where $l \in \m{L}$ and $r \in {R}$ , along with $2$-cells $\delta \colon f_1 \Rightarrow f_2$ and $\gamma \colon g_1 \Rightarrow g_2$ for which $\gamma \ast l = r \ast \delta$, there exists a unique $2$-cell $\mu \colon m_1 \Rightarrow m_2$, where $m_1$ and $m_2$ are induced by the $1$-cell lifting property, satisfying $\mu \ast l = \delta$ and $r \ast \mu = \gamma$.
\item Given maps $l \in \m{L}$, $r \in \m{R}$ and an invertible $2$-cell $\alpha \colon rf \Rightarrow gl$
    \[
        \xy
            (0,0)*+{A}="00";
            (15,0)*+{C}="10";
            (0,-15)*+{B}="01";
            (15,-15)*+{D}="11";
            {\ar^{f} "00" ; "10"};
            {\ar^{r} "10" ; "11"};
            {\ar_{l} "00" ; "01"};
            {\ar_{g} "01" ; "11"};
            {\ar@{=>}^{\alpha} (9.375,-5.625) ; (5.625,-9.375)};
            (22.5,-7.5)*+{=};
            (30,0)*+{A}="20";
            (45,0)*+{C}="30";
            (30,-15)*+{B}="21";
            (45,-15)*+{D}="31";
            {\ar^{f} "20" ; "30"};
            {\ar^{r} "30" ; "31"};
            {\ar_{l} "20" ; "21"};
            {\ar_{g} "21" ; "31"};
            {\ar^{m} "21" ; "30"};
            {\ar@{=>}^{\beta} (41,-8) ; (38,-12)};
        \endxy
    \]
there exists a unique pair $(m,\beta)$ where $m \colon B \rightarrow C$ is a $1$-cell and $\beta \colon rm \Rightarrow g$ is an invertible $2$-cell such that $ml = f$ and $\beta \ast l = \alpha$.
\end{enumerate}
\end{Defi}

\begin{thm}[Theorem~3.4 \cite{power-gen},~Theorem~4.6 \cite{lack-cod}]\label{thm:power}
Let $K$ be a 2-category, and $T$ be a $2$-monad on $K$. If $K$ has an enhanced factorization system $\mathcal{L}, \mathcal{R}$ such that 
\begin{enumerate}
\item for every $r \colon C \to D$ in $\mathcal{R}$, 1-cell $s \colon D \to C$, and isomorphism $\alpha \colon rs \cong 1_{D}$, there exists an isomorphism $\beta \colon sr \cong 1_C$; and
\item for every $l \in \mathcal{L}$, the 1-cell $Tl$ is also in $\mathcal{L}$;
\end{enumerate}
then the inclusion 2-functor  \[
        U \colon T\mbox{-}\mb{Alg}_s \rightarrow \mb{Ps}\mbox{-}T\mbox{-}\mb{Alg}
    \]
has a left 2-adjoint, and the components of the unit of the adjunction are equivalences in $\mb{Ps}\mbox{-}T\mbox{-}\mb{Alg}$. In particular, every pseudo-$T$-algebra is equivalent to a strict one.
\end{thm}

\begin{lem}[Lemma~3.3 \cite{power-gen}]\label{lem:efs-cat}
The 2-category $\mb{Cat}$ has an enhanced factorization system in which the class $\mathcal{L}$ consists of the functors that are bijective on objects and the class $\mathcal{R}$ consists of the functors that are full and faithful.
\end{lem}

\begin{prop}\label{prop:P-boo}
For any $\Lambda$-operad $P$, the $2$-monad $\underline{P}$ preserves bijective-on-objects functors.
\end{prop}
\begin{proof}
This follows immediately from the description of $\EL$ in \cref{prop:hom-set-lemma}.
\end{proof}

\begin{cor}\label{cor:coherence-for-ulP}
Every pseudo-$\underline{P}$-algebra is equivalent to a strict $\underline{P}$-algebra.
\end{cor}
\begin{proof}
We use the enhanced factorization system from \cref{lem:efs-cat}, and check the hypotheses of \cref{thm:power}. If a functor $r$ is full and faithful, and there exists a functor $s$ together with an isomorphism $\alpha \colon rs \cong 1$, then the components of $\alpha$ exhibit $r$ as essentially surjective. Thus $r$ is an equivalence of categories, and so there exists an isomorphism $\beta \colon sr \cong 1$ as required. For any bijective on objects functor $l$, \cref{prop:P-boo} shows that $\underline{P}l$ is also bijective on objects. Thus both of the hypotheses of \cref{thm:power} are satisfied, completing the proof.
\end{proof}

\begin{rem}[(Pseudo-$\EL$-algebras are weak, unbiased)]\label{rem:ps-P-alg-unpack}
The translation between pseudoalgebras and traditional notions of non-strict monoidal categories is not that of a direct correspondence. The pseudo-$\EL$-algebras are a \emph{weak} and \emph{unbiased} form of $\Lambda$-monoidal categories.
\begin{itemize}
\item Here \emph{weak} means equational axioms at the level of objects are replaced by coherent isomorphisms. On its own, the reader might expect such a claim to mean that pseudo-$\EL$-algebras have an underlying monoidal, rather than strict monoidal, structure. This is not the case.
\item These pseudoalgebras are also \emph{unbiased}, meaning they have prescribed $n$-ary tensor product operations $\otimes_n \colon X^n \to X$ for every $n \in \mathbb{N}$, and these are related by the isomorphisms in the previous point using operadic composition. Thus instead of an associativity isomorphism $(x \otimes y) \otimes z \cong x \otimes (y \otimes z)$, there is an isomorphism
\[
\otimes_2\left( \otimes_2(x,y), z\right) \cong \otimes_3(x,y,z).
\]
\end{itemize}
We refer the reader to Section 3.1 of \cite{leinster} for a further discussion of the relationship between strict structures and unbiased, weak ones.
\end{rem}

We end this section by exploring a variant of $\Lambda$-monoidal categories in which the underlying monoidal structure is weak, but the tensor product is not unbiased as above.

\begin{nota}[(Standard association)]\label{nota:standard-assoc}
Let $(M, \otimes, I, a, l, r)$ be a monoidal category. The \emph{standard association} of a tuple $x_1, \ldots, x_n$ of objects is defined inductively as follows.
\begin{enumerate}
\item The standard association of the empty tuple, written $\underline{\emptyset}$, is the unit object $I$.
\item The standard association of a single object $x$, written $\underline{x}$, is $x$ itself.
\item Assume that the standard association of $n$ objects $x_1, \ldots, x_n$ has been given as $\underline{x_1 \cdots x_n}$. The standard association of $n+1$ objects $x_1, \ldots, x_{n+1}$ is defined by the formula
\[
\underline{x_1 \cdots x_{n+1}} = x_1 \otimes \underline{x_2 \cdots x_{n+1}}.
\]
\end{enumerate}
\end{nota}

\begin{Defi}[(Weak $\Lambda$-monoidal categories)]\label{Defi:wk-lmc}
A \emph{weak $\Lambda$-monoidal category} consists of
\begin{itemize}
\item a monoidal category $(M, \otimes, I, a, l, r)$ and
\item a natural isomorphism
\[
[g] \colon \underline{x_1 \dots x_n} \cong \underline{x_{g^{-1}(1)} \cdots x_{g^{-1}(n)} }
\]
for each $g \in \Lambda(n)$
\end{itemize}
satisfying the following three axioms.
\begin{enumerate}
\item Let $g, h \in \Lambda(n)$. The composite $[h] \circ [g]$ shown below
\[
\underline{x_1 \dots x_n} \stackrel{[g]}{\to} \underline{x_{g^{-1}(1)} \cdots x_{g^{-1}(n)} } \stackrel{[h]}{\to} \underline{x_{g^{-1}(h^{-1}(1))} \cdots x_{g^{-1}(h^{-1}(n))} } 
\]
equals $[hg]$, where $hg \in \Lambda(n)$ is given by multiplication using the group structure.
\item Let $h_i \in \Lambda(k_i)$ for $i=1, \ldots, n$, and let $x_{ij}$ be objects of $M$ for $i = 1, \ldots, n$ and double indices $ij$ such that $1 \leq j \leq k_i$. Then the isomorphism
\[
[\beta(h_1, \ldots, h_n)] \colon \underline{x_{ij}} \to \underline{x_{ih_i^{-1}(j)}}
\]
is equal to the composite
\[
\underline{x_{ij}} \cong \underline{\underline{x_{1j}} \cdots \underline{x_{nj}}} 
\stackrel{\underline{[h_i]}}{\to} 
\underline{\underline{x_{1h_1^{-1}(j)}} \cdots \underline{x_{nh_n^{-1}(j)}}} \cong
\underline{x_{ih_i^{-1}(j)}},
\]
where the two unlabeled isomorphisms are the unique reassociations given by coherence for monoidal categories.
\item Let $g \in \Lambda(n)$, and let $x_{ij}$ be objects of $M$ for $i = 1, \ldots, n$ and double indices $ij$ such that $1 \leq j \leq k_i$. Then the isomorphism
\[
[\delta_{n; k_1, \ldots, k_n}(g)] \colon \underline{x_{ij}} \to \underline{x_{g^{-1}(1)1} x_{g^{-1}(1)2}\cdots x_{g^{-1}(1)k_{g^{-1}(1)}} \cdots x_{g^{-1}(n)k_{g^{-1}(n)}}}
\]
is equal to the composite
\[
\underline{x_{ij}} \cong \underline{y_i} \stackrel{[g]}{\to} \underline{y_{g^{-1}(i)}} \cong \underline{x_{g^{-1}(1)1} x_{g^{-1}(1)2}\cdots x_{g^{-1}(1)k_{g^{-1}(1)}} \cdots x_{g^{-1}(n)k_{g^{-1}(n)}}},
\]
where $y_i = \underline{x_{i1} \cdots x_{ik_i}}$ and the two unlabeled isomorphism are the unique reassociations given by coherence for monoidal categories.
\end{enumerate}
\end{Defi}

\begin{nota}
By coherence for monoidal functors in the form \cite{gj-pseudo}, every monoidal functor $(F, F_2, F_0)$ induces a unique isomorphism
\[
F_n \colon \underline{Fx_1 \cdots Fx_n} \cong F(\underline{x_1 \cdots x_n}).
\]
In the case that $n=0, 2$, these isomorphisms agree with the isomorphisms $F_0, F_2$ in the data defining $F$ as a monoidal functor.
\end{nota}

\begin{Defi}[(Weak $\Lambda$-monoidal functors)]\label{Defi:wk-lmf}
Let $M, N$ be weak $\Lambda$-monoidal categories. A \emph{weak $\Lambda$-monoidal functor} $F \colon M \to N$ consists of a monoidal functor $(F, F_0, F_2) \colon M \to N$ of the underlying monoidal categories such that for all $g \in \Lambda(n)$ and all tuples of objects $x_1, \ldots, x_n \in M$, the following diagram commutes.
     \[
        \xy
            (0,0)*+{\underline{Fx_1 \cdots Fx_n}}="00";
            (50,0)*+{\underline{Fx_{g^{-1}(1)} \cdots Fx_{g^{-1}(n)}}}="10";
            (0,-15)*+{F\underline{x_1 \cdots x_n}}="01";
            (50,-15)*+{F\underline{x_{g^{-1}(1)} \cdots x_{g^{-1}(n)}}}="11";
            {\ar^{[g]} "00" ; "10"};
            {\ar^{F_n} "10" ; "11"};
            {\ar_{F_n} "00" ; "01"};
            {\ar_{F[g]} "01" ; "11"};
        \endxy
     \]
\end{Defi}

We leave the proof of the following proposition to the reader, as the details are simple to fill in and mimic similar proofs for braided or symmetric monoidal categories.

\begin{prop}\label{prop:2cat-of-weaklambda}
There is a 2-category with 
\begin{itemize}
\item objects the weak $\Lambda$-monoidal category, 
\item 1-cells the weak $\Lambda$-monoidal functors,
\item 2-cells the monoidal transformations,
\item 1-cell identities $1_M \colon M \to M$ given by the identity functor equipped with $F_0 = \id_I$ and $(F_2)_{x,y} = \id_{x \otimes y}$,
and
\item composition of 1-cells given by composition of the underlying monoidal functors.
\end{itemize}
\end{prop}

\begin{nota}[(2-category of weak $\Lambda$-monoidal categories)]\label{nota:2cat-of-weaklambda}
The 2-category in \cref{prop:2cat-of-weaklambda} is called the \emph{2-category of weak $\Lambda$-monoidal categories}, and is denoted $\wlmc$.
\end{nota}

\begin{rem}\label{rem:eq-in-wlmc}
The internal equivalences in $\wlmc$ are, by definition, those weak $\Lambda$-monoidal functors $F \colon M \to N$ for which there exists a weak $\Lambda$-monoidal functor $G \colon N \to M$ and invertible monoidal transformations $GF \cong 1_M$, $FG \cong 1_N$. By a generalization of the standard argument for plain monoidal functors a weak $\Lambda$-monoidal functor $F$ is an internal equivalence if and only if the underlying functor of $F$ is an equivalence of categories, see \cite[Proposition 3.4]{lcmv-psmonad} or \cite[Corollary 7.4.2]{jy-2dim}, for example.
\end{rem}

\begin{thm}\label{thm:wlmc-to-lmc}
Let $\Lambda$ be an action operad.
\begin{enumerate}
\item There is an inclusion 2-functor
\[
i \colon \lmc \to \wlmc,
\]
the image of which consists of those weak $\Lambda$-monoidal categories for which the underlying monoidal category is strict.
\item Every weak $\Lambda$-monoidal category is equivalent, in $\wlmc$, to one in the image of $i$.
\end{enumerate}
\end{thm}
\begin{proof}
Let $X$ be an $\EL$-algebra given by functors $\mu_n \colon \coeq{\EL}{X}{\Lambda}{n} \to X$, or equivalently a single functor $\mu \colon \underline{\EL}(X) \to X$. 
We will equip $X$ with a strict monoidal structure, and then extend that to a weak $\Lambda$-monoidal category structure.
Let $T$ be the trivial action operad, and let $T \to \Lambda$ be the unique map of action operads (see \cref{example:aop-triv}). Then $\underline{ET}$ is easily seen to the free monoid 2-monad on $\mb{Cat}$, and the description of $\EL$ in \cref{prop:hom-set-lemma} shows that $\underline{ET}(X)$ embeds as the subcategory of $\underline{\EL}(X)$ consisting of all the objects and but only the morphisms in each summand corresponding to $e_n \in \Lambda(n)$. Thus $X$ obtains an $ET$-algebra structure via the composite
\[
\underline{ET}(X) \hookrightarrow \underline{\EL}(X) \stackrel{\mu}{\to} X.
\]
This $ET$-algebra structure is the desired strict monoidal structure.

Let $g \in \Lambda(n)$ and $x_1, \ldots, x_n$ be objects of $X$. There is a unique isomorphism $\tilde{g} \colon e_n \cong g$ in $\EL(n)$, and applying $\mu_n$ to 
\[
[e_n; x_1, \ldots, x_n] \stackrel{[\tilde{g}; \id, \ldots, \id]}{\longrightarrow} [g; x_1, \ldots, x_n] = [e_n; x_{g^{-1}(1)}, \ldots, x_{g^{-1}(n)}]
\]
produces an isomorphism 
\[
[\tilde{g}; \underline{\id}] \colon x_1 \otimes \cdots \otimes x_n \cong x_{g^{-1}(1)} \otimes \ldots \otimes x_{g^{-1}(n)}.
\]
Define the isomorphism $[g]$ in \cref{Defi:wk-lmc} to be $[\tilde{g}; \underline{\id}]$. The three axioms in \cref{Defi:wk-lmc} follow immediately from the action operad axioms, completing the construction of the 2-functor $i$ on objects. Similar arguments apply to strict $\EL$-morphisms and $\EL$-transformations, and we leave them to the reader. The arguments above together with \cref{prop:hom-set-lemma} show that the resulting 2-functor $i$ has image those weak $\Lambda$-monoidal category with strict underlying monoidal structure, finishing the proof of the first claim.

Now let $M$ be a weak $\Lambda$-monoidal category, and let $M_u$ denote its underlying monoidal category. By coherence for monoidal categories \cite{js,coh3d}, there is a strict monoidal category $\textrm{st} M_u$ and a monoidal equivalence $e \colon \textrm{st} M_u \to M_u$ given as follows.
\begin{itemize}
\item The objects of $\textrm{st} M_u$ consist of a natural number $n$ and then an ordered list $x_1, \ldots, x_n$ of objects of $M_u$. There is a unique such object when $n=0$.
\item The functor $e$ maps $(n; x_1, \ldots, x_n)$ to the standard association (\cref{nota:standard-assoc}) $\underline{x_1 \cdots x_n}$.
\item The set of morphisms from $(m; x_1, \ldots, x_m)$ to $(n; y_1, \ldots, y_n)$ in $\textrm{st} M_u$ is defined to be the set of morphism from $\underline{x_1 \cdots x_m}$ to $\underline{y_1 \cdots y_n}$ in $M_u$.
\item The monoidal structure is given on objects by the sum of natural numbers and the concatenation of lists, and on morphisms is given by
\[
\underline{x_1 \cdots x_m y_1 \cdots y_n} \cong \underline{x_1 \cdots x_m} \otimes \underline{y_1 \cdots y_n} \stackrel{f \otimes g}{\to} \underline{u_1 \cdots u_k} \otimes \underline{v_1 \cdots v_j} \cong \underline{u_1 \cdots u_kv_1 \cdots v_j}.
\]
\end{itemize}
We will equip $\textrm{st} M_u$ with the structure of a $\Lambda$-monoidal category in such a way that $e$ induces an equivalence between it and $M$.

Let $y_i = (n_i; x_{i1}, \ldots, x_{in_i})$ be objects of $\textrm{st} M_u$ for $i=1, \ldots, m$. For an element $g \in \Lambda(k)$, define $[g]$ to be the isomorphism $[\delta_{m; k_1, \ldots, k_m}(g)]$ from \cref{Defi:wk-lmc}.
We must now verify the three axioms from \cref{Defi:wk-lmc} using this definition of $[g]$ in $\textrm{st} M_u$.
Each axiom follows from the characterization of action operads in \cref{thm:charAOp}.
The first axiom is a consequence of Axiom \eqref{eq6}, the second axiom is a consequence of Axiom \eqref{eq9}, and the third axiom is a consequence of Axiom \eqref{eq7}. We write $\textrm{st}^{\Lambda} M$ for this $\Lambda$-monoidal structure on $\textrm{st} M_u$. Since $\textrm{st} M_u$ is strict monoidal by construction and $e \colon \textrm{st} M_u \to M$ is known to be a monoidal equivalence, we can prove that $e$ is actually a $\Lambda$-monoidal equivalence $\textrm{st}^{\Lambda} M \to M$ by showing that $e$ is a $\Lambda$-monoidal functor. The single axiom in \cref{Defi:wk-lmf} requires the commutativity of the diagram below,
 \[
        \xy
            (0,0)*+{\underline{ey_1 \cdots ey_m}}="00";
            (50,0)*+{\underline{ey_{g^{-1}(1)} \cdots ey_{g^{-1}(m)}}}="10";
            (0,-15)*+{e\big(\underline{y_1 \cdots y_m}\big)}="01";
            (50,-15)*+{e\big(\underline{y_{g^{-1}(1)} \cdots y_{g^{-1}(m)}}\big)}="11";
            {\ar^{[g]} "00" ; "10"};
            {\ar^{e_m} "10" ; "11"};
            {\ar_{e_m} "00" ; "01"};
            {\ar_{e[g]} "01" ; "11"};
        \endxy
     \]
and that follows immediately from the third axiom in \cref{Defi:wk-lmc}. By \cref{rem:eq-in-wlmc}, this observation completes the proof that $e$ is an equivalence in the 2-category of weak $\Lambda$-monoidal categories.
\end{proof}

\begin{rem}[(Pseudo-$\EL$-algebras versus weak $\Lambda$-monoidal categories)]\label{rem:psELalg-vs-wlmc}
While \cref{cor:coherence-for-ulP} is satisfying in its brevity, one would expect it to be less useful in practice than the second part of \cref{thm:wlmc-to-lmc}, as that is the case for plain monoidal categories.
\end{rem}

\section{Group Actions and Cartesian $2$-monads}\label{sec:cart}

This is the first of two sections to investigate the interaction between operads and pullbacks. The monads arising from a non-symmetric operad are always cartesian, as described in \cite[Appendix C]{leinster}. The monads that arise from symmetric operads, however, are not always cartesian and an example of where this fails is the symmetric operad for which the algebras are commutative monoids. 
Moving to the context of $2$-monads on a 2-category, we can consider those that are 2-cartesian (\cref{Defi:2cart-monad}).
The goal of this section is to characterize those $\Lambda$-operads $P$ for which the induced 2-monad $\underline{P}$ is 2-cartesian. 
We prove in \cref{cor:cart_cor,thm:cart_thm} that $P$ being 2-cartesian is equivalent to either free group actions, in the symmetric case, or a slight weakening of free group actions, in the general $\Lambda$-operad case.

\begin{Defi}[(2-cartesian $2$-monad)]\label{Defi:2cart-monad}
A $2$-monad $T \colon \mathcal{K} \rightarrow \mathcal{K}$ is said to be \textit{$2$-cartesian} if
    \begin{itemize}
        \item the $2$-category $\mathcal{K}$ has $2$-pullbacks,
        \item the functor $T$ preserves $2$-pullbacks (up to isomorphism), and
        \item the naturality squares for the unit and multiplication of the $2$-monad are $2$-pullbacks.
    \end{itemize}
\end{Defi}

\begin{rem}
As discussed in \cref{rem:lim-v-2lim}, the  $2$-pullback of a diagram is actually the same as the ordinary pullback in $\mb{Cat}$. We will therefore drop the prefix 2-, and interpret cartesian to mean 2-cartesian.
\end{rem}

We begin our study of the cartesian property in the context of symmetric operads.
We will individually examine when the unit for $\underline{P}$ is cartesian, when the multiplication for $\underline{P}$ is cartesian, and when $\underline{P}$ is a cartesian 2-functor.

\begin{prop}\label{prop:cart_unit}
Let $P$ be a symmetric operad. Then the unit $\eta \colon \id \Rightarrow \underline{P}$ for the associated monad is a cartesian transformation.
\end{prop}
\begin{proof}
In order to show that $\eta$ is cartesian, we must prove that for a functor $f \colon X \rightarrow Y$, the pullback of the diagram below is the category $X$.
    \[
        \xy
            (40,0)*+{Y}="10";
            (0,-15)*+{\coprod \coeqsig{P}{X}{n}}="01";
            (40,-15)*+{\coprod \coeqsig{P}{Y}{n}}="11";
            {\ar^{\eta_Y} "10" ; "11"};
            {\ar_{\underline{P}(f)} "01" ; "11"};
        \endxy
    \]
The pullback of this diagram is isomorphic to the coproduct of the pullbacks of diagrams of the following form, where on the left we note that no coequalizer is needed because $\Sigma_1$ is the trivial group.
\[
        \xy
            (30,0)*+{Y}="10";
            (0,-15)*+{P(1) \times X}="01";
            (30,-15)*+{P(1) \times Y}="11";
            {\ar^{} "10" ; "11"};
            {\ar_{1 \times f} "01" ; "11"};
            (90,0)*+{\emptyset}="60";
            (60,-15)*+{\coeqsig{P}{X}{n}}="51";
            (90,-15)*+{\coeqsig{P}{Y}{n}}="61";
            {\ar^{} "60" ; "61"};
            {\ar_{1 \times f^n} "51" ; "61"};
            (75,-21)*{n \neq 1}
        \endxy
    \]
It is easy then to see that $X$ is the pullback of the $n=1$ cospan, and that the empty category is the pullback of each of the other cospans, making $X$ the pullback of the original diagram and verifying that $\eta$ is cartesian.
\end{proof}

\begin{prop}\label{prop:mu-2cart}
Let $P$ be a symmetric operad. If the $\Sigma_n$-actions are all free, then the multiplication $\mu \colon  \underline{P}^{2} \Rightarrow \underline{P}$ of the associated monad is a cartesian transformation.
\end{prop}
\begin{proof}
Note that if all of the diagrams
    \[
        \xy
            (0,0)*+{\underline{P}^2(X)}="00";
            (20,0)*+{\underline{P}^2(1)}="10";
            (0,-15)*+{\underline{P}(X)}="01";
            (20,-15)*+{\underline{P}(1)}="11";
            {\ar^{\underline{P}^2(!)} "00" ; "10"};
            {\ar^{\mu_1} "10" ; "11"};
            {\ar_{\mu_X} "00" ; "01"};
            {\ar_{\underline{P}(!)} "01" ; "11"};
        \endxy
    \]
are pullbacks then the outside of the diagram
    \[
        \xy
            (0,0)*+{\underline{P}^2(X)}="00";
            (20,0)*+{\underline{P}^2(Y)}="10";
            (40,0)*+{\underline{P}^2(1)}="20";
            (0,-15)*+{\underline{P}(X)}="01";
            (20,-15)*+{\underline{P}(Y)}="11";
            (40,-15)*+{\underline{P}(1)}="21";
            {\ar^{\underline{P}^2(f)} "00" ; "10"};
            {\ar^{\underline{P}^2(!)} "10" ; "20"};
            {\ar^{\mu_{1}} "20" ; "21"};
            {\ar_{\mu_X} "00" ; "01"};
            {\ar_{\underline{P}(f)} "01" ; "11"};
            {\ar_{\underline{P}(!)} "11" ; "21"};
            {\ar_{\mu_Y} "10" ; "11"};
        \endxy
    \]
is also a pullback and so each of the naturality squares for $\mu$ must therefore be a pullback. 

Now we can split up the square above, much like we did for $\eta$, and prove that each of the squares below is a pullback.
    \[
        \xy
            (0,0)*+{\coprod_{m} P(m) \otimes_{\Sigma_m} \prod_{k_1+\cdots+k_m=n} \left(\coeqsig{P}{X}{{k_i}}\right)}="00";
            (60,0)*+{\coprod P(m) \otimes_{\Sigma_m} \prod_i \left(P(k_i) / \Sigma_{k_i}\right)}="10";
            (0,-20)*+{\coeqsig{P}{X}{n}}="01";
            (60,-20)*+{P(n) / \Sigma_{n}}="11";
            {\ar "00" ; "10"};
            {\ar "00" ; "01"};
            {\ar "01" ; "11"};
            {\ar "10" ; "11"};
        \endxy
    \]
The map along the bottom is the obvious one, sending $[p; x_1, \ldots, x_n]$ simply to the equivalence class $[p]$. The map along the right hand side is induced by operadic composition, and sends $[q;[p_1],\ldots,[p_m]]$ to $[\mu^P(q;p_1,\ldots,p_n)]$. The pullback of these maps would be the category consisting of pairs
    \[
        \left([p;x_1,\ldots,x_{n}],[q;[p_1],\ldots,[p_m]]\right),
    \]
where $q \in P(m)$, $p_i \in P(k_i)$, $k_1 + \cdots + k_m = n$, $p \in P(n)$, and for which $[p] = [\mu^P(q;p_1,\ldots,p_n)]$; we will denote the pullback by $U$. The upper left category in the diagram, which we denote by $Q$, has objects
    \[
        \left[q;\left[p_1;\underline{x}_1\right],\ldots,\left[p_m;\underline{x}_m\right]\right].
    \]

The uniquely induced functor $F \colon Q \to U$ is defined on objects by the formula
    \[
F\Big(\big[q;\left[p_1;\underline{x}_1\right],\ldots,\left[p_m;\underline{x}_m\right]\big] \Big) =          \left(\left[\mu^P(q;p_1,\ldots,p_m);\underline{x}\right], \big[q;[p_1],\ldots,[p_m]\big]\right),
    \]
where the list $\underline{x}$ is the concatenation $\underline{x}_1, \ldots, \underline{x}_m$.
We define an inverse $G \colon U \to Q$ as follows.
Let 
    \[
        \left([p;x_1,\ldots,x_{n}],[q;[p_1],\ldots,[p_m]]\right)
    \]
be an object of $U$, with $p_i \in P(k_i)$ as above. 
Since the action of $\Sigma_n$ on $P(n)$ is free, there is a unique $g \in \Sigma_n$ such that $p  = \mu^P(q;p_1,\ldots,p_m) \cdot g$. Then
\begin{align*}
[p; x_1, \ldots, x_n] & = \big[ \mu^P(q;p_1,\ldots,p_m) \cdot g; x_1, \ldots, x_n \big] \\
& = \big[ \mu^P(q;p_1,\ldots,p_m); x_{g^{-1}(1)}, \ldots, x_{g^{-1}(n)} \big],
\end{align*}
so by reindexing the $x_i$'s if necessary we can assume that $p =  \mu^P(q;p_1,\ldots,p_m)$.
Then define
\[
G\Big( [p;x_1,\ldots,x_{n}],[q;[p_1],\ldots,[p_m]] \Big) = \big[ q; [p_1; \underline{y}_1], \ldots, [p_m; \underline{y}_m] \big],
\]
where 
\[
\underline{y_j} = x_{k_1 + \cdots + k_{j-1} + 1}, \ldots, x_{k_1 + \cdots + k_{j-1} + k_j}.
\]
The reader can check that this is well-defined, and an inverse to $F$ on objects. A similar formula holds for morphisms. The functor $G$ is an inverse to $F$, so the desired square is a pullback, completing the proof.
\end{proof}

\begin{rem}[(Pullback cancellation)]\label{rem:pb-cancel}
We call the technique in the first paragraph of the previous proof \emph{pullback cancellation}.
\end{rem}

\begin{prop}\label{prop:P-2cart}
Let $P$ be a symmetric operad. Then the $2$-monad $\underline{P}$ preserves pullbacks if and only if $\Sigma_{n}$ acts freely on $P(n)$ for all $n$.
\end{prop}
\begin{proof}
Consider the following pullback of discrete categories.
    \[
        \xy
            (0,0)*+{\lbrace (x,y), (x,y'), (x',y), (x',y') \rbrace}="00";
            (40,0)*+{\lbrace y,y' \rbrace}="10";
            (0,-15)*+{\lbrace x, x' \rbrace}="01";
            (40,-15)*+{\lbrace z \rbrace}="11";
            {\ar "00" ; "10"};
            {\ar "10" ; "11"};
            {\ar "00" ; "01"};
            {\ar "01" ; "11"};
        \endxy
    \]
Letting $\mathbf{4}$ denote the pullback and similarly writing $\mathbf{2}_X = \{ x, x' \}$ and $\mathbf{2}_Y = \{y, y'\}$, the following diagram results as the image of this pullback square under $\underline{P}$.
    \[
        \xy
            (0,0)*+{\coprod \coeqsig{P}{\mathbf{4}}{n}}="00";
            (40,0)*+{\coprod \coeqsig{P}{\mathbf{2}_Y}{n}}="10";
            (0,-15)*+{\coprod \coeqsig{P}{\mathbf{2}_X}{n}}="01";
            (40,-15)*+{\coprod P(n)/\Sigma_n}="11";
            {\ar "00" ; "10"};
            {\ar "10" ; "11"};
            {\ar "00" ; "01"};
            {\ar "01" ; "11"}:
        \endxy
    \]
The projection map $\pi_Y \colon \underline{P}(\mb{4}) \rightarrow \underline{P}(\mb{2}_Y)$ is defined by
    \[
        \pi_Y\big(\ [p;(x_1,y_1), \ldots, (x_n,y_n)]\ \big) = [p;y_1,\ldots,y_n],
    \]
and likewise for the projection $\pi_X$ to $\underline{P}(\mb{2}_X)$.

Now assume that, for some $n$, the action of $\Sigma_n$ on $P(n)$ is not free. Then there exist $p \in P(n)$ and a non-identity $g \in \Sigma_n$ such that $p \cdot g = p$. We will show that the existence of $g$ proves that $\underline{P}$ is not cartesian.
Since $g \neq e$, so there exists an $i \in \{1, \ldots, n\}$ such that $g(i) \neq i$; without loss of generality, we may take $i=1$ and assume $g(1)=2$. Consider the two distinct elements
    \[
      a_1=  \left[p;(x',y),(x,y'),(x,y),\ldots,(x,y)\right]
    \]
and
    \[
     a_2=   \left[p;(x,y),(x',y'),(x,y),\ldots,(x,y)\right]
    \]
in $\underline{P}(\mb{4})$, where all the elements of these lists are given by $(x,y)$ unless otherwise indicated. Both of these elements are mapped to the same elements in $\underline{P}(\mb{2}_X)$:
    \begin{align*}
           \pi_X(a_1) & = \left[p; x', x, \ldots, x\right] \\
           &= \left[p \cdot g; x', x, \ldots, x\right]\\
          &= \left[p;g\cdot (x', x, \ldots, x)\right]\\
          &= \left[p;x,x',x,\ldots,x\right] \\
          & = \pi_X(a_2).
    \end{align*}
Similarly, 
    \[
        \pi_Y(a_1) = \left[p;y,y',y, \ldots, y\right] = \pi_Y(a_2).
    \]
The pullback of this diagram, however, has a unique element which is projected to the ones we have considered, so $\underline{P}(\mb{4})$ is not a pullback of the square displayed above. This completes the proof that $\underline{P}$ does not preserve pullbacks if for some $n$ the action of $\Sigma_n$ on $P(n)$ is not free.

Now assume that each $\Sigma_n$ acts freely on $P(n)$. Given a pullback
    \[
        \xy
            (0,0)*+{A}="00";
            (15,0)*+{B}="10";
            (0,-15)*+{C}="01";
            (15,-15)*+{D}="11";
            {\ar^{F} "00" ; "10"};
            {\ar^{S} "10" ; "11"};
            {\ar_{R} "00" ; "01"};
            {\ar_{H} "01" ; "11"};
        \endxy
    \]
we must show that the image of the diagram under $\underline{P}$ is also a pullback. Now this will be true if and only if each individual diagram
        \[
            \xy
                (0,0)*+{\coeqsig{P}{A}{n}}="00";
                (30,0)*+{\coeqsig{P}{B}{n}}="10";
                (0,-15)*+{\coeqsig{P}{C}{n}}="01";
                (30,-15)*+{\coeqsig{P}{D}{n}}="11";
                {\ar^{\coeqb{1}{F^n}{\Sigma_{n}}} "00" ; "10"};
                {\ar^{\coeqb{1}{S^n}{\Sigma_{n}}} "10" ; "11"};
                {\ar_{\coeqb{1}{R^n}{\Sigma_{n}}} "00" ; "01"};
                {\ar_{\coeqb{1}{H^n}{\Sigma_{n}}} "01" ; "11"}:
            \endxy
    \]
is also a pullback. 

Suppose that 
        \[
            \xy
                (0,0)*+{X}="00";
                (30,0)*+{\coeqsig{P}{B}{n}}="10";
                (0,-15)*+{\coeqsig{P}{C}{n}}="01";
                (30,-15)*+{\coeqsig{P}{D}{n}}="11";
                {\ar^{K} "00" ; "10"};
                {\ar^{\coeqb{1}{S^n}{\Sigma_{n}}} "10" ; "11"};
                {\ar_{L} "00" ; "01"};
                {\ar_{\coeqb{1}{H^n}{\Sigma_{n}}} "01" ; "11"}:
            \endxy
    \]
commutes. For an object $x \in X$, write 
\begin{align*}
K(x) & = [u; x_1, \ldots, x_n],\\
L(x) & = [v; x_1', \ldots, x_n'].
\end{align*}
Since the action of $\Sigma_n$ is free on $P(n)$, \cref{lem:coeq-lem} and \cref{rem:operadcoeq-as-quotient} imply that the equation $\coeqb{1}{S^n}{\Sigma_{n}}\circ K(x) = \coeqb{1}{H^n}{\Sigma_{n}}\circ L(x)$ is equivalent to the existence of a unique $g \in \Sigma_n$ such that
\begin{itemize}
\item $u \cdot g = v$ and
\item for each $i$, $Sx_i = Hx'_{g^{-1}(i)}$.
\end{itemize}
Since $A$ is the pullback of the original square, there exists a unique $a_i$ such that $F(a_i) = x_i$ and $R(a_i) = x'_{g^{-1}(i)}$. Define $J \colon X \to \coeqsig{P}{A}{n}$ on objects by 
\[
J(x) = [u; a_1, \ldots, a_n].
\]
We compute
\begin{align*}
\coeqb{1}{F^n}{\Sigma_{n}}\circ J(x) & = \coeqb{1}{F^n}{\Sigma_{n}}[u; a_1, \ldots, a_n] \\
& = [u; Fa_1, \ldots, Fa_n] \\
& = [u; x_1, \ldots, x_n]\\
& = K(x), \\
\coeqb{1}{R^n}{\Sigma_{n}}\circ J(x) & = \coeqb{1}{R^n}{\Sigma_{n}}[u; a_1, \ldots, a_n] \\
& = [u; Ra_1, \ldots, Ra_n] \\
& = [u; x'_{g^{-1}(1)}, \ldots, x'_{g^{-1}(n)}]\\
& = [u; g \cdot(x'_1, \ldots, x'_n)] \\
& = [u \cdot g; x'_1, \ldots, x'_n] \\
& = [v; x'_1, \ldots, x'_n] \\
& = L(x).
\end{align*}
The same uniqueness arguments, using \cref{lem:coeq-lem} and \cref{rem:operadcoeq-as-quotient}, show that this is the unique assignment on objects making both $\coeqb{1}{F^n}{\Sigma_{n}}\circ J = K$ and $\coeqb{1}{R^n}{\Sigma_{n}}\circ J = L$ hold at the level of objects. The argument above applies equally to morphisms, and it is simple to show that the resulting $J$ is the unique functor $X \to \coeqsig{P}{A}{n}$ satisfying $\coeqb{1}{F^n}{\Sigma_{n}}\circ J = K$ and $\coeqb{1}{R^n}{\Sigma_{n}}\circ J = L$. Therefore $\coeqsig{P}{A}{n}$ is the pullback as required, completing the proof.
\end{proof}

Collecting \cref{prop:cart_unit,prop:P-2cart,prop:mu-2cart} together gives the following corollary.

\begin{cor}\label{cor:cart_cor}
The $2$-monad associated to a symmetric operad $P$ is cartesian if and only if the action of $\Sigma_n$ is free on each $P(n)$.
\end{cor}

We require one simple technical lemma before giving a complete characterization of $\Lambda$-operads that induce cartesian $2$-monads.

\begin{lem}\label{lem:kernel_lem}
Let $C$ be a category with a right action of some group $\Lambda$ via $\mu \colon C \times \Lambda \to C$, and let $\pi \colon  \Lambda \rightarrow \Sigma$ be a group homomorphism to any other group $\Sigma$. Then the right $\Sigma$-action on the category $\coeqb{C}{\Sigma}{\Lambda}$, defined as the coequalizer below
    \[
        \xy
            (0,0)*+{C \times \Lambda \times \Sigma}="00";
            (30,0)*+{C \times \Sigma}="10";
            (60,0)*+{\coeqb{C}{\Sigma}{\Lambda}}="20";
            {\ar@<1ex>^{m \circ 1 \times \pi} "00" ; "10"};
            {\ar@<-1ex>_{\mu \times 1} "00" ; "10"};
            {\ar^{\varepsilon} "10" ; "20"};
        \endxy
    \]
where $m \colon \Sigma \times \Sigma \to \Sigma$ is the group multiplication,
is free if and only if the kernel of $\pi$ contains all the elements of $\Lambda$ that fix any object of $C$.
\end{lem}
\begin{proof}
By \cref{lem:free-on-obj}, we only need to check that the action is free on objects.
Since the set of objects functor preserves colimits, the objects of $\coeqb{C}{\Sigma}{\Lambda}$ are equivalence classes $[c;g]$ where $c \in C$ and $g \in \Sigma$, with $[c\cdot r;g] = [c; \pi(r)g]$. 
These classes can also be described as: $[c_1; \sigma_1] = [c_2; \sigma_2]$ if and only if there exists an $r \in \Lambda$ such that $ c_1 \cdot r = c_2$ and $\sigma_1 = \pi(r^{-1})\sigma_2$.

First assume the $\Sigma$-action is free. Then noting that $[c;e]\cdot g =[c;g]$, we have if $[c;g] = [c;e]$ then $g=e$. Let $r \in \Lambda$ be an element such that $c\cdot r = c$. Then
  \[
    [c;e] = [c\cdot r; e] = [c; \pi(r)],
  \]
so $\pi(r) = e$.

Now assume that every element of $\Lambda$ fixing an object lies in the kernel of $\pi$. Let $\tau \in \Sigma$, and assume it fixes $[p; \sigma]$, so that $[p; \sigma] = [p; \sigma \tau]$.
Then there exists an element $r \in \Lambda$ such that $p\cdot r = p$ and $\sigma = \pi(r^{-1})\sigma\tau$. 
Then $r$ fixes $p$, so lies in the kernel of $\pi$, and the second equation reduces to $\sigma = \sigma \tau$ which immediately implies that $\tau = e$. Therefore the action of $\Sigma$ is free on $\coeqb{C}{\Sigma}{\Lambda}$.
\end{proof}

\begin{thm}\label{thm:cart_thm}
The $2$-monad $\underline{P}$ associated to a $\Lambda$-operad $P$ is cartesian if and only if whenever $p \cdot g = p$ for an object $p \in P(n)$, $g \in \textrm{Ker} \, \pi (n)$.
\end{thm}
\begin{proof}
By \cref{cor:sym-preserve-alg}, the monad $\underline{P}$ is isomorphic to $\underline{\pi_{!}P}$, we need only verify when $\underline{\pi_{!}P}$ is $2$-cartesian. Thus the theorem is a direct consequence of \cref{lem:kernel_lem} and \cref{cor:cart_cor}.
\end{proof}

\begin{cor}\label{cor:EL-2cart}
Let $\Lambda$ be an action operad in $\mb{Sets}$. Then the $2$-monad $\underline{\EL}$ is cartesian.
\end{cor}
\begin{proof}
The action of $\Lambda(n)$ on $\EL(n)$ is free for all $n$, so in particular satisfies the conditions in \cref{thm:cart_thm}.
\end{proof}

\section{Action Operads as Clubs}\label{sec:club}

This is the second section to investigate the interaction of action operads and pullbacks using the clubs of Kelly \cite{kelly_club1, kelly_club0, kelly_club2}. Kelly's theory of clubs  was designed to simplify and explain how coherence results for a $2$-monad $T$ can often be extracted from information about the 
specific free object $T1$, where $1$ denotes the terminal category. 
We shall see that the $\Lambda$-monoidal structure on $\EL(1)$ recovers the entire action operad structure on $\Lambda$. Furthermore, the presentations of action operads from \cref{sec:pres-aop} match up with presentations of clubs from \cite[Section 3]{kelly_club1}. This fact gives a conceptual explanation for the calculations in \cref{ex:sigma-pres}.

We begin by reminding the reader of the notion of a club, or more specifically what Kelly \cite{kelly_club1,kelly_club2} calls a club over $\mb{P}$. We will only be interested in clubs over $\mb{P}$, and thusly shorten the terminology to club from this point onward. We define clubs succinctly using Leinster's terminology of generalized operads \cite{leinster}.

\begin{Defi}[($T$-Collections, $T$-operads)]\label{Defi:Tcoll-op}
Let $C$ be a category with finite limits.
\begin{enumerate}
\item A monad $T \colon C \rightarrow C$ is \textit{cartesian} if the functor $T$ preserves pullbacks, and the naturality squares for the unit $\eta$ and the multiplication $\mu$ for $T$ are all pullbacks.
\item The category of \textit{$T$-collections}, $T\mbox{-}\mb{Coll}$, is the slice category $C/T1$, where $1$ denotes the terminal object.
\item Given a pair of $T$-collections $X \stackrel{x}{\rightarrow} T1, Y \stackrel{y}{\rightarrow} T1$, their \textit{composition product} $X \circ Y$ is given by the pullback below together with the morphism along the top.
  \[
    \xy
      (0,0)*+{X \circ Y} ="00";
      (15,0)*+{TY} ="10";
      (30,0)*+{T^{2}1} ="20";
      (45,0)*+{T1} ="30";
      (0,-10)*+{X} ="01";
      (15,-10)*+{T1} ="11";
      {\ar^{} "00" ; "10"};
      {\ar^{Ty} "10" ; "20"};
      {\ar^{\mu} "20" ; "30"};
      {\ar^{T!} "10" ; "11"};
      {\ar_{x} "01" ; "11"};
      {\ar^{} "00" ; "01"};
      (3,-3)*{\lrcorner};
    \endxy
  \]
\item The composition product, along with the unit of the adjunction $\eta \colon 1 \rightarrow T1$, give $T\mbox{-}\mb{Coll}$ a monoidal structure. A \textit{$T$-operad} is a monoid in $T\mbox{-}\mb{Coll}$.
\end{enumerate}
\end{Defi}

\begin{rem}
Everything in the above definition can be $\mb{Cat}$-enriched without any substantial modifications. Thus we require our ground $2$-category to have finite limits in the enriched sense, and the slice and pullbacks are the $2$-categorical (and not bicategorical) versions. If we take this $2$-category to be $\mb{Cat}$, then in each case the underlying category of the $2$-categorical construction is given by the corresponding $1$-categorical version. From this point, we will not distinguish between the $1$-dimensional and $2$-dimensional theory. 
\end{rem}

Let $\Sigma$ be the operad of symmetric groups. This is the terminal object of the category of action operads, with each $\pi_{n}$ the identity map. Then $\underline{E\Sigma}$ is a $2$-monad on $\mb{Cat}$, and by \cref{cor:EL-2cart} it is cartesian.

\begin{Defi}[(Club)]\label{Defi:club}
A \textit{club} is a $T$-operad in $\mb{Cat}$ for $T = \underline{E\Sigma}$.
\end{Defi}

\begin{rem}[($\mb{P} = B\Sigma = \underline{E\Sigma}(1)$)]
The category $\mb{P}$ in Kelly's notation (defined in \cite[Section 2]{kelly_club0}) is the category $B\Sigma$ of \cref{Defi:BLambda}, or equivalently
 the result of applying $\underline{E\Sigma}$ to $1$ by \cref{cor:el1=bl}.
\end{rem}

\begin{rem}[(Explicit description of clubs)]\label{rem:exp-club}
It is useful to break down the definition of a club. A club consists of
\begin{enumerate}
\item a category $K$ together with a functor $k \colon K \rightarrow B \Sigma$,
\item a multiplication map $K \circ K \rightarrow K$, and
\item a unit map $1 \rightarrow K$,
\end{enumerate}
satisfying the axioms to be a monoid in the monoidal category of $E\Sigma$-collections. The objects of $K \circ K$  are tuples of objects of $K$, written $(x; y_{1}, \ldots, y_{n})$, where $k(x) = n$. In order to describe the morphisms of $K \circ K$, recall the description of the hom-sets of $E\Sigma(K)$ from \cref{prop:hom-set-lemma}.
A morphism
  \[
    (x; y_{1}, \ldots, y_{n}) \rightarrow (z; w_{1}, \ldots, w_{m})
  \]
exists only when $n=m$ (since $B\Sigma$ only has endomorphisms) and then consists of a morphism $f \colon x \rightarrow z$ in $K$ together with morphisms $g_{i} \colon y_{i} \rightarrow w_{f(i)}$ in $K$; here we have written $f(i)$ for the permutation $k(f)$ applied to the element $i$, following \cref{nota:perm_shorthand}.
\end{rem}

\begin{nota}\label{nota:clubmult}
For a club $K$ and a morphism $(f; g_{1}, \ldots, g_{n})$ in $K \circ K$, we write $f(g_{1}, \ldots, g_{n})$ for the image of the morphism under the functor $K \circ K \rightarrow K$.
\end{nota}

We will usually just refer to a club by its underlying category $K$.

\begin{Defi}\label{Defi:2monad-from-club}
Let $K$ be a club. The 2-monad $K^{\textrm{m}}$ on $\mb{Cat}$ is defined as follows.
\begin{itemize}
\item The underlying 2-functor of $K^{\textrm{m}}$ is given by $K^{\textrm{m}}(X) = K \circ X$, where the category $X$ is equipped with the $\underline{E\Sigma}$-collection structure $X \stackrel{!}{\to} 1 \stackrel{\eta}{\to} E\Sigma(1)$.
\item The multiplication and unit are induced from $K$, using its multiplication and unit as a club.
\end{itemize}
\end{Defi}

\begin{thm}\label{thm:pi-club}
Let $\Lambda$ be an action operad. Then the map of operads $\pi \colon \Lambda \rightarrow \Sigma$ gives the category $B\Lambda = \coprod B\Lambda(n)$ from \cref{Defi:BLambda} the structure of a club.
\end{thm}
\begin{proof}
To give the functor $B\pi \colon B\Lambda \rightarrow B \Sigma$ the structure of a club it suffices (see \cite[Section 6.2]{leinster}) to show that
\begin{itemize}
\item the induced monad, which we will show to be $\underline{\EL}$, is a cartesian monad on $\mb{Cat}$,
\item the transformation $\tilde{\pi} \colon \underline{\EL} \Rightarrow \underline{E\Sigma}$ induced by the functor $E\pi$ is cartesian, and
\item $\tilde{\pi}$ commutes with the monad structures.
\end{itemize}
The monad $\underline{\EL}$ is always cartesian by \cref{cor:EL-2cart}. The transformation $\tilde{\pi}$ is the coproduct of the maps $\tilde{\pi}_{n}$ that are induced by the universal property of the coequalizer as shown below.
  \[
    \xy
      (0,0)*+{\scriptstyle \ELn \times \Lambda(n) \times X^n} ="00";
      (0,-15)*+{\scriptstyle E\Sigma_{n} \times \Sigma_{n} \times X^n} ="01";
      (30,0)*+{\scriptstyle \ELn \times X^n} ="10";
      (30,-15)*+{\scriptstyle E\Sigma_{n} \times X^n} ="11";
      (60,0)*+{\scriptstyle \ELn \otimes_{\Lambda(n)} X^n} ="20";
      (60,-15)*+{\scriptstyle E\Sigma_{n} \otimes_{\Sigma_{n}}  X^n} ="21";
      {\ar (11,1)*{}; (22,1)*{} };
      {\ar (11,-1)*{}; (22,-1)*{} };
      {\ar_{E\pi \times \pi \times 1} "00" ; "01"};
      {\ar (10,-14)*{}; (23,-14)*{} };
      {\ar (10,-16)*{}; (23,-16)*{} };
      {\ar_{E\pi \times 1} "10" ; "11"};
      {\ar@{.>}^{\tilde{\pi}_{n}} "20" ; "21"};
      {\ar "10" ; "20"};
      {\ar "11" ; "21"};
    \endxy
  \]
Naturality is immediate, and since $\pi$ is a map of operads $\tilde{\pi}$ also commutes with the monad structures.

It only remains to show that $\tilde{\pi}$ is cartesian and that the induced monad is actually $\underline{\EL}$. 
By pullback cancellation (\cref{rem:pb-cancel}), these will both follow if we prove that $\EL(X) \cong B\Lambda \circ X$, or equivalently if
 \begin{equation}\label{eqn:pb-club}
    \xy
      (0,0)*+{E\Lambda(X)} ="00";
      (0,-10)*+{B\Lambda} ="01";
      (35,0)*+{E\Sigma(X)} ="10";
      (35,-10)*+{B\Sigma} ="11";
      {\ar^{} "00" ; "10"};
      {\ar^{} "10" ; "11"};
      {\ar^{} "00" ; "01"};
      {\ar^{} "01" ; "11"};
    \endxy
  \end{equation}
  is a pullback. This fact follows immediately from the description of the free objects in \cref{prop:hom-set-lemma}.
\end{proof}

Let $(\Lambda, \pi)$ be an action operad, and $B\Lambda$ the club from \cref{thm:pi-club}. The proof above shows that the 2-monad $\underline{\EL}$ is (isomorphic to) the 2-monad $B\Lambda^{\textrm{m}}$ from \cref{Defi:2monad-from-club}.
The club $B\Lambda$ has the following properties. First, the category $B\Lambda$ is a groupoid. Second, the functor $B\pi \colon B\Lambda \rightarrow B\Sigma$ is  bijective on objects. We claim that these properties characterize those clubs that arise from action operads. Thus the clubs arising from action operads are a special class of PROPs \cite{mac_prop, markl_prop}.

\begin{thm}\label{thm:club=operad}
Let $(K, k)$ be a club such that
\begin{itemize}
\item the map $k \colon K \rightarrow B\Sigma$ is bijective on objects and
\item $K$ is a groupoid.
\end{itemize}
Then $(K,k) \cong (B\Lambda, B\pi)$ for some action operad $\Lambda$. The assignment $\Lambda \mapsto (B\Lambda, B\pi)$ is a full and faithful embedding of the category of action operads $\mb{AOp}$ into the category of clubs.
\end{thm}
\begin{proof}
Let $(K,k)$ be such a club. Our hypotheses immediately imply that $K$ is a groupoid with objects in bijection with the natural numbers; we will now assume the functor $K \rightarrow B\Sigma$ is the identity on objects. Furthermore, we can conclude that there is an isomorphism $m \cong n$ in $K$ if and only if $m = n$. Let $\Lambda(n) = K(n,n)$, so that $K = \coprod B\Lambda(n)$ as groupoids. Define $B\pi = k$, or equivalently define the homomorphism $\pi_{n} \colon \Lambda(n) \rightarrow \Sigma_{n}$ to the morphism part of the composite
\[
B\Lambda(n) \hookrightarrow K \stackrel{k}{\to} B\Sigma.
\]
 We claim that the club structure on $K$ makes the collection of groups $\{ \Lambda(n) \}$ an action operad. In order to do so, we will employ \cref{thm:charAOp}.

First, we give the group homomorphism $\beta$ using \cref{rem:exp-club,nota:clubmult}. Define
  \[
    \beta(g_{1}, \ldots, g_{n}) = e_{n}(g_{1}, \ldots, g_{n})
  \]
where $e_{n}$ is the identity morphism $n \rightarrow n$ in $K(n,n)$. Functoriality of the club multiplication map immediately implies that this is a group homomorphism. Second, we define the function $\delta$ in a similar fashion:
  \[
    \delta_{n; k_{1}, \ldots, k_{n}}(f) = f(e_{k_1}, \ldots, e_{k_n}),
  \]
where here $e_{k_i}$ is the identity morphism of $k_{i}$ in $K$.

There are now nine axioms to verify in \cref{thm:charAOp}. The club multiplication functor is a map of collections, so a map over $B\Sigma$; this fact immediately implies that Axioms \eqref{eq1} (using morphisms in $K \circ K$ with only $g_{i}$ parts) and \eqref{eq4} (using morphisms in $K \circ K$ with only $f$ parts) hold. The mere fact that multiplication is a functor also implies Axioms \eqref{eq6} (once again using morphisms with only $f$ parts) and \eqref{eq8} (by considering the composite of a morphism with only an $f$ with a morphism with only $g_{i}$'s). Axiom \eqref{eq2} is the equation $e_{1}(g) = g$ which is a direct consequence of the unit axiom for the club $K$; the same is true of Axiom \eqref{eq5}. Axioms \eqref{eq3}, \eqref{eq7},  and \eqref{eq9} all follow from the associativity of the club multiplication. Thus $(\Lambda, \pi)$ is an action operad, and using the pullback square \cref{eqn:pb-club} we see that $(K,k) \cong (B\Lambda, B\pi)$ as clubs.

Finally, we will show that the construction above gives a full and faithful embedding
    \begin{equation}\label{eqn:B-aop}
        B \colon \mb{AOp} \rightarrow \mb{Club}
    \end{equation}
of the category of action operads into the category of clubs. Let $f, f' \colon \Lambda \rightarrow \Lambda'$ be maps between action operads. Then if $Bf = Bf'$ as maps between clubs, then they must be equal as functors $B\Lambda \rightarrow B\Lambda'$. But these functors are nothing more than the coproducts of the functors
  \[
    B(f_{n}), B(f_{n}') \colon B\Lambda(n) \rightarrow B\Lambda'(n),
  \]
and the functor $B$ from groups to categories is faithful, so \cref{eqn:B-aop} is also faithful. Now let $f \colon B\Lambda \rightarrow B\Lambda'$ be a maps of clubs. We clearly get group homomorphisms $f_{n} \colon \Lambda(n) \rightarrow \Lambda'(n)$ such that $\pi^{\Lambda}_{n} = \pi^{\Lambda'}_{n} f_{n}$, so we must only show that the $f_{n}$ also constitute an operad map. Using the description of the club structure above in terms of the maps $\beta, \delta$, we conclude that commuting with the club multiplication implies commuting with both of these, which in turn is equivalent to commuting with operad multiplication. Thus \cref{eqn:B-aop} is full as well.
\end{proof}

\begin{rem}[(Relaxing the hypotheses in \cref{thm:club=operad})]
First, one should note that if $K$ is a club over $B\Sigma$, then every $K$-algebra has an underlying strict monoidal structure. Second, requiring that $K \rightarrow B\Sigma$ be bijective on objects ensures that $K$ does not have  operations other than $\otimes$, such as duals or internal hom-objects, from which to build new types of objects. Finally, $K$ being a groupoid ensures that all of the ``constraint morphisms'' that exist in algebras for $K$ are invertible.

These hypotheses could be relaxed somewhat. Instead of having a club over $B\Sigma$, we could have a club over the free symmetric monoidal category on one object (note that the free symmetric monoidal category monad on $\mb{Cat}$ is still cartesian). This would produce $K$-algebras with underlying monoidal structures that are not necessarily strict. This change should have relatively little impact on how the theory is developed. Changing $K$ to be a category instead of a groupoid would likely have a larger impact, as the resulting action operads would have monoids instead of groups at each level. We have made repeated use of inverses throughout the proofs in the basic theory of action operads, and these would have to be revisited if groups were replaced by monoids in the definition of action operads.
\end{rem}

In \cite[Section 3]{kelly_club1}, Kelly discusses clubs given by generators and relations of the type that are relevant to our study of action operads; we call particular attention to his Theorem 3.1. His generators include functorial operations more general than what we are interested in here, and the natural transformations are not required to be invertible. In our case, the only generating operations we require are those of a unit and tensor product, as the algebras for $\EL$ are always strict monoidal categories with additional structure. 

\begin{nota}\label{nota:tensor-n}
Suppose that $M$ is a monoidal category via $\otimes, I$. Write $\otimes_n$ for the functor $M^n \to M$ given inductively by
\begin{itemize}
\item $\otimes_0$ is the functor $1 \to M$ sending the unique object of the terminal category 1 to the unit $I \in M$, and
\item given $\otimes_k$, define 
\[
\otimes_{k+1}(m_1, \ldots, m_k, m_{k+1}) = \otimes_2\big( m_1, \otimes_k(m_2, \ldots, m_{k+1}) \big).
\]
\end{itemize}
We note that since $\mb{Cat}$ is a symmetric monoidal category under the cartesian product, for any $\sigma \in \Sigma_n$ there is an isomorphism $\sigma \colon M^n \to M^n$ given explicitly on objects by
\[
\sigma(m_1, \ldots, m_n) = \big( m_{\sigma^{-1}(1)}, \ldots, m_{\sigma^{-1}(n)} \big). 
\]
\end{nota}

Tracing through Kelly's discussion of generators and relations for a club gives the following theorem.

\begin{thm}\label{thm:pres1}
Let $\Lambda$ be an action operad with presentation given by $(\mathbf{g},\mathbf{r}, s_{i}, p)$. Then the club $\EL$ is generated by
\begin{itemize}
  \item functors giving the unit object and tensor product, and
  \item natural transformations given by the set $\mathbf{g}$:  each element $x$ of $\mathbf{g}$ with $\pi(x) = \sigma_{x} \in \Sigma_{|x|}$ gives a natural transformation $\alpha_x \colon \otimes_{|x|} \Rightarrow \otimes_{|x|} \circ \sigma$,
\end{itemize}
subject to relations enforcing the following axioms.
\begin{itemize}
  \item The monoidal structure given by the unit and tensor product is strict.
  \item The transformations given by the elements of $\mathbf{g}$ are all natural isomorphisms.
  \item For each element $y \in \mathbf{r}$, the equation $s_{1}(y) = s_{2}(y)$ holds.
\end{itemize}
\end{thm}

\begin{example}\label{ex:S-moncats}
We can use \cref{thm:pres1} to express a given type of monoidal structure by constructing a presentation of the corresponding action operad.
\begin{enumerate}
\item Strict monoidal categories are the algebras for the club with $\mathbf{g} = \emptyset$. Thus they are the algebras for $\EL$ for $\Lambda$ the initial action operad. The category of action operads is pointed, so $\Lambda = T$.
\item Symmetric strict monoidal categories are the algebras for the club with a single generating natural isomorphism $\beta$ with $\pi(\beta) = (1 \ 2) \in \Sigma_2$. The only axioms needed are
\begin{itemize}
\item $\beta_{y,x} \circ \beta_{x,y} = 1_{x \otimes y}$,
\item $\beta_{x,yz} = (1_y \otimes \beta_{x,z}) \circ (\beta_{x,y} \otimes 1_z)$,
\end{itemize}
recovering the presentation for $\Sigma$ as an action operad from \cref{ex:sigma-pres}.
\item Braided strict monoidal categories are the algebras for the club with a single generating natural isomorphism $\beta$ with $\pi(\beta) = (1 \ 2) \in \Sigma_2$. The only axioms needed are
\begin{itemize}
\item $\beta_{x,yz} = (1_y \otimes \beta_{x,z}) \circ (\beta_{x,y} \otimes 1_z)$,
\item $\beta_{xy,z} = (\beta_{x,z} \otimes 1_y) \circ (1_{x} \otimes \beta_{y,z})$.
\end{itemize}
These axioms give a presentation for the braid action operad $B$.
\end{enumerate}
\end{example}

\begin{rem}[(Group presentation versus club presentation)]\label{rem:two-presentations}
The examples above demonstrate that the presentations of well-known action operads are often quite compact, but the computations in \cref{ex:sigma-pres} show that it is a nontrivial problem to translate between a sequence of presentations for the individual groups $\Lambda(n)$ and a single presentation for $\Lambda$ as an action operad.
\end{rem}

\section{Extended Example: Coboundary Categories}\label{sec:exex-cactus}

We now turn to an example that is not as widely known in the categorical literature, that of coboundary categories \cite{drin-quasihopf}. These arise in the representation theory of quantum groups and in the theory of crystals \cite{hk-cobound, hk-quantum}. Our goal here is to refine the relationship between coboundary categories and the operad of $n$-fruit cactus groups in \cite{hk-cobound} by using presentations of action operads. We begin by recalling the definition of a coboundary category.

\begin{Defi}\label{def:cobcat}
A \textit{coboundary category} is a monoidal category $C$ equipped with a natural isomorphism $\sigma_{x,y} \colon x \otimes y \rightarrow y \otimes x$ (called the \textit{commutor}) such that
\begin{itemize}
\item $\sigma_{y,x} \circ \sigma_{x,y} = 1_{x \otimes y}$ and
\item the diagram below, called the \emph{cactus relation}, commutes (in which the unlabeled morphisms are an associator and an inverse associator).
  \[
    \xy
      (0,0)*+{(x \otimes y) \otimes z} ="00";
      (35,0)*+{x \otimes (y \otimes z)} ="10";
      (70,0)*+{x \otimes (z \otimes y)} ="20";
      (0,-15)*+{(y \otimes x) \otimes z} ="01";
      (35,-15)*+{z \otimes (y \otimes x)} ="11";
      (70,-15)*+{(z \otimes y )\otimes x} ="21";
      {\ar "00"; "10" };
      {\ar^{1 \sigma_{y,z}} "10"; "20" };
      {\ar^{\sigma_{x,zy}} "20"; "21" };
      {\ar_{\sigma_{x,y}1} "00"; "01" };
      {\ar_{\sigma_{yx,z}} "01"; "11" };
      {\ar "11"; "21" };
    \endxy
  \]
\end{itemize}
\end{Defi}

\begin{example}\label{ex:cobcats}
\begin{enumerate}
\item As noted by Savage \cite{savage-braidcob}, any braiding automatically satisfies the cactus relation (the diagram in \cref{def:cobcat}). However, since braidings need not be involutions this does not mean that any braided monoidal category is a coboundary category. However, it should then be clear that any symmetric monoidal category is also a coboundary category.
\item The name coboundary category comes from the original work of Drinfeld \cite{drin-quasihopf} in which he shows that the category of representations of a coboundary Hopf algebra has the structure of coboundary category.
\item Henriques and Kamnitzer \cite{hk-cobound} show that the category of crystals for a finite dimensional complex reductive Lie algebra has the structure of a coboundary category. 
\end{enumerate}
\end{example}

\begin{rem}
By \cref{thm:wlmc-to-lmc}, we restrict ourself to the case in which the underlying monoidal structure is strict.
\end{rem}

We now turn to the operadic description of coboundary categories.

\begin{Defi}[(Contains, disjoint)]
Fix $n>1$, and let $1 \leq p < q \leq n$, $1 \leq k < l \leq n$.
\begin{enumerate}
\item $p<q$ \textit{contains} $k<l$ if $p \leq k < l \leq q$.
\item $p<q$ is \textit{disjoint} from $k<l$ if $q<k$ or $l<p$.
\end{enumerate}
\end{Defi}

\begin{Defi}
Let $1 \leq p < q \leq n$, and define $\hat{s}_{p,q} \in \Sigma_{n}$ to be the permutation defined below.
  \[
    \begin{array}{r|ccccccccccccc}
      i & 1 & 2 & \cdots & p-1 & p & p+1 & p+2 & \cdots & q-1 & q & q+1 & \cdots & n \\
      \hat{s}_{p,q}(i) & 1 & 2 & \cdots & p-1 & q & q-1 & q-2 & \cdots & p+1 & p & q+1 & \cdots & n
    \end{array}
  \]
\end{Defi}

The $n$-fruit cactus group is then defined as follows.

\begin{Defi}\label{Defi:defcactus}
Let $J_{n}$ be the group generated by symbols $s_{p,q}$ for $1 \leq p < q \leq n$ subject to the following relations.
  \begin{enumerate}
    \item For all $p < q$, $s_{p,q}^{2}=e$.
    \item If $p<q$ is disjoint from $k<l$, then $s_{p,q}s_{k,l} = s_{k,l}s_{p,q}$.
    \item If $p<q$ contains $k<l$, then $s_{p,q}s_{k,l} = s_{m,n}s_{p,q}$ where
      \begin{itemize}
        \item $m = \hat{s}_{p,q}(l)$ and
        \item $n = \hat{s}_{p,q}(k)$.
      \end{itemize}
  \end{enumerate}
\end{Defi}

It is easy to check that the elements $\hat{s}_{p,q} \in \Sigma_{n}$ satisfy the three relations in \cref{Defi:defcactus}, so $s_{p,q} \mapsto \hat{s}_{p,q}$ extends to a group homomorphism $\pi_{n} \colon J_{n} \rightarrow \Sigma_{n}$. This is the first step in proving the following.

\begin{thm}\label{thm:J_aop}
The collection of groups $J = \{ J_{n} \}$ form an action operad.
\end{thm}
\begin{proof}
We will use \cref{thm:charAOp} to determine the rest of the action operad structure. Thus we must give, for any collection of natural numbers $n, k_{1}, \ldots, k_{n}$ and $K = \sum k_{i}$, group homomorphisms $\beta \colon J_{k_{1}} \times \cdots \times J_{k_{n}} \rightarrow J_{K}$ and functions $\delta \colon J_{n} \rightarrow J_{K}$ satisfying nine axioms. We define both of these on generators, starting with $\beta$.

Let $s_{p_{i}, q_{i}} \in J_{k_{i}}$. Let $r_{i} = k_{1} + k_{2} + \cdots + k_{i-1}$ for $i > 1$. Define $\beta$ by
  \[
    \beta(s_{p_{1}, q_{1}}, \ldots, s_{p_{n}, q_{n}}) = s_{p_{1}, q_{1}} s_{p_{2}+r_{2}, q_{2}+r_{2}} \cdots s_{p_{n}+r_{n}, q_{n}+r_{n}}.
  \]
Note that $s_{p_{i}+r_{i}, q_{i}+r_{i}}$ and $s_{p_{j}+r_{j}, q_{j}+r_{j}}$ are disjoint when $i \neq j$.
It is easy to check that this disjointness property ensures that $\beta$ gives a well-defined group homomorphism
  \[
    J_{k_{1}} \times \cdots \times J_{k_{n}} \rightarrow J_{K}.
  \]

To define $\delta \colon J_{n} \rightarrow J_{K}$ for natural numbers $n, k_{1}, \ldots, k_{n}$ and $K = \sum k_{i}$, let $t_{k} = s_{1,k} \in J_{k}$. Then we start by defining
  \[
    \delta(t_{n}) = t_{K} \cdot \beta(t_{k_{1}}, t_{k_{2}}, \ldots, t_{k_{n}}).
  \]
By the third axiom in \cref{Defi:defcactus}, this is equal to
  \[
    \beta(t_{k_{n}}, t_{k_{n-1}}, \ldots, t_{k_{1}}) \cdot t_{K}.
  \]
Now $s_{p,q} \in J_{n}$ is equal to $\beta(e_{p-1}, t_{q-p+1}, e_{n-q})$ (here $e_{i}$ is the identity element in $J_{i}$) by definition of the $t_{i}$ and $\beta$, so we can define $\delta$ on any generator $s_{p,q}$ by
  \[
    \delta(s_{p,q}) = \beta ( e_{A}, M, e_{B} )
  \]
with
  \begin{itemize}
    \item $A = k_{1} + k_{2} + \cdots + k_{p-1}$,
    \item $M = t_{k_{p}+ \cdots +k_{q}} \cdot \beta(t_{k_{p}}, t_{k_{p+1}}, \ldots, t_{k_{q}})$, and
    \item $B = k_{q+1} + k_{q+2} + \cdots + k_{n}$.
  \end{itemize}
Unpacking this yields the following formula:
  \[
  \delta(s_{p,q}) = s_{k_{1}+\cdots+k_{p-1}+1, k_{1}+\cdots+k_{q}} \cdot \beta(e_{k_{1}+\cdots+k_{p-1}}, t_{k_{p}}, \ldots, t_{k_{q}}, e_{k_{q+1}+\cdots+k_{n}}).
  \]

We extend $\delta$ to arbitrary elements of $J_n$ using \cref{thm:charAOp}. 
Now $\delta$ is not a group homomorphism, but it does satisfy a twisted version in Axiom \eqref{eq6}.
Define
  \[
    \delta_{n; j_1,\ldots,j_n}(gh) = \delta_{n; k_1,\ldots,k_n}(g)\delta_{n; j_1,\ldots,j_n}(h)
  \]
where $k_{i} = j_{h^{-1}(i)}$. There are three relations we must verify for compatibility.
\begin{itemize}
\item We must show that $\delta_{n; j_1,\ldots,j_n}\left(s_{p,q}^{2}\right) = e$. By definition, we have
  \[
    \delta_{n; j_1,\ldots,j_n}\left(s_{p,q}^{2}\right) = \delta_{n; k_1,\ldots,k_n}\left(s_{p,q}\right)\delta_{n; j_1,\ldots,j_n}\left(s_{p,q}\right)
  \]
which is
  \[
    t_{j_1 + \cdots + j_n}\beta(t_{j_{n}}, \ldots, t_{j_{1}}) t_{\underline{j}} \beta(t_{j_{1}}, \ldots, t_{j_{n}}).
  \]
By the definition of $\delta$ and the fact that $s_{p,q}^{2}=e$, the element above is easily seen to be the identity.
\item We must show that $\delta(s_{p,q}s_{k,l}) = \delta(s_{k,l}s_{p,q})$ when $(p,q)$ is disjoint from $(k,l)$. This is another simple calculation using the definition of $\delta$ and the disjointness of the terms involved.
\item We must show that $\delta(s_{p,q}s_{k,l}) = \delta(s_{a,b}s_{p,q})$,  where $a = \hat{s}_{p,q}(l), b = \hat{s}_{p,q}(k)$, if $p < k < l < q$. In this case, we use all of the relations in the cactus groups to show that each side is equal to
  \[
    \beta\left(\underline{e}, t_{j_{p}+\cdots + j_{q}} \cdot \beta \left(t_{j_{p}}, \ldots t_{j_{k-1}}, t_{j_{k}+ \cdots +j_{l}}, t_{j_{l+1}}, \ldots, t_{j_{q}}\right), t_{j_{q+1}}, \ldots, t_{j_{n}}\right)
  \]
where $\underline{e} = e_{j_{1}}, \ldots, e_{j_{p-1}}$.
\end{itemize}
In order to show that this gives a well-defined function on products of three or more generators, a straightforward induction argument shows that $\delta\left((fg)h\right) = \delta\left(f(gh)\right)$ using the formula above. This concludes the definition of the family of functions $\delta_{n; j_{i}}$.

There are now nine axioms to check in \cref{thm:charAOp}. Axioms \eqref{eq1} - \eqref{eq3} all concern $\beta$, and are immediate from the defining formula. Axiom \eqref{eq4} is obvious for the elements $t_{k}$, from which it follows in general by the formulas defining $\delta$. For Axiom \eqref{eq5}, one can check easily that
  \[
    \delta_{n; 1, \ldots, 1}(t_{n}) = t_{n}, \quad \delta_{n;k_1,\ldots,k_n}(e_n) = e_{k_1 + \cdots + k_n}
  \]
following the description of $\delta$ above and once again the general case follows from these. Axiom \eqref{eq6} holds by the construction of $\delta$. Axiom \eqref{eq8} can be verified with only one $h_{i}$ nontrivial at a time, and then it is a simple consequence of the second and third relations for $J_{n}$.

Axiom \eqref{eq9} is straightforward to check when only a single $g_{i}$ is a generator and the rest are identities using the defining formulas, and the general case then follows using Axiom \eqref{eq6}. Using Axiom \eqref{eq9}, we can then prove Axiom \eqref{eq7} as follows; we suppress the subscripts on different $\delta$'s for clarity. We must show
  \[
    \delta_{m_1 + \cdots + m_n; p_{11}, \ldots, p_{1m_{1}}, p_{21}, \ldots, p_{nm_{m}}}\left( \delta_{n; m_{1}, \ldots, m_{n}}(f) \right) = \delta_{n; P_{1}, \ldots, P_{n}}(f),
  \]
and we do so on $t_{n}$. By definition, we have
  \[
    \delta \left( \delta(t_{n}) \right) = \delta \left( t_{K} \beta(t_{k_{1}}, \ldots, t_{k_{n}}) \right),
  \]
which by Axiom \eqref{eq6} is equal to
  \[
    t_{P_{1} + \cdots + P_{n}} \cdot \beta(t_{p_{11}}, \ldots, t_{p_{n,m_{n}}}) \cdot \delta\left( \beta(t_{k_{1}}, \ldots, t_{k_{n}}) \right).
  \]
Now this last term is equal to $\beta \left( \delta(t_{k_{1}}), \ldots, \delta(t_{k_{n}}) \right)$ by Axiom \eqref{eq9}, which is then equal to
  \[
    \beta \left( t_{P_{1}}\cdot \beta(t_{p_{11}}, \ldots, t_{p_{1,m_{1}}}), \ldots,  t_{P_{n}}\cdot \beta(t_{p_{n1}}, \ldots, t_{p_{1,m_{n}}}) \right).
  \]
Taken all together, the left hand side of Axiom \eqref{eq9} is then
  \[
    t_{P_{1} + \cdots + P_{n}} \cdot \beta(t_{p_{11}}, \ldots, t_{p_{n,m_{n}}}) \cdot \beta \left( t_{P_{1}}\cdot \beta(\underline{t_{p_{1}}}), \ldots,  t_{P_{n}}\cdot \beta(\underline{t_{p_{n}}}) \right).
  \]
where $\underline{t_{p_{i}}} = t_{p_{i,1}}, \ldots, t_{i,m_{i}}$
All of the terms coming from an $t_{p_{ij}}$ can be collected together, and since $s_{p,q}^{2} = e$ for all $p,q$, these cancel. This leaves
  \[
    t_{P_{1} + \cdots + P_{n}} \cdot \beta \left( t_{P_{1}}, \ldots,  t_{P_{n}} \right)
  \]
which is the right hand side of Axiom \eqref{eq9} as desired.
\end{proof}

\begin{lem}
The $2$-monad $C$ for strict coboundary categories is a club.
\end{lem}
\begin{proof}
This is obvious by \cref{thm:pres1}.
\end{proof}

\begin{thm}
The free coboundary category on one element, $C1$, is isomorphic to $BJ = \coprod BJ_{n}$.
\end{thm}
\begin{proof}
We must give $BJ$ the structure of a strict coboundary category and then prove that, for any strict coboundary category $X$, there is a natural isomorphism between strict coboundary functors $F \colon BJ \to X$ and objects of $X$. We note here that a strict coboundary functor is a strict monoidal functor mapping the commutor of its source to the commutor of its target.

The category $BJ$ has natural numbers as objects, and addition as its tensor product. The tensor product of two morphisms is given by $\beta$ as in \cref{thm:J_aop}, and it is simple to check that this is a strict monoidal structure. The commutor $\sigma_{m,n}$ is defined to be the product $s_{1, m+n}s_{1,m}s_{m+1,m+n}$. Using the relations in $J_{n}$, it is clear that $\sigma_{m,n}\sigma_{n,m}$ is the identity, so we only have one more axiom to verify in order to give a coboundary structure. By definition, this axiom is equivalent to the equation
  \[
    \sigma_{m, p+n}\cdot \beta(e_{m}, \sigma_{n,p}) = \sigma_{n+m,p}\cdot \beta(\sigma_{m,n},e_{p})
  \]
holding for all $m,n,p$. Each side has six terms when written out using the definitions of $\sigma$ and $\beta$, two terms on each side cancel using $s_{p,q}^{2} = e$ and the disjointness relation, and the other four terms match after using the disjointness relation. This establishes the coboundary structure on $BJ$; note that $\sigma_{1,1} = s_{1,2}$, the nontrivial element of $J(2)$.

Every strict coboundary functor $F \colon BJ \rightarrow X$ determines an object of $X$ by evaluation at $1$. Conversely, given an object $x$ of a strict coboundary category $X$, there is a group homomorphism of $J_{n} \to X(x^{n},x^{n})$ by Theorem 7 of \cite{hk-cobound}. The proof in \cite{hk-cobound} shows that these group homomorphisms are compatible with the homomorphisms $\beta \colon J_n \times J_m \to J_{n+m}$, and so define a strict monoidal functor $\overline{x} \colon BJ \rightarrow X$ with $\overline{x}(1) = x$. By construction, this strict monoidal functor is in fact a strict coboundary functor since it sends the commutor $\sigma_{1,1}$ in $BJ$ to $\sigma_{x,x}$ in $X$. In fact, the calculations in \cite{hk-cobound} leading up to Theorem 7 show that every element of $J_{n}$ is given as an operadic composition of $\sigma$'s, so requiring $\overline{x}$ to be a strict coboundary functor with $\overline{x}(1) = x$ determines the rest of the functor uniquely. This observation establishes the bijection between strict coboundary functors $F \colon BJ \rightarrow X$ and objects of $X$. Naturality is immediate from the construction, so $BJ$ is the free strict coboundary category on one object.
\end{proof}

\begin{cor}\label{cor:J=coboundary}
The $2$-monad $C$ for coboundary categories corresponds, using  \cref{thm:club=operad}, to the action operad $J$.
\end{cor}

\begin{rem}[(Comparision of presentations)]
As with the symmetric groups, we have two different presentations: presentations for each individual group given separately but in a uniform fashion, and a single presentation for the entire action operad.
The calculations in \cref{ex:sigma-pres} unify those two presentations for $\Sigma$, and those in \cref{thm:J_aop} and \cite{hk-cobound} combine via \cref{cor:J=coboundary} and \cref{thm:club=operad} to do the same for $J$.
\end{rem}

\section{Pseudo-commutativity}\label{sec:pscomm}

This section gives conditions sufficient to equip the $2$-monad $\underline{P}$ induced by a $\Lambda$-operad $P$ in $\mb{Cat}$ with a pseudo-commutative structure in the sense of \cite{HP}. Such a pseudo-commutativity will then give the $2$-category $\mb{Ps}\mbox{-}\underline{P}\mbox{-}\mb{Alg}$ a closed monoidal structure, as well as construct a two-dimensional analogue of a multicategory for which $\mb{Ps}\mbox{-}\underline{P}\mbox{-}\mb{Alg}$ is the underlying 2-category.
The 1-dimensional version of this theory is that of commutative monads, as developed by Kock \cite{kock-closed, kock-monads, kock-strong}.

\begin{Defi}\label{Defi:strengths}
A \textit{left strength} for an endo-$2$-functor $T \colon \m{K} \rightarrow \m{K}$ on a $2$-category with products and terminal object $1$ consists of a $2$-natural transformation $d$ with components
    \[
        d_{A,B} \colon A \times TB \rightarrow T(A \times B)
    \]
satisfying the following unit and associativity axioms \cite{kock-monads}.
  \[
    \xy
    (0,0)*+{1 \times TA}="ul1";
    (30,0)*+{T(1 \times A)}="ur1";
    (30,-13)*+{TA}="br1";
    (50,0)*+{A \times B}="ul2";
    (80,0)*+{A \times TB}="ur2";
    (80,-13)*+{T(A \times B)}="br2";
    {\ar^{d_{1,A}} "ul1"; "ur1"};
    {\ar^{\cong} "ur1"; "br1"};
    {\ar_{\cong} "ul1"; "br1"};
    {\ar^{1 \times \eta} "ul2"; "ur2"};
    {\ar^{d_{A,B}} "ur2"; "br2"};
    {\ar_{\eta} "ul2"; "br2"};
    \endxy
  \]
  \[
    \xy
    (0,0)*+{(A \times B) \times TC}="ul";
    (70,0)*+{T \left((A \times B) \times C \right)}="ur";
    (0,-15)*+{A \times (B \times TC)}="ll";
    (35,-15)*+{A \times T(B \times C)}="m";
    (70,-15)*+{ T \left(A \times (B \times C) \right)}="lr";
    {\ar^{d_{AB,C}} "ul"; "ur"};
    {\ar^{Ta} "ur"; "lr"};
    {\ar_{a} "ul"; "ll"};
    {\ar_{1 \times d_{B,C}} "ll"; "m"};
    {\ar_{d_{A,BC}} "m"; "lr"};
    \endxy
  \]
  \[
    \xy
    (0,0)*+{A \times T^{2}B}="ul";
    (60,0)*+{T^{2}(A \times B)}="ur";
    (0,-15)*+{A \times TB}="ll";
    (30,0)*+{T(A \times TB)}="m";
    (60,-15)*+{ T(A \times B)}="lr";
    {\ar^{d_{A,TB}} "ul"; "m"};
    {\ar^{Td_{A,B}} "m"; "ur"};
    {\ar^{\mu} "ur"; "lr"};
    {\ar_{1 \times \mu} "ul"; "ll"};
    {\ar_{d_{A,B}} "ll"; "lr"};
    \endxy
  \]
Similarly, a \emph{right strength} for $T$ consists of a $2$-natural transformation $d^{\ast}$ with components
  \[
      d^{\ast}_{A,B} \colon TA \times B \rightarrow T(A \times B)
  \]
again satisfying unit and associativity axioms.
\end{Defi}
The strengths for the associated $2$-monad $\underline{P}$ are quite simple to define. We define the left strength $d$ for $\underline{P}$ as follows. The component $d_{A,B}$ is a functor
    \[
        d_{A,B} \colon A \times \left(\amalg P(n) \times_{\Lambda(n)} B^n\right) \rightarrow \amalg P(n) \times_{\Lambda(n)} \left(A \times B \right)^n
    \]
which sends an object $(a, [p;b_1,\ldots,b_n])$ to the object $[p;(a,b_1),\ldots,(a,b_n)]$. We also define the right strength similarly, sending an object $([p;a_1,\ldots,a_n],b)$ to the object which is an equivalence class $[p;(a_1,b), \ldots, (a_n, b)]$. Both the left and the right strengths are defined in the obvious way on morphisms.

\begin{remark}[(Change of terminology: costrengths)]\label{rem:strength}
In a minor change of terminology, what we refer to as a \emph{right} strength in \cref{Defi:strengths} is in \cite[Section 3.1]{HP} simply a strength, while our \emph{left} strength corresponds to a costrength. We stress this difference to match more contemporary usage \cite[Definition 3.3]{mu-strong}, avoiding the confusion that a prefix of `co-' generally means a reversal of directions.
\end{remark}

\begin{rem}[(Strengths arise non-equivariantly)]\label{rem:strength-nonsym}
It is crucial to note that the left strength $d$ and the right strength $d^{*}$ do not depend on the $\Lambda$-actions in the following sense. The $\Lambda$-operad $P$ has an underlying non-symmetric operad that we also denote $P$, and it has a left strength
  \[
    d_{A,B} \colon A \times \left(\amalg P(n) \times B^n\right) \rightarrow \amalg P(n) \times \left(A \times B \right)^n
  \]
given by essentially the same formula:
  \[
    \left( a; (p; b_{1}, \ldots, b_{n}) \right) \mapsto \left(p; (a,b_{1}), \ldots, (a, b_{n})\right).
  \]
The left strength for the $\Lambda$-equivariant $P$ is just the induced functor between coequalizers.
\end{rem}

\begin{Defi}[(Pseudo-commutative structure)]\label{Defi:pscommute}
    Given a $2$-monad $T \colon \m{K} \rightarrow \m{K}$ with left strength $d$ and right strength $d^{\ast}$, a \textit{pseudo-commutative structure} consists of an invertible modification $\gamma$ with components
      \[
        \xy
            (0,0)*+{TA \times TB}="00";
            (30,0)*+{T(A \times TB)}="10";
            (60,0)*+{T^2(A \times B)}="20";
            (0,-15)*+{T(TA \times B)}="01";
            (30,-15)*+{T^2(A \times B)}="11";
            (60,-15)*+{T(A \times B)}="21";
            {\ar^{d^{\ast}_{A,TB}} "00" ; "10"};
            {\ar^{Td_{A,B}} "10" ; "20"};
            {\ar^{\mu_{A \times B}} "20" ; "21"};
            {\ar_{d_{TA,B}} "00" ; "01"};
            {\ar_{Td^{\ast}_{A,B}} "01" ; "11"};
            {\ar_{\mu_{A \times B}} "11" ; "21"};
            {\ar@{=>}^{\gamma_{A,B}} (30,-5.5) ; (30,-9.5)};
        \endxy
      \]
satisfying the following three strength axioms, two unit (or $\eta$) axioms, and two multiplication (or $\mu$) axioms for all $A$, $B$, and $C$.
    \begin{enumerate}
        \item\label{axiom:ps_comm_str_1} $\gamma_{A \times B,C} * (d_{A,B} \times 1_{TC}) = d_{A,B \times C} * (1_A \times \gamma_{B,C})$.
        \item\label{axiom:ps_comm_str_2} $\gamma_{A,B \times C} * (1_{TA} \times d_{B,C}) = \gamma_{A \times B, C} * (d^{\ast}_{A,B} \times 1_{TC})$.
        \item\label{axiom:ps_comm_str_3} $\gamma_{A,B \times C} * (1_{TA} \times d^{\ast}_{B,C}) = d^{\ast}_{A \times B,C} * (\gamma_{A,B} \times 1_{C})$.
        \item\label{axiom:ps_comm_unit_1} $\gamma_{A,B} * (\eta_A \times 1_{TB})$  is the identity on $d$.
        \item\label{axiom:ps_comm_unit_2} $\gamma_{A,B} * (1_{TA} \times \eta_B)$ is the identity on $d^{*}$.
        \item\label{axiom:ps_comm_mult_1} $\gamma_{A,B} * (\mu_A \times 1_{TB})$ is equal to the pasting below.
          \[
            \xy
                (0,0)*+{\scriptstyle T^2A \times TB}="00";
                (30,0)*+{\scriptstyle T(TA \times TB)}="10";
                (60,0)*+{\scriptstyle T^2(A \times TB)}="20";
                (90,0)*+{\scriptstyle T^3(A \times B)}="30";
                (0,-15)*+{\scriptstyle T(T^2A \times B)}="01";
                (30,-15)*+{\scriptstyle T^2(TA \times B)}="11";
                (60,-15)*+{\scriptstyle T^3(A \times B)}="21";
                (90,-15)*+{\scriptstyle T^2(A \times B)}="31";
                (0,-30)*+{\scriptstyle T^2(TA \times B)}="02";
                (30,-30)*+{\scriptstyle T(TA \times B)}="12";
                (60,-30)*+{\scriptstyle T^2(A \times B)}="22";
                (90,-30)*+{\scriptstyle T(A \times B)}="32";
                {\ar^{d^{\ast}_{TA,TB}} "00" ; "10"};
                {\ar^{Td^{\ast}_{A,TB}} "10" ; "20"};
                {\ar^{T^2 d_{A,B}} "20" ; "30"};
                {\ar_{d_{T^2A,B}} "00" ; "01"};
                {\ar_{Td_{TA,B}} "10" ; "11"};
                {\ar^{T\mu_{A \times B}} "30" ; "31"};
                {\ar_{T^2 d^{\ast}_{A,B}} "11" ; "21"};
                {\ar_{T\mu_{A \times B}} "21" ; "31"};
                {\ar_{Td^{\ast}_{TA,B}} "01" ; "02"};
                {\ar_{\mu_{TA \times B}} "11" ; "12"};
                {\ar_{\mu_{T(A \times B)}} "21" ; "22"};
                {\ar^{\mu_{A \times B}} "31" ; "32"};
                {\ar_{\mu_{TA \times B}} "02" ; "12"};
                {\ar_{Td^{\ast}_{A,B}} "12" ; "22"};
                {\ar_{\mu_{A \times B}} "22" ; "32"};
                {\ar@{=>}^{T\gamma_{A,B}} (60,-5.5) ; (60,-9.5)};
                {\ar@{=>}^{\gamma_{TA,B}} (12.5,-13) ; (12.5,-17)};
            \endxy
          \]
        \item\label{axiom:ps_comm_mult_2} $\gamma_{A,B} * (1_{TA} \times \mu_B)$ is equal to the pasting below.
          \[
            \xy
                (0,0)*+{\scriptstyle TA \times T^2B}="00";
                (30,0)*+{\scriptstyle T(A \times T^2B)}="10";
                (60,0)*+{\scriptstyle T^2(A \times TB)}="20";
                (0,-15)*+{\scriptstyle T(TA \times TB)}="01";
                (30,-15)*+{\scriptstyle T^2(A \times TB)}="11";
                (60,-15)*+{\scriptstyle T(A \times TB)}="21";
                (0,-30)*+{\scriptstyle T^2(TA \times B)}="02";
                (30,-30)*+{\scriptstyle T^3(A \times B)}="12";
                (60,-30)*+{\scriptstyle T^2(A \times B)}="22";
                (0,-45)*+{\scriptstyle T^3(A \times B)}="03";
                (30,-45)*+{\scriptstyle T^2(A \times B)}="13";
                (60,-45)*+{\scriptstyle T(A \times B)}="23";
                {\ar^{d^{\ast}_{A,T^2B}} "00" ; "10"};
                {\ar^{Td_{A,TB}} "10" ; "20"};
                {\ar_{d_{TA,TB}} "00" ; "01"};
                {\ar^{\mu_{A \times TB}} "20" ; "21"};
                {\ar_{Td^{\ast}_{A,TB}} "01" ; "11"};
                {\ar_{\mu_{A \times TB}} "11" ; "21"};
                {\ar_{Td_{TA,B}} "01" ; "02"};
                {\ar^{T^2 d_{A,B}} "11" ; "12"};
                {\ar^{Td_{A,B}} "21" ; "22"};
                {\ar^{\mu_{T(A \times B)}} "12" ; "22"};
                {\ar_{T^2 d^{\ast}_{A,B}} "02" ; "03"};
                {\ar^{T\mu_{A \times B}} "12" ; "13"};
                {\ar^{\mu_{A \times B}} "22" ; "23"};
                {\ar_{T\mu_{A \times B}} "03" ; "13"};
                {\ar_{\mu_{A \times B}} "13" ; "23"};
                {\ar@{=>}^{T\gamma_{A,B}} (13,-28) ; (13,-32)};
                {\ar@{=>}^{\gamma_{A,TB}} (30,-5.5) ; (30,-9.5)};
            \endxy
          \]
    \end{enumerate}
\end{Defi}

\begin{rem}[(Redundant axioms)]
    It is noted in \cite[Proposition 1]{HP} that this definition has some redundancy and therein it is claimed that any two of the strength axioms (Axioms 1 - 3) immediately implies the third. Furthermore, the three strength axioms are equivalent when the $\eta$ and $\mu$ axioms hold (Axioms 4-6) as well as the following associativity axiom:
        \[
            \gamma_{A,B \times C} \circ (1_{TA} \times \gamma_{B,C}) = \gamma_{A \times B,C} \times (\gamma_{A,B} \times 1_{TC}).
        \]
\end{rem}

We need some further notation before stating our main theorem. 

\begin{nota}[(Lexicographic and colexicographic orderings)]\label{nota:double-underlines}
Suppose we are given two finite ordered lists, $\underline{a} = a_{1}, \ldots , a_{m}$ and $\underline{b} = b_{1}, \ldots, b_{n}$. We use the following notation for the lexicographic and colexicographic orderings on the set $\underline{a} \times \underline{b} = \{ (a_{i}, b_{j})\}$. 
\begin{enumerate}
\item The \emph{lexicographic ordering} is denoted $\underline{(a, \underline{b})}$, and has the order given by
  \[
    (a_{p}, b_{q}) < (a_{r}, b_{s}) \textrm{ if } \left\{ \begin{array}{l} p < r, \textrm{ or } \\ p=r \textrm{ and } q < s. \end{array} \right.
  \]
\item The \emph{colexicographic ordering} is denoted $\underline{(\underline{a}, b)}$, and has the order given by
\[
    (a_{p}, b_{q}) < (a_{r}, b_{s}) \textrm{ if } \left\{ \begin{array}{l} q < s, \textrm{ or } \\ q=s \textrm{ and } p < r. \end{array} \right.
  \]
\end{enumerate}
\end{nota}

\begin{rem}[(Intuition for underlining convention)]
The notation $(a, \underline{b})$ is meant to indicate that there is a single $a$ but a list of $b$'s, so then $\underline{(a, \underline{b})}$ would represent a list which itself consists of lists of that form. 
\end{rem}

\begin{nota}[(Constant lists)]\label{nota:constant-list}
When $x$ is a single element, and $n$ is a given natural number, we write $\underline{x}$ for the list $x, x, \ldots, x$ of length $n$.
\end{nota}

\begin{Defi}[(The transposition permutation, $\tau$)]\label{Defi:tau}
Let $\underline{a} = a_{1}, \ldots , a_{m}$ and $\underline{b} = b_{1}, \ldots, b_{n}$ be two ordered finite lists. 
The permutation $\tau_{m,n} \in \Sigma_{mn}$ is defined uniquely by the property that $\tau_{m,n}(i) = j$ if the $i$th element of the ordered set $\underline{(a, \underline{b})}$ is equal to the $j$th element of the ordered set $\underline{(\underline{a}, b)}$.
\end{Defi}

We illustrate these permutations with a couple of examples.
    \[
        \xy
            {\ar@{-} (0,0) ; (0,-10)};
            {\ar@{-} (5,0) ; (10,-10)};
            {\ar@{-} (10,0) ; (20,-10)};
            {\ar@{-} (15,0) ; (5,-10)};
            {\ar@{-} (20,0) ; (15,-10)};
            {\ar@{-} (25,0) ; (25,-10)};
            (12.5,-13)*{\tau_{2,3}};
            {\ar@{-} (45,0) ; (45,-10)};
            {\ar@{-} (50,0) ; (65,-10)};
            {\ar@{-} (55,0) ; (50,-10)};
            {\ar@{-} (60,0) ; (70,-10)};
            {\ar@{-} (65,0) ; (55,-10)};
            {\ar@{-} (70,0) ; (75,-10)};
            {\ar@{-} (75,0) ; (60,-10)};
            {\ar@{-} (80,0) ; (80,-10)};
            (62.5,-13)*{\tau_{4,2}}
        \endxy
    \]

\begin{rem}
We make the following elementary remarks about the transposition permutations $\tau_{m,n}$.
\begin{itemize}
\item By construction, we have $\tau_{m,n} = \tau_{n,m}^{-1}$.  
\item We call these transposition permutations as $\tau_{m,n}$ is the permutation given by taking the transpose of the $m \times n$ matrix with entries $(a_{i}, b_{j})$, where the entries are ordered lexicographically. In other words, the permutation $\tau_{m,n}$ has the effect of rearranging $m$ groups of $n$ things into $n$ groups of $m$ things.
\item The transposition permutations $\tau_{m,n}$ satisfy an additional naturality-type relation. For $\alpha \in \Sigma_n$ and $\beta \in \Sigma_m$, we have the equality
\[
\mu(\alpha; \underline{\beta}) \tau_{m,n} = \tau_{m,n} \mu(\beta; \underline{\alpha}).
\]
This equation is 3.9 in \cite{guillou_multiplicative}.
\end{itemize}
\end{rem}

\begin{nota}
Let $\mathbb{N}_{+}$ denote the set of positive integers.
\end{nota}

\begin{Defi}[(Pseudo-commutative structure for operads)]\label{Defi:ps-comm_operad}
Let $P$ be a $\Lambda$-operad in $\mb{Cat}$. A \emph{pseudo-commutative structure} on $P$ consists of the following data.
    \begin{itemize}
        \item For each pair $(m,n) \in \mathbb{N}_{+}^2$, an element $t_{m,n} \in \Lambda(mn)$ such that  $\pi(t_{m,n}) = \tau_{m,n}$.     
        \item For each object $p \in P(n)$ and object $q \in P(m)$, a natural isomorphism
            \[
                \lambda_{p,q} \colon \mu(p;q,\ldots,q) \cdot t_{m,n} \cong \mu(q;p,\ldots,p).
            \]
            Naturality of $\lambda_{p,q}$ means that for all $f \colon p \to p'$ in $P(n)$ and $g \colon q \to q'$ in $P(m)$, the following square commutes.
             \[
    \xy
      (0,0)*+{\mu(p;q,\ldots,q) \cdot t_{m,n}} ="00";
      (0,-10)*+{\mu(p';q',\ldots,q') \cdot t_{m,n}} ="01";
      (40,0)*+{\mu(q;p,\ldots,p)} ="10";
      (40,-10)*+{\mu(q';p',\ldots,p')} ="11";
      {\ar^{\lambda_{p,q}} "00" ; "10"};
      {\ar^{\mu(g; f, \ldots, f)} "10" ; "11"};
      {\ar_{\mu(f; g, \ldots, g) \cdot t_{m,n}} "00" ; "01"};
      {\ar_{\lambda_{p',q'}} "01" ; "11"};
    \endxy
  \]
Using \cref{nota:constant-list}, we write this as $\lambda_{p,q}\colon \mu(p; \underline{q}) \cdot t_{m,n} \cong \mu(q; \underline{p})$.
    \end{itemize}
These data are required to satisfy the following axioms.  
    \begin{enumerate}
        \item\label{axiom:t_id} For all $m,n \in \mathbb{N}_+$,
            \[
                t_{1,n} = e_n = t_{m,1}.
            \]
             For all $p \in P(n)$, the isomorphism $\lambda_{p, \id}\colon p \cdot t_{1,n} \cong p$ is the identity map.
             For all $q \in P(m)$, the isomorphism $\lambda_{\id, q}\colon q \cdot t_{m,1} \cong q$ is the identity map.
             \item\label{axiom:t_nat} For all $g \in \Lambda(n)$ and $h \in \Lambda(m)$, the equality
        \[
        \mu^{\Lambda}(g; \underline{h}) t_{m,n} = t_{m,n} \mu^{\Lambda}(h; \underline{g})
        \]
        holds in $\Lambda(mn)$.
        \item\label{axiom:t_equiv} For all $p \in P(n)$, $g \in \Lambda(n)$, $q \in P(m)$, and $h \in \Lambda(m)$, the equality of morphisms
        \[
        \lambda_{p,q} \cdot \mu^{\Lambda}(h; \underline{g}) = \lambda_{p\cdot g, q \cdot h}
        \]
        holds. The source of the left morphism is $\mu(p; \underline{q}) \cdot t_{m,n} \cdot \mu^{\Lambda}(h; \underline{g})$ and the source of the right morphism is $\mu(p \cdot g; \underline{q \cdot h}) \cdot t_{m,n}$, and these are equal by Axiom \eqref{axiom:t_nat} and the $\Lambda$-operad axioms; the target of the left morphism is $\mu(q; \underline{p}) \cdot \mu^{\Lambda}(h; \underline{g})$ and the target of the right morphism is $\mu(q \cdot h; \underline{p \cdot g})$, and these are equal by the $\Lambda$-operad axioms.
        \item\label{axiom:t_sumR} For all $l, m_1, \ldots, m_l, n \in \mathbb{N}_+$, with $M = \sum m_i$,
            \[
              \beta(t_{n,m_1},\ldots,t_{n,m_l}) \cdot \delta_{M,\ldots,M}(t_{n,l}) = t_{n,M}.
            \]
        \item\label{axiom:t_sumL} For all $l, m, n_1,\ldots, n_m \in \mathbb{N}_+$, with $N = \sum n_i$,
            \[
              \delta_{\underline{n_1},\ldots,\underline{n_m}}(t_{m,l}) \cdot \beta(t_{n_1,l},\ldots,t_{n_m,l}) = t_{N,l}.
            \]
            Here $\underline{n_{i}}$ indicates that each subscript $n_{i}$ is repeated $l$ times.
        \item\label{axiom:t_diagR} For any $l, m_i, n \in \mathbb{N}_+$, with $1 \leq i \leq n$, and $p \in P(l)$, $q_i \in P(m_i)$ and $r \in P(n)$, the following diagram commutes. (Note that we maintain the convention that anything underlined indicates a list, and double underlining indicates a list of lists. Each instance should have an obvious meaning from context and the equations appearing above.)
          \[
            \xy
                (0,0)*+{\mu\left(p;\underline{\mu(q_i;\underline{r})}\right) \cdot \mu(e_l;\underline{t_{n,m_i}})\mu(t_{n,l};\underline{\underline{e_{m_i}}})}="00";
                (60,0)*+{\mu\left(p;\underline{\mu(q_i;\underline{r})}\right) \cdot t_{n,M}}="10";
                (0,-15)*+{\mu\left(p;\underline{\mu(q_i;\underline{r})\cdot t_{n,m_i}}\right) \cdot \mu(t_{n,l};\underline{e_{m_1},\ldots,e_{m_l}})}="01";
                (60,-20)*+{\mu\left(\mu(p;q_1,\ldots,q_n);\underline{\underline{r}}\right)\cdot t_{n,M}}="11";
                (0,-30)*+{\mu\left(p;\underline{\mu(r;\underline{q_i})}\right) \cdot \mu(t_{n,l};\underline{e_{m_1},\ldots,e_{m_l}})}="02";
                (60,-40)*+{\mu\left(\mu(p;q_1,\ldots,q_n);\underline{\underline{r}}\right)}="12";
                (0,-45)*+{\mu\left(\mu(p;\underline{r}) \cdot t_{n,l} ; \underline{q_1,\ldots,q_n}\right)}="03";
                (60,-60)*+{\mu\left(r;\underline{\mu(p;q_1,\ldots,q_n)}\right)}="13";
                (0,-60)*+{\mu\left(\mu(r;\underline{p});\underline{q_1,\ldots,q_n}\right)}="04";
                {\ar@{=} "00" ; "10"};
                {\ar@{=} "00" ; "01"};
                {\ar@{=} "10" ; "11"};
                {\ar_{\mu(1;\underline{\lambda_{q_i,r}}) \cdot 1} "01" ; "02"};
                {\ar@{=} "02" ; "03"};
                {\ar@{=} "04" ; "13"};
                {\ar_{\mu(\lambda_{p,r};1)} "03" ; "04"};
                {\ar^{\lambda_{\mu(p;q_1,\ldots,q_n),r}} "11" ; "12"};
                {\ar@{=} "12" ; "13"};
            \endxy
          \]
        \item\label{axiom:t_diagL} For any $l,m, n_i \in \mathbb{N}_+$, with $1 \leq i \leq m$, and $p \in P(l)$, $q \in P(m)$ and $r_i \in P(n_i)$, the following diagram commutes.
          \[
            \xy
                (0,0)*+{\mu\left(\mu(p;\underline{q}) \cdot t_{m,l} ; \underline{\underline{r_i}}\right) \cdot \mu(e_m;\underline{t_{n_i,l}})}="00";
                (60,0)*+{\mu\left(\mu(p;\underline{q});\underline{\underline{r_i}}\right) \cdot \mu(t_{m,l};\underline{\underline{e_{n_i}}})\mu(e_{m};\underline{t_{n_i,l}})}="10";
                (60,-15)*+{\mu\left(p;\underline{\mu(q;\underline{r_i})}\right) \cdot \mu(t_{m,l};\underline{\underline{e_{n_i}}})\mu(e_{m};\underline{t_{n_i,l}})}="11";
                (0,-20)*+{\mu\left(\mu(q;\underline{p}); \underline{r_1},\ldots,\underline{r_m}\right) \cdot \mu(e_m;\underline{t_{n_i,l}})}="01";
                (0,-40)*+{\mu\left(q;\underline{\underline{\mu(p;r_i)}}\right) \cdot \mu(e_m;\underline{t_{n_i,l}})}="02";
                (0,-60)*+{\mu\left(q;\underline{\mu(p;\underline{r_i}) \cdot t_{n_i,l}}\right)}="03";
                (60,-30)*+{\mu\left(p;\underline{\mu(q;r_1,\ldots,r_m)}\right) \cdot t_{N,l}}="12";
                (60,-45)*+{\mu\left(\mu(q;r_1,\ldots,r_m); \underline{\underline{p}}\right)}="13";
                (60,-60)*+{\mu\left(q;\underline{\mu(r_i;\underline{p})}\right)}="14";
                {\ar@{=} "00" ; "10"};
                {\ar@{=} "10" ; "11"};
                {\ar@{=} "11" ; "12"};
                {\ar^{\lambda_{p,\mu(q;r_1,\ldots,r_m)}} "12" ; "13"};
                {\ar@{=} "13" ; "14"};
                {\ar_{\mu(\lambda_{p,q};1) \cdot 1} "00" ; "01"};
                {\ar@{=} "01" ; "02"};
                {\ar@{=} "02" ; "03"};
                {\ar_{\mu(1;\underline{\lambda_{p,r_i}})} "03" ; "14"};
            \endxy
          \]
    \end{enumerate}
\end{Defi}

\begin{rem}[(Updated axioms)]\label{rem:updated}
Axioms \eqref{axiom:t_nat} and \eqref{axiom:t_equiv} were absent from the original definition we gave in the preprint \cite{cg-preprint}. The need for Axiom \eqref{axiom:t_nat} was noted in \cite{guillou_symmetric} and appears as \cite[Axiom (iii) of 11.1]{guillou_multiplicative}. Since the authors of \cite{guillou_multiplicative} only worked with contractible operads, they did not include Axiom \eqref{axiom:t_equiv}.
\end{rem}

\begin{thm}\label{thm:pscomm}
Let $P$ be a $\Lambda$-operad in $\mb{Cat}$ equipped with a pseudo-commutative structure. Then the 2-monad $\underline{P}$ has a pseudo-commutative structure.
\end{thm}
\begin{proof}
We refer to the Axioms in \cref{Defi:ps-comm_operad} throughout. We begin the proof by defining an invertible modification $\gamma$ for the pseudo-commutativity for which the components are natural transformations $\gamma_{A,B}$. Let $[p;a_1,\ldots,a_n]$ be an object of $\coeq{P}{A}{\Lambda}{n}$ and $[q;b_1,\ldots,b_m]$ be an object of $\coeq{P}{B}{\Lambda}{m}$.
The required transformation $\gamma_{A,B}$ has a component at the pair $\big([p;a_1,\ldots,a_n], [q;b_1,\ldots,b_m] \big)$ with source
  \[
    \left[\mu\left(p; \underline{q}\right); \underline{(a, \underline{b})}\right]
  \]
and target
  \[
    \left[\mu\left(q; \underline{p}\right); \underline{(\underline{a},b)}\right].
  \]
Now $ \lambda_{p,q} \colon \mu(p;q,\ldots,q) \cdot t_{m,n} \cong \mu(q;p,\ldots,p)$ gives rise to another map by multiplication on the right by $t_{m,n}^{-1}$,
  \[
    \lambda_{p,q}\cdot t_{m,n}^{-1} \colon \mu(p;q,\ldots,q) \cong \mu(q;p,\ldots,p) \cdot t_{m,n}^{-1},
  \]
so we define $(\gamma_{A,B})_{[p;a_1,\ldots,a_n],[q;b_1,\ldots,b_m]}$ to be the morphism which is the image of $(\lambda_{p,q}\cdot t_{m,n}^{-1}, 1)$ under the map
  \[
   P(nm) \times (A \times B)^{nm} \rightarrow  \coeq{P}{(A \times B)}{\Lambda}{nm}.
  \]
  We will write this morphism as $[\lambda_{p,q}t_{m,n}^{-1}, 1]$.
  
We must first verify that $[\lambda_{p,q}t_{m,n}^{-1}, 1]$ is well-defined. Let $g \in \Lambda(n)$ and $h \in \Lambda(m)$, and consider the objects $[p \cdot g; \underline{a}] = [p; g \cdot \underline{a}]$ in $\coeq{P}{A}{\Lambda}{n}$ and $[q \cdot h; \underline{b}] = [q; h \cdot \underline{b}]$ in $\coeq{P}{B}{\Lambda}{m}$ (see \cref{rem:obj-coeq}). We will verify that the morphisms
\[
[\lambda_{p \cdot g, q \cdot h} \cdot t_{m,n}^{-1}, 1] \colon   \left[\mu\left(p \cdot g; \underline{q \cdot h}\right); \underline{(a, \underline{b})}\right] \to  \left[\mu\left(q \cdot h; \underline{p \cdot g}\right); \underline{(\underline{a},b)}\right]
\]
and
\[
[\lambda_{p, q} \cdot t_{m,n}^{-1}, 1] \colon   \left[\mu\left(p; \underline{q}\right); g \cdot \underline{(a, h \cdot \underline{ b})}\right] \to  \left[\mu\left(q; \underline{p}\right); h \cdot \underline{(g \cdot \underline{a},b)}\right]
\]
are equal.
It is a straightforward calculation to show that these have the same source and target, using Axiom \eqref{axiom:t_nat}. By Axiom \eqref{axiom:t_equiv} followed by Axiom \eqref{axiom:t_nat}, we obtain
\begin{align*}
\lambda_{p \cdot g, q \cdot h} \cdot t_{m,n}^{-1} & = \lambda_{p,q} \cdot \mu^{\Lambda}(h; \underline{g}) \cdot t_{m,n}^{-1} \\
& = \lambda_{p,q} \cdot t_{m,n}^{-1} \cdot \mu^{\Lambda}(g; \underline{h}).
\end{align*}
Thus we conclude that the morphism $[\lambda_{p \cdot g, q \cdot h} \cdot t_{m,n}^{-1}, 1]$ above is equal to $
[\lambda_{p,q} \cdot t_{m,n}^{-1} \cdot \mu^{\Lambda}(g; \underline{h}), 1]$, and using the equality $[f \cdot \sigma, g] = [f, \sigma \cdot g]$ in $\coeq{P}{(A \times B)}{\Lambda}{nm}$ it is therefore equal to
\begin{align*}
\left[\mu\left(p; \underline{q}\right); \mu^{\Lambda}(g; \underline{h}) \cdot \underline{(a, \underline{b})}\right] & \stackrel{[\lambda_{p,q} \cdot t_{m,n}^{-1}, 1]}{\longrightarrow}  \left[\mu\left(q; \underline{p}\right) \cdot t_{m,n}^{-1}; \mu^{\Lambda}(g; \underline{h})\cdot \underline{(\underline{a},b)}\right] \\
& = \left[\mu\left(q; \underline{p}\right); t_{m,n}^{-1} \cdot \mu^{\Lambda}(g; \underline{h})\cdot \underline{(\underline{a},b)}\right] \\
& =\left[\mu\left(q; \underline{p}\right); \mu^{\Lambda}(h; \underline{g})\cdot \underline{(a,\underline{b})}\right].
\end{align*}
The reader can verify that the sources and targets in this calculation match those of $[\lambda_{p, q} \cdot t_{m,n}^{-1}, 1]$, proving the desired equality. Thus the components of $\gamma_{A,B}$ are well-defined.
Naturality of the components of $\gamma_{A,B}$ in the objects $[p;a_1,\ldots,a_n],[q;b_1,\ldots,b_m]$ follows from that of each $\lambda_{p,q}$.  

We show that this is a modification by noting that it does not rely on objects in the lists $a_1, \ldots, a_n$ or $b_1, \ldots, b_m$, only on their lengths and the operations $p$ and $q$. As a result, if there are functors $f \colon A \rightarrow A'$ and $g \colon B \rightarrow B'$, then it is clear that
    \[
        (\underline{P}(f\times g) \circ \gamma_{A,B})_{\left[p;\underline{a}\right],\left[q;\underline{b}\right]} = [\lambda_{p,q},\underline{1}] = (\gamma_{A',B'} \circ (\underline{P}f\times \underline{P}g))_{\left[p;\underline{a}\right],\left[q;\underline{b}\right]}.
    \]
As such we can simply write $(\gamma_{A,B})_{[p;\underline{a}],[q;\underline{b}]}$ in shorthand as $\gamma_{p,q}$.

There are now seven axioms to check for a pseudo-commutativity: three strength axioms, two unit axioms, and two multiplication axioms. For the first strength axiom, we must verify that two different $2$-cells of shape
  \[
    \xy
      (0,0)*+{A \times TB \times TC}="0";
      (50,0)*+{T(A \times B \times C)}="1";
      {\ar@/^1pc/ "0"; "1"};
      {\ar@/_1pc/ "0"; "1"};
      (25,0)*{\Downarrow}
    \endxy
  \]
are equal. The first of these is $\gamma$ precomposed with $d \times 1$, and so is the component of $\gamma$ at an object
  \[
    \left( [p;(a,b_1),\ldots,(a,b_n)], [q; c_{1}, \ldots, c_{m}] \right).
  \]
The second of these is $d$ applied to the component of $1 \times \gamma$ at
  \[
    \left(a, ([p;b_1,\ldots,b_n], [q; c_{1}, \ldots, c_{m}]) \right).
  \]
It is straightforward to compute that each of these maps is the image of $\left(\lambda_{p,q}\cdot t_{m,n}^{-1},1\right)$ under the functor
  \[
    \coprod P(n) \times (A \times B)^{n} \rightarrow \coprod \coeq{P}{A \times B}{\Lambda}{n}.
  \]
The other two strength axioms follow by analogous calculations for other whiskerings of $\gamma$ with $d$ or $d^{*}$.

For the unit axioms, we must compute the components of $\gamma$ precomposed with $\eta \times 1$ for the first axiom and $1 \times \eta$ for the second. Thus for the first unit axiom, we must compute the component of $\gamma$ at $\left( [\id;a], [q; b_{1}, \ldots, b_{m}] \right)$. By definition, this is the image of $(\lambda_{\id,q}\cdot t^{-1}_{m,1}, 1)$, and by Axiom \eqref{axiom:t_id} of \cref{Defi:ps-comm_operad} this is the identity. The second unit axiom follows similarly, using that $\lambda_{p, \id}$ and $t^{-1}_{1,n}$ are both the identity.

For the multiplication axioms, first note that Axiom \eqref{axiom:t_sumR} is necessary in order to ensure the existence of the top horizontal equality in the diagram of Axiom \eqref{axiom:t_diagR} for the pseudo-commutative structure; the same goes for Axioms \eqref{axiom:t_sumL} and \eqref{axiom:t_diagL}. We now explain how Axioms \eqref{axiom:t_sumR} and \eqref{axiom:t_diagR} for the pseudo-commutative structure ensure that the first multiplication axiom holds, with the same reasoning showing that Axioms \eqref{axiom:t_sumL} and \eqref{axiom:t_diagL} imply the second multiplication axiom.

We begin by studying the pasting diagram in the first multiplication axiom, but computing its values using the strengths for the non-symmetric operad underlying $P$; this means that we evaluate on objects of the form $(p; a_{1}, \ldots, a_{n})$ rather than on their equivalence classes. Let $p \in P(l), q_{i} \in P(m_{i})$ for $1 \leq i \leq l$, and $r \in P(n)$. Computing the top and right leg around the pasting diagram gives the function on objects which sends
  \[
    \left( (p; (q_{1}; \un{a_{1}}), \ldots, (q_{l}; \un{a_{l}})), (r; \un{b}) \right)
  \]
to
  \[
    \left( \mu(p; \mu(q_{1}; \un{r}), \ldots, \mu(q_{l}; \un{r})); (\un{(a_{1\bullet}, \un{b})}), \ldots, (\un{(a_{l\bullet}, \un{b})}) \right),
  \]
where $(\un{(a_{i\bullet}, \un{b})})$ is the list of pairs
  \[
    (a_{i1}, b_{1}), \ldots, (a_{i1}, b_{m}), (a_{i2}, b_{1}), \ldots, (a_{in_{i}}, b_{m}).
  \]
Then $\un{P}\gamma$ is the image of the morphism which is the identity on the $(a_{ij}, b_{k})$'s, and is the morphism
  \[
    \mu\left(1;\lambda_{q_1,r}t^{-1}_{n,m_1},\ldots,\lambda_{q_l,r}t^{-1}_{n,m_l}\right)
  \]
on the first component with domain and codomain shown below.
  \[
    \mu\left(p;\mu\left(q_1;\un{r}\right),\ldots,\mu\left(q_n;\un{r}\right)\right) \longrightarrow \mu\left(p;\mu\left(r;\un{q_1}\right)t^{-1}_{n,m_1},\ldots,\mu\left(r;\un{q_l}\right)t^{-1}_{n,m_l}\right)
  \]
By the $\Lambda$-operad axioms, the target of this morphism is equal to
  \[
    \mu\left(p; \mu\left(r; \un{q_{1}}\right), \ldots, \mu\left(r; \un{q_{l}}\right) \right)\mu\left(e_{l}; t^{-1}_{n,m_{1}}, \ldots, t^{-1}_{n,m_{l}}\right).
  \]
Note that this is not the same object as one obtains by computing $T\mu \circ T^{2}d^{*} \circ Td \circ d^{*}$ using the underlying non-symmetric operad of $P$ as we are required to use the $\Lambda$-equivariance to ensure that the target of $\gamma$ is the correct one.

Next we compute the source of $(\mu \circ Td^{*})*\gamma$, the other $2$-cell in the pasting appearing in the first multiplication axiom. We compute this once again using the strengths for the underlying non-symmetric operad, and note once again that this will not match our previous calculations precisely, but only up to an application of $\Lambda$-equivariance. This functor has its map on objects given by
  \[
    \left( (p; (q_{1}; \un{a_{1}}), \ldots, (q_{l}; \un{a_{l}})), (r; \un{b}) \right) \mapsto \left(\mu(\mu(p; \un{r}); \un{q_{1}}, \ldots, \un{q_{l}}); \un{(\un{a_{1}}, b_{\bullet})}, \ldots, \un{(\un{a_{l}}, b_{\bullet})} \right).
  \]
  Note that if we apply $\Lambda$-equivariance, this matches the target computed above. Once again the component of $\gamma$ is the image of a morphism which is the identity on the $(a_{ij}, b_{k})$'s, and its first component is
  \[
    \xy
      {\ar^{\mu\left(\lambda_{p,r} \cdot t^{-1}_{n,l}; 1, \ldots, 1\right)} (0,0)*+{\mu\left(\mu(p; \un{r}); \un{q_{1}}, \ldots, \un{q_{l}}\right)}; (60,0)*+{\mu\left(\mu(r; \un{p})\cdot t^{-1}_{n,l}; \un{q_{1}}, \ldots, \un{q_{l}}\right).} }
    \endxy
  \]

We cannot compose these morphisms in $\coprod P(n) \times (A \times B)^{n}$ as they do not have matching source and target, but we can in $\coprod P(n) \otimes_{\Lambda} (A \times B)^{n}$. The resulting morphism has first component given by the image of
  \[
    \xy
      {\ar^{\scriptstyle \mu\left(1; \lambda_{q_{1}, r} t^{-1}_{n,m_{1}}, \ldots, \lambda_{q_{1}, r} t^{-1}_{n,m_{l}}\right)} (0,0)*+{\scriptstyle \mu\left(p; \mu\left(q_{1}; \un{r}\right), \ldots, \mu\left(q_{n}; \un{r}\right)\right)}; (75,0)*+{\scriptstyle \mu\left(p; \mu\left(r; \un{q_{1}}\right) t^{-1}_{n,m_{1}}, \ldots, \mu\left(r; \un{q_{l}}\right) t^{-1}_{n,m_{l}} \right)} };
      {\ar^<<<<<<<<<<<<<<<<<<<<<<{\scriptstyle \mu\left(\lambda_{p,r} \cdot t^{-1}_{n,l}; 1, \ldots, 1\right)\cdot \mu\left(e_{l}; t^{-1}_{n,m_{1}}, \ldots, t^{-1}_{n,m_{l}}\right)} (0,-10)*+{}; (75,-10)*+{\scriptstyle \mu\left(\mu\left(r; \un{p}\right)\cdot t^{-1}_{n,l}; \un{q_{1}}, \ldots, \un{q_{l}}\right)\cdot \mu\left(e_{l}; t^{-1}_{n,m_{1}}, \ldots, t^{-1}_{n,m_{l}}\right),} }
    \endxy
  \]
where we have made use of the operad axioms in identifying the target of the first map with the source of the second. Using the $\Lambda$-operad axioms again on the target, we find that
  \[
    \mu\left(\mu(r; \un{p})\cdot t^{-1}_{n,l}; \un{q_{1}}, \ldots, \un{q_{l}}\right)\cdot \mu(e_{l}; t^{-1}_{n,m_{1}}, \ldots, t^{-1}_{n,m_{l}})
  \]
is equal to
  \[
    \mu\left(\mu(r; \un{p}); \un{q_{1}, \ldots, q_{l}}\right) \cdot \mu(t^{-1}_{n,l}; \un{e}) \cdot \mu(e_{l}; t^{-1}_{n,m_{1}}, \ldots, t^{-1}_{n,m_{l}}).
  \]
This composite of two morphisms, together with the necessary identities coming from operad axioms, is precisely the left and bottom leg of the diagram in Axiom \eqref{axiom:t_diagR}. Using the same method, one then verifies that $\gamma * (\mu \times 1)$ has its first component the image of the morphism appearing along the top and right leg of the diagram in Axiom \eqref{axiom:t_diagR}. The second component of these morphisms are all identities arising from $\Lambda$-equivariance, so the first multiplication axiom is a consequence of Axioms \eqref{axiom:t_sumR} and \eqref{axiom:t_diagR} for the pseudo-commutative structure. We leave the calculations for the second multiplication axiom to the reader as they are of the same nature, using Axioms \eqref{axiom:t_sumL} and \eqref{axiom:t_diagL}.
\end{proof}

\begin{cor}\label{cor:not-pc}
Let $P$ be a non-symmetric operad, ie, a $\Lambda$-operad over the terminal action operad $T$. Then the induced monad $\underline{P}$ is never pseudo-commutative.
\end{cor}
\begin{proof}
In the non-symmetric case, the $2$-monad is just given using coproducts and products, i.e., there is no coequalizer. In order to define $\gamma$, we then need an isomorphism
  \[
    \left(\mu(p; \underline{q}); \underline{(a, \underline{b})}\right) \cong \left(\mu(q; \underline{p}); \underline{(\underline{a},b)}\right).
  \]
When $A,B$ are discrete, there is no isomorphism $\underline{\left(a,\underline{b}\right)} \cong \underline{\left(\underline{a},b\right)}$, and therefore no such $\gamma$ can exist.
\end{proof}

Hyland and Power also define a symmetry for a pseudo-commutative structure on a $2$-monad $T$. This symmetry is then reflected in the monoidal structure on the $2$-category of algebras, which will then also have a symmetric tensor product (in a suitable, $2$-categorical sense) \cite[Theorem 13]{HP}.

\begin{Defi}
Let $T \colon \m{K} \rightarrow \m{K}$ be a $2$-monad on a symmetric monoidal $2$-category $\m{K}$ with symmetry $c$. We then say that a pseudo-commutativity $\gamma$ for $T$ is \textit{symmetric} when the following is satisfied for all $A$, $B \in \m{K}$:
    \[
        Tc_{A,B} \circ \gamma_{A,B} \circ c_{TB, TA} = \gamma_{B,A}.
    \]
\end{Defi}

With the notion of symmetry at hand we are able to extend the above theorem, stating when $\underline{P}$ is symmetric.
\begin{thm}
The pseudo-commutativity of $\underline{P}$ given by \cref{thm:pscomm}  is symmetric if for all $m,n \in \mathbb{N}_+$ the two conditions below hold.
    \begin{enumerate}
        \item $t_{m,n} = t_{n,m}^{-1}$.
        \item The following diagram commutes:
          \[
              \xy
                (0,0)*+{\mu\left(p;\underline{q}\right) \cdot t_{m,n}t_{n,m}}="00";
                (30,0)*+{\mu\left(p;\underline{q}\right) \cdot e_{mn}}="10";
                (0,-15)*+{\mu\left(q;\underline{p}\right) \cdot t_{n,m}}="01";
                (30,-15)*+{\mu\left(p;\underline{q}\right)}="11";
                {\ar@{=} "00" ; "10"};
                {\ar_{\lambda_{p,q} \cdot 1} "00" ; "01"};
                {\ar@{=} "10" ; "11"};
                {\ar_{\lambda_{q,p}} "01" ; "11"};
              \endxy
          \]
    \end{enumerate}
\end{thm}
\begin{proof}
The commutativity of the diagram above ensures that the first component of the symmetry axiom commutes in $P(n)$ before taking equivalence classes in the coequalizer, just as in the proof of \cref{thm:pscomm}.
\end{proof}

\begin{Defi}
Let $P$ be a $\Lambda$-operad in $\mb{Cat}$. We say that $P$ is \textit{contractible} if each category $P(n)$ is equivalent to the terminal category.
\end{Defi}

\begin{cor}\label{cor:contract-to-psc}
If $P$ is contractible and there exist $t_{m,n}$ as in \cref{Defi:ps-comm_operad}, then $P$ acquires a pseudo-commutative structure. Furthermore, it is symmetric if $t_{n,m} = t_{m,n}^{-1}$.
\end{cor}
\begin{proof}
The only thing left to define is the collection of natural isomorphisms $\lambda_{p,q}$. But since each $P(n)$ is contractible, $\lambda_{p,q}$ must be the unique isomorphism between its source and target, and furthermore the last two axioms hold since any pair of parallel arrows are equal in a contractible category.
\end{proof}

\begin{cor}\label{cor:contractplussym-to-psc}
If $P$ is a contractible symmetric operad then the operad $P$ has a unique pseudo-commutative structure. The 2-monad $\underline{P}$ then obtains a symmetric pseudo-commutativity.
\end{cor}
\begin{proof}
The only possible choice of the elements $t_{m,n}$ is $t_{m,n} = \tau_{m,n}$.
\end{proof}

\begin{rem}[(Symmetrization and contractibility)]\label{rem:symm-and-contract}
If a $\Lambda$-operad $P$ is contractible, it is not the case that its symmetrization $\pi_{!}P$ (see \cref{Defi:symmetrization}) will also be contractible. For example, consider the braid operad $B$ and the corresponding $B$-operad $EB$ in $\mb{Cat}$. Then $\pi_{!}EB(2)$ has two objects $[0,\id], [0,\sigma]$ corresponding to the two elements of the symmetric group $\Sigma_2$ by considering the quotient $(\mathbb{Z} \times \Sigma_2)/\mathbb{Z}$ as in \cref{lem:coeq-lem}. The object $[0,\id]$ has as its automorphism group the subgroup $PB_2 = \textrm{ker}(\pi) \leq B_2$ of pure braids as follows. The group $B_2$ is isomorphic to the integers, so $EB_2$ has an object for every integer and a unique isomorphism $k \cong j$ for every pair $k, j \in \mathbb{Z}$. In particular, for every $k \in \mathbb{Z}$ there is a unique isomorphism $0 \cong 2k$. Using the description in \cref{lem:coeq-lem}, we can verify that $[0 \cong 2k, 1_{\id}]$ and $[0 \cong 2j, 1_{\id}]$ are not in the same orbit unless $k=j$, so give distinct isomorphisms $[0, \id] \cong [2k, \id] = [0, \id]$, $[0, \id] \cong [2j, \id] = [0, \id]$.
Thus we see that a given $\Lambda$-operad $P$ might satisfy the hypotheses of \cref{cor:contract-to-psc} without its symmetrization $S(P)$ satisfying the hypotheses of \cref{cor:contractplussym-to-psc}.
\end{rem}

\begin{rem}\label{rem:EB-fail}
An earlier version of this article \cite{cg-preprint} constructed a pseudo-commutative structure for $EB$ as a $B$-operad, but contained an error and has been removed.
\end{rem}